\documentclass{article}
\usepackage[english]{babel}
\usepackage{amsmath,amssymb,stmaryrd,bbm,latexsym,theorem}
\usepackage[top=2.5cm,bottom=2.5cm,right=2cm,left=2cm]{geometry}
\usepackage{authblk}

\numberwithin{equation}{section}

\catcode`\<=\active \def<{
\fontencoding{T1}\selectfont\symbol{60}\fontencoding{\encodingdefault}}
\catcode`\>=\active \def>{
\fontencoding{T1}\selectfont\symbol{62}\fontencoding{\encodingdefault}}
\newcommand{\assign}{:=}
\newcommand{\comma}{{,}}
\newcommand{\mathi}{\mathrm{i}}
\newcommand{\nobracket}{}
\newcommand{\nocomma}{}
\newcommand{\nosymbol}{}
\newcommand{\tmem}[1]{{\em #1\/}}
\newcommand{\tmmathbf}[1]{\ensuremath{\boldsymbol{#1}}}
\newcommand{\tmop}[1]{\ensuremath{\operatorname{#1}}}
\newcommand{\tmrsub}[1]{\ensuremath{_{\textrm{#1}}}}
\newcommand{\tmtextbf}[1]{{\bfseries{#1}}}
\newcommand{\tmtextit}[1]{{\itshape{#1}}}
\newenvironment{itemizeminus}{\begin{itemize} }{\end{itemize}}
\newenvironment{proof}{\noindent\textbf{Proof\ }}{\hspace*{\fill}$\Box$\medskip}
\newenvironment{tmindent}{\begin{tmparmod}{1.5em}{0pt}{0pt} }{\end{tmparmod}}
\newenvironment{tmparmod}[3]{\begin{list}{}{\setlength{\topsep}{0pt}\setlength{\leftmargin}{#1}\setlength{\rightmargin}{#2}\setlength{\parindent}{#3}\setlength{\listparindent}{\parindent}\setlength{\itemindent}{\parindent}\setlength{\parsep}{\parskip}} \item[]}{\end{list}}
\newtheorem{theorem}{Theorem}[section]
\newtheorem{corollary}[theorem]{Corollary}
\newtheorem{lemma}[theorem]{Lemma}
\newtheorem{proposition}[theorem]{Proposition}
{\theorembodyfont{\rmfamily}\newtheorem{remark}[theorem]{Remark}}

\usepackage[hidelinks]{hyperref}

\date{}
\author{David Chiron}
\author{Eliot Pacherie}
\affil{Universit{\'e} C\^{o}te d'Azur, CNRS, LJAD, France}

\begin{document}

\title{ Smooth branch of travelling waves for the Gross-Pitaevskii equation in
$\mathbbm{R}^2$ for small speed}

\maketitle

\begin{abstract}
  We construct a smooth branch of travelling wave solutions for the 2
  dimensional Gross-Pitaevskii equations for small speed. These travelling
  waves exhibit two vortices far away from each other. We also compute the
  leading order term of the derivatives with respect to the speed. We
  construct these solutions by an implicit function type argument.
\end{abstract}

\section{Introduction and statement of the result}

We consider the Gross-Pitaevskii equation
\[ 0 = (\tmop{GP}) (\mathfrak{u}) \assign i \partial_t \mathfrak{u}+ \Delta
   \mathfrak{u}- (| \mathfrak{u} |^2 - 1) \mathfrak{u} \]
in dimension 2 for $\mathfrak{u}: \mathbbm{R}_t \times \mathbbm{R}_x^2 \mapsto
\mathbbm{C}$. The Gross-Pitaevskii equation is a physical model for
Bose-Einstein condensates {\cite{Ginz-Pit_58a}}, {\cite{Neu_90}}, and is
associated with the Ginzburg-Landau energy
\[ E (v) \assign \frac{1}{2} \int_{\mathbbm{R}^2} | \nabla v |^2 + \frac{1}{4}
   \int_{\mathbbm{R}^2} (1 - | v |^2)^2 . \]
The condition at infinity for $(\tmop{GP})$ will be
\[ | \mathfrak{u} | \rightarrow 1 \quad \tmop{as} \quad | x | \rightarrow +
   \infty . \]
We look for travelling wave solutions of $(\tmop{GP})$:
\[ \mathfrak{u} (t, x) = v (x_1, x_2 + c t) \]
where $x = (x_1, x_2)$ and $c > 0$ is the speed of the travelling wave, which
moves along the direction $- \overrightarrow{e_2}$. The equation on $v$ is
\[ 0 = (\tmop{TW}_c) (v) \assign - i c \partial_{x_2} v - \Delta v - (1 - | v
   |^2) v. \]

We want to construct travelling waves for small speed that look like the
product of two well-separated vortices. Vortices are stationary solutions of
$(\tmop{GP})$ of degrees $n \in \mathbbm{Z}^{\ast}$ (see {\cite{CX}} and
{\cite{HH}}):
\[ V_n (x) = \rho_n (r) e^{i n \theta}, \]
where $x = r e^{i \theta}$, solving
\[ \left\{ \begin{array}{l}
     \Delta V_n - (| V_n |^2 - 1) V_n = 0\\
     | V_n | \rightarrow 1 \quad \tmop{as} \quad | x | \rightarrow \infty .
   \end{array} \right. \]
A vortex alone in the plane is a stationnary solution of $(\tmop{GP})$, and
vortices might interact when there are several of them. It is expected that if
they are far away from each other, their dynamic is governed, at least at
first order, by the point vortex system (see {\cite{PV}} and references
therein). In particular, a vortex $V_1$ of degree 1 with a vortex $V_{- 1}$ of
degree $- 1$ should move at constant speed in the direction orthogonal to the
line that connects their centers.

\

The main result of this paper is the construction of a branch of solution by
perturbation of the product of two vortices at any small speed $c > 0$, and
the fact that this branch of solution is $C^1$ with respect to the speed.

\begin{theorem}
  \label{th1}There exists $c_0 > 0$ a small constant such that, for any $0 < c
  \leqslant c_0$, there exists a solution of $(\tmop{TW}_c)$ of the form
  \[ Q_c = V_1 (. - d_c \overrightarrow{e_1}) V_{- 1} (. + d_c
     \overrightarrow{e_1}) + \Gamma_{c, d_c}, \]
  where $d_c = \frac{1 + o_{c \rightarrow 0} (1)}{c}$ is a continuous function
  of c. This solution has finite energy $(E (Q_c) < + \infty)$ and $Q_c
  \rightarrow 1$ when $| x | \rightarrow + \infty$.
  
  Furthermore, for all $+ \infty \geqslant p > 2$, there exists $c_0 (p) > 0$
  such that if $c < c_0 (p)$, for the norm
  \[ \| h \|_{X_p} \assign \| h \|_{L^p (\mathbbm{R}^2)} + \| \nabla h
     \|_{L^{p - 1} (\mathbbm{R}^2)} \]
  and the space $X_p \assign \{ f \in L^p (\mathbbm{R}^2), \nabla f \in L^{p -
  1} (\mathbbm{R}^2) \}$, one has
  \[ \| \Gamma_{c, d_c} \|_{X_p} = o_{c \rightarrow 0} (1) . \]
  In addition,
  \[ c \mapsto Q_c - 1 \in C^1 (] 0, c_0 (p) [, X_p), \]
  with the estimate (for $\nu (c) = \frac{1 + o_{c \rightarrow 0} (1)}{c^2}$)
  \[ \| \partial_c Q_c + \nu (c) \partial_d (V_1 (. - d \overrightarrow{e_1})
     V_{- 1} (. + d \overrightarrow{e_1}))_{| d = d_c \nobracket} \|_p = o_{c
     \rightarrow 0} \left( \frac{1}{c^2} \right) . \]
\end{theorem}

Existence of travelling waves solutions for this equation with finite energy
has already been proven for small speeds in {\cite{MR1669387}} (see also
{\cite{MR2461988}} and {\cite{MR3686002}} for results in dimension 2 and 3).
Moreover, the decay at infinity conjectured in {\cite{0305-4470-15-8-036}} has
been established in {\cite{G1}}. Here, we use an implicit function argument to
construct the solution, using techniques developed in {\cite{DP2}} or
{\cite{LW}} for instance, displaying a clear understanding of the shape of the
solution (see Lemma \ref{Qcfest} for instance). We show in addition that the
constructed branch is $C^1$, which is, to the best of our knowledge, the first
result of this kind in dimension larger than one.

In the Gross-Pitaevskii equation, vortices play the role of solitons (as we
can see in NLS or other such equations). In particular here we show that two
vortices interact at long range since the speed $c$ is of order
$\frac{1}{d_c}$, the half distance between the vortices. This is due to the
slow decay of the vortex: $\nabla V_1$ is of order $\frac{1}{r}$ at infinity
due to the phase.

The formal method for this kind of construction is well known. It has been
done rigorously in a bounded domain for the Ginzburg Landau equation with no
speed ({\cite{DP2}}). One of the difficulties here is to find the right
functional setting to construct the $C^1$ branch, in particular with regards
to the transport term $i c \partial_{x_2} v$. On the contrary of what is claim
in {\cite{LW}}, the transport term can not be treated perturbatively. This is
why we use another functional setting than {\cite{LW}} or {\cite{LW2B}} (see
Remark \ref{normsetup2} for more details)

In this paper, we start by doing the construction of the solution to fill
these gaps in the case of two vortices in $(\tmop{GP})$. This construction is
also a good introduction to the proof of the differentiability of the branch,
which uses many of the same ideas, but with a more technical setting.

\

We start by reducing the problem to a one dimensional one in section
\ref{sec2}. The construction of the travelling wave $Q_c$ is completed in
section \ref{n3}. Furthermore, in subsection \ref{th1proof}, we show that
$Q_c$ has finite energy and we compute some estimates particular to the branch
of solutions. Finally section \ref{qcc1} is devoted to the proof of the
differentiability of the branch.

\

We use the scalar product for $f, g \in L^2 (\mathbbm{R}^2)$,
\[ \langle f, g \rangle \assign \mathfrak{R}\mathfrak{e} \int_{\mathbbm{R}^2}
   f \bar{g}, \]
For $X = (X_1, X_2), Y = (Y_1, Y_2) \in \mathbbm{C}^2$, we define
\[ X.Y \assign X_1 Y_1 + X_2 Y_2, \]
which is the scalar product if $X, Y \in \mathbbm{R}^2$. We use the notation
$B (x, r)$ to define the closed ball in $\mathbbm{R}^2$ of center $x \in
\mathbbm{R}^2$ and radius $r > 0$ for the Euclidean norm. In the estimates, a
constant $K > 0$ is a universal constant independent of any parameter of the
problem.

{\noindent}\tmtextbf{Acknowledgments . }The authors would like to thank Pierre
Rapha{\"e}l for helpful discussions. E.P. is supported by the ERC-2014-CoG
646650 SingWave.{\hspace*{\fill}}{\medskip}

\section{Lyapunov-Schmidt reduction}\label{sec2}

The proof of Theorem \ref{th1} follows closely the construction done in
{\cite{DP2}} or {\cite{LW}}. The main idea is to use perturbation methods on
an approximate solution.

In subsection \ref{V} we define this approximate solution $V$ which consists
in two vortices at distance $2 d$ from each other. We then look for a solution
of $(\tmop{TW}_c)$ as a perturbation of $V$, with an additive perturbation
close to the vortices and a multiplicative one far from them. This is computed
in subsection \ref{perturbation}. We define suitable spaces in subsection
\ref{normsetup} that we will use to invert the linear part and use a
contraction argument. We ask for an orthogonality on the perturbation, and the
norms are a little better but more technical than the ones in Theorem
\ref{th1}. In particular $\Gamma_{c, d_c}$ in Theorem \ref{th1} verifies
better estimates which are discussed for instance in Corollary \ref{sigmarem}
and in Lemma \ref{Qcfest}. We invert the linearized operator in Proposition
\ref{invertop} and show that the perturbation is a fixed point of a
contracting functional in Proposition \ref{contractionest}. The orthogonality
condition create a Lagragian multiplier (see subsection \ref{orthocond}),
which left us with a problem in one dimension. This multiplier will be
cancelled for a good choice of the parameter $d$ in section \ref{n3}.

\subsection{Estimates on vortices}\label{V}

From {\cite{HH}}, we can find stationary solution of $(\tmop{GP})$:
\[ V_n (x) = \rho_n (r) e^{i n \theta} \]
where $x = r e^{\mathi \theta}, n \in \mathbbm{Z}^{\ast}$, solving
\[ \left\{ \begin{array}{l}
     \Delta V_n - (| V_n |^2 - 1) V_n = 0\\
     | V_n | \rightarrow 1 \quad \tmop{as} \quad | x | \rightarrow \infty .
   \end{array} \right. \]
These solutions are well understood and, in particular, we have some estimates
(see {\cite{HH}} for instance) that we will use. We also know the kernel of
the linearized operator around $V_{\pm 1}$ ({\cite{DP}}), which we will need
for inverting the linearized operator around the approximate solution $V$
defined using these vortices
\[ V (x) \assign V_1 (x - d \vec{e_1}) V_{- 1} (x + d \overrightarrow{e_1}) \]
where $d > 0$, $x = (x_1, x_2)$. The function $V$ is the product of two
vortices with opposite degrees at a distance $2 d$ from each other. One vortex
alone in $\mathbbm{R}^2$ is a stationary solution, and it is expected that two
vortices interact and translate at a constant speed of order $c \simeq
\frac{1}{d}$, see {\cite{PV}}. Hence for the two parameters of this problem
$c, d > 0$, we let them be free from each other, but with the condition $c$ is
of order $1 / d$ by imposing that $\frac{1}{2 c} < d < \frac{2}{c}$.

\

We will study in particular areas near the center of each vortices. We will
use coordinates adapted to this problem:
\begin{eqnarray}
  x & = & (x_1, x_2) = r e^{i \theta}, \nonumber\\
  y & = & (y_1, y_2) \assign x - d \overrightarrow{e_1} = r_1 e^{i \theta_1},
  \nonumber\\
  z & = & (z_1, z_2) \assign y + 2 d \overrightarrow{e_1} = x + d
  \overrightarrow{e_1} = r_{- 1} e^{i \theta_{- 1}}, \nonumber\\
  \tilde{r} & \assign & \min (r_1, r_{- 1}) . \label{notation} 
\end{eqnarray}
Using $y$ coordinate mean that we are centered around $V_1$, and $z$
coordinate for around $V_{- 1}$. Note that we have
\[ V (x) = V_1 (y) V_{- 1} (z) \]
using these notations. If it is not precised, $V$ will be taken in $x$, $V_1$
in $y$ and $V_{- 1}$ in $z$. If we compute $(\tmop{TW}_c)$ for $V$, i.e. $- i
c \partial_{x_2} V - \Delta V - (1 - | V |^2) V$, we get
\[ (\tmop{TW}_c) (V) = E - i c \partial_{x_2} V, \]
where we defined
\[ E \assign - \Delta V - (1 - | V |^2) V. \]
We have $V = V_1 V_{- 1}$ and, by using $- \Delta V_{\varepsilon} = (1 - |
V_{\varepsilon} |^2) V_{\varepsilon}$ for $\varepsilon = \pm 1$, we compute
\[ E = - 2 \nabla V_1 . \nabla V_{- 1} + V_1 V_{- 1} (1 - | V_1 |^2 + 1 - |
   V_{- 1} |^2 - 1 + | V_1 V_{- 1} |^2) . \]
Hence
\begin{equation}
  E = - 2 \nabla V_1 . \nabla V_{- 1} + (1 - | V_1 |^2) (1 - | V_{- 1} |^2)
  V_1 V_{- 1} . \label{E2}
\end{equation}
The rest of this subsection is devoted to the computation of estimates on $V,
E, \partial_d V$ and $i c \partial_{x_2} V$ using estimates on $V_1$ and $V_{-
1}$. Let us start with the properties on $V_{\pm 1}$ we need.

\begin{lemma}[{\cite{HH}}]
  \label{lemme3} $V_1 (x) = \rho_1 (r) e^{i \theta}$ verifies $V_1 (0) = 0$,
  and there exists a constant $\kappa > 0$ such that, for all $r > 0$, $0 <
  \rho_1 (r) < 1$, $\rho_1' (r) > 0$, and
  \[ \rho_1 (r) \sim_{r \rightarrow 0} \kappa r, \]
  \[ \rho_1' (r) = O_{r \rightarrow \infty} \left( \frac{1}{r^3} \right), \]
  \[ \rho_1'' (r) = o_{r \rightarrow \infty} \left( \frac{1}{r^3} \right), \]
  \[ 1 - | V_1 (x) | = \frac{1}{2 r^2} + O_{r \rightarrow \infty} \left(
     \frac{1}{r^3} \right), \]
  \[ \nabla V_1 (x) = i V_1 (x) \frac{x^{\bot}}{r^2} + O_{r \rightarrow
     \infty} \left( \frac{1}{r^3} \right) \]
  where $x^{\perp} = (- x_2, x_1)$, $x = r e^{i \theta}$. Furthermore we have
  similar properties for $V_{- 1}$ since
  \[ V_{- 1} (x) = \overline{V_1 (x)_{}} . \]
\end{lemma}

We will use the $O$ notation for convergence independent of any other
quantity. Now let us write all the derivatives of a vortex in polar
coordinate, which will be useful all along the proof of the results.

\begin{lemma}
  \label{dervor}We define $u \assign \frac{\rho_1' (r_1)}{\rho_1 (r_1)}$.
  Then,
  \[ \partial_{x_1} V_1 (y) = \left( \cos (\theta_1) u - \frac{i}{r_1} \sin
     (\theta_1) \right) V_1, \]
  \[ \partial_{x_2} V_1 (y) = \left( \sin (\theta_1) u + \frac{i}{r_1} \cos
     (\theta_1) \right) V_1, \]
  \[ \partial_{x_1 x_1} V_1 (y) = \left( \cos^2 (\theta_1) (u^2 + u') + \sin^2
     (\theta_1) \left( \frac{u}{r_1} - \frac{1}{r_1^2} \right) + 2 i \sin
     (\theta_1) \cos (\theta_1) \left( \frac{1}{r_1^2} - \frac{u}{r_1} \right)
     \right) V_1, \]
  \[ \partial_{x_1 x_2} V_1 (y) = \left( \sin (\theta_1) \cos^{\nosymbol}
     (\theta_1) \left( u^2 + u' + \frac{1}{r_1^2} - \frac{u}{r_1} \right) - i
     \cos (2 \theta_1) \left( \frac{1}{r_1^2} - \frac{u}{r_1} \right) \right)
     V_1 . \]
\end{lemma}

We obtain the derivatives of $V_{- 1}$ by changing $i \rightarrow - i$, $y
\rightarrow z$, $\theta_1 \rightarrow \theta_{- 1}$, $r_1 \rightarrow r_{- 1}$
and $V_1 \rightarrow V_{- 1}$. We remark in particular that the first
derivatives are of first order $\frac{1}{r_1}$ and the second derivatives are
of first order $\frac{1}{r_1^2}$ for large values of $r_1$. From {\cite{HH}},
we can check that, more generally, we have
\begin{equation}
  | D^{(n)} V_1 (y) | \leqslant \frac{K (n)}{(1 + r_1)^n} . \label{230210}
\end{equation}
\begin{proof}
  With the notation of (\ref{notation}) in radial coordinate around $d
  \overrightarrow{e_1}$, the center of $V_1$:
  \[ \partial_{x_1} = \cos (\theta_1) \partial_{r_1} - \frac{\sin
     (\theta_1)}{r_1} \partial_{\theta_1} \]
  \[ \partial_{x_2} = \sin (\theta_1) \partial_{r_1} + \frac{\cos
     (\theta_1)}{r_1} \partial_{\theta_1}, \]
  we compute directly the first two equalities of the lemma. Now, we compute
  \[ \partial_{x_1 x_1} V_1 = \cos (\theta_1) \partial_{r_1} (\partial_{x_1}
     V_1) - \frac{\sin (\theta_1)}{r_1} \partial_{\theta_1} (\partial_{x_1}
     V_1) \]
  with
  \[ \partial_{r_1} (\partial_{x_1} V_1) = \left( u \left( \cos (\theta_1) u -
     \frac{i}{r_1} \sin (\theta_1) \right) + \cos (\theta_1) u' +
     \frac{i}{r_1^2} \sin (\theta) \right) V_1 \]
  and
  \[ \partial_{\theta_1} (\partial_{x_1} V_1) = \left( i \cos (\theta_1) u +
     \frac{1}{r_1} \sin (\theta_1) - \sin (\theta_1) u - \frac{i}{r_1} \cos
     (\theta_1) \right) V_1 \]
  for the third inequality. We use them also in
  \[ \partial_{x_1 x_2} V_1 = \sin (\theta_1) \partial_{r_1} (\partial_{x_1}
     V_1) + \frac{\cos (\theta_1)}{r_1} \partial_{\theta_1} (\partial_{x_1}
     V_1) \]
  for the fourth relation, with $\cos^2 (\theta_1) - \sin^2 (\theta_1) = \cos
  (2 \theta_1)$.
\end{proof}

Now, we compute some basic estimates on $V$.

\begin{lemma}
  \label{nonmodV}There exists a universal constant $K > 0$ and a constant $K
  (d) > 0$ depending only on $d > 1$ such that
  \[ | 1 - V |^2 \leqslant \frac{K (d)}{(1 + r)^2}, \]
  \[ 0 \leqslant 1 - | V |^2 \leqslant \frac{K}{(1 + \tilde{r})^2}, \]
  \[ | \nabla (| V |) | \leqslant \frac{K}{(1 + \tilde{r})^3}, \]
  and we have
  \[ | \nabla V | \leqslant \frac{K}{(1 + \tilde{r})}, \]
  as well as
  \[ | \nabla V | \leqslant \frac{K d}{(1 + \tilde{r})^2}, \]
  where $\tilde{r} = \min (r_1, r_{- 1})$. Furthermore,
  \[ | \nabla^2 V | \leqslant \frac{K}{(1 + \tilde{r})^2} \]
  and
  \[ | \nabla^2 V | \leqslant \frac{K d}{(1 + \tilde{r})^3} . \]
\end{lemma}

\begin{proof}
  For the first inequality, we are at fixed $d$. Since $V = | V_1 V_{- 1} |
  e^{i (\theta_1 - \theta_{- 1})}$ and $\theta_1, \theta_{- 1}$ are angles
  from points separated by $2 d$, we infer
  \[ e^{i (\theta_1 - \theta_{- 1})} = 1 + O^d_{r \rightarrow \infty} \left(
     \frac{1}{r} \right), \]
  and $| V_1 V_{- 1} | = 1 + O^d_{r \rightarrow \infty} \left(
  \frac{1}{r_{\nosymbol}^2} \right)$ from Lemma \ref{lemme3} where $O^d_{r
  \rightarrow \infty} \left( \frac{1}{r} \right)$ is a quantity that decay in
  $\frac{1}{r}$ is at fixed $d$. Therefore,
  \[ | 1 - V |^2 = | 1 - | V_1 V_{- 1} | e^{i (\theta_1 - \theta_{- 1})} |^2 =
     \left| K (d) O_{r \rightarrow \infty} \left( \frac{1}{r} \right)
     \right|^2 \leqslant \frac{K (d)}{(1 + r)^2} . \]
  From Lemma \ref{lemme3}, we compute
  \[ 1 - | V |^2 = 1 - | V_1 |^2 + | V_1^{\nosymbol} |^2 (1 - | V_{- 1} |^2)
     \leqslant K \left( \frac{1}{(1 + r_1)^2} + \frac{1}{(1 + r_{- 1})_{}^2}
     \right) \leqslant \frac{K}{(1 + \tilde{r})^2}, \]
  and
  \[ | \nabla (| V |) | \leqslant | \nabla (| V_1 |) | V_{- 1} | | + | \nabla
     (| V_{- 1} |) | V_1 | | \leqslant K \left( \frac{1}{(1 + r_1)^3} +
     \frac{1}{(1 + r_{- 1})^3} \right) \leqslant \frac{K}{(1 + \tilde{r})^3} .
  \]
  We check that $\nabla V = \nabla V_1 V_{- 1} + \nabla V_{- 1} V_1$, and
  therefore, with Lemma \ref{dervor}, we have
  \[ | \nabla V | \leqslant \frac{K}{(1 + r_1)} + \frac{K}{(1 + r_{- 1})}
     \leqslant \frac{K}{(1 + \tilde{r})} . \]
  Furthermore, by Lemma \ref{lemme3},
  \[ \nabla V_{\pm 1} = \frac{\pm i}{r_{\pm 1}}
     \overrightarrow{e_{}}_{\theta_{\pm 1}} + O_{r_{\pm 1} \rightarrow \infty}
     \left( \frac{1}{r_{\pm 1}^3} \right) . \]
  For $\tilde{r} \geqslant 1$ (the last estimate on $| \nabla V |$ for
  $\tilde{r} \leqslant 1$ is a consequence of $| \nabla V | \leqslant
  \frac{K}{(1 + \tilde{r})}$), since $r_{\pm 1} e^{i \theta_{\pm 1}} = x \mp d
  \vec{e}_1$,
  \begin{eqnarray*}
    \frac{\cos (\theta_1)}{r_1} - \frac{\cos (\theta_{- 1})}{r_{- 1}} & = &
    \frac{x_1 - d}{(x_1 - d)^2 + x_2^2} - \frac{x_1 + d}{(x_1 + d)^2 +
    x_2^2}\\
    & = & \frac{x_1}{r_1^2 r_{- 1}^2} ((x_1 + d)^2 + x_2^2 - ((x_1 - d)^2 +
    x_2^2)) - d \left( \frac{1}{r_1^2} + \frac{1}{r_{- 1}^2} \right)\\
    & = & \frac{d}{r_1^2 r_{- 1}^2} (2 x_1^2 - r_1^2 - r_{- 1}^2),
  \end{eqnarray*}
  therefore
  \[ \left| \frac{\cos (\theta_1)}{r_1} - \frac{\cos (\theta_{- 1})}{r_{- 1}}
     \right| \leqslant \frac{K d}{(1 + \tilde{r})^2} \]
  since $\frac{x_1}{r_1 r_{- 1}} \leqslant \frac{1}{\tilde{r}}$ if $\tilde{r}
  \geqslant 1$. With a similar estimation for $\frac{\sin (\theta_1)}{r_1} -
  \frac{\sin (\theta_{- 1})}{r_{- 1}}$, we infer
  \begin{eqnarray*}
    | \nabla V | & \leqslant & \left| \frac{\vec{e}_{\theta_1}}{r_1} -
    \frac{\vec{e}_{\theta_{- 1}}}{r_{- 1}} \right| + \frac{K}{(1 +
    \tilde{r})^3}\\
    & \leqslant & \frac{K d}{(1 + \tilde{r})^2} + \frac{K}{(1 +
    \tilde{r})^3}\\
    & \leqslant & \frac{K d}{(1 + \tilde{r})^2} .
  \end{eqnarray*}
  Finally, for the second derivatives, we have for $j, k \in \{ 1, 2 \}$
  \[ \partial_{x_j x_k} V = \partial_{x_j x_k} V_1 V_{- 1} + \partial_{x_j}
     V_1 \partial_{x_k} V_{- 1} + \partial_{x_k} V_1 \partial_{x_j} V_{- 1} +
     \partial_{x_j x_k} V_{- 1} V_1, \]
  therefore, with (\ref{230210}),
  \[ | \nabla^2 V | \leqslant \frac{K}{(1 + r_1)^2} + \frac{K}{(1 + r_{- 1})
     (1 + r_1)} + \frac{K}{(1 + r_{- 1})^2} \leqslant \frac{K}{(1 +
     \tilde{r})^2} . \]
  We check that $\frac{x_1}{r_1 r_{- 1}} \leqslant \frac{1}{\tilde{r}}$ and
  $\left| \nabla \left( \frac{1}{r_{\pm 1}} \right) \right| \leqslant
  \frac{K}{r_{\pm 1}^2}$ if $\tilde{r} \geqslant 1$, hence
  \[ \left| \nabla \left( \frac{\cos (\theta_1)}{r_1} - \frac{\cos (\theta_{-
     1})}{r_{- 1}} \right) \right| \leqslant \frac{K d}{(1 + \tilde{r})^3} .
  \]
  With a similar estimation for $\nabla \left( \frac{\sin (\theta_1)}{r_1} -
  \frac{\sin (\theta_{- 1})}{r_{- 1}} \right)$ and Lemma \ref{lemme3}, we
  conclude with
  \[ | \nabla^2 V | \leqslant \left| \nabla \left(
     \frac{\vec{e}_{\theta_1}}{r_1} - \frac{\vec{e}_{\theta_{- 1}}}{r_{- 1}}
     \right) \right| + \frac{K}{(1 + \tilde{r})^3} \leqslant \frac{K d}{(1 +
     \tilde{r})^3} . \]
\end{proof}

Now we look at the convergence of some quantities when we are near the center
of $V_1$ and $d \rightarrow \infty$. When we are close to the center of $V_1$
and $d$ goes to infinity, we expect that the second vortex as no influence.

\begin{lemma}
  \label{infest}As $d \rightarrow \infty$, we have, locally uniformly in
  $\mathbbm{R}^2$,
  \[ V (. + d \vec{e_1}) = V_1 (.) V_{- 1} (. + 2 d \vec{e_1}) \rightarrow V_1
     (.), \]
  \[ E (. + d \vec{e_1}) \rightarrow 0 \]
  and
  \[ \partial_d V (. + d \overrightarrow{e_1}) \rightarrow - \partial_{x_1}
     V_1 (.) . \]
\end{lemma}

\begin{proof}
  In the limit $d \rightarrow \infty$, for $y \in \mathbbm{R}^2$,
  \[ V (y + d \vec{e_1}) = V_1 (y) e^{- i \theta_{- 1}} \left( 1 + O \left(
     \frac{1}{r_{- 1}^2} \right) \right) \]
  by Lemma \ref{lemme3}, hence
  \[ V (.) \rightarrow V_1 (.) \]
  locally uniformly since $\theta_{- 1} \rightarrow 0, r_{- 1} \rightarrow +
  \infty$ when $d \rightarrow \infty$ locally uniformly. On the other hand,
  since $V (x) = V_1 (y) V_{- 1} (y + 2 d \vec{e}_1)$, we have
  \[ (\partial_d V) (y + d \vec{e_1}) = - \partial_{x_1} V_1 (y) V_{- 1} (y +
     2 d \vec{e_1}) + V_1 (y) \partial_{x_1} V_{- 1} (y + 2 d \vec{e_1}) . \]
  Since $\partial_{x_1} V_{- 1} (y + 2 d \vec{e_1}) = \nabla V_{- 1} (y + 2 d
  \vec{e_1}) . \vec{e_1} \rightarrow 0$ locally uniformly as $d \rightarrow
  \infty$, we have
  \[ \partial_d V (.) \rightarrow - \partial_{x_1} V_1 (.) \]
  locally uniformly. Finally, from (\ref{E2}), we have that
  \[ E (x) = - 2 \nabla V_1 (y) . \nabla V_{- 1} (z) + (1 - | V_1 (y) |^2) (1
     - | V_{- 1} (z) |^2) V_1 (y) V_{- 1} (z) \]
  with the notations from (\ref{notation}), therefore, locally uniformly,
  \[ E (. + d \vec{e_1}) \rightarrow 0 \]
  as $\nabla V_{- 1} \rightarrow 0$ and $| V_{- 1} | \rightarrow 1$ locally
  uniformly when $d \rightarrow \infty$.
\end{proof}

We now do a precise computation on the term $i c \partial_{x_2} V$, which
appears in $(\tmop{TW}_c) (V)$.

\begin{lemma}
  \label{fineest}There exists a universal constant $C > 0$ (independent of
  $d$) such that if $r_1, r_{- 1} \geqslant 1$,
  \[ \left| i \frac{\partial_{x_2} V}{V} - 2 d \frac{x_1^2 - d^2 -
     x_2^2}{r_{1^{\nosymbol \nosymbol}}^2 r_{- 1}^2} \right| \leqslant C
     \left( \frac{1}{r_1^3} + \frac{1}{r_{- 1}^3} \right) . \]
\end{lemma}

Remark that this shows that the first order term of $i \frac{\partial_{x_2}
V}{V}$ is real-valued and the dependence on $d$ of this term is explicit.

\begin{proof}
  Recall from Lemma \ref{dervor} that for $\varepsilon = \pm 1$,
  \[ \partial_{x_2} V_{\varepsilon} = \frac{i \varepsilon}{r_{\varepsilon}}
     \cos (\theta_{\varepsilon}) V_{\varepsilon} + O_{r_1 \rightarrow \infty}
     \left( \frac{1}{r_1^3} \right) . \]
  We have
  \[ \frac{\partial_{x_2} V}{V} = \frac{\partial_{x_2} V_1}{V_1} +
     \frac{\partial_{x_2} V_{- 1}}{V_{- 1}} \]
  and
  \[ \cos (\theta_{\varepsilon}) = \frac{x_1 - \varepsilon
     d}{r_{\varepsilon}}, \]
  yielding
  \begin{eqnarray*}
    \frac{\partial_{x_2} V}{V} & = & i \left( \frac{x_1 - d}{r_1^2} -
    \frac{x_1 + d}{r_{- 1}^2} \right)\\
    & = & i \left( x_1 \left( \frac{1}{r_1^2} - \frac{1}{r_{- 1}^2} \right) -
    d \left( \frac{1}{r_1^2} + \frac{1}{r_{- 1}^2} \right) \right) + O_{r_1
    \rightarrow \infty} \left( \frac{1}{r_1^3} \right) + O_{r_{- 1}
    \rightarrow \infty} \left( \frac{1}{r_{- 1}^3} \right) .
  \end{eqnarray*}
  We compute with (\ref{notation}) that
  \[ \frac{1}{r_1^2} - \frac{1}{r_{- 1}^2} = \frac{(x_1 + d)^2 + x_2^2 - (x_1
     - d)^2 - x_2^2}{r_1^2 r_{- 1}^2} = \frac{4 d x_1}{r_1^2 r_{- 1}^2} \]
  and
  \[ \frac{1}{r_1^2} + \frac{1}{r_{- 1}^2} = \frac{(x_1 + d)^2 + x_2^2 + (x_1
     - d)^2 + x_2^2}{r_1^2 r_{- 1}^2} = 2 \frac{x_1^2 + d^2 + x_2^2}{r_1^2
     r_{- 1}^2}, \]
  yielding the estimate.
\end{proof}

Finally, we show an estimate on $\partial_d V = \partial_d (V_1 (x - d
\overrightarrow{e_1}) V_{- 1} (x + d \overrightarrow{e_1})) = - \partial_{x_1}
V_1 V_{- 1} + \partial_{x_1} V_{- 1} V_1$.

\begin{lemma}
  \label{ddVest}There exists a constant $K > 0$ such that
  \[ | \partial_d V | \leqslant \frac{K}{(1 + \tilde{r})}, \]
  \[ | \nabla \partial_d V | \leqslant \frac{K}{(1 + \tilde{r})^2} \]
  and
  \[ | \mathfrak{R}\mathfrak{e} (\bar{V} \partial_d V) | \leqslant \frac{K}{(1
     + \tilde{r})^3} . \]
  Furthermore,
  \[ | \partial_d^2 V | \leqslant \frac{K}{(1 + \tilde{r})^2} \]
  and
  \[ | \partial_d^2 \nabla V | \leqslant \frac{K}{(1 + \tilde{r})^3} . \]
\end{lemma}

\begin{proof}
  We have that $\partial_d V = - \partial_{x_1} V_1 V_{- 1} + \partial_{x_1}
  V_{- 1} V_1$ and from Lemma \ref{dervor},
  \[ | \partial_{x_1} V_1 | \leqslant \frac{K}{(1 + r_1)} \leqslant
     \frac{K}{(1 + \tilde{r})} . \]
  Similarly, $| \partial_{x_1} V_{- 1} | \leqslant \frac{K}{(1 + \tilde{r})}$
  and this proves the first inequality. Furthemore, for $\nabla \partial_d V$,
  every terms has two derivatives, each one bringing a $\frac{1}{(1 +
  \tilde{r})}$ by (\ref{230210}), this shows the second inequality. Finally,
  we compute
  \[ \mathfrak{R}\mathfrak{e} (\bar{V} \partial_d V) = - | V_{- 1} |^2
     \mathfrak{R}\mathfrak{e} (\overline{V_1} \partial_{x_1} V_1) + | V_1 |^2
     \mathfrak{R}\mathfrak{e} (\overline{V_{- 1}} \partial_{x_1} V_{- 1}) . \]
  From Lemma \ref{lemme3}, $| \mathfrak{R}\mathfrak{e} (\overline{V_1}
  \partial_{x_1} V_1) | \leqslant \frac{K}{(1 + r_1)^3} \leqslant \frac{K}{(1
  + \tilde{r})^3}$ and $| V_{- 1} |^2 \leqslant 1$. Similarly we have
  \[ | | V_1 |^2 \mathfrak{R}\mathfrak{e} (\overline{V_{- 1}} \partial_{x_1}
     V_{- 1}) | \leqslant \frac{K}{(1 + \tilde{r})^3} . \]
  Furthermore, since $\partial_d^2 V = \partial_{x_1}^2 V_1 V_{- 1} - 2
  \partial_{x_1} V_1 \partial_{x_1} V_{- 1} + \partial_{x_1}^2 V_{- 1} V_1$,
  with equation (\ref{230210}), we check easily the estimations on
  $\partial_d^2 V$ and $\partial_d^2 \nabla V$.
\end{proof}

\subsection{Setup of the proof}\label{perturbation}

In the same way as in {\cite{DP2}} (see also {\cite{LW}}), we will look at a
solution of $(\tmop{TW}_c)$ as a perturbation of $V$ of the form
\[ v \assign \eta V (1 + \Psi) + (1 - \eta) V e^{\Psi} \]
where $\eta (x) = \tilde{\eta} (r_1) + \tilde{\eta} (r_{- 1})$ and
$\tilde{\eta}$ is a $C^{\infty}$ positive cutoff with $\tilde{\eta} (r) = 1$
if $r \leqslant 1$ and 0 if $r \geqslant 2$. The perturbation is $\Psi$ and we
will also use
\[ \Phi \assign V \Psi . \]
We use such a perturbation because we want it to be additive (in $\Phi$) near
the center of the vortices (where $v = V + \Phi$), and multiplicative (in
$\Psi$) far from them (where $v = V e^{\Psi}$). We shall require $\Phi$ to be
bounded (and small) near the vortices. The problem becomes an equation on
$\Psi$, with the following Lemma \ref{lemma7}, we shall write
\[ \eta L (\Phi) + (1 - \eta) V L' (\Psi) + F (\Psi) = 0 \]
where $L$ and $L'$ are linear. The main part of the proof of the construction
consists of inverting the linearized operator $\eta L (\Phi) + (1 - \eta) V L'
(\Psi)$ in suitable spaces, and then use a contraction argument by showing
that $F$ is small and conclude on the existence of a solution $\Psi$ by a
fixed point theorem.

\begin{lemma}
  \label{lemma7}The function $v = \eta V (1 + \Psi) + (1 - \eta) V e^{\Psi}$
  is solution of $(\tmop{TW}_c)$ if and only if
  \[ \eta L (\Phi) + (1 - \eta) V L' (\Psi) + F (\Psi) = 0, \]
  where $\Phi = V \Psi$,
  \[ L' (\Psi) \assign - \Delta \Psi - 2 \frac{\nabla V}{V} . \nabla \Psi + 2
     | V |^2 \mathfrak{R}\mathfrak{e} (\Psi) - i c \partial_{x_2} \Psi, \]
  \[ L (\Phi) \assign - \Delta \Phi - (1 - | V |^2) \Phi +
     2\mathfrak{R}\mathfrak{e} (\bar{V} \Phi) V - i c \partial_{x_2} \Phi, \]
  \[ F (\Psi) \assign E - i c \partial_{x_2} V + V (1 - \eta) (- \nabla \Psi .
     \nabla \Psi + | V |^2 S (\Psi)) + R (\Psi), \]
  with
  \[ E = - \Delta V - (1 - | V |^2) V, \]
  \[ S (\Psi) \assign e^{2\mathfrak{R}\mathfrak{e} (\Psi)} - 1 -
     2\mathfrak{R}\mathfrak{e} (\Psi) \]
  and $R (\Psi)$ is a sum of terms at least quadratic in $\Psi$ or $\Phi$
  localized in the area where $\eta \neq 0$. Furthermore, there exists $C, C_0
  > 0$ such that the estimate
  \[ | R (\Psi) | + | \nabla R (\Psi) | \leqslant C \| \Phi \|^2_{C^2 (\{
     \nobracket \tilde{r} \leqslant 2 \} \nobracket)} \]
  holds if $\| \Phi \|_{C^2 (\mathbbm{R}^2)} \leqslant C_0$ (a constant
  independent of $c$), where $\tilde{r} = \min (| x - d \vec{e_1} |, | x + d
  \overrightarrow{e_1} |)$ for $x \in \mathbbm{R}^2$. Additionally, $L (\Phi)$
  and $L' (\Psi)$ are related by
  \[ L (\Phi) = (E - i c \partial_{x_2} V) \Psi + V L' (\Psi) . \]
\end{lemma}

\

See Appendix $\ref{AA}$ for the proof of this result.

\

The main reason for such a perturbation ansatz is because $V (d
\overrightarrow{e_1}) = V (- d \overrightarrow{e_1}) = 0$, so we can not
divide by $V$ as done in $L'$ for instance when we look near the vortices,
therefore an additive perturbation is more suitable. But far from the
vortices, the perturbation is easier to compute when written multiplicatively
with a factorisation by $V$. Remark also that this allows us to take $\Psi$ to
explode at $d \vec{e}_1$ and $- d \vec{e}_1$ as long as $\Phi = V \Psi$ does
not. This is needed for the norm we use in subsection \ref{normsetup}.

As we look for $\Phi$ small (it is a perturbation), the conditions $\| \Phi
\|_{C^2 (\mathbbm{R}^2)} \leqslant C_0$ will always be true. We need them
because some of the error terms have an exponential contribution in $\Psi$,
and not only quadratic. We recall that, with our notations, $\nabla \Psi .
\nabla \Psi$ is complex-valued.

Remark that the quantity $F$ contains only nonlinear terms and the source
term, which is $E - i c \partial_{x_2} V$. Furthermore, contrary to the work
{\cite{LW}}, the transport term is in the linearized operator, and not
considered as an error term in $F$.

\subsection{Setup of the norms}\label{normsetup}

For a given $\sigma \in \mathbbm{R}$, we define, similarly as in {\cite{DP2}}
and {\cite{LW}}, for $\Psi = \Psi_1 + i \Psi_2$ and $h = h_1 + i h_2$, the
norms
\begin{eqnarray*}
  \| \Psi \|_{\ast, \sigma, d} & \assign & \| V \Psi \|_{C^2 (\{ \tilde{r}
  \leqslant 3 \})}\\
  & + & \| \tilde{r}^{1 + \sigma} \Psi_1 \|_{L^{\infty} (\{ \tilde{r}
  \geqslant 2 \})} + \| \tilde{r}^{2 + \sigma} \nabla \Psi_1 \|_{L^{\infty}
  (\{ \tilde{r} \geqslant 2 \})} + \| \tilde{r}^{2 + \sigma} \nabla^2 \Psi_1
  \|_{L^{\infty} (\{ \tilde{r} \geqslant 2 \})}\\
  & + & \| \tilde{r}^{\sigma} \Psi_2 \|_{L^{\infty} (\{ \tilde{r} \geqslant 2
  \})} + \| \tilde{r}^{1 + \sigma} \nabla \Psi_2 \|_{L^{\infty} (\{ \tilde{r}
  \geqslant 2 \})} + \| \tilde{r}^{2 + \sigma} \nabla^2 \Psi_2 \|_{L^{\infty}
  (\{ \tilde{r} \geqslant 2 \})},\\
  &  & \\
  \| h \|_{\ast \ast, \sigma, d} & \assign & \| V h \|_{C^1 (\{ \tilde{r}
  \leqslant 3 \})}\\
  & + & \| \tilde{r}^{1 + \sigma} h_1 \|_{L^{\infty} (\{ \tilde{r} \geqslant
  2 \})} + \| \tilde{r}^{2 + \sigma} \nabla h_1 \|_{L^{\infty} (\{ \tilde{r}
  \geqslant 2 \})}\\
  & + & \| \tilde{r}^{2 + \sigma} h_2 \|_{L^{\infty} (\{ \tilde{r} \geqslant
  2 \})} + \| \tilde{r}^{2 + \sigma} \nabla h_2 \|_{L^{\infty} (\{ \tilde{r}
  \geqslant 2 \})},
\end{eqnarray*}
where $\tilde{r} = \min (r_1, r_{- 1})$ (which depends on $d$). These are the
spaces we shall use for the inversion of the linear operator for suitable
values of $\sigma$.

This norm is not the ``natural'' energy norm that we could expect, for
instance:
\[ \| \Phi \|_{H_{V_{\nosymbol}}}^2 \assign \int_{\mathbbm{R}^2} | \nabla \Phi
   |^2 + (1 - | V_{\nosymbol} |^2) | \Phi |^2 +\mathfrak{R}\mathfrak{e}
   (\bar{V} \Phi)^2 . \]
In particular, we require different conditions on the decay at infinity (with,
in a way, less decay). As a consequence, the decay we have in Theorem
\ref{th1} is not optimal (see {\cite{G1}}). This decay will be recovered later
on by showing that the solution has finite energy. The main advantage of the
norms $\| . \|_{\ast, \sigma, d}$ and $\| . \|_{\ast \ast, \sigma, d}$ is that
they will allow us to have uniform estimates on the error, without constants
depending on $c$ or $d$.

\

We are looking for a solution $\Psi$ on a space of symmetric functions: we
suppose that
\[ \forall x = (x_1, x_2) \in \mathbbm{R}^2, \Psi (x_1, x_2) = \overline{\Psi
   (x_1, - x_2)} = \Psi (- x_1, x_2) \]
because $V$ and the equation has the same symmetries. With only those
symmetries we will not be able to invert the linearized operator because it
has a kernel, we also need an orthogonal condition. We define
\[ Z_d (x) \assign \partial_d V (x) (\tilde{\eta} (4 r_1) + \tilde{\eta} (4
   r_{- 1})), \]
where $\tilde{\eta}$ is the same function as the one used for $v$: it is a
$C^{\infty}$ non negative smooth cutoff with $\tilde{\eta} (r) = 1$ if $r
\leqslant 1$ and 0 if $r \geqslant 2$. In particular $Z_d (x) = 0$ if
$\tilde{r} \geqslant 1 / 2$, which will make some computations easier. The
other interest of the cutoff function is that without it
\[ \partial_d V (x) = - \partial_{x_1} V_1 V_{- 1} + \partial_{x_1} V_{- 1}
   V_1 \]
is not integrable in all $\mathbbm{R}^2$. We define the Banach spaces we shall
use for inverting the linear part:
\[ \mathcal{E}_{\ast, \sigma, d} \assign \]
\[ \{ \Phi = V \Psi \in C^2 (\mathbbm{R}^2, \mathbbm{C}), \| \Psi \|_{\ast,
   \sigma, d} < + \infty ; \langle \Phi, Z_d \rangle = 0 ; \forall x \in
   \mathbbm{R}^2, \Psi (x_1, x_2) = \overline{\Psi (x_1, - x_2)} = \Psi (-
   x_1, x_2) \}, \]
\[ \mathcal{E}_{\ast \ast, \sigma', d} \assign \{ V h \in C^1 (\mathbbm{R}^2,
   \mathbbm{C}), \| h \|_{\ast \ast, \sigma', d} < + \infty \} \]
for $\sigma, \sigma' \in \mathbbm{R}$. We shall omit the subscript $d$ in the
construction and use only $\mathcal{E}_{\ast, \sigma}$, $\mathcal{E}_{\ast
\ast, \sigma'}$. Remark that $\mathcal{E}_{\ast, \sigma}$ contains an
orthogonality condition as well as the symmetries.

\

Our first goal is to invert the linearized operator. This is a difficult
part, which requires a lot of computations and critical elliptic estimates.
The next subsection is devoted to the proof of the elliptic tools use in the
proof of the inversion. In particular, our paper diverges here from
{\cite{LW}} (see Remark \ref{normsetup2} thereafter).

\subsection{Some elliptic estimates}

In this subsection, we provide some tools for elliptic estimate adapted to
$L^{\infty}$ norms.

\subsubsection{Weighted $L^{\infty}$ estimates on a Laplacian problem}

\begin{lemma}
  \label{lapzeta}For $d \geqslant 5$, $0 < \alpha < 1$, there exists a
  constant $K (\sigma) > 0$ such that, for $f \in C^0 (\mathbbm{R}^2,
  \mathbbm{C})$ such that
  \[ \forall (x_1, x_2) \in \mathbbm{R}^2, \quad f (x_1, x_2) = - f (x_1, -
     x_2) \]
  and with
  \[ \varepsilon_{f, \alpha} \assign \| f (x) (1 + \tilde{r})^{2 + \alpha}
     \|_{L^{\infty} (\mathbbm{R}^2)} < + \infty, \]
  there exists a unique $C^1 (\mathbbm{R}^2)$ function $\zeta$ such that
  \[ \Delta \zeta = f \]
  in the distribution sense,
  \[ \forall (x_1, x_2) \in \mathbbm{R}^2, \quad \zeta (x_1, x_2) = - \zeta
     (x_1, - x_2) \]
  and $\zeta$ satisfies the following two estimates:
  \[ \forall x \in \mathbbm{R}^2, | \zeta (x) | \leqslant \frac{K (\sigma)
     \varepsilon_{f, \alpha}}{(1 + \tilde{r})^{\alpha}} \]
  and
  \[ \forall x \in \mathbbm{R}^2, | \nabla \zeta (x) | \leqslant \frac{K
     (\sigma) \varepsilon_{f, \alpha}}{(1 + \tilde{r})^{1 + \alpha}} . \]
\end{lemma}

See Appendix \ref{AB1} for the proof of this result.

\

Remark here that for a given function $f$, if it satisfies two inequalities
with different values of $(\varepsilon_{f, \alpha}, \alpha)$, then the
associated function $\zeta$ satisfies the estimates with both sets of values
by uniqueness. Furthermore, with only the hypothesis $f \in C^0
(\mathbbm{R}^2)$, we do not have $\zeta \in C_{\tmop{loc}}^2 (\mathbbm{R}^2)$
a priori.

\subsubsection{Fundamental solution for $- \Delta + 2$}

We will use the fundamental solution of $- \Delta + 2$. It can be deduce from
the fundamental solution of $- \Delta + 1$, which has the following
properties.

\begin{lemma}[{\cite{MR0167642}}]
  \label{Kzero}The fundamental solution of $- \Delta + 1$ in $\mathbbm{R}^2$
  is $\frac{1}{2 \pi} K_0 (| . |)$, where $K_0$ is the modified Bessel
  function of second kind. It satifies $K_0 \in C^{\infty} (\mathbbm{R}^{+
  \ast})$ and
  \[ K_0 (r) \sim_{r \rightarrow \infty} \sqrt{\frac{\pi}{2 r}} e^{- r}, \]
  \[ K_0 (r) \sim_{r \rightarrow 0} - \ln (r), \]
  
  \[ K_0' (r) \sim_{r \rightarrow \infty} - \sqrt{\frac{\pi}{2 r}} e^{- r}, \]
  \[ K_0' (r) \sim_{r \rightarrow 0} \frac{- 1}{r}, \]
  \[ \forall r > 0, K_0 (r) > 0, K'_0 (r) < 0 \quad \mathit{and} \quad K_0''
     (r) > 0. \]
\end{lemma}

\begin{proof}
  The first three equivalents are respectively equations 9.7.2, 9.6.8 and
  9.7.4 of {\cite{MR0167642}}. The fourth one can be deduced from equations
  9.6.27 and 9.6.9 of {\cite{MR0167642}}. For $\nu \in \mathbbm{N}$, $K_{\nu}$
  is $C^{\infty} (\mathbbm{R}, \mathbbm{R})$ since it solves 9.6.1 of
  {\cite{MR0167642}} and from the end of 9.6 of {\cite{MR0167642}}, we have
  that $K_{\nu}$ has no zeros. In particular with the asymptotics of 9.6.8,
  this implies that $K_{\nu} (r) > 0$. Furthermore, from 9.6.27 of
  {\cite{MR0167642}}, we have $K'_0 = - K_1 < 0$ and $K''_0 = - K_1' =
  \frac{K_0 + K_2}{2} > 0$.
\end{proof}

We end this subsection by the proof an elliptic estimate that will be used in
the proof of Proposition \ref{invertop}.

\begin{lemma}
  \label{see}For any $\alpha > 0$, there exists a constant $C (\alpha) > 0$
  such that, for any $d > 1$, if two real-valued functions $\Psi \in H^1
  (\mathbbm{R}^2), h \in C^0 (\mathbbm{R}^2)$ satisfy in the distribution
  sense
  \[ (- \Delta + 2) \Psi = h, \]
  and
  \[ \| (1 + \tilde{r})^{\alpha} h \|_{L^{\infty} (\mathbbm{R}^2)} < + \infty,
  \]
  then $\Psi \in C^1 (\mathbbm{R}^2)$ with
  \[ | \Psi | + | \nabla \Psi | \leqslant \frac{C (\alpha) \| (1 +
     \tilde{r})^{\alpha} h \|_{L^{\infty} (\mathbbm{R}^2)}}{(1 +
     \tilde{r})^{\alpha}} . \]
\end{lemma}

See Appendix \ref{AB2} for the proof of this result.

\begin{remark}
  \label{normsetup2}Lemma \ref{see} is different from the equivalent one of
  {\cite{LW}} for the gradient, which is equation (5.21) there. They claim
  that:
  
  for any $0 < \sigma < 1$, there exists $C > 0$ such that, if two real-valued
  functions $\Psi \in C^1 (\mathbbm{R}^2), h \in C^0 (\mathbbm{R}^2)$ satisfy
  \[ (- \Delta + 2) \Psi = h \]
  in the distribution sense, and
  \[ \| \Psi (1 + \tilde{r})^{1 + \sigma} \|_{L^{\infty} (\mathbbm{R}^2)} + \|
     \nabla \Psi (1 + \tilde{r})^{2 + \sigma} \|_{L^{\infty} (\mathbbm{R}^2)}
     + \| (1 + \tilde{r})^{1 + \sigma} h \|_{L^{\infty} (\mathbbm{R}^2)} < +
     \infty, \]
  then
  \[ | \Psi | \leqslant \frac{C \| (1 + \tilde{r})^{1 + \sigma} h
     \|_{L^{\infty} (\mathbbm{R}^2)}}{(1 + \tilde{r})^{1 + \sigma}} \]
  and
  \[ | \nabla \Psi | \leqslant \frac{C \| (1 + \tilde{r})^{1 + \sigma} h
     \|_{L^{\infty} (\mathbbm{R}^2)}}{(1 + \tilde{r})^{2 + \sigma}} . \]
  The main difference they claim would be a stronger decay for the gradient.
  However, such a result can not hold, because of the following
  counterexample:
  \[ \Psi_{\varepsilon} (x) = \left\{ \begin{array}{ll}
       0 & \tmop{if} | x | \leqslant 1 / \varepsilon\\
       \frac{\sin^2 (r)}{(1 + r)^{2 + \sigma}} & \tmop{if} | x | \geqslant 1 /
       \varepsilon .
     \end{array} \right. \]
  For $\varepsilon > 0$ small enough (in particular such that
  $\frac{1}{\varepsilon} \gg \frac{1}{c}$, and such that
  $\frac{1}{\varepsilon}$ is an integer multiple of $\pi$, so that
  $\Psi_{\varepsilon}$ is $C^2$), we have
  \[ \| (1 + \tilde{r})^{1 + \sigma} h (x) \|_{L^{\infty} (\mathbbm{R}^2)} =
     \| (1 + \tilde{r})^{1 + \sigma} ((- \Delta + 2) \Psi) (x) \|_{L^{\infty}
     (\mathbbm{R}^2)} \leqslant K \varepsilon \]
  and
  \[ \| (1 + \tilde{r})^{2 + \sigma} | \nabla \Psi (x) | \|_{L^{\infty}
     (\mathbbm{R}^2)} \geqslant 1 / 2. \]
  Therefore, taking $\varepsilon \rightarrow 0$, we see that the estimate $|
  \nabla \Psi (x) | \leqslant \frac{C \| (1 + \tilde{r})^{1 + \sigma} h
  \|_{L^{\infty} (\mathbbm{R}^2)}}{(1 + \tilde{r})^{2 + \sigma}}$ can not
  hold.
\end{remark}

For our proof of the inversion of the linearized operator (Proposition
\ref{invertop} below), we did not choose the same norms $\| . \|_{\ast,
\sigma, d}$ and $\| . \|_{\ast \ast, \sigma', d}$ as in {\cite{LW}} (at the
beginning of subsection \ref{normsetup}). In particular, we require decays on
the second derivatives for $\| . \|_{\ast \ast, \sigma', d}$. Our proof of the
inversion of the linearized operator (the equivalent of Lemma 5.1 of
{\cite{LW}}) will be different, and will follow more closely the proof of
{\cite{DP2}}.

\subsubsection{Estimates for the Gross-Pitaevskii kernels}\label{gpkernel}

We are interested here in solving the following equation on $\psi$, given a
source term $h$ and $c \in \left] 0, \sqrt{2} \right[$:
\[ - i c \partial_{x_2} \psi - \Delta \psi + 2\mathfrak{R}\mathfrak{e} (\psi)
   = h. \]
It will appear in the inversion of the linearized operator around $V$. See
Lemma \ref{P3starest} for the exact result. We give here a way to construct a
solution formally. We will highlight all the important quantities, as well as
all the difficulties that arise when trying to solve this equation rigorously.

In this subsection, we want to check that a solution of this equation, with
$\psi = \psi_1 + i \psi_2$ and $h = h_1 + i h_2$ (where $\psi_1, \psi_2, h_1,
h_2$ are real valued) can be written
\begin{equation}
  \psi_1 = K_0 \ast h_1 + c H, \label{P32931}
\end{equation}
with $H$ a function that satisfies
\[ \partial_{x_j} H \assign K_j \ast h_2, \]
and
\begin{equation}
  \partial_{x_j} \psi_2 = G_j - c K_j \ast h_1, \label{P32932}
\end{equation}
where similarly $G_j$ satisfies
\[ \partial_{x_k} G_j \assign (c^2 L_{j, k} - R_{j, k}) \ast h_2, \]
where, for $j, k \in \{ 1, 2 \}$, $\xi = (\xi_1, \xi_2) \in \mathbbm{R}^2$,
\[ \widehat{K_0} (\xi) \assign \frac{| \xi |^2}{| \xi |^4 + 2 | \xi |^2 - c^2
   \xi^2_2}, \]
\[ \widehat{K_j} (\xi) \assign \frac{\xi_2 \xi_j}{| \xi |^4 + 2 | \xi |^2 -
   c^2 \xi^2_2}, \]
\[ \widehat{L_{j, k}} (\xi) \assign \frac{\xi_2^2 \xi_j \xi_k}{| \xi |^2 (|
   \xi |^4 + 2 | \xi |^2 - c^2 \xi^2_2)}, \]
and
\[ \widehat{R_{j, k}} (\xi) \assign \frac{\xi_j \xi_k}{| \xi |^2} . \]
We will check later on that, for continuous and sufficiently decaying
functions $h$, these quantities are well defined, and that $H, G_j, \psi_2$
can be defined from there derivatives. The Gross-Pitaevskii kernels, $K_0,
K_j, L_{j, k}$, and the Riesz kernels $R_{j, k}$ have been studied in
{\cite{MR2086751}}, and we will recall some of the results obtained there.

We write the system in real and imaginary part:
\[ \left\{ \begin{array}{l}
     c \partial_{x_2} \psi_2 - \Delta \psi_1 + 2 \psi_1 = h_1\\
     - c \partial_{x_2} \psi_1 - \Delta \psi_2 = h_2 .
   \end{array} \right. \]
Now, taking the Fourier transform of the system, we have
\[ \left\{ \begin{array}{l}
     i \xi_2 c \widehat{\psi_2} + (| \xi |^2 + 2) \widehat{\psi_1} =
     \widehat{h_1}\\
     - i \xi_2 c \widehat{\psi_1} + | \xi |^2 \widehat{\psi_2} =
     \widehat{h_2},
   \end{array} \right. \]
and we write it
\[ \left(\begin{array}{cc}
     | \xi |^2 + 2 & i c \xi_2\\
     - i c \xi_2 & | \xi |^2
   \end{array}\right) \left(\begin{array}{c}
     \widehat{\psi_1}\\
     \widehat{\psi_2}
   \end{array}\right) = \left(\begin{array}{c}
     \widehat{h_1}\\
     \widehat{h_2}
   \end{array}\right) . \]
Here, we suppose that $\psi$ is a tempered distributions and $h \in L^p
(\mathbbm{R}^2, \mathbbm{C})$ for some $p > 1$.

Now, we want to invert the matrix, and for that, we have to divide by its
determinant, $| \xi |^4 + 2 | \xi |^2 - c^2 \xi^2_2$. For $0 < c < \sqrt{2}$,
this quantity is zero only for $\xi = 0$. Thus, for $\xi \neq 0$,
\[ \left(\begin{array}{c}
     \widehat{\psi_1}\\
     \widehat{\psi_2}
   \end{array}\right) = \frac{1}{| \xi |^4 + 2 | \xi |^2 - c^2 \xi_2^2}
   \left(\begin{array}{c}
     | \xi |^2  \widehat{h_1} - i c \xi_2 \widehat{h_2}\\
     (| \xi |^2 + 2) \widehat{h_2} + i c \xi_2 \widehat{h_1}
   \end{array}\right), \]
which implies that
\[ \widehat{\psi_1} = \frac{| \xi |^2  \widehat{h_1}}{| \xi |^4 + 2 | \xi |^2
   - c^2 \xi_2^2} + \frac{- i c \xi_2 \widehat{h_2}}{| \xi |^4 + 2 | \xi |^2 -
   c^2 \xi_2^2} . \]
With the definition of $K_0$, we have $\frac{| \xi |^2}{| \xi |^4 + 2 | \xi
|^2 - c^2 \xi_2^2}  \widehat{h_1} = \widehat{K_0} \widehat{h_1}$ and, defining
the distribution $H$ by $\hat{H} = \frac{- i \xi_2}{| \xi |^4 + 2 | \xi |^2 -
c^2 \xi_2^2} \widehat{h_2}$, we have, for $\xi \neq 0$,
\[ \widehat{\partial_{x_j} H} = \frac{\xi_j \xi_2 \widehat{h_2}}{| \xi |^4 + 2
   | \xi |^2 - c^2 \xi_2^2} = \widehat{K_j} \widehat{h_2} . \]
Remark that $\frac{- i \xi_2}{| \xi |^4 + 2 | \xi |^2 - c^2 \xi_2^2} \in L^{3
/ 2} (\mathbbm{R}^2, \mathbbm{C})$ and thus is a tempered distribution.

\

Now, we have
\[ \widehat{\partial_{x_j} \psi_2} = \frac{i \xi_j (| \xi |^2 + 2)
   \widehat{h_2}}{| \xi |^4 + 2 | \xi |^2 - c^2 \xi_2^2} + \frac{- c \xi_j
   \xi_2 \widehat{h_1}}{| \xi |^4 + 2 | \xi |^2 - c^2 \xi_2^2} . \]
We check that $\frac{- c \xi_j \xi_2}{| \xi |^4 + 2 | \xi |^2 - c^2 \xi_2^2}
\widehat{h_1} = - c \widehat{K_j} \widehat{h_1}$, and we compute
\[ \frac{| \xi |^2 + 2}{| \xi |^4 + 2 | \xi |^2 - c^2 \xi_2^2} = \frac{1}{|
   \xi |^2} \left( 1 - \frac{c^2 \xi_2^2}{| \xi |^4 + 2 | \xi |^2 - c^2
   \xi_2^2} \right) = \frac{1}{| \xi |^2} - c^2 \frac{\xi_2^2}{| \xi |^2 (|
   \xi |^4 + 2 | \xi |^2 - c^2 \xi_2^2)}, \]
thus, denoting $\widehat{G_j} = \frac{i \xi_j (| \xi |^2 + 2)}{| \xi |^4 + 2 |
\xi |^2 - c^2 \xi_2^2} \widehat{h_2}$, we have
\[ \widehat{\partial_{x_k} G_j} \assign (c^2 \widehat{L_{j, k}} -
   \widehat{R_{j, k}}) \widehat{h_2} . \]
We therefore have that, at least formally, for $\xi \neq 0$, $\widehat{- i c
\partial_{x_2} \psi - \Delta \psi + 2\mathfrak{R}\mathfrak{e} (\psi) - h}
(\xi) = 0$. We deduce that there exists $P \in \mathbbm{C} [X_1, X_2]$ such
that $- i c \partial_{x_2} \psi - \Delta \psi + 2\mathfrak{R}\mathfrak{e}
(\psi) - h = P$ Now, if the function $\psi$ and $h$ are such that the left
hand side is bounded and goes to $0$ at infinity, this implies that $P = 0$.
This will hold under a condition on $h$ (which will be $\int_{\mathbbm{R}^2}
h_2 = 0$ and some decay estimates, that $\psi$ will inherit). Another remark
is that $\psi$ is here in part defined through its derivatives, and we need an
argument to construct a primitive. See Lemma \ref{P3starest} for a rigorous
proof of this construction. Remark that $- i c \partial_{x_2} \psi - \Delta
\psi + 2\mathfrak{R}\mathfrak{e} (\psi) = 0$ has some nonzero or unbounded
polynomial solutions, for instance $\psi = i$ or $\psi = i x_2 - \frac{c}{2}$.

The kernels $K_0, K_j$ and $L_{j, k}$ have been studied in details in
{\cite{MR2086751}}, {\cite{MR2191764}} and {\cite{G1}}. In particular, we
recall the following result.

\begin{theorem}[{\cite{MR2086751}}, Theorems 5 and 6]
  \label{P3gravejatftw}For $\mathcal{K} \in \{ K_0, K_j, L_{j, k} \}$ and any
  $0 < c_0 < \sqrt{2}$, there exist a constant $K (c_0) > 0$ such that, for
  all $0 < c < c_0$,
  \[ | \mathcal{K} (x) | \leqslant \frac{K (c_0)}{| x |^{1 / 2} (1 + | x |)^{3
     / 2}} \]
  and
  \[ | \nabla \mathcal{K} (x) | \leqslant \frac{K (c_0)}{| x |^{3 / 2} (1 + |
     x |)^{3 / 2}} . \]
\end{theorem}

\begin{proof}
  This is the main result of Theorems 5 and 6 of {\cite{MR2086751}}. We added
  the fact that the constant $K$ is uniform in $c$, given that $c$ is small.
  This can be easily shown by following the proof of Theorem 5 and 6 of
  {\cite{MR2086751}}, and verifying that the constants depends only on weighed
  $L^{\infty}$ norms on $\hat{\mathcal{K}}$ and its first derivatives, which
  are uniforms in $c$ if $c > 0$ is small. The condition $c < c_0$ is taken in
  ordrer to avoid $c \rightarrow \sqrt{2}$, where this does not hold (the
  singularity near $\xi = 0$ of $\hat{\mathcal{K}}$ changes of order at the
  limit). Furthermore, the factor $1 / 2$ for the growth near $x = 0$ is not
  at all optimal, but we will not require more here.
  
  Remark that the speed in {\cite{MR2086751}} is in the direction
  $\overrightarrow{e_1}$, whereas it is in the direction
  $\overrightarrow{e_2}$ in our case, which explains the swap between $\xi_2$
  and $\xi_1$ in the two papers.
\end{proof}

We recall that $\tilde{r} = \min (r_1, r_{- 1})$ with $r_{\pm 1} = | x \mp d
\overrightarrow{e_1} |$. We give some estimates of convolution with these
kernels.

\begin{lemma}
  \label{P3kerest1}Take $K \in \{ K_0, K_j, L_{j, k} \}$ and $h \in C^0
  (\mathbbm{R}^2, \mathbbm{R})$, and suppose that, for some $\alpha > 0$, $\|
  h (1 + \tilde{r})^{\alpha} \|_{L^{\infty} (\mathbbm{R}^2)} < + \infty$.
  Then, for any $0 < \alpha' < \alpha$, there exists $C (\alpha, \alpha') > 0$
  such that, for $0 < c < 1$, if either
  \begin{itemizeminus}
    \item $\alpha < 2$
    
    \item $2 < \alpha < 3$, $\forall (x_1, x_2) \in \mathbbm{R}^2, h (- x_1,
    x_2) = h (x_1, x_2)$ and $\int_{\mathbbm{R}^2} h = 0$,
  \end{itemizeminus}
  then
  \[ | K \ast h | \leqslant \frac{C (\alpha, \alpha') \| h (1 +
     \tilde{r})^{\alpha} \|_{L^{\infty} (\mathbbm{R}^2)}}{(1 +
     \tilde{r})^{\alpha'}} . \]
  Furthermore, if $\alpha < 3$ (without any other conditions), then
  \[ | \nabla K \ast h | \leqslant \frac{C (\alpha, \alpha') \| h (1 +
     \tilde{r})^{\alpha} \|_{L^{\infty} (\mathbbm{R}^2)}}{(1 +
     \tilde{r})^{\alpha'}} . \]
\end{lemma}

See Appendix \ref{AB3} for the proof of this result.

\

The symmetry and $\int_{\mathbbm{R}^2} h = 0$ in the case $2 < \alpha < 3$
could be removed, if we suppose instead that $\int_{\{ x_1 \geqslant 0 \}} h =
\int_{\{ x_1 \leqslant 0 \}} h = 0$. In particular, if we suppose that
$\forall (x_1, x_2) \in \mathbbm{R}^2, h (x_1, x_2) = - h (x_1, - x_2)$, then
the condition $\int_{\mathbbm{R}^2} h = 0$ is automatically satisfied.

\

We complete these estimates with some for $R_{j, k}$.

\begin{lemma}
  \label{P3kerest2}Take $h \in C^1 (\mathbbm{R}^2, \mathbbm{R})$ with $\forall
  x = (x_1, x_2) \in \mathbbm{R}^2, h (- x_1, x_2) = h (x_1, x_2)$, and
  suppose that for some $\alpha > 0$, $\| h (1 + \tilde{r})^{\alpha}
  \|_{L^{\infty} (\mathbbm{R}^2)} + \| \nabla h (1 + \tilde{r})^{\alpha}
  \|_{L^{\infty} (\mathbbm{R}^2)} < + \infty$. Then, for any $0 < \alpha' <
  \alpha$, for $0 < c < 1$, if either
  \begin{itemizeminus}
    \item $\alpha < 2$
    
    \item $2 < \alpha < 3$ and $\int_{\mathbbm{R}^2} h = 0$,
  \end{itemizeminus}
  then, there exists $C (\alpha, \alpha') > 0$ such that
  \[ | R_{j, k} \ast h | \leqslant \frac{C (\alpha, \alpha') (\| h (1 +
     \tilde{r})^{\alpha} \|_{L^{\infty} (\mathbbm{R}^2)} + \| \nabla h (1 +
     \tilde{r})^{\alpha} \|_{L^{\infty} (\mathbbm{R}^2)})}{(1 +
     \tilde{r})^{\alpha'}} . \]
\end{lemma}

See Appendix \ref{AB4} for the proof of this result.

\

We can now solve the problem
\[ - i c \partial_{x_2} \psi - \Delta \psi + 2\mathfrak{R}\mathfrak{e} (\psi)
   = h, \]
\[ \int_{\mathbbm{R}^2} \mathfrak{I}\mathfrak{m} (h) = 0 \]
in some suitable spaces. We define the norms, for $\sigma, \sigma' \in
\mathbbm{R}$,
\begin{eqnarray*}
  \| \psi \|_{\otimes, \sigma, \infty} & \assign & \| (1 + \tilde{r})^{1 +
  \sigma} \psi_1 \|_{L^{\infty} (\mathbbm{R}^2)} + \| (1 + \tilde{r})^{2 +
  \sigma} \nabla \psi_1 \|_{L^{\infty} (\mathbbm{R}^2)}\\
  & + & \| (1 + \tilde{r})^{2 + \sigma} \nabla^2 \psi_1 \|_{L^{\infty}
  (\mathbbm{R}^2)} + \| (1 + \tilde{r})^{\sigma} \psi_2 \|_{L^{\infty}
  (\mathbbm{R}^2)}\\
  & + & \| (1 + \tilde{r})^{1 + \sigma} \nabla \psi_2 \|_{L^{\infty}
  (\mathbbm{R}^2)} + \| (1 + \tilde{r})^{2 + \sigma} \nabla^2 \psi_2
  \|_{L^{\infty} (\mathbbm{R}^2)}
\end{eqnarray*}
and
\begin{eqnarray*}
  \| h \|_{\otimes \otimes, \sigma', \infty} & \assign & \| (1 + \tilde{r})^{1
  + \sigma'} h_1 \|_{L^{\infty} (\mathbbm{R}^2)} + \| (1 + \tilde{r})^{2 +
  \sigma'} \nabla h_1 \|_{L^{\infty} (\mathbbm{R}^2)}\\
  & + & \| (1 + \tilde{r})^{2 + \sigma'} h_2 \|_{L^{\infty} (\mathbbm{R}^2)}
  + \| (1 + \tilde{r})^{2 + \sigma'} \nabla h_2 \|_{L^{\infty}
  (\mathbbm{R}^2)},
\end{eqnarray*}
as well as the spaces
\[ \mathcal{E}_{\otimes, \sigma}^{\infty} \assign \{ \psi \in C^2
   (\mathbbm{R}^2, \mathbbm{C}), \| \psi \|_{\otimes, \sigma, \infty} < +
   \infty, \forall (x_1, x_2) \in \mathbbm{R}^2, \psi (x_1, x_2) = \psi (-
   x_1, x_2) \}, \]
and
\[ \mathcal{E}_{\otimes \otimes, \sigma'}^{\infty} \assign \{ h \in C^1
   (\mathbbm{R}^2, \mathbbm{C}), \| h \|_{\otimes \otimes, \sigma', \infty} <
   + \infty, \forall (x_1, x_2) \in \mathbbm{R}^2, h (x_1, x_2) = h (- x_1,
   x_2) \} . \]
The norms $\| . \|_{\otimes, \sigma, \infty}$ and $\| . \|_{\ast, \sigma}$
differ only on $\{ \tilde{r} \leqslant 3 \}$, and $\mathcal{E}_{\otimes,
\sigma}^{\infty}$ has one less symmetry than $\mathcal{E}_{\ast, \sigma}$, but
they are equivalents at infinity in position. Same remarks hold for $\| .
\|_{\otimes \otimes, \sigma', \infty}$ and $\| . \|_{\ast \ast, \sigma'}$ and
their associated spaces. Remark that if $\chi \geqslant 0$ is a smooth cutoff
function with value $0$ on $\{ \tilde{r} \leqslant R / 2 \}$ and $1$ on $\{
\tilde{r} \geqslant R \}$, then for any $\sigma \in \mathbbm{R}$,
\begin{equation}
  \| \psi \|_{\ast, \sigma} \leqslant K (R, \sigma) \| V \psi \|_{C^2 (\{
  \tilde{r} \leqslant R \})} + K \| \chi \psi \|_{\otimes, \sigma, \infty} .
  \label{P3zelda}
\end{equation}
\begin{lemma}
  \label{P3starest}Given $1 > \sigma' > \sigma > 0$, there exists $K_1
  (\sigma, \sigma') > 0$ such that, for any $h \in
  \mathcal{E}^{\infty}_{\otimes \otimes, \sigma'}$ with $\int_{\mathbbm{R}^2}
  \mathfrak{I}\mathfrak{m} (h) = 0$ and $0 < c < 1$, there exists a unique
  solution to the problem
  \[ - i c \partial_{x_2} \psi - \Delta \psi + 2\mathfrak{R}\mathfrak{e}
     (\psi) = h, \]
  in $\mathcal{E}^{\infty}_{\otimes, \sigma}$. This solution $\psi \in
  \mathcal{E}^{\infty}_{\otimes, \sigma}$ satisfies
  \[ \| \psi \|_{\otimes, \sigma, \infty} \leqslant K_1 (\sigma, \sigma') \| h
     \|_{\otimes \otimes, \sigma', \infty} . \]
  Furthermore, if instead $\sigma \in] - 1, 0 [$ and $1 > \sigma' > \sigma$,
  there exists then $K_2 (\sigma, \sigma') > 0$ such that, for any $h \in
  \mathcal{E}^{\infty}_{\otimes \otimes, \sigma'}$ with $\forall (x_1, x_2)
  \in \mathbbm{R}^2, h (x_1, x_2) = \overline{h (x_1, - x_2)}$, there exists a
  unique solution to the problem
  \[ - i c \partial_{x_2} \psi - \Delta \psi + 2\mathfrak{R}\mathfrak{e}
     (\psi) = h \]
  in $\{ \Psi \in \mathcal{E}^{\infty}_{\otimes, \sigma}, \forall (x_1, x_2)
  \in \mathbbm{R}^2, \Psi (x_1, x_2) = \overline{\Psi (x_1, - x_2)} \}$. This
  solution $\psi \in \mathcal{E}^{\infty}_{\otimes, \sigma}$ satisfies
  \[ \| \psi \|_{\otimes, \sigma, \infty} \leqslant K_2 (\sigma, \sigma') \| h
     \|_{\otimes \otimes, \sigma', \infty} . \]
\end{lemma}

The case $\sigma \in] - 1, 0 [$ is particular and such a norm will be used
only in the proof of Lemma \ref{funky} (if $\| \psi \|_{\otimes, \sigma,
\infty} < + \infty$ for $\sigma < 0$, the function $\psi$ is not necessarily
bounded for instance). Remark that the condition $\int_{\mathbbm{R}^2}
\mathfrak{I}\mathfrak{m} (h) = 0$ is automatically satisfied if $\forall (x_1,
x_2) \in \mathbbm{R}^2, h (x_1, x_2) = \overline{h (x_1, - x_2)}$.

\begin{proof}
  For $1 > \sigma' > \sigma > - 1$, we write in real and imaginary parts $h =
  h_1 + i h_2$. We define, for $j \in \{ 1, 2 \}$,
  \[ \Psi_{1, j} \assign K_0 \ast \partial_{x_j} h_1 + c K_j \ast h_2 . \]
  If $1 > \sigma' > \sigma > 0$, since $\partial_{x_j} h_1, h_2 \in L^1
  (\mathbbm{R}^2)$ (because $\sigma' > 0$ and $h \in
  \mathcal{E}^{\infty}_{\otimes \otimes, \sigma'}$), and $\int_{\mathbbm{R}^2}
  h_2 = \int_{\mathbbm{R}^2} \partial_{x_2} h_1 = 0$, by Lemma \ref{P3kerest1}
  (applied with $0 < \alpha = 2 + \sigma' < 3$, $0 < \alpha' = 2 + \sigma <
  \alpha$), the function $\Psi_{1, 2}$ is well defined and satisfies
  \[ | \nabla \Psi_{1, 2} | + | \Psi_{1, 2} | \leqslant \frac{K (\sigma,
     \sigma') \| h \|_{\otimes \otimes, \sigma', \infty}}{(1 + \tilde{r})^{2 +
     \sigma}} . \]
  This result still holds if $\sigma \in] - 1, 0 [$ and $1 > \sigma' >
  \sigma$, since $0 < \alpha = 2 + \sigma' < 3, 0 < \alpha' = 2 + \sigma <
  \alpha$. We check, still with Lemma \ref{P3kerest1} (applied with $0 <
  \alpha = 2 + \sigma' < 3$, $0 < \alpha' = 2 + \sigma < \alpha$), that
  $\Psi_{1, 1}$ is well defined and
  \[ | \nabla \Psi_{1, 1} | \leqslant \frac{K (\sigma, \sigma') \| h
     \|_{\otimes \otimes, \sigma', \infty}}{(1 + \tilde{r})^{2 + \sigma}} . \]
  If $\sigma \in] - 1, 0 [$, we have $| \Psi_{1, 1} | \leqslant \frac{K
  (\sigma, \sigma') \| h \|_{\otimes \otimes, \sigma', \infty}}{(1 +
  \tilde{r})^{2 + \sigma}}$ by Lemma \ref{P3kerest1} ($2 + \sigma < 2$). But
  since $\partial_{x_1} h_1$ is not even in $x_1$, we can not apply Lemma
  \ref{P3kerest1} to estimate $\Psi_{1, 1}$ with the same decay in the case
  $\sigma > 0$. However, following the proof of Lemma \ref{P3kerest1}, we
  check that the estimate holds if $| x + d \vec{e}_1 | \leqslant 1$ or $| x -
  d \vec{e}_1 | \leqslant 1$, and that otherwise
  \[ | \Psi_{1, 1} | \leqslant \frac{K (\sigma, \sigma') \| h \|_{\otimes
     \otimes, \sigma', \infty}}{(1 + \tilde{r})^{2 + \sigma}} + \left| K (x +
     d \vec{e}_1) \int_{\{ y_1 \leqslant 0 \}} \partial_{x_1} h (y) d y + K (x
     - d \vec{e}_1) \int_{\{ y_1 \geqslant 0 \}} \partial_{x_1} h (y) d y
     \right| . \]
  Since
  \[ \int_{\{ y_1 \leqslant 0 \}} \partial_{x_1} h (y) d y = - \int_{\{ y_1
     \geqslant 0 \}} \partial_{x_1} h (y) d y = \int_{\mathbbm{R}} h (0, y_2)
     d y_2, \]
  and
  \[ \left| \int_{\mathbbm{R}} h (0, y_2) d y_2 \right| \leqslant
     \int_{\mathbbm{R}} \frac{\| h \|_{\otimes \otimes, \sigma', \infty}}{(1 +
     \tilde{r})^{1 + \sigma'}} d y_2 \leqslant c^{\sigma} \| h \|_{\otimes
     \otimes, \sigma', \infty} \int_{\mathbbm{R}} \frac{d y_2}{(1 + | y_2
     |)^{1 + \sigma' - \sigma}} \leqslant K (\sigma, \sigma') c^{\sigma} \| h
     \|_{\otimes \otimes, \sigma', \infty}, \]
  we have
  \begin{eqnarray*}
    &  & \left| K (x + d \vec{e}_1) \int_{\{ y_1 \leqslant 0 \}}
    \partial_{x_1} h (y) d y + K (x - d \vec{e}_1) \int_{\{ y_1 \geqslant 0
    \}} \partial_{x_1} h (y) d y \right|\\
    & \leqslant & K (\sigma, \sigma') | K (x + d \vec{e}_1) - K (x - d
    \vec{e}_1) | c^{\sigma} \| h \|_{\otimes \otimes, \sigma', \infty} .
  \end{eqnarray*}
  By Theorem \ref{P3gravejatftw}, if $| x + d \vec{e}_1 |, | x - d \vec{e}_1 |
  \geqslant 1$,
  \[ | K (x + d \vec{e}_1) - K (x - d \vec{e}_1) | \leqslant \frac{K}{(1 + | x
     + d \vec{e}_1 |)^2} + \frac{K}{(1 + | x - d \vec{e}_1 |)^2} \leqslant
     \frac{K}{(1 + \tilde{r})^2}, \]
  and, if $\tilde{r} \leqslant 3 d$,
  \[ | K (x + d \vec{e}_1) - K (x - d \vec{e}_1) | \leqslant \frac{K}{(1 +
     \tilde{r})^2} \leqslant \frac{K d}{(1 + \tilde{r})^3}, \]
  or if $\tilde{r} \geqslant 3 d$,
  \[ | K (x + d \vec{e}_1) - K (x - d \vec{e}_1) | \leqslant K d \sup_{\nu \in
     [- d, d]} | \nabla K (x + \nu \vec{e}_1) | \leqslant \frac{K d}{(1 +
     \tilde{r})^3}, \]
  therefore, by interpolation,
  \[ | K (x + d \vec{e}_1) - K (x - d \vec{e}_1) | \leqslant \left(
     \frac{K}{(1 + \tilde{r})^2} \right)^{1 - \sigma} \times \left( \frac{K
     d}{(1 + \tilde{r})^3} \right)^{\sigma} \leqslant \frac{K d^{\sigma}}{(1 +
     \tilde{r})^{2 + \sigma}} . \]
  We deduce
  \begin{eqnarray*}
    &  & \left| K (x + d \vec{e}_1) \int_{\{ y_1 \leqslant 0 \}}
    \partial_{x_1} h (y) d y + K (x - d \vec{e}_1) \int_{\{ y_1 \geqslant 0
    \}} \partial_{x_1} h (y) d y \right|\\
    & \leqslant & K (\sigma, \sigma') | K (x + d \vec{e}_1) - K (x - d
    \vec{e}_1) | c^{\sigma} \| h \|_{\otimes \otimes, \sigma', \infty}\\
    & \leqslant & \frac{K (\sigma, \sigma')  (d c)^{\sigma}}{(1 +
    \tilde{r})^{2 + \sigma}} \| h \|_{\otimes \otimes, \sigma', \infty}\\
    & \leqslant & \frac{K (\sigma, \sigma')}{(1 + \tilde{r})^{2 + \sigma}} \|
    h \|_{\otimes \otimes, \sigma', \infty} .
  \end{eqnarray*}
  Combining the previous estimates, we conclude that, for $j \in \{ 1, 2 \}$,
  \begin{equation}
    | \nabla \Psi_{1, j} | + | \Psi_{1, j} | \leqslant \frac{K (\sigma,
    \sigma') \| h \|_{\otimes \otimes, \sigma', \infty}}{(1 + \tilde{r})^{2 +
    \sigma}} . \label{P3zelda1}
  \end{equation}

  Let us show that $\Psi_{1, j} \in C^1 (\mathbbm{R}^2, \mathbbm{C})$ by
  dominated convergence theorem (it is not clear at this point that $\nabla
  \Psi_{1, j}$ is continuous). For $x, \varepsilon \in \mathbbm{R}^2$,
  \begin{eqnarray*}
    \nabla \Psi_{1, j} (x + \varepsilon) - \nabla \Psi_{1, j} (x) & = &
    \int_{\mathbbm{R}^2} \nabla K_0 (y) (\partial_{x_j} h_1 (x + \varepsilon -
    y) - \partial_{x_j} h_1 (x - y)) d y,\\
    & + & c \int_{\mathbbm{R}^2} \nabla K_j (y) (h_2 (x + \varepsilon - y) -
    h_2 (x - y)) d y.
  \end{eqnarray*}
  We check that for any $y \in \mathbbm{R}^2$, $\partial_{x_j} h_1 (x +
  \varepsilon - y) - \partial_{x_j} h_1 (x - y) \rightarrow 0, h_2 (x +
  \varepsilon - y) - h_2 (x - y) \rightarrow 0$ pointwise when $| \varepsilon
  | \rightarrow 0$ (by continuity of $\partial_{x_j} h_1$ and $h_2$), and
  \begin{eqnarray*}
    &  & | \nabla K_0 (y) (\partial_{x_j} h_1 (x + \varepsilon - y) -
    \partial_{x_j} h_1 (x - y)) |\\
    & + & c | \nabla K_j (y) (h_2 (x + \varepsilon - y) - h_2 (x - y)) |\\
    & \leqslant & K (\sigma) \frac{| \nabla K_0 (y) |}{(1 + \tilde{r} (x -
    y))^{2 + \sigma'}} \| \partial_{x_j} h_1 (1 + \tilde{r})^{2 + \sigma'}
    \|_{L^{\infty} (\mathbbm{R}^2)}\\
    & + & K (\sigma) \frac{c | \nabla K_j (y) |}{(1 + \tilde{r} (x - y))^{2 +
    \sigma'}} \| h_2 (1 + \tilde{r})^{2 + \sigma'} \|_{L^{\infty}
    (\mathbbm{R}^2)}\\
    & \leqslant & K (\sigma, x) \frac{| \nabla K_0 (y) |}{(1 + \tilde{r}
    (y))^{2 + \sigma'}} \| \partial_{x_j} h_1 (1 + \tilde{r})^{2 + \sigma'}
    \|_{L^{\infty} (\mathbbm{R}^2)}\\
    & + & K (\sigma, x) \frac{c | \nabla K_j (y) |}{(1 + \tilde{r} (y))^{2 +
    \sigma'}} \| h_2 (1 + \tilde{r})^{2 + \sigma'} \|_{L^{\infty}
    (\mathbbm{R}^2)} \in L^1 (\mathbbm{R}^2)
  \end{eqnarray*}
  for $| \varepsilon | \leqslant 1$ and a constant $K (\sigma, x) > 0$, giving
  the domination.
  
  \
  
  Now, we check, by taking their Fourier transforms, that $\partial_{x_1}
  \Psi_{1, 2} = \partial_{x_2} \Psi_{1, 1} \in L^2 (\mathbbm{R}^2,
  \mathbbm{C})$ (see the computations at the beginning of subsection
  \ref{gpkernel}), and thus the integral of the vector field
  $\left(\begin{array}{c}
    \Psi_{1, 1}\\
    \Psi_{1, 2}
  \end{array}\right)$ on any closed curve of $\mathbbm{R}^2$ is $0$. For a
  large constant $D > 0$, taking, for $x_1 \in \mathbbm{R}$, the path
  \[ \{ (x_1, y), y \in [- D, D] \} \cup \{ Y = (y_1, y_2) \in \mathbbm{R}^2,
     | (x_1, 0) - Y | = D, y_1 \geqslant 0 \}, \]
  since $| \Psi_{1, 2} | \leqslant \frac{K (\sigma, \sigma') \| h \|_{\otimes
  \otimes, \sigma', \infty}}{(1 + \tilde{r})^{2 + \sigma}}$ and
  \[ \int_{\{ Y = (y_1, y_2) \in \mathbbm{R}^2, | (x_1, 0) - Y | = D, y_1
     \geqslant 0 \}} | \Psi_{1, 2} | \leqslant \frac{K (c, \sigma, \sigma',
     h)}{D^{1 + \sigma}} \rightarrow 0 \]
  when $D \rightarrow \infty$ (since $1 + \sigma > 0$), we deduce that
  \begin{equation}
    \int_{- \infty}^{+ \infty} \Psi_{1, 2} (x_1, y_2) d y_2 = 0.
    \label{P33429}
  \end{equation}
  We then define for $(x_1, x_2) \in \mathbbm{R}^2$,
  \[ \psi_1 (x_1, x_2) = \int_{+ \infty}^{x_2} \Psi_{1, 2} (x_1, y_2) d y_2,
  \]
  and thus, if $x_2 < 0$,
  \[ \psi_1 (x_1, x_2) = \int_{- \infty}^{x_2} \Psi_{1, 2} (x_1, y_2) d y_2 .
  \]
  With (\ref{P3zelda1}), we check that $\psi_1 \in C^1 (\mathbbm{R}^2,
  \mathbbm{C})$, and by simple integration from infinity using the equations
  above (with $\tilde{r} = \min (| x - d_c \overrightarrow{e_1} |, | x + d_c
  \overrightarrow{e_1} |)$, and since $1 + \sigma > 0$), that
  \[ | \psi_1 | \leqslant \frac{K (\sigma, \sigma') \| h \|_{\otimes \otimes,
     \sigma', \infty}}{(1 + \tilde{r})^{1 + \sigma}} . \]
  Furthermore, we check that
  \[ \partial_{x_2} \psi_1 = \Psi_{1, 2} \in C^1 (\mathbbm{R}^2, \mathbbm{C}),
  \]
  and (by taking their Fourier transforms)
  \[ \partial_{x_1} \psi_1 = \Psi_{1, 1} \in C^1 (\mathbbm{R}^2, \mathbbm{C}),
  \]
  therefore $\psi_1 \in C^2 (\mathbbm{R}^2, \mathbbm{C})$, and by
  (\ref{P3zelda1}),
  \[ | \nabla \psi_1 | \leqslant | \Psi_{1, 1} | + | \Psi_{1, 2} | \leqslant
     \frac{K (\sigma, \sigma') \| h \|_{\otimes \otimes, \sigma', \infty}}{(1
     + \tilde{r})^{2 + \sigma}} . \]
  For $j, k \in \{ 1, 2 \}$, we have $\partial^2_{x_j x_k} \psi_1 =
  \partial_{x_j} \Psi_{1, k}$, thus, by (\ref{P3zelda1}),
  \[ | \nabla^2 \psi_1 | \leqslant \frac{K (\sigma, \sigma') \| h \|_{\otimes
     \otimes, \sigma', \infty}}{(1 + \tilde{r})^{2 + \sigma}} . \]

  Now, we define
  \[ \Psi_{2, j, k} \assign (c^2 L_{j, k} - R_{j, k}) \ast h_2 - c K_j \ast
     \partial_{x_k} h_1 . \]
  In the case $1 > \sigma' > \sigma > 0$, $\partial_{x_k} h_1, h_2 \in L^1
  (\mathbbm{R}^2)$ and since $\int_{\mathbbm{R}^2} h_2 = \int_{\mathbbm{R}^2}
  \partial_{x_k} h_1 = 0$, by Lemmas \ref{P3kerest1} and \ref{P3kerest2} (for
  $\alpha = 2 + \sigma' < 3$, $\alpha' = 2 + \sigma < \alpha$, and the same
  variant for $K_j \ast \partial_{x_1} h_1$ as in the proof of
  (\ref{P3zelda1})), this function is well defined in $L^{\infty}
  (\mathbbm{R}^2, \mathbbm{C})$, and satisfies,
  \begin{equation}
    | \Psi_{2, j, k} | \leqslant \frac{K (\sigma, \sigma') \| h \|_{\otimes
    \otimes, \sigma', \infty}}{(1 + \tilde{r})^{2 + \sigma}} . \label{P335299}
  \end{equation}
  We check, as for the proof of (\ref{P3zelda1}), that this result holds if
  $\sigma \in] - 1, 0 [$ and $1 > \sigma' > \sigma$.
  
  Remark here that we do not have $\Psi_{2, j, k} \in C^1 (\mathbbm{R}^2,
  \mathbbm{C})$, since in Lemma \ref{P3kerest2}, the estimate on $R_{j, k}
  \ast h_2$ uses $\nabla h_2$ in the norm (showing that $\Psi_{2, j, k} \in
  C^1 (\mathbbm{R}^2, \mathbbm{C})$ would require estimates on $\nabla^2
  h_2$). However, we have that $\Psi_{2, j, k} \in C^0 (\mathbbm{R}^2,
  \mathbbm{C})$ by dominated convergence and continuity of $h_2$ and
  $\partial_{x_k} h_1$ (as for $\nabla \Psi_{1, j}$). Furthermore, we check
  (by taking their Fourier transforms) that $\partial_{x_1} \Psi_{2, j, 2} =
  \partial_{x_2} \Psi_{2, j, 1}$ in the distribution sense. We infer that the
  integral of $\left(\begin{array}{c}
    \Psi_{2, j, 1}\\
    \Psi_{2, j, 2}
  \end{array}\right)$ on any bounded closed curve of $\mathbbm{R}^2$ is $0$.
  Indeed, taking $\chi_n$ a mollifier sequence, then $\chi_n \ast \Psi_{2, j,
  1}, \chi_n \ast \Psi_{2, j, 2} \in C^1 (\mathbbm{R}^2, \mathbbm{C})$,
  \[ \partial_{x_1} (\chi_n \ast \Psi_{2, j, 2}) - \partial_{x_2} (\chi_n \ast
     \Psi_{2, j, 1}) = \chi_n \ast (\partial_{x_1} \Psi_{2, j, 2} -
     \partial_{x_2} \Psi_{2, j, 1}) = 0, \]
  therefore, for any closed curve $\mathcal{C}$, the integral of the field
  $\left(\begin{array}{c}
    \chi_n \ast \Psi_{2, j, 1}\\
    \chi_n \ast \Psi_{2, j, 2}
  \end{array}\right)$ is $0$. Using $\chi_n \ast \Psi_{2, j, k} \rightarrow
  \Psi_{2, j, k}$ pointwise (by continuity of $\Psi_{2, j, k}$) and the
  domination
  \[ \| \chi_n \ast \Psi_{2, j, 1} \|_{L^{\infty} (\mathbbm{R}^2)} \leqslant
     \| \Psi_{2, j, 1} \|_{L^{\infty} (\mathbbm{R}^2)} < + \infty, \]
  we infer that this result holds for $\left(\begin{array}{c}
    \Psi_{2, j, 1}\\
    \Psi_{2, j, 2}
  \end{array}\right)$. We deduce, as for the proof of (\ref{P33429}), that
  \begin{equation}
    \int_{- \infty}^{+ \infty} \Psi_{2, j, 2} (x_1, y_2) d y_2 = 0.
    \label{P335299re}
  \end{equation}

  We then define for $(x_1, x_2) \in \mathbbm{R}^2, j \in \{ 1, 2 \}$,
  \[ \Psi_{2, j} (x_1, x_2) = \int_{+ \infty}^{x_2} \Psi_{2, j, 2} (x_1, y_2)
     d y_2, \]
  and if $x_2 < 0$, by (\ref{P335299re}),
  \[ \Psi_{2, j} (x_1, x_2) = \int_{- \infty}^{x_2} \Psi_{2, j, 2} (x_1, y_2)
     d y_2 . \]
  With arguments similar to the proof for $\Psi_{1, j}$, we check that
  $\Psi_{2, j} \in C^1 (\mathbbm{R}^2, \mathbbm{C})$ with $\partial_{x_k}
  \Psi_{2, j} = \Psi_{2, j, k}$,
  \[ | \Psi_{2, j} | \leqslant \frac{K (\sigma, \sigma') \| h \|_{\otimes
     \otimes, \sigma', \infty}}{(1 + \tilde{r})^{1 + \sigma}}, \]
  as well as
  \[ | \nabla \Psi_{2, j} | \leqslant \frac{K (\sigma, \sigma') \| h
     \|_{\otimes \otimes, \sigma', \infty}}{(1 + \tilde{r})^{2 + \sigma}} . \]
  Finally, since $\partial_{x_1} \Psi_{2, 2} = \Psi_{2, 2, 1} = \Psi_{2, 1, 2}
  = \partial_{x_2} \Psi_{2, 1} \in L^2 (\mathbbm{R}^2, \mathbbm{C})$ (by
  taking their Fourier transforms, it follows from $R_{j, k} = R_{k, j}$,
  $L_{j, k} = L_{k, j}$ and $\hat{K}_j \xi_k = \hat{K}_k \xi_j$), we have, as
  before, that
  \[ \int_{- \infty}^{+ \infty} \Psi_{2, 2} (x_1, y_2) d y_2 = 0. \]
  We define

  \[ \psi_2 (x_1, x_2) = \int_{+ \infty}^{x_2} \Psi_{2, 2} (x_1, y_2) d y_2,
  \]
  and thus, if $x_2 < 0$,
  \[ \psi_2 (x_1, x_2) = \int_{- \infty}^{x_2} \nobracket \Psi_{2, 2} (x_1,
     y_2 \nobracket) d y_2 . \]
  We check, as previously, by integration from infinity, that $\psi_2 \in C^2
  (\mathbbm{R}^2, \mathbbm{C})$, $\partial^2_{x_j x_k} \psi_2 = \Psi_{2, j,
  k}$, and
  \[ | \nabla^2 \psi_2 | \leqslant \frac{K (\sigma, \sigma') \| h \|_{\otimes
     \otimes, \sigma', \infty}}{(1 + \tilde{r})^{2 + \sigma}}, \]
  \[ | \nabla^{} \psi_2 | \leqslant \frac{K (\sigma, \sigma') \| h \|_{\otimes
     \otimes, \sigma', \infty}}{(1 + \tilde{r})^{1 + \sigma}}, \]
  as well as (if $\sigma > 0$)
  \[ | \psi_2 | \leqslant \frac{K (\sigma, \sigma') \| h \|_{\otimes \otimes,
     \sigma', \infty}}{(1 + \tilde{r})^{\sigma}} . \]
  Remark that if $h$ satisfies $\forall (x_1, x_2) \in \mathbbm{R}^2, h (x_1,
  x_2) = \overline{h (x_1, - x_2)}$, then by the definition of $\psi_1$ and
  $\psi_2$ above, for $\psi = \psi_1 + i \psi_2$, we have that $\forall (x_1,
  x_2) \in \mathbbm{R}^2, \psi (x_1, x_2) = \overline{\psi (x_1, - x_2)}$.
  Therefore, in the case $\sigma \in] - 1, 0 [$, since $\forall (x_1, x_2) \in
  \mathbbm{R}^2, \psi_2 (x_1, x_2) = - \psi_2 (x_1, - x_2)$, we have $\psi_2
  (x_1, 0) = 0$, and we integrate $\nabla \psi_2$ from the line $\{ x_2 = 0
  \}$ instead of infinity to show that $| \psi_2 | \leqslant \frac{K (\sigma,
  \sigma') \| h \|_{\otimes \otimes, \sigma', \infty}}{(1 +
  \tilde{r})^{\sigma}}$.
  
  \
  
  We deduce that, in either cases, $\psi = \psi_1 + i \psi_2 \in
  \mathcal{E}_{\otimes, \sigma}^{\infty}$, and it satisfies
  \[ \| \psi \|_{\otimes, \sigma, \infty} \leqslant K (\sigma, \sigma') \| h
     \|_{\otimes \otimes, \sigma', \infty} . \]
  Now, let us show that $- i c \partial_{x_2} \psi - \Delta \psi +
  2\mathfrak{R}\mathfrak{e} (\psi) = h$. From the computations at the
  beginning of subsection \ref{gpkernel}, we check that the Fourier transform
  (in the distribution sense) of both side of the equation are equals on $\{
  \xi \in \mathbbm{R}^2, \xi \neq 0 \}$ (remark that they are both in $L^p
  (\mathbbm{R}^2, \mathbbm{C})$ for some $p > 2$ large enough). This implies
  that
  \[ \tmop{Supp} \left( \widehat{- i c \partial_{x_2} \psi - \Delta \psi +
     2\mathfrak{R}\mathfrak{e} (\psi) - h} \right) \subset \{ 0 \}, \]
  and thus $- i c \partial_{x_2} \psi - \Delta \psi +
  2\mathfrak{R}\mathfrak{e} (\psi) - h = P \in \mathbbm{C} [X_1, X_2]$. With
  the decay estimates on $\psi$ and $h$, we check that $P$ is bounded and goes
  to $0$ at infinity (since $\sigma, \sigma' > - 1$), thus $P = 0$.
  
  Finally, if $\tilde{\psi} \in \mathcal{E}^{\infty}_{\otimes, \sigma}$
  satisfies $- i c \partial_{x_2} \tilde{\psi} - \Delta \tilde{\psi} +
  2\mathfrak{R}\mathfrak{e} (\tilde{\psi}) = h$, then $\psi - \tilde{\psi} \in
  C^2 (\mathbbm{R}^2, \mathbbm{C})$ and
  \[ (- i c \partial_{x_2} - \Delta + 2\mathfrak{R}\mathfrak{e}) (\psi -
     \tilde{\psi}) = 0. \]
  With the computations at the beginning of subsection \ref{gpkernel}, since
  $\psi - \tilde{\psi}$ is a tempered distribution, we check that $\tmop{Supp}
  \widehat{\psi - \tilde{\psi}} \subset \{ 0 \}$, therefore $\psi -
  \tilde{\psi} = P \in \mathbbm{C} [X_1, X_2]$. If $\sigma > 0$, since $\psi -
  \tilde{\psi}$ goes to $0$ at infinity, $P = 0$. If $\sigma \in] - 1, 0 [$,
  then $P = i \lambda$ for some $\lambda \in \mathbbm{R}$
  ($\mathfrak{R}\mathfrak{e} (\psi - \tilde{\psi}) \rightarrow 0$ at infinity
  and $\tilde{r}^{- \sigma} \mathfrak{I}\mathfrak{m} (\psi - \tilde{\psi})$ is
  bounded), and by the symmetry on $\psi, \tilde{\psi}$ we have in that case,
  $\lambda = 0$. This shows the uniqueness of a solution in
  $\mathcal{E}^{\infty}_{\otimes, \sigma}$ (with the symmetry if $\sigma \in]
  - 1, 0 [$), and thus concludes the proof of this lemma.
\end{proof}

\subsection{ Reduction of the problem}

\subsubsection{Inversion of the linearized operator}\label{2512809}

One of the key element in the inversion of the linearized operator is the
computation of the kernel for only one vortex. The kernel of the linearized
operator around one vortex has been studied in {\cite{DP}}, with the following
result.

\begin{theorem}[Theorem 1.2 of {\cite{DP}}]
  \label{delpino}Consider the linearized operator around one vortex of degree
  $\varepsilon = \pm 1$,
  \[ L_{V_{\varepsilon}} (\Phi) \assign - \Delta \Phi - (1 - | V_{\varepsilon}
     |^2) \Phi + 2\mathfrak{R}\mathfrak{e} (\overline{V_{\varepsilon}} \Phi)
     \overline{V_{\varepsilon}} . \]
  Suppose that
  \[ \| \Phi \|_{H_{V_{\varepsilon}}}^2 \assign \int_{\mathbbm{R}^2} | \nabla
     \Phi |^2 + (1 - | V_{\varepsilon} |^2) | \Phi |^2
     +\mathfrak{R}\mathfrak{e}^2 (\overline{V_{\varepsilon}} \Phi) < + \infty
  \]
  and
  \[ L_{V_{\varepsilon}} (\Phi) = 0. \]
  Then, there exist two constants $c_1, c_2 \in \mathbbm{R}$ such that
  \[ \Phi = c_1 \partial_{x_1} V_{\varepsilon} + c_2 \partial_{x_2}
     V_{\varepsilon} . \]
\end{theorem}

This result describes the kernel of $L_{V_{\varepsilon}}$ that will appear in
the proof of Proposition \ref{invertop}. It shows that the kernel in
$H_{V_{\varepsilon}} \assign \left\{ \Phi \in H_{\tmop{loc}}^1
(\mathbbm{R}^2), \| \Phi \|_{H_{V_{\varepsilon}}} < + \infty \right\}$
contains only the two elements we expect: $\partial_{x_1} V_{\varepsilon},
\partial_{x_2} V_{\varepsilon}$, which are due to the invariance by
translation of $(\tmop{GP})$. One of the directions will be killed by the
symmetry and the other one by the orthogonality.

Now, we shall invert the linear part $\eta L (\Phi) + (1 - \eta) V L' (\Psi)$.
We recall that $\Phi = V \Psi$. We first state an a priori estimate result. We
recall the definitions, for $\sigma, \sigma' \in] 0, 1 [$,
\[ \mathcal{E}_{\ast, \sigma, d} = \]
\[ \{ \Phi = V \Psi \in C^2 (\mathbbm{R}^2, \mathbbm{C}), \| \Psi \|_{\ast,
   \sigma, d} < + \infty ; \langle \Phi, Z_d \rangle = 0 ; \forall x \in
   \mathbbm{R}^2, \Psi (x_1, x_2) = \overline{\Psi (x_1, - x_2)} = \Psi (-
   x_1, x_2) \} \]
and
\[ \mathcal{E}_{\ast \ast, \sigma', d} = \{ V h \in C^1 (\mathbbm{R}^2,
   \mathbbm{C}), \| h \|_{\ast \ast, \sigma', d} < + \infty \}, \]
with
\begin{eqnarray*}
  \| \Psi \|_{\ast, \sigma, d} & = & \| V \Psi \|_{C^2 (\{ \tilde{r} \leqslant
  3 \})}\\
  & + & \| \tilde{r}^{1 + \sigma} \Psi_1 \|_{L^{\infty} (\{ \tilde{r}
  \geqslant 2 \})} + \| \tilde{r}^{2 + \sigma} \nabla \Psi_1 \|_{L^{\infty}
  (\{ \tilde{r} \geqslant 2 \})} + \| \tilde{r}^{2 + \sigma} \nabla^2 \Psi_1
  \|_{L^{\infty} (\{ \tilde{r} \geqslant 2 \})}\\
  & + & \| \tilde{r}^{\sigma} \Psi_2 \|_{L^{\infty} (\{ \tilde{r} \geqslant 2
  \})} + \| \tilde{r}^{1 + \sigma} \nabla \Psi_2 \|_{L^{\infty} (\{ \tilde{r}
  \geqslant 2 \})} + \| \tilde{r}^{2 + \sigma} \nabla^2 \Psi_2 \|_{L^{\infty}
  (\{ \tilde{r} \geqslant 2 \})},\\
  &  & \\
  \| h \|_{\ast \ast, \sigma', d} & = & \| V h \|_{C^1 (\{ \tilde{r} \leqslant
  3 \})}\\
  & + & \| \tilde{r}^{1 + \sigma'} h_1 \|_{L^{\infty} (\{ \tilde{r} \geqslant
  2 \})} + \| \tilde{r}^{2 + \sigma'} \nabla h_1 \|_{L^{\infty} (\{ \tilde{r}
  \geqslant 2 \})}\\
  & + & \| \tilde{r}^{2 + \sigma'} h_2 \|_{L^{\infty} (\{ \tilde{r} \geqslant
  2 \})} + \| \tilde{r}^{2 + \sigma'} \nabla h_2 \|_{L^{\infty} (\{ \tilde{r}
  \geqslant 2 \})} .
\end{eqnarray*}
\begin{proposition}
  \label{invertop}For $1 > \sigma' > \sigma > 0$, consider the problem, in the
  distribution sense
  \[ \left\{ \begin{array}{l}
       \eta L (\Phi) + (1 - \eta) V L' (\Psi) = V h\\
       \Phi = V \Psi \in \mathcal{E}_{\ast, \sigma}, V h \in \mathcal{E}_{\ast
       \ast, \sigma'} .
     \end{array} \right. \]
  Then, there exist constants $c_0 (\sigma, \sigma') > 0$ small and $C
  (\sigma, \sigma') > 0$ depending only on $\sigma$ and $\sigma'$, such that,
  for any solution of this problem with $0 < c \leqslant c_0 (\sigma,
  \sigma')$, $\frac{1}{2} < c d < 2$, it holds
  \[ \| \Psi \|_{\ast, \sigma, d} \leqslant C (\sigma, \sigma') \| h \|_{\ast
     \ast, \sigma', d} . \]
\end{proposition}

\begin{proof}
  This proof is similar to the ones done in {\cite{DP2}} for the inversion of
  their linearized operator. The main difference is that we have a transport
  term. Fix $1 > \sigma' > \sigma > 0$. We argue by contradiction. Suppose
  that for given $1 > \sigma' > \sigma > 0$, there is no threshold $c_0
  (\sigma, \sigma') > 0$ such that, if $0 < c \leqslant c_0 (\sigma, \sigma')$
  we have $\| \Psi \|_{\ast, \sigma, d} \leqslant C (\sigma, \sigma') \| h
  \|_{\ast \ast, \sigma', d}$. We can then find a sequence of $c_n \rightarrow
  0$ (and so $d_n \rightarrow \infty$), functions $\Phi_n = V \Psi_n \in
  \mathcal{E}_{\ast, \sigma}$ and $V h_n \in \mathcal{E}_{\ast \ast, \sigma'}$
  solutions of the problem and such that
  \[ \| \Psi_n \|_{\ast, \sigma, d_n} = 1 \nocomma \]
  and
  \[ \| h_n \|_{\ast \ast, \sigma', d_n} \rightarrow 0. \]
  We look in the region $\Sigma \assign \{ x_1 \geqslant 0 \}$ thanks to the
  symetry $\Psi (x_{1,} x_2) = \Psi (- x_1, x_2)$. The orthogonality condition
  of $\mathcal{E}_{\ast, \sigma}$ becomes $2\mathfrak{R}\mathfrak{e}
  \int_{\Sigma} \overline{\Phi_n} Z_{d_n} = 0$.
  
\
  
  \begin{tmindent}
    Step 1.  Inner estimates.
  \end{tmindent}

\  
  
  The problem can be written (using $V L' (\Psi_n) = - (E - i c_n
  \partial_{x_2} V) \Psi_n + L (\Phi_n)$ from Lemma \ref{lemma7}) as
  \[ V h_n = L (\Phi_n) - (1 - \eta) (E - i c_n \partial_{x_2} V) \Psi_n . \]
  First, we recall that $V$ and $E$ are depending on $n$. The sequence
  $(\Phi_n (. + d_n \overrightarrow{e_1}))_{n \in \mathbbm{N}}$ is
  equicontinuous and bounded ($1 = \| \Psi_n \|_{\ast, \sigma, d}$ controls
  $\Phi_n$ and its derivatives in $L^{\infty} (\mathbbm{R}^2)$ uniformly in
  $n$).
  
  Such a function $\Phi_n$, as a solution of
  \begin{equation}
    \Delta \Phi_n = - (1 - | V |^2) \Phi_n + 2\mathfrak{R}\mathfrak{e}
    (\bar{V} \Phi_n) V - i c \partial_{x_2} \Phi_n - (1 - \eta) (E - i c_n
    \partial_{x_2} V) \Psi_n - V h_n \label{jaune}
  \end{equation}
  in the distribution sense, by Theorem 8.8 of {\cite{MR1814364}} is
  $H^2_{\tmop{loc}} (\mathbbm{R}^2)$ (since the right hand side is $C^0
  (\mathbbm{R}^2)$). Furthermore, still by Theorem 8.8 of {\cite{MR1814364}},
  we have, for $x \in \mathbbm{R}^2$,
  \[ \| \Phi_n \|_{H^2 (B (x, 1))} \leqslant K (\| \Phi_n \|_{H^1 (B (x, 2))}
     + \| \Delta \Phi_n \|_{L^2 (B (x, 2))}) . \]
  By $\| \Psi_n \|_{\ast, \sigma, d} = 1$, the quantities $\| \Phi_n
  \|_{L^{\infty} (B (x, 2))}$, $\| \nabla \Phi_n \|_{L^{\infty} (B (x, 2))}$
  and $\| \Delta \Phi_n \|_{L^{\infty} (B (x, 2))}$ are bounded by a constant
  independent of $n$. Therefore, $(\Phi_n)_{n \in \mathbbm{N}}$ is bounded in
  $H^2_{\tmop{loc}} (\mathbbm{R}^2)$.
  
  We deduce, by compact embedding, that there exists a function
  $\Phi_{\nosymbol}$ such that $\Phi_n (. + d_n \overrightarrow{e_1})
  \rightarrow \Phi$ in $H^1_{\tmop{loc}} (\mathbbm{R}^2)$ (up to a
  subsequence).
  
  \
  
  Now, since $L (\Phi_n) = - \Delta \Phi_n - (1 - | V |^2) \Phi_n +
  2\mathfrak{R}\mathfrak{e} (\bar{V} \Phi_n) V - i c \partial_{x_2} \Phi_n$,
  we have, in the weak sense,
  \[ \Delta \Phi_n + V h_n = - (1 - | V |^2) \Phi_n +
     2\mathfrak{R}\mathfrak{e} (\bar{V} \Phi_n) V - i c_n \partial_{x_2}
     \Phi_n - (1 - \eta) (E - i c_n \partial_{x_2} V) \Psi_n, \]
  therefore $\Delta \Phi_n (. + d_n \overrightarrow{e_1}) + V h_n (. + d_n
  \overrightarrow{e_1})$ is equicontinuous and bounded uniformly and then, by
  Ascoli's Theorem, up to a subsequence converges to a limit $l$ in
  $\mathcal{C}^0_{\tmop{loc}} (\mathbbm{R}^2)$. Since $V h_n (. + d_n
  \overrightarrow{e_1}) \rightarrow 0$ in $\mathcal{C}^0_{\tmop{loc}}
  (\mathbbm{R}^2)$ by $\| h_n \|_{\ast \ast, \sigma', d} \rightarrow 0$ and
  $\Delta \Phi_n (. + d_n \overrightarrow{e_1}) \rightarrow \Delta \Phi$ in
  the distribution sense, this limit must be $\Delta \Phi$ (in the $H^{-
  1}_{\tmop{loc}} (\mathbbm{R}^2)$ sense).
  
  \
  
  We have locally uniformly that $V h_n (. + d_n \overrightarrow{e_1})
  \rightarrow 0$ because $\| h_n \|_{\ast \ast, \sigma', d} \rightarrow 0$ and
  $| V | \leqslant 1$, and we have, from Lemma \ref{infest}, that $E (y + d_n
  \overrightarrow{e_1}) \rightarrow 0$ and $V (y + d_n \overrightarrow{e_1})
  \rightarrow V_1 (y)$ when $n \rightarrow \infty$ locally uniformly. Lastly,
  $\partial_{x_2} \Phi_n$ and $(1 - \eta) \partial_{x_2} V \Psi_n$ are
  uniformly bounded in $\mathbbm{R}^2$ independently of $n$. Therefore when we
  take the locally uniform limit when $d_n \rightarrow \infty$ in
  \[ (V h_n) \nobracket (y + d_n \overrightarrow{e_1} \nobracket) = (L
     (\Phi_n)) (y + d_n \overrightarrow{e_1}) - ((1 - \eta) (E - i c_n
     \partial_{x_2} V) \Psi_n) (y + d_n \overrightarrow{e_1}), \]
  we have (in the distribution sense)
  \[ L_{V_1} (\Phi) = 0. \]
  Using $\partial_d V (. + d \overrightarrow{e_1}) \rightarrow -
  \partial_{x_1} V_1 (.)$ locally uniformly from Lemma \ref{infest}, we show
  that
  \[ 0 = 2\mathfrak{R}\mathfrak{e} \int_{\Sigma} \overline{\Phi_n} Z_d
     \rightarrow 2 \langle \nobracket \Phi | \nobracket \tilde{\eta} (. / 4)
     \partial_{x_1} V_1 \rangle \nobracket \]
  since $Z_d$ is compactly supported around $0$ when we take the equation in
  $y + d_n \overrightarrow{e_1}$. The problem at the limit $n \rightarrow
  \infty$ becomes (in the $H^{- 1}_{\tmop{loc}} (\mathbbm{R}^2)$ sense)
  \[ \left\{ \begin{array}{l}
       L_{V_1} (\Phi) = 0\\
       \langle \nobracket \Phi \left| \tilde{\eta} \left( \frac{.}{4} \right)
       \partial_{x_1} V_1 \rangle \nobracket = 0 \right.,
     \end{array} \right. \]
  with $\Phi = V_1 \Psi$ (since $V (y + d \overrightarrow{e_1}) \rightarrow
  V_1 (y)$ from Lemma \ref{infest}).
  
  Let us show that $\| \Phi \|_{H_{V_1}} < + \infty$. For that, we will show
  that
  \[ \int_{B (d_n \overrightarrow{e_1}, d_n^{1 / 2})} | \nabla \Phi_n |^2 +
     \frac{| \Phi_n |^2}{(1 + r_1)^2} +\mathfrak{R}\mathfrak{e}^2
     (\overline{V_1} (. - d_n \overrightarrow{e_1}) \Phi_n) \leqslant K
     (\sigma), \]
  where $K (\sigma) > 0$ is independent of $n$, which shall imply (by Lemma
  \ref{nonmodV})
  \[ \| \Phi \|^2_{H_{V_1}} \leqslant \limsup_{n \rightarrow \infty} \int_{B
     (d_n \overrightarrow{e_1}, d_n^{1 / 2})} | \nabla \Phi_n |^2 + \frac{|
     \Phi_n |^2}{(1 + r_1)^2} +\mathfrak{R}\mathfrak{e} (\overline{V_1} (. -
     d_n \overrightarrow{e_1}) \Phi_n)^2 \leqslant K (\sigma) < + \infty . \]
  First, $\Phi_n \in C^2 (\mathbbm{R}^2)$ hence $\Phi_n \in H_{\tmop{loc}}^1
  (\mathbbm{R}^2)$. We have
  \[ | \nabla \Phi_n |^2 \leqslant 2 | \nabla V_1 |^2 | \Psi_n |^2 + 2 |
     \nabla \Psi_n |^2 | V_1 |^2, \]
  with $| \nabla V_1 |^2 = O_{r_1 \rightarrow \infty} \left( \frac{1}{r_1^2}
  \right) \tmop{by} \tmop{Lemma} \ref{dervor}, \tmop{and}, \tmop{in} B (d_n
  \overrightarrow{e_1}, d_n^{1 / 2}), | \Psi_n |^2 \leqslant \frac{C}{(1 +
  r_1)^{2 \sigma}} \| \Psi_n \|_{\ast, \sigma, d}^2, | \nabla \Psi_n |^2
  \leqslant \frac{C}{(1 + r_1)^{2 + 2 \sigma}} \| \Psi_n \|_{\ast, \sigma,
  d}^2$. Therefore since $\| \Psi_n \|_{\ast, \sigma, d} \leqslant 1$,
  \[ \int_{B (d_n \overrightarrow{e_1}, d_n^{1 / 2})} | \nabla \Phi_n |^2
     \leqslant \int_{B (d_n \overrightarrow{e_1}, d_n^{1 / 2})} \frac{K}{(1 +
     r_1)^{2 + 2 \sigma}} \leqslant K (\sigma) . \]
  In addition, in $B (d_n \overrightarrow{e_1}, d_n^{1 / 2})$, $| \Phi_n |^2 =
  | V_1 |^2 | \Psi_n |^2 \leqslant \frac{K}{(1 + r_1)^{2 \sigma}} \| \Psi_n
  \|_{\ast, \sigma, d_n}^2$ hence $\frac{| \Phi_n |^2}{(1 + r_1)^2} \leqslant
  \frac{K}{(1 + r_1)^{2 + 2 \sigma}}$ and
  \[ \int_{B (d_n \overrightarrow{e_1}, d_n^{1 / 2})} \frac{| \Phi_n |^2}{(1 +
     r_1)^2} \leqslant \int_{B (d_n \overrightarrow{e_1}, d_n^{1 / 2})}
     \frac{K}{(1 + r_1)^{2 + 2 \sigma}} \leqslant K (\sigma) . \]
  Lastly, still in $B (d_n \overrightarrow{e_1}, d_n^{1 / 2})$, by Lemma
  \ref{nonmodV},
  \[ \mathfrak{R}\mathfrak{e} (\overline{V_1} \Phi_n)^2 = | V_1 |^4
     \mathfrak{R}\mathfrak{e} (V_{- 1} \Psi_n)^2 \leqslant | V_1 |^4
     (\mathfrak{R}\mathfrak{e} (\Psi_n)^2 + (1 - | V_{- 1} |^2) | \Psi_n |^2)
     \leqslant \frac{K}{(1 + r_1)^{2 + 2 \sigma}}, \]
  giving the same result. We then have $\| \Phi \|_{H_{V_1}} < + \infty$,
  therefore, we can apply Theorem \ref{delpino}. We deduce that
  \[ \Phi = c_1 \partial_{x_1} V_1 + c_2 \partial_{x_2} V_1 \]
  for some constants $c_1, c_2 \in \mathbbm{R}$.
  
  \
  
  Since $\forall x \in \mathbbm{R}^2, \Psi_n (x_{1,} x_2) = \overline{\Psi_n
  (x_1, - x_2)}$, we have $\forall y \in \mathbbm{R}^2, \Phi (y_1, y_2) =
  \overline{\Phi (y_1, - y_2)}$. The function $\partial_{x_1} V_1$ enjoys also
  this symmetry, therefore so does $c_2 \partial_{x_2} V_1$. It is possible
  only if $c_2 = 0$. The orthogonality condition then imposes
  \[ c_1  \int_{\Sigma \nosymbol} | \partial_{x_1} V_1 (y) |^2 \tilde{\eta}
     \left( \frac{y}{4} \right) d y = 0, \]
  implying that $c_1 = 0$. Hence
  \[ \Phi_n (. + d_n \overrightarrow{e_1}) \rightarrow 0 \]
  $\tmop{in} \mathcal{C}_{\tmop{loc}}^1 (\mathbbm{R}^2)$. By equation
  (\ref{jaune}) and standard elliptic estimates, this convergence also hold in
  $\mathcal{C}_{\tmop{loc}}^2 (\mathbbm{R}^2)$. The same proof works for the
  $z$ coordinate (around the center of the $- 1$ vortex). As a consequence,
  for any $R > 0$, we have
  \begin{equation}
    \| \Phi_n \|_{L^{\infty} (\{ \tilde{r}_{\nosymbol} \leqslant R \})} + \|
    \nabla \Phi_n \|_{L^{\infty} (\{ \tilde{r} \leqslant R \})} + \| \nabla^2
    \Phi_n \|_{L^{\infty} (\{ \tilde{r} \leqslant R \})}
    \rightarrow_{\nosymbol} 0 \label{ie}
  \end{equation}
  as $n \rightarrow \infty$. With this result, to obtain a contradiction
  (which will be $\| \Psi_n \|_{\ast, \sigma, d} \rightarrow 0$) we still need
  to have estimates near the infinity in space.
  
\
  
  \begin{tmindent}
    Step 2.  Outer computations.
  \end{tmindent}
  
 \ 
  
  Thanks to the previous step, we can take a cutoff to look only at the
  infinity in space. For $R \geqslant 4$, we define $\chi_R$ a smooth cutoff
  function with value $\chi_R (x) = 1$ if $\tilde{r} \geqslant R$ and $\chi_R
  (x) = 0$ if $\tilde{r} \leqslant \frac{R}{2}$, with $| \nabla \chi_R |
  \leqslant \frac{4}{R}$. We then define
  \[ \tilde{\Psi}_n \assign \chi_R \Psi_n, \]
  \[ \tilde{h}_n \assign \chi_R h_n \]
  and we choose $\chi_R$ such that $\tilde{\Psi}_n$ and $\tilde{h}_n$ enjoy
  the same symmetries than $\Psi_n$ and $h_n$ respectively. We compute on
  $\mathbbm{R}^2 \backslash (B (d_n \overrightarrow{e_1}, R) \cup B (- d_n
  \overrightarrow{e_1}, R))$:
  \[ \nabla \tilde{\Psi}_n = \nabla \chi_R \Psi_n + \chi_R \nabla \Psi_n =
     \nabla \Psi_n, \]
  \[ \Delta \tilde{\Psi}_n = \Delta \chi_R \Psi_n + 2 \nabla \chi_R \nabla
     \Psi_n + \chi_R \Delta \Psi_n = \Delta \Psi_n . \]
  We deduce that $\tilde{\Psi}_n \in \mathcal{E}^{\infty}_{\otimes, \sigma}$
  and $\tilde{h}_n \in \mathcal{E}^{\infty}_{\otimes \otimes, \sigma'}$ by
  (\ref{P3zelda}), since $\tilde{\Psi}_n \in C^2 (B (d_n \overrightarrow{e_1},
  R) \cup B (- d_n \overrightarrow{e_1}, R), \mathbbm{C})$, $\tilde{h}_n \in
  C^1 (B (d_n \overrightarrow{e_1}, R) \cup B (- d_n \overrightarrow{e_1}, R),
  \mathbbm{C})$ and, outside of $B (d_n \overrightarrow{e_1}, R) \cup B (- d_n
  \overrightarrow{e_1}, R)$, $\tilde{\Psi}_n = \Psi_n$ with $\| \Psi_n
  \|_{\ast, \sigma, d_n} = 1$, as well as $\tilde{h}_n = h_n$, with $\| h_n
  \|_{\ast \ast, \sigma', d_n} \rightarrow 0$ when $n \rightarrow \infty$. In
  particular,
  \[ \| \tilde{h}_n \|_{\otimes \otimes, \sigma', \infty} = o_{n \rightarrow
     \infty}^R (1), \]
  where $o_{n \rightarrow \infty}^R (1)$ is a sequence that, for fixed $R
  \geqslant 4$, goes to $0$ when $n \rightarrow \infty$ (it also depends on
  $\sigma$ and $\sigma'$).
  
  \
  
  Since $\chi_R = 1$ on $\mathbbm{R}^2 \backslash (B (d_n
  \overrightarrow{e_1}, R) \cup B (- d_n \overrightarrow{e_1}, R))$, we have
  there $L' (\tilde{\Psi}_n) = \tilde{h}_n$. Therefore, we can write that in
  $\mathbbm{R}^2$ that $L' (\tilde{\Psi}_n) = \tilde{h}_n + \tmop{Loc}
  (\Psi_n)$, with
  \[ \tmop{Loc} (\Psi_n) \assign - \frac{\chi_R \eta}{V} L (V \Psi_n) + (1 -
     \eta) (L' (\chi_R \Psi_n) - \chi_R L' (\Psi_n)), \]
  a term that is supported in $\mathbbm{R}^2 \backslash (B (d_n
  \overrightarrow{e_1}, R) \cup B (- d_n \overrightarrow{e_1}, R))$. By
  (\ref{ie}) and $\| h_n \|_{\ast \ast, \sigma', d_n} \rightarrow 0$ when $n
  \rightarrow \infty$, it satisfies
  \begin{eqnarray*}
    &  & \| \tmop{Loc} (\Psi_n) \|_{\otimes \otimes, \sigma', \infty}\\
    & \leqslant & K (R) \| \tmop{Loc} (\Psi_n) \|_{C^1 (\mathbbm{R}^2
    \backslash (B (d_n \overrightarrow{e_1}, R) \cup B (- d_n
    \overrightarrow{e_1}, R)))}\\
    & \leqslant & K (R) \| \Phi_n \|_{C^2 (\mathbbm{R}^2 \backslash (B (d_n
    \overrightarrow{e_1}, R) \cup B (- d_n \overrightarrow{e_1}, R)))}\\
    & = & o_{n \rightarrow \infty}^R (1) .
  \end{eqnarray*}
  We recall that $L' (\Psi) = - \Delta \Psi - 2 \frac{\nabla V}{V} . \nabla
  \Psi + 2 | V |^2 \mathfrak{R}\mathfrak{e} (\Psi) - i c \partial_{x_2} \Psi$,
  therefore
  \begin{equation}
    - \Delta \tilde{\Psi}_n - i c \partial_{x_2} \tilde{\Psi}_n +
    2\mathfrak{R}\mathfrak{e} (\tilde{\Psi}_n) = \tilde{h}_n + \tmop{Loc}
    (\Psi_n) + 2 \frac{\nabla V}{V} . \nabla \tilde{\Psi}_n + 2 (1 - | V |^2)
    \mathfrak{R}\mathfrak{e} (\tilde{\Psi}_n) \label{systss} .
  \end{equation}
  We define
  \[ \tilde{h}_n' \assign \tilde{h}_n + \tmop{Loc} (\Psi_n) + 2 \frac{\nabla
     V}{V} . \nabla \tilde{\Psi}_n + 2 (1 - | V |^2) \mathfrak{R}\mathfrak{e}
     (\tilde{\Psi}_n) . \]
  Let us show that $\tilde{h}'_n \in \mathcal{E}^{\infty}_{\otimes \otimes,
  \sigma'}$ with
  \[ \| \tilde{h}_n' \|_{\otimes \otimes, \sigma', \infty} \leqslant
     \nobracket o^R_{n \rightarrow \infty} (1) + o_{R \rightarrow \infty} (1
     \nobracket), \]
  where $o_{R \rightarrow \infty} (1)$ is a quantity that goes to $0$ when $R
  \rightarrow \infty$ (in particular, independently of $n$). By Lemma
  \ref{P3starest}, (the condition $\int_{\mathbbm{R}^2}
  \mathfrak{I}\mathfrak{m} (\tilde{h}_n') = 0$ is a consequence of the
  symmetries on $\tilde{h}_n$ and $\tilde{\Psi}_n$), this would imply, with
  equation (\ref{systss}) (and since $\tilde{\Psi}_n \in
  \mathcal{E}^{\infty}_{\otimes, \sigma}$), that
  \begin{equation}
    \| \tilde{\Psi} \|_{\otimes, \sigma, \infty} \leqslant \nobracket o^R_{n
    \rightarrow \infty} (1) + o_{R \rightarrow \infty} (1 \nobracket) .
    \label{shuffle2}
  \end{equation}
  This estimate has already been done for the terms $\tmop{Loc} (\Psi_n)$ and
  $\tilde{h}_n$. Therefore, we only have to check that
  \[ \left\| 2 \frac{\nabla V}{V} . \nabla \tilde{\Psi}_n + 2 (1 - | V |^2)
     \mathfrak{R}\mathfrak{e} (\tilde{\Psi}_n) \right\|_{\otimes \otimes,
     \sigma', \infty} \leqslant \nobracket o^R_{n \rightarrow \infty} (1) +
     o_{R \rightarrow \infty} (1 \nobracket) . \]

  First, remark that the term $(1 - | V |^2) \mathfrak{R}\mathfrak{e}
  (\tilde{\Psi}_n)$ is real-valued. By Lemma \ref{nonmodV},
  \[ | 1 - | V |^2 | + \nabla (| V |^2) \leqslant \frac{K}{(1 + \tilde{r})^2},
  \]
  and with (\ref{ie}), $\tilde{\Psi}_n = \Psi_n$ in $\{ \tilde{r} \geqslant R
  \}$, $\| \Psi_n \|_{\ast, \sigma} = 1$, $0 < \sigma < \sigma' < 1$,
  \begin{eqnarray*}
    &  & \| (1 + \tilde{r})^{1 + \sigma'} (1 - | V |^2)
    \mathfrak{R}\mathfrak{e} (\tilde{\Psi}_n) \|_{L^{\infty}
    (\mathbbm{R}^2)}\\
    & \leqslant & o_{n \rightarrow \infty}^R (1) + K \left\| \frac{(1 +
    \tilde{r})^{1 + \sigma'}}{(1 + \tilde{r})^{3 + \sigma}}
    \right\|_{L^{\infty} (\{ \tilde{r} \geqslant R \})}\\
    & \leqslant & \nobracket o^R_{n \rightarrow \infty} (1) + o_{R
    \rightarrow \infty} (1 \nobracket)
  \end{eqnarray*}
  and
  \begin{eqnarray*}
    &  & \| (1 + \tilde{r})^{2 + \sigma'} \nabla ((1 - | V |^2)
    \mathfrak{R}\mathfrak{e} (\tilde{\Psi})) \|_{L^{\infty} (\mathbbm{R}^2)}\\
    & \leqslant & \| (1 + \tilde{r})^{2 + \sigma'} \nabla (\nobracket | V
    |^2) \mathfrak{R}\mathfrak{e} (\tilde{\Psi} \nobracket) \|_{L^{\infty}
    (\mathbbm{R}^2)} + \| (1 + \tilde{r})^{2 + \sigma'} \nobracket (1 - | V
    |^2) \mathfrak{R}\mathfrak{e} (\nabla \tilde{\Psi} \nobracket)
    \|_{L^{\infty} (\mathbbm{R}^2)}\\
    & \leqslant & o_{n \rightarrow \infty}^R (1) + K \left( \left\| \frac{(1
    + \tilde{r})^{2 + \sigma'}}{(1 + \tilde{r})^{3 + \sigma}}
    \right\|_{L^{\infty} (\{ \tilde{r} \geqslant R \})} \right)\\
    & \leqslant & \nobracket o^R_{n \rightarrow \infty} (1) + o_{R
    \rightarrow \infty} (1 \nobracket) .
  \end{eqnarray*}
  This concludes the proof of
  \[ \| 2 (1 - | V |^2) \mathfrak{R}\mathfrak{e} (\tilde{\Psi}_n) \|_{\otimes
     \otimes, \sigma', \infty} \leqslant \nobracket o^R_{n \rightarrow \infty}
     (1) + o_{R \rightarrow \infty} (1 \nobracket) . \]
  Now, we compute
  \[ \frac{\nabla V}{V} (x) = \frac{\nabla V_1}{V_1} (y) + \frac{\nabla V_{-
     1}}{V_{- 1}} (z), \]
  and recall, by Lemma \ref{lemme3}, that $\nabla V_{\varepsilon} (x) = i
  \varepsilon V_{\varepsilon} (x) \frac{x^{\bot}}{| x |^2} + O \left(
  \frac{1}{r^3} \right)$ for $\varepsilon = \pm 1$. We deduce that, far from
  the vortices (for instance on $\mathbbm{R}^2 \backslash (B (d
  \overrightarrow{e_1}, 4) \cup B (- d \overrightarrow{e_1}, 4))$), we have
  \[ \frac{\nabla V}{V} (x) = i \left( \frac{y^{\bot}}{r_1^2} -
     \frac{z^{\bot}}{r_{- 1}^2} \right) + O_{r_1 \rightarrow \infty} \left(
     \frac{1}{r_1^3} \right) + O_{r_{- 1} \rightarrow \infty} \left(
     \frac{1}{r_{- 1}^3} \right) . \]
  In particular, the first order of $\frac{\nabla V}{V}$ is purely imaginary,
  and the next term is of order $\frac{1}{r_1^3} + \frac{1}{r_{- 1}^3}$. We
  check in particular, using Lemma \ref{nonmodV}, that on $\mathbbm{R}^2
  \backslash (B (d \overrightarrow{e_1}, 4) \cup B (- d \overrightarrow{e_1},
  4))$,
  \begin{eqnarray*}
    &  & \left| (1 + \tilde{r}) \mathfrak{I}\mathfrak{m} \left( \frac{\nabla
    V}{V} \right) \right| + \left| (1 + \tilde{r})^3 \mathfrak{R}\mathfrak{e}
    \left( \frac{\nabla V}{V} \right) \right|\\
    & + & \left| (1 + \tilde{r})^2 \nabla \mathfrak{I}\mathfrak{m} \left(
    \frac{\nabla V}{V} \right) \right| + \left| (1 + \tilde{r})^3 \nabla
    \mathfrak{R}\mathfrak{e} \left( \frac{\nabla V}{V} \right) \right|\\
    & \leqslant & K.
  \end{eqnarray*}
  Therefore, with $R \geqslant 4$, equation (\ref{ie}), $\tilde{\Psi}_n =
  \Psi_n$ in $\{ \tilde{r} \geqslant R \}$, $\| \Psi_n \|_{\ast, \sigma} = 1$
  and $0 < \sigma < \sigma' < 1$,
  \begin{eqnarray*}
    &  & \left\| (1 + \tilde{r})^{1 + \sigma'} \mathfrak{R}\mathfrak{e}
    \left( \frac{\nabla V}{V} . \nabla \tilde{\Psi} \right)
    \right\|_{L^{\infty} (\mathbbm{R}^2)}\\
    & \leqslant & o^R_{n \rightarrow \infty} (1) + K \left\| \frac{(1 +
    \tilde{r})^{1 + \sigma'}}{(1 + \tilde{r})^{2 + \sigma}}
    \right\|_{L^{\infty} (\{ \tilde{r} \geqslant R \})}\\
    & \leqslant & \nobracket o^R_{n \rightarrow \infty} (1) + o_{R
    \rightarrow \infty} (1 \nobracket),
  \end{eqnarray*}
  \begin{eqnarray*}
    &  & \left\| (1 + \tilde{r})^{2 + \sigma'} \nabla
    \mathfrak{R}\mathfrak{e} \left( \frac{\nabla V}{V} . \nabla \tilde{\Psi}
    \right) \right\|_{L^{\infty} (\mathbbm{R}^2)}\\
    & \leqslant & o^R_{n \rightarrow \infty} (1) + K \left\| \frac{(1 +
    \tilde{r})^{2 + \sigma'}}{(1 + \tilde{r})^{3 + \sigma}}
    \right\|_{L^{\infty} (\{ \tilde{r} \geqslant R \})}\\
    & \leqslant & \nobracket o^R_{n \rightarrow \infty} (1) + o_{R
    \rightarrow \infty} (1 \nobracket),
  \end{eqnarray*}
  \begin{eqnarray*}
    &  & \left\| (1 + \tilde{r})^{2 + \sigma'} \mathfrak{I}\mathfrak{m}
    \left( \frac{\nabla V}{V} . \nabla \tilde{\Psi} \right)
    \right\|_{L^{\infty} (\mathbbm{R}^2)}\\
    & \leqslant & \left\| (1 + \tilde{r})^{2 + \sigma'}
    \mathfrak{I}\mathfrak{m} \left( \frac{\nabla V}{V} \right)
    .\mathfrak{R}\mathfrak{e} (\nabla \tilde{\Psi}) \right\|_{L^{\infty}
    (\mathbbm{R}^2)} + \left\| (1 + \tilde{r})^{2 + \sigma'}
    \mathfrak{R}\mathfrak{e} \left( \frac{\nabla V}{V} \right)
    .\mathfrak{I}\mathfrak{m} (\nabla \tilde{\Psi}) \right\|_{L^{\infty}
    (\mathbbm{R}^2)}\\
    & \leqslant & o^R_{n \rightarrow \infty} (1) + K \left\| \frac{(1 +
    \tilde{r})^{2 + \sigma'}}{(1 + \tilde{r})^{3 + \sigma}}
    \right\|_{L^{\infty} (\{ \tilde{r} \geqslant R \})} + K \left\| \frac{(1 +
    \tilde{r})^{2 + \sigma'}}{(1 + \tilde{r})^{3 + \sigma}}
    \right\|_{L^{\infty} (\{ \tilde{r} \geqslant R \})}\\
    & \leqslant & \nobracket o^R_{n \rightarrow \infty} (1) + o_{R
    \rightarrow \infty} (1 \nobracket),
  \end{eqnarray*}
  and, with a similar decomposition,
  \[ \left\| (1 + \tilde{r})^{2 + \sigma'} \nabla \mathfrak{I}\mathfrak{m}
     \left( \frac{\nabla V}{V} . \nabla \tilde{\Psi} \right)
     \right\|_{L^{\infty} (\mathbbm{R}^2)} \leqslant \nobracket o^R_{n
     \rightarrow \infty} (1) + o_{R \rightarrow \infty} (1 \nobracket) . \]
  This conclude the proof of $\left\| 2 \frac{\nabla V}{V} . \nabla
  \tilde{\Psi} \right\|_{\otimes \otimes, \sigma', \infty} \leqslant
  \nobracket o^R_{n \rightarrow \infty} (1) + o_{R \rightarrow \infty} (1
  \nobracket),$ and thus of (\ref{shuffle2}).
 
 \ 
  
  \begin{tmindent}
    Step 3.  Conclusion.
  \end{tmindent}
 
 \ 
  
  We have $\| \Psi_n \|_{\ast, \sigma, d_n} \leqslant K (R) \| \Phi_n
  \|_{C^{^2} (\{ \tilde{r}_{\nosymbol} \leqslant R \})} + K \| \tilde{\Psi}_n
  \|_{\ast, \sigma, d_n}$ by (\ref{P3zelda}), therefore, with equations
  (\ref{ie}) and (\ref{shuffle2}),
  \[ \| \Psi_n \|_{\ast, \sigma, d_n} \leqslant \nobracket o^R_{n \rightarrow
     \infty} (1) + o_{R \rightarrow \infty} (1 \nobracket) . \]
  If we take $R$ large enough (depending on $\sigma, \sigma'$) so that $o_{R
  \rightarrow \infty} (1) \leqslant 1 / 10$ and then $n$ large enough
  (depending on $R, \sigma$ and $\sigma'$) so that $o^R_{n \rightarrow \infty}
  (1) \leqslant 1 / 10$, we have, for $n$ large, $\| \Psi_n \|_{\ast, \sigma,
  d_n} \leqslant 1 / 5$, which is in contradiction with
  \[ \| \Psi_n \|_{\ast, \sigma, d_n} = 1. \]
\end{proof}

\subsubsection{Existence of a solution}

At this point, we do not have existence of a solution to the linear problem
\[ \left\{ \begin{array}{l}
     \eta L (\Phi) + (1 - \eta) V L' (\Psi) = V h\\
     \Phi \in \mathcal{E}_{\ast, \sigma}, V h \in \mathcal{E}_{\ast \ast,
     \sigma'},
   \end{array} \right. \]
only an a priori estimate. The existence of a solution is done in Proposition
\ref{p2191802}, its proof being the purpose of this subsection. In
{\cite{DP2}}, the existence proof is done using mainly the fact that the
domain is bounded. We provide here a proof of existence by approximation on
balls of large radii for a particular Hilbertian norm. Given $c > 0$ and $a >
10 / c^2$, we define
\[ H_a \assign \]
\[ \left\{ \Phi = Q_c \Psi \in H^1_{\tmop{loc}} (B (0, a)), \| \Phi \|^2_{H_a}
   \assign \| \Phi \|^2_{H^1 (\{ \tilde{r} \leqslant 3 \})} + \int_{\{
   \tilde{r} \geqslant 2 \} \cap \{ r \leqslant a \}} | \nabla \Psi |^2
   +\mathfrak{R}\mathfrak{e}^2 (\Psi) + \frac{\mathfrak{I}\mathfrak{m}^2
   (\Psi)}{(1 + r)^{5 / 2}} \right\}, \]
and we also allow $a = + \infty$. We first state a result on functions in
$H_{\infty}$.

\begin{lemma}
  \label{funky}There exists $c_0 > 0$ such that, for $0 < c < c_0$, $0 <
  \sigma < \sigma' < 1$, $V h \in \mathcal{E}_{\ast \ast, \sigma'}$, if a
  function $\Phi \in H_{\infty} \cap C^1 (\mathbbm{R}^2)$ satisfies, in the
  weak sense,
  \[ \eta L (\Phi) + (1 - \eta) V L' (\Psi) = V h, \]
  and $\Phi = V \Psi$, $\langle V \Psi, Z_d \rangle = 0 ; \forall x \in
  \mathbbm{R}^2, \Psi (x_1, x_2) = \overline{\Psi (x_1, - x_2)} = \Psi (- x_1,
  x_2)$, then
  \[ \Phi \in \mathcal{E}_{\ast, \sigma} . \]
\end{lemma}

See Appendix $\ref{AB5}$ for the proof of this result.

\

The next step is to construct a solution on a large ball in the space $H_a$.

\begin{lemma}
  \label{funky2}For $0 < \sigma' < 1$, there exists $c_0 (\sigma') > 0$ such
  that, for $0 < c < c_0 (\sigma')$, there exists $a_0 (c, \sigma') >
  \frac{10}{c^2}$ such that, for $V h \in \mathcal{E}_{\ast \ast, \sigma'}$,
  $a > a_0 (c, \sigma')$, the problem
  \[ \left\{ \begin{array}{l}
       \eta L (\Phi) + (1 - \eta) V L' (\Psi) = V h \quad \tmop{in} B (0, a)\\
       \Phi \in H_a, \Phi = V \Psi, \langle V \Psi, Z_d \rangle = 0 ; \forall
       x \in B (0, a), \Psi (x_1, x_2) = \overline{\Psi (x_1, - x_2)} = \Psi
       (- x_1, x_2)\\
       \Phi = 0 \quad \tmop{on} \partial B (0, a)\\
       \langle V h, Z_d \rangle = 0
     \end{array} \right. \]
  admits a unique solution, and furthermore, there exists $K (\sigma', c) > 0$
  independent of $a$ such that
  \[ \| \Phi \|_{H_a} \leqslant K (\sigma', c) \| h \|_{\ast \ast, \sigma'} .
  \]
\end{lemma}

Here, $a > 10 / c^2$ is not necessary, the condition $a > 10 / c$ should be
enough. However, this simplifies some estimates in the proof, and it will be
enough for us here. Here, we require $\langle V h, Z_d \rangle = 0$ in order
to apply the Fredholm alternative in $\{ \varphi \in H^1_0 (B (0, a)), \langle
\varphi, Z_d \rangle = 0 \}$ to show the existence of a solution.

\

\begin{proof}
  We argue by contradiction on the estimation. Assuming the existence, take
  any $0 < \sigma' < 1$, and choose $c_0 (\sigma') > 0$ smaller than the one
  from Proposition \ref{invertop}, and $0 < c < c_0 (\sigma')$. Suppose that
  there exists a sequence $a_n > \frac{10}{c^2}$, $a_n \rightarrow \infty$,
  functions $\Phi_n \in H_{a_n}$, $\Phi_n = 0$ on $\partial B (0, a_n)$ and $V
  h_n \in \mathcal{E}_{\ast \ast, \sigma'}$ such that $\| \Phi_n \|_{H_{a_n}}
  = 1$, $\| h_n \|_{\ast \ast, \sigma'} \rightarrow 0$ and $\eta L (\Phi_n) +
  (1 - \eta) V L' (\Psi_n) = V h_n$ on $B (0, a_n)$. In particular, remark
  here that $c$ is independent of $n$, only the size of the ball grows. Our
  goal is to show that $\| \Phi_n \|_{H_{a_n}} = o^c_{n \rightarrow \infty}
  (1)$, where $o^c_{n \rightarrow \infty} (1)$ is a quantity going to $0$ when
  $n \rightarrow \infty$ at fixed $c$, which leads to the contradiction.
  
  Following the same arguments as in step 1 of the proof of Proposition
  \ref{invertop}, we check that $\Phi_n \rightarrow \Phi$ in $C^2_{\tmop{loc}}
  (\mathbbm{R}^2)$ and $\eta L (\Phi) + (1 - \eta) V L' (\Psi) = 0$ in
  $\mathbbm{R}^2$. Furthermore, it is easy to check that, since $\| \Phi_n
  \|_{H_{a_n}} = 1$, we have $\| \Phi \|_{H_{\infty}} \leqslant 1$. Then, by
  Lemma \ref{funky}, since the orthogonality and the symmetries pass at the
  limit, this implies that $\Phi \in \mathcal{E}_{\ast, \sigma}$ for any $0 <
  \sigma < \sigma'$, and therefore, by Proposition \ref{invertop}, $\Phi = 0$.
  
  We deduce that $\| \Phi_n \|_{C^2 (B (0, 10 / c^2))} = o^c_{n \rightarrow
  \infty} (1)$. Now, we use the same cutoff as in the proof of Lemma
  \ref{funky}, and we have the system on $\tilde{\Psi}_n = \tilde{\Psi}_1 + i
  \tilde{\Psi}_2$, with $\tilde{h}_n = \tilde{h}_1 + i \tilde{h}_2$ (see
  equation (\ref{systnew})):
  \[ \left\{ \begin{array}{l}
       \Delta \tilde{\Psi}_1 - 2 \tilde{\Psi}_1 - c \partial_{x_2}
       \tilde{\Psi}_2 = - \tilde{h}_1 - 2\mathfrak{R}\mathfrak{e} \left(
       \frac{\nabla V}{V} . \nabla \tilde{\Psi}_n \right) + \tmop{Loc}_1
       (\Psi_n) - 2 (1 - | V |^2) \tilde{\Psi}_1\\
       \Delta \tilde{\Psi}_2 + c \partial_{x_2} \tilde{\Psi}_1 = - \tilde{h}_2
       - 2\mathfrak{I}\mathfrak{m} \left( \frac{\nabla V}{V} . \nabla
       \tilde{\Psi}_n \right) + \tmop{Loc}_2 (\Psi_n) .
     \end{array} \right. \]
  Now, multiplying the first equation by $\tilde{\Psi}_1$ and integrating on
  $\Omega = B (0, a) \backslash B (0, 5 / c^2)$, we have
  \[ \int_{\Omega} (\Delta \tilde{\Psi}_1 - 2 \tilde{\Psi}_1) \tilde{\Psi}_1 =
  \]
  \[ \int_{\Omega} \left( c \partial_{x_2} \tilde{\Psi}_2 - \tilde{h}_1 -
     2\mathfrak{R}\mathfrak{e} \left( \frac{\nabla V}{V} . \nabla
     \tilde{\Psi}_n \right) + \tmop{Loc}_1 (\Psi_n) - 2 (1 - | V |^2)
     \tilde{\Psi}_1 \right) \tilde{\Psi}_1 . \]
  We integrate by parts. Recall that $\| \Phi_n \|_{C^2 (B (0, 10 / c^2))} =
  o^c_{n \rightarrow \infty} (1)$ and $\Phi_n = V \Psi_n = 0$ on $\partial B
  (0, a_n)$, thus
  \[ \int_{\Omega} \Delta \tilde{\Psi}_1 \tilde{\Psi}_1 = - \int_{\Omega} |
     \nabla \tilde{\Psi}_1 |^2 + o^c_{n \rightarrow \infty} (1) . \]
  Furthermore, since $V h_n \in \mathcal{E}_{\ast \ast, \sigma'}$, we check
  easily that $\| \tilde{h}_1 \|_{L^2 (\Omega)} \leqslant o^{\sigma'}_{c
  \rightarrow 0} (1)$, and we compute with Lemma \ref{nonmodV} and $\| \Phi_n
  \|_{C^2 (B (0, 10 / c^2))} = o^c_{n \rightarrow \infty} (1)$ that, since for
  $x \in \Omega, r \geqslant 5 / c^2$,
  \[ \left\| \frac{\nabla V}{V} \right\|_{L^{\infty} (\Omega)} + \|
     \tmop{Loc}_1 (\Psi_n) \|_{L^{\infty} (\Omega)} + \| \tmop{Loc}_2 (\Psi_n)
     \|_{L^{\infty} (\Omega)} + \| (1 - | V |^2) \|_{L^{\infty} (\Omega)}
     \leqslant o_{c \rightarrow 0} (1) + o^c_{n \rightarrow \infty} (1) . \]
  This allows us to estimate the right hand side: by Cauchy-Schwarz,
  \[ \| \nabla \tilde{\Psi}_1 \|^2_{L^2 (\Omega)} + 2 \| \tilde{\Psi}_1
     \|^2_{L^2 (\Omega)} \leqslant \]
  \[ c \| \nabla \tilde{\Psi}_2 \|_{L^2 (\Omega)} \| \tilde{\Psi}_1 \|_{L^2
     (\Omega)} + (o_{c \rightarrow 0} (1) + o^c_{n \rightarrow \infty} (1))
     (\| \nabla \tilde{\Psi}_n \|_{L^2 (\Omega)} + \| \tilde{\Psi}_1 \|_{L^2
     (\Omega)}) + o^c_{n \rightarrow \infty} (1) . \]
  Now, we multiply the second equation by $\tilde{\Psi}_2$, and we integrate
  on $\Omega$. By integration by parts, we check
  \[ \| \nabla \tilde{\Psi}_2 \|^2_{L^2 (\Omega)} \leqslant \]
  \[ c \left| \int_{\Omega} \partial_{x_2} \tilde{\Psi}_1 \tilde{\Psi}_2
     \right| + \left| \int_{\Omega} \tilde{h}_2 \tilde{\Psi}_2 \right| + 2
     \int_{\Omega} \left| \mathfrak{I}\mathfrak{m} \left( \frac{\nabla V}{V} .
     \nabla \tilde{\Psi}_n \right) \tilde{\Psi}_2 \right| + \int_{\Omega} |
     \tmop{Loc}_2 (\Psi_n) \tilde{\Psi}_2 | + o^c_{n \rightarrow \infty} (1) .
  \]
  By integration by parts, since $\| \Phi_n \|_{C^2 (B (0, 10 / c^2))} =
  o^c_{n \rightarrow \infty} (1)$ and $\Phi_n = 0$ on $\partial B (0, a_n)$,
  we have
  \[ c \left| \int_{\Omega} \partial_{x_2} \tilde{\Psi}_1 \tilde{\Psi}_2
     \right| \leqslant o^c_{n \rightarrow \infty} (1) + c \left| \int_{\Omega}
     \tilde{\Psi}_1 \partial_{x_2} \tilde{\Psi}_2 \right| \leqslant o^c_{n
     \rightarrow \infty} (1) + c \| \tilde{\Psi}_1 \|_{L^2 (\Omega)} \| \nabla
     \tilde{\Psi}_2 \|_{L^2 (\Omega)} . \]
  We recall that $| \tilde{\Psi}_2 | = o^c_{n \rightarrow \infty} (1)$ on
  $\partial B (0, 5 / c^2)$, therefore
  \begin{eqnarray*}
    \int_{r = 5 / c^2}^a \frac{| \tilde{\Psi}_2 |^2}{r^{2 + \sigma'}} r d r &
    = & \frac{- 1}{\sigma'} \int_{r = 5 / c^2}^a \partial_r \left(
    \frac{1}{r^{\sigma'}} \right) | \tilde{\Psi}_2 |^2 d r\\
    & \leqslant & \frac{K (c)}{\sigma'} | \tilde{\Psi}_2 |^2 (5 / c^2) +
    \frac{2}{\sigma'} \int_{r = 5 / c^2}^a \frac{1}{r^{\sigma'}} | \nabla
    \tilde{\Psi}_2 | | \tilde{\Psi}_2 | d r\\
    & \leqslant & o^{c, \sigma'}_{n \rightarrow \infty} (1) +
    \frac{2}{\sigma'} \sqrt{\int_{r = 5 / c^2}^a | \nabla \tilde{\Psi}_2 |^2 r
    d r \int_{r = 5 / c^2}^a \frac{| \tilde{\Psi}_2 |^2}{r^{2 + \sigma'}} r d
    r} .
  \end{eqnarray*}
  We deduce that
  \[ \int_{r = 5 / c^2}^a \frac{| \tilde{\Psi}_2 |^2}{r^{2 + \sigma'}} r d r
     \leqslant o^{c, \sigma'}_{n \rightarrow \infty} (1) + \frac{K}{\sigma'}
     \int_{r = 5 / c^2}^a | \nabla \tilde{\Psi}_2 |^2 r d r, \]
  and therefore
  \[ \left| \int_{\Omega} \frac{| \tilde{\Psi}_2 |^2}{(1 + | x |)^{2 +
     \sigma'}} \right| \leqslant o^{c, \sigma'}_{n \rightarrow \infty} (1) +
     \frac{K}{\sigma'} \| \nabla \tilde{\Psi}_2 \|^2_{L^2 (\Omega)} . \]
  Since $V h_n \in \mathcal{E}_{\ast \ast, \sigma'}$, we estimate, by
  Cauchy-Schwarz, that
  \[ \left| \int_{\Omega} \tilde{h}_2 \tilde{\Psi}_2 \right| \leqslant o_{c
     \rightarrow 0} (1) \sqrt{\int_{\Omega} \frac{| \tilde{\Psi}_2 |^2}{(1 + |
     x |)^{2 + \sigma'}}} \leqslant o^{\sigma'}_{c \rightarrow 0} (1) \|
     \nabla \tilde{\Psi}_2 \|_{L^2 (\Omega)} + o^{c, \sigma}_{n \rightarrow
     \infty} (1) . \]
  Furthermore, since $\tmop{Loc}_2 (\Psi_n)$ is supported in $B (0, 10 / c^2)$
  and $\| \Phi_n \|_{C^1 (B (0, 10 / c^2))} = o^c_{n \rightarrow \infty} (1)$,
  we check that
  \[ \int_{\Omega} | \tmop{Loc}_2 (\Psi_n) \tilde{\Psi}_2 | \leqslant o^c_{n
     \rightarrow \infty} (1) . \]
  Finally, from Lemma \ref{dervor}, we check that, in $\mathbbm{R}^2$,
  \[ \left| \frac{\nabla V}{V} \right| \leqslant K \left| i \left(
     \frac{y^{\bot}}{| y |^2} - \frac{z^{\bot}}{| z |^2} \right) \right| +
     \frac{K}{c (1 + | x |)^2} \leqslant \frac{K}{c (1 + | x |)^2}, \]
  and thus, by Cauchy-Schwarz,
  \begin{eqnarray*}
    \int_{\Omega} \left| \mathfrak{I}\mathfrak{m} \left( \frac{\nabla V}{V} .
    \nabla \tilde{\Psi}_n \right) \tilde{\Psi}_2 \right| & \leqslant & \|
    \nabla \tilde{\Psi}_n \|_{L^2 (\Omega)} \sqrt{\int_{\Omega} \left|
    \frac{\nabla V}{V} \right|^2 | \tilde{\Psi}_2 |^2}\\
    & \leqslant & \frac{K \| \nabla \tilde{\Psi}_n \|_{L^2 (\Omega)}}{c}
    \sqrt{\int_{\Omega} \frac{| \tilde{\Psi}_2 |^2}{(1 + | x |)^4}} .
  \end{eqnarray*}
  In $\Omega$, $| x | \geqslant 5 / c^2$, thus
  \[ \int_{\Omega} \frac{| \tilde{\Psi}_2 |^2}{(1 + | x |)^4} \leqslant c^{2
     (2 - \sigma')} \int_{\Omega} \frac{| \tilde{\Psi}_2 |^2}{(1 + | x |)^{2 +
     \sigma'}} \leqslant c^{2 (2 - \sigma')} K (\sigma') \| \nabla
     \tilde{\Psi}_2 \|^2_{L^2 (\Omega)} + o^c_{n \rightarrow \infty} (1), \]
  hence
  \[ \int_{\Omega} \left| \mathfrak{I}\mathfrak{m} \left( \frac{\nabla V}{V} .
     \nabla \tilde{\Psi}_n \right) \tilde{\Psi}_2 \right| \leqslant
     o^{\sigma'}_{c \rightarrow 0} (1) \| \nabla \tilde{\Psi}_2 \|_{L^2
     (\Omega)} + o^c_{n \rightarrow \infty} (1) . \]
  We conclude that
  \begin{eqnarray*}
    &  & \| \nabla \tilde{\Psi}_1 \|^2_{L^2 (\Omega)} + 2 \| \tilde{\Psi}_1
    \|^2_{L^2 (\Omega)}\\
    & \leqslant & c \| \nabla \tilde{\Psi}_2 \|_{L^2 (\Omega)} \|
    \tilde{\Psi}_1 \|_{L^2 (\Omega)} + (o_{c \rightarrow 0} (1) + o^c_{n
    \rightarrow \infty} (1)) (\| \nabla \tilde{\Psi}_n \|_{L^2 (\Omega)} + \|
    \tilde{\Psi}_1 \|_{L^2 (\Omega)}) + o^c_{n \rightarrow \infty} (1),
  \end{eqnarray*}
  and
  \[ \| \nabla \tilde{\Psi}_2 \|_{L^2 (\Omega)}^2 \leqslant o^c_{n \rightarrow
     \infty} (1) + c \| \tilde{\Psi}_1 \|_{L^2 (\Omega)} \| \nabla
     \tilde{\Psi}_2 \|_{L^2 (\Omega)} + o^{\sigma'}_{c \rightarrow 0} (1) \|
     \nabla \tilde{\Psi}_2 \|_{L^2 (\Omega)}, \]
  therefore
  \[ \| \nabla \tilde{\Psi}_1 \|_{L^2 (\Omega)} + \| \tilde{\Psi}_1 \|_{L^2
     (\Omega)} + \| \nabla \tilde{\Psi}_2 \|_{L^2 (\Omega)} \leqslant o^c_{n
     \rightarrow \infty} (1) + o^{\sigma'}_{c \rightarrow 0} (1) . \]
  We have shown that for any $\sigma' > 0$,
  \[ \left| \int_{\Omega} \frac{| \tilde{\Psi}_2 |^2}{(1 + | x |)^{2 +
     \sigma'}} \right| \leqslant \left( o^{c, \sigma'}_{n \rightarrow \infty}
     (1) + \frac{K}{\sigma'} \| \nabla \tilde{\Psi}_2 \|^2_{L^2 (\Omega)}
     \right), \]
  thus
  \[ \left| \int_{\Omega} \frac{| \tilde{\Psi}_2 |^2}{(1 + | x |)^{5 / 2}}
     \right| \leqslant o^c_{n \rightarrow \infty} (1) + o_{c \rightarrow 0}
     (1) . \]
  Together with $\| \Phi_n \|_{C^2 (B (0, 10 / c^2))} = o^c_{n \rightarrow
  \infty} (1)$, this is in contradiction with $\| \Phi_n \|_{H_{a_n}} = 1$.
  
  This concludes the proof of the estimation. Now, for the existence, we argue
  by Fredholm's alternative in $\{ \varphi \in H^1_0 (B (0, a)), \langle
  \varphi, Z_d \rangle = 0 \}$, and we remark that the norms $\| . \|_{H_a}$
  and $\| . \|_{H^1}$ are equivalent on $B (0, a)$. By Riesz's representation
  theorem, the elliptic equation $\eta L (\Phi) + (1 - \eta) V L' (\Psi) = V
  h$ can be rewritten in the operational form $\Phi +\mathcal{K} (\Phi)
  =\mathcal{S} (h)$ where $\mathcal{K}$ is a compact operator in $H^1_0 (B (0,
  a))$, and it has no kernel in $H_a$ (i.e. in $\{ \varphi \in H^1_0 (B (0,
  a)), \langle \varphi, Z_d \rangle = 0 \}$) by the estimation we just showed.
  Therefore, there exists a unique solution $\Phi \in H_a$, and it then
  satisfies
  \[ \| \Phi \|_{H_a} \leqslant K (\sigma', c) \| h \|_{\ast \ast, \sigma'} .
  \]
\end{proof}

\begin{proposition}
  \label{p2191802}Consider the problem, for $0 < \sigma < \sigma' < 1$,
  \[ \left\{ \begin{array}{l}
       \eta L (\Phi) + (1 - \eta) V L' (\Psi) = V h\\
       V h \in \mathcal{E}_{\ast \ast, \sigma'}, \langle V h, Z_d \rangle = 0.
     \end{array} \right. \]
  Then, there exist constants $c_0 (\sigma, \sigma') > 0$ small and $C
  (\sigma, \sigma') > 0$ depending only on $\sigma, \sigma'$, such that, for
  $0 < c \leqslant c_0 (\sigma, \sigma')$ and $V h \in \mathcal{E}_{\ast \ast,
  \sigma'}$ with $\langle V h, Z_d \rangle = 0$, there exists $\Phi \in
  \mathcal{E}_{\ast, \sigma}$, $\Phi = V \Psi$ solution of this problem, with
  \[ \| \Psi \|_{\ast, \sigma, d} \leqslant C (\sigma, \sigma') \| h \|_{\ast
     \ast, \sigma', d} . \]
\end{proposition}

\begin{proof}
  By Lemma \ref{funky2}, For $a > a_0 (c, \sigma')$, there exists a solution
  to the problem
  \[ \left\{ \begin{array}{l}
       \eta L (\Phi_a) + (1 - \eta) V L' (\Psi_a) = V h \quad \tmop{on} B (0,
       a)\\
       \Phi_a \in H_a, \Phi_a = V \Psi_a, \langle V \Psi_a, Z_d \rangle = 0 ;
       \forall x \in B (0, a), \Psi_a (x_1, x_2) = \overline{\Psi_a (x_1, -
       x_2)} = \Psi_a (- x_1, x_2)\\
       \Phi_a = 0 \quad \tmop{on} \partial B (0, a)\\
       \langle h, Z_d \rangle = 0
     \end{array} \right. \]
  with $\| \Phi_a \|_{H_a} \leqslant K (\sigma', c) \| h \|_{\ast \ast,
  \sigma'}$. Taking a sequence of values $a_n > a_0$ going to infinity, we can
  construct by a diagonal argument a function $\Phi \in H^1_{\tmop{loc}}
  (\mathbbm{R}^2)$ which satisfies in the distribution sense
  \[ \eta L (\Phi) + (1 - \eta) V L' (\Phi) = V h \]
  (hence $\Phi \in C^2 (\mathbbm{R}^2)$ by standard elliptic arguments), such
  that
  \[ \| \Phi \|_{H_{\infty}} \leqslant \limsup_{n \rightarrow \infty} \|
     \Phi_n \|_{H_{a_n}} \leqslant K (\sigma', c) \| h \|_{\ast \ast,
     \sigma'}, \]
  thus $\Phi \in H_{\infty}$, and $\Phi = V \Psi$, $\langle V \Psi, Z_d
  \rangle = 0 ; \forall x \in \mathbbm{R}^2, \Psi (x_1, x_2) = \overline{\Psi
  (x_1, - x_2)} = \Psi (- x_1, x_2)$. From Lemma \ref{funky}, we deduce that
  $\Phi \in \mathcal{E}_{\ast, \sigma}$, and is thus a solution to the
  problem. Furthermore, by Proposition \ref{invertop}, $\| \Psi \|_{\ast,
  \sigma, d} \leqslant C (\sigma, \sigma') \| h \|_{\ast \ast, \sigma', d} .$
  Still by Proposition \ref{invertop}, this solution is unique in
  $\mathcal{E}_{\ast, \sigma'}$.
\end{proof}

\subsubsection{Estimates for the contraction in the orthogonal
space}\label{2532809}

We showed in Proposition \ref{p2191802} that the operator $\nobracket \eta L
(.) + (1 - \eta) V L' (. / V \nobracket)$ is invertible from
$\mathcal{E}_{\ast \ast, \sigma', d} \cap \{ \langle ., Z_d \rangle = 0 \}$ to
$\mathcal{E}_{\ast, \sigma, d}$. The operator $(\eta L (.) + (1 - \eta) V L'
(. / V))^{- 1}$ is the one that, for a given $V h \in \mathcal{E}_{\ast \ast,
\sigma', d}$ such that $\langle V h, Z_d \rangle = 0$, returns the unique
function $\Phi = V \Psi \in \mathcal{E}_{\ast, \sigma, d}$ such that $\eta L
(\Phi) + (1 - \eta) V L' (\Psi) = V h$, and this function satisfies the
estimate $\| \Psi \|_{\ast, \sigma, d} \leqslant C (\sigma, \sigma') \| h
\|_{\ast \ast, \sigma', d}$.

Now, we define (for any $\Phi \in C^0 (\mathbbm{R}^2, \mathbbm{C})$)
\[ \Pi_d^{\bot} (\Phi) \assign \Phi - \langle \Phi, Z_d \rangle \frac{Z_d}{\|
   Z_d \|_{L^2 (\mathbbm{R}^2)}^2}, \]
the projection on the orthogonal of $Z_d$. We want to apply a fixed-point
theorem on the functional
\[ (\eta L (.) + (1 - \eta) V L' (. / V))^{- 1} (\Pi_d^{\bot} (- F (. / V))) :
   \mathcal{E}_{\ast, \sigma} \rightarrow \mathcal{E}_{\ast, \sigma}, \]
and for that we need some estimates on the function $\Pi_d^{\bot} o F
(. / V) : \mathcal{E}_{\ast, \sigma} \rightarrow \{ V h \in \mathcal{E}_{\ast
\ast, \sigma'}, \langle V h, Z_d \rangle = 0 \}$. The function $F$ contains
the source term $E - i c \partial_{x_2} V$ and nonlinear terms. The source
term requires a precise computation (see Lemma \ref{fineestimate}) to show its
smallness in the spaces of invertibility. The nonlinear terms will be small if
we do the contraction in an area with small $\Psi$ (which is the case since we
will do it in the space of function $\Phi = V \Psi \in \mathcal{E}_{\ast,
\sigma}$ such that $\| \Psi \|_{\ast, \sigma, d} \leqslant K_0 (\sigma,
\sigma') c^{1 - \sigma'}$ for a well chosen constant $K_0 (\sigma, \sigma') >
0$). This subsection is devoted to the proof of the following result.

\begin{proposition}
  \label{contractionest}For $0 < \sigma < \sigma' < 1$, there exist constants
  $K_0 (\sigma, \sigma'), c_0 (\sigma, \sigma') > 0$ depending only on
  $\sigma, \sigma'$ such that for $0 < c < c_0 (\sigma, \sigma')$, the
  function (from $\mathcal{E}_{\ast, \sigma, d}$ to $\mathcal{E}_{\ast,
  \sigma, d}$)
  \[ \Phi \mapsto (\eta L (.) + (1 - \eta) V L' (. / V))^{- 1} (\Pi_d^{\bot}
     (- F (\Phi / V))) \]
  is a contraction in the space of functions $\Phi = V \Psi \in
  \mathcal{E}_{\ast, \sigma, d}$ such that $\| \Psi \|_{\ast, \sigma, d}
  \leqslant K_0 (\sigma, \sigma') c^{1 - \sigma'}$. As such, by the
  contraction mapping theorem, it admits a unique fixed point $\Phi \in
  \mathcal{E}_{\ast, \sigma, d}$ in $\{ \Phi \in \mathcal{E}_{\ast, \sigma,
  d}, \| \Psi \|_{\ast, \sigma, d} \leqslant K_0 (\sigma, \sigma') c^{1 -
  \sigma'} \}$, and there exists $\lambda (c, d) \in \mathbbm{R}$ such that
  \[ \eta L (\Phi) + (1 - \eta) V L' (\Psi) + F (\Psi) = \lambda (c, d) Z_d \]
  in the distribution sense.
\end{proposition}

We recall that, from the definition of $\mathcal{E}_{\ast, \sigma, d}$ in
subsection $\ref{normsetup}$, $\Phi \in \mathcal{E}_{\ast, \sigma, d}$ implies
that $\langle \Phi, Z_d \rangle = 0$, which is the origin of the fact that
$\eta L (\Phi) + (1 - \eta) V L' (\Psi) + F (\Psi)$ is not zero, but only
proportional to $Z_d$.

We start with some estimates on the terms contained in $F (\Psi)$. These are
done in the following three lemmas.

\begin{lemma}
  \label{fineestimate} For any $0 < \sigma' < 1$, there exists a constant $C_1
  (\sigma') > 0$ depending only on $\sigma'$ such that
  \[ \left\| \frac{i c \partial_{x_{2 \nosymbol}} V}{V} \right\|_{\ast \ast,
     \sigma', d} + \left\| \frac{E}{V} \right\|_{\ast \ast, \sigma', d}
     \leqslant C_1 (\sigma') c^{1 - \sigma'} . \]
\end{lemma}

\begin{proof}
  We have defined the norm
  \[ \| h \|_{\ast \ast, \sigma', d} = \| V h \|_{C^1 (\{ \tilde{r} \leqslant
     3 \})} + \| \tilde{r}^{1 + \sigma'} h_1 \|_{L^{\infty} (\{ \tilde{r}
     \geqslant 2 \})} + \| \tilde{r}^{2 + \sigma'} h_2 \|_{L^{\infty} (\{
     \tilde{r} \geqslant 2 \})} + \| \tilde{r}^{2 + \sigma'} \nabla h
     \|_{L^{\infty} (\{ \tilde{r} \geqslant 2 \})}, \]
  thus we separate two areas for the computation: the first one is where
  $\tilde{r} \leqslant 3$ which will be easy and then far from the vortices,
  i.e. in $\{ \tilde{r} \geqslant 2 \}$, where the division by $V$ is not a
  problem.
  
\ 
  
  \begin{tmindent}
    Step 1.  Estimates for $E$.
  \end{tmindent}
  
\
  
  In (\ref{E2}), we showed that
  \[ E = (1 - | V_1 |^2) (1 - | V_{- 1} |^2) V_1 V_{- 1} - 2 \nabla V_1 .
     \nabla V_{- 1} . \]
  Near $V_1$, i.e. in $B (d \overrightarrow{e_1}, 3)$, we have from Lemma
  \ref{lemme3},
  \[ \| (1 - | V_{- 1} |^2) \|_{C^1 (\{ r_1 \leqslant 3 \})} \leqslant K c^2
     \quad \tmop{and} \quad \| \nabla V_{- 1} \|_{C^1 (\{ r_1 \leqslant 3 \})}
     \leqslant K c, \]
  hence
  \begin{equation}
    \left\| \frac{E}{V} V \right\|_{C^1 (\{ r_1 \leqslant 3 \})} \leqslant K c
    \leqslant o^{\sigma'}_{c \rightarrow 0} (1) c^{1 - \sigma'}, \label{iee3}
  \end{equation}
  where $o^{\sigma'}_{c \rightarrow 0} (1)$ is a quantity that for a fixed
  $\sigma' > 0$, goes to $0$ when $c \rightarrow 0$. By symmetry, the result
  holds in the area where $\tilde{r} \leqslant 3$.
  
  \
  
  We now turn to the estimates for $\tilde{r} \geqslant 2$. The first term
  $(1 - | V_1 |^2) (1 - | V_{- 1} |^2) $ of $\frac{E}{V}$ is real valued.
  Using the definition of $r_1$ and $r_{- 1}$ from (\ref{notation}), in the
  right half-plane, where $r_1 \leqslant r_{- 1}$ and $r_{- 1} \geqslant d
  \geqslant \frac{K}{c}$, we have from Lemma \ref{lemme3}
  \[ \| r_1^{1 + \sigma'} (1 - | V_1 |^2) (1 - | V_{- 1} |^2) \|_{L^{\infty}
     (\{ 2 \leqslant r_1 \leqslant r_{- 1} \})} \leqslant K \left\|
     \frac{1}{r_1^{1 - \sigma'} r_{- 1}^2} \right\|_{L^{\infty} (\{2 \leqslant
     r_1 \leqslant r_{- 1} \})} \]
  and
  \[ \| r_1^2 (1 - | V_1 |^2) (1 - | V_{- 1} |^2) \|_{L^{\infty} (\{ 2
     \leqslant r_1 \leqslant r_{- 1} \})} \leqslant K. \]
  In this area, $\frac{1}{r_{- 1}^2} \leqslant K c^2$ and $\frac{1}{r_1^{1 -
  \sigma'}} \leqslant \frac{1}{2^{1 - \sigma'}}$, thus
  \[ \| r_1^{1 + \sigma'} (1 - | V_1 |^2) (1 - | V_{- 1} |^2) \|_{L^{\infty}
     (\{ 2 \leqslant r_1 \leqslant r_{- 1} \})} \leqslant K (\sigma') c^2
     \leqslant o^{\sigma'}_{c \rightarrow 0} (1) c^{1 - \sigma'} . \]
  By symmetry, the same result holds for the other half-plane, hence
  \begin{equation}
    \| \tilde{r}^{1 + \sigma'} (1 - | V_1 |^2) (1 - | V_{- 1} |^2)
    \|_{L^{\infty} (\{ \tilde{r} \geqslant 2 \})} \leqslant o^{\sigma'}_{c
    \rightarrow 0} (1) c^{1 - \sigma'} \label{iee4} .
  \end{equation}
  From Lemma \ref{lemme3}, we have
  \[ \nabla V_{\varepsilon} (x) = i \varepsilon V_{\varepsilon} (x)
     \frac{x^{\bot}}{r^2} + O \left( \frac{1}{r^3} \right), \]
  hence
  \[ \frac{\nabla V_1 . \nabla V_{- 1}}{V_1 V_{- 1}} = \frac{y^{\bot}
     .z^{\bot}}{r_1^2 r_{- 1}^2} + O \left( \frac{1}{r_1^3 r_{- 1}} \right) +
     O \left( \frac{1}{r_{- 1}^3 r_1} \right) . \]
  Remark that the first term is real-valued. We compute first in the right
  half-plane, where $r_1 \leqslant r_{- 1}$ and $r_{- 1} \geqslant d \geqslant
  \frac{K}{c}$,
  \[ \left\| r_1^{1 + \sigma'} \frac{y^{\bot} .z^{\bot}}{r_1^2 r_{- 1}^2}
     \right\|_{L^{\infty} (\{ 2 \leqslant r_1 \leqslant r_{- 1} \})} \leqslant
     \left\| \frac{r_1^{1 + \sigma'}}{r_1 r_{- 1}} \right\|_{L^{\infty} (\{ 2
     \leqslant r_1 \leqslant r_{- 1} \})} . \]
  Since
  \[ \frac{r_1^{1 + \sigma'}}{r_1 r_{- 1}} = \left( \frac{r_1}{r_{- 1}}
     \right)^{\sigma'} \frac{1}{r_{- 1}^{1 - \sigma'}} \leqslant K (\sigma')
     c^{1 - \sigma'}, \]
  we deduce
  \[ \left\| r_1^{1 + \sigma'} \frac{y^{\bot} .z^{\bot}}{r_1^2 r_{- 1}^2}
     \right\|_{L^{\infty} (\{ 2 \leqslant r_1 \leqslant r_{- 1} \})} \leqslant
     K (\sigma') c^{1 - \sigma'} \]
  and by symmetry,
  \begin{equation}
    \left\| \tilde{r}^{1 + \sigma'} \frac{y^{\bot} .z^{\bot}}{r_1^2 r_{- 1}^2}
    \right\|_{L^{\infty} (\{ \tilde{r} \geqslant 2 \})} \leqslant K (\sigma')
    c^{1 - \sigma'} . \label{iee5}
  \end{equation}
  For the last two terms $O \left( \frac{1}{r_1^3 r_{- 1}} \right) + O \left(
  \frac{1}{r_{- 1}^3 r_1} \right)$,we will show that in the right half-plane
  \begin{equation}
    \left\| r_1^{2 + \sigma'} \frac{1}{r_1^3 r_{- 1}} \right\|_{L^{\infty} (\{
    2 \leqslant r_1 \leqslant r_{- 1} \})} + \left\| r_1^{2 + \sigma'}
    \frac{1}{r_{- 1}^3 r_1} \right\|_{L^{\infty} (\{ 2 \leqslant r_1 \leqslant
    r_{- 1} \})} \leqslant o^{\sigma'}_{c \rightarrow 0} (1) c^{1 - \sigma'} .
    \label{iee1}
  \end{equation}
  This immediately implies
  \begin{equation}
    \left\| r_1^{1 + \sigma'} \frac{1}{r_1^3 r_{- 1}} \right\|_{L^{\infty} (\{
    2 \leqslant r_1 \leqslant r_{- 1} \})} + \left\| r_1^{1 + \sigma'}
    \frac{1}{r_{- 1}^3 r_1} \right\|_{L^{\infty} (\{ 2 \leqslant r_1 \leqslant
    r_{- 1} \})} \leqslant o^{\sigma'}_{c \rightarrow 0} (1) c^{1 - \sigma'} .
    \label{iee2}
  \end{equation}
  We compute in the right half-plane where $r_1 \leqslant r_{- 1}$ and $r_{-
  1} \geqslant d \geqslant \frac{K}{c}$, $\frac{1}{r_{- 1}} \leqslant K
  c^{\nosymbol}$ and $\frac{1}{r_1^{1 - \sigma'}} \leqslant K (\sigma')$, thus
  \[ r_1^{2 + \sigma'} \frac{1}{r_1^3 r_{- 1}}_{} = \frac{1}{r_1^{1 - \sigma'}
     r_{- 1}} \leqslant K c \leqslant o^{\sigma'}_{c \rightarrow 0} (1) c^{1 -
     \sigma'} . \]
  Furthermore, still in the right half-plane,
  \[ r_1^{2 + \sigma'}  \frac{1}{r_{- 1}^3 r_1} = \left( \frac{r_1}{r_{- 1}}
     \right)^{1 + \sigma'} \frac{1}{r_{- 1}^{2 - \sigma'}} \leqslant K
     (\sigma') c^{2 - \sigma'} \leqslant o^{\sigma'}_{c \rightarrow 0} (1)
     c^{1 - \sigma'} . \]
  Gathering (\ref{iee1}) to (\ref{iee2}) and using the symmetry for the left
  half-plane, we deduce with the previous esimates (\ref{iee3}), (\ref{iee4}),
  (\ref{iee5}) that
  \[ \left\| V \left( \frac{E}{V} \right) \right\|_{C^1 (\{ \tilde{r}
     \leqslant 3 \})} + \left\| \tilde{r}^{1 + \sigma'}
     \mathfrak{R}\mathfrak{e} \left( \frac{E}{V} \right) \right\|_{L^{\infty}
     (\{ \tilde{r} \geqslant 2 \})} + \left\| \tilde{r}^{2 + \sigma'}
     \mathfrak{I}\mathfrak{m} \left( \frac{E}{V} \right) \right\|_{L^{\infty}
     (\{ \tilde{r} \geqslant 2 \})} \leqslant K (\sigma') c^{1 - \sigma'} . \]
  Now, for the estimate on $\nabla \left( \frac{E}{V} \right)$ in $\{
  \tilde{r} \geqslant 2 \}$, we have from Lemma \ref{lemme3}, for $\tilde{r}
  \geqslant 2$,
  \[ | \nabla ((1 - | V_1 |^2) (1 - | V_{- 1} |^2)) | \leqslant | \nabla | V_1
     |^2 (1 - | V_{- 1} |^2) | + | (1 - | V_1 |^2) \nabla | V_{- 1} |^2 |
     \leqslant \frac{K}{r_1^3 r_{- 1}^2} + \frac{K}{r_1^2 r_{- 1}^3}, \]
  and
  \[ \left| \nabla \left( \frac{\nabla V_1 . \nabla V_{- 1}}{V_1 V_{- 1}}
     \right) \right| \leqslant \left| \nabla \left( \frac{\nabla V_1}{V_1 }
     \right) . \frac{\nabla V_{- 1}}{V_{- 1}} \right| + \left| \frac{\nabla
     V_1}{V_1} . \nabla \left( \frac{\nabla V_{- 1}}{V_{- 1} } \right) \right|
     \leqslant \frac{K}{r_1^2 r_{- 1}} + \frac{K}{r_1 r_{- 1}^2}, \]
  thus, with similar estimates as previously, we deduce
  \begin{equation}
    \left\| \tilde{r}^{2 + \sigma'} \nabla \left( \frac{E}{V} \right)
    \right\|_{L^{\infty} (\{ \tilde{r} \geqslant 2 \})} \leqslant K (\sigma')
    c^{1 - \sigma'} . \label{halplz}
  \end{equation}
  This concludes the proof of
  \[ \left\| \frac{E}{V} \right\|_{\ast \ast, \sigma', d} \leqslant C'_1
     (\sigma') c^{1 - \sigma'} \]
  for some constant $C_1' (\sigma') > 0$ depending only on $\sigma'$.
 
 \ 
  
  \begin{tmindent}
    Step 2.  Estimates for $i c \frac{\partial_{x_2} V}{V}$.
  \end{tmindent}

\  
  
  First, near the vortices, we have $| \partial_{x_2} V | + | \nabla
  \partial_{x_2} V | \leqslant K$ a universal constant, therefore
  \[ \left\| i c \frac{\partial_{x_2} V}{V} V \right\|_{C^1 (\{ \tilde{r}
     \leqslant 3 \})} \leqslant K c \leqslant o^{\sigma'}_{c \rightarrow 0}
     (1) c^{1 - \sigma'} . \]
  We now turn to the estimate for $\tilde{r} \geqslant 2$. Recall Lemma
  \ref{fineest}, stating that for a universal constant $C > 0$, since $r_1,
  r_{- 1} \geqslant 2$,
  \[ \left| i c \frac{\partial_{x_2} V}{V} - 2 c d \frac{x_1^2 - d^2 -
     x_2^2}{r_{1^{\nosymbol \nosymbol}}^2 r_{- 1}^2} \right| \leqslant C
     \left( \frac{c}{r_1^3} + \frac{c}{r_{- 1}^3} \right) . \]
  Remark that $2 c d \frac{x_1^2 - d^2 - x_2^2}{r_{1^{\nosymbol \nosymbol}}^2
  r_{- 1}^2}$ is real-valued. Using that $c d \leqslant 2$, that
  \[ | x_1^2 - d^2 | = | (x_1 - d) (x_1 + d) | \leqslant r_1 r_{- 1} \]
  and also that $x_2^2 \leqslant r_1 r_{- 1}$, we deduce that in the right
  half-plane, where $r_1 \leqslant r_{- 1}$ and $r_{- 1} \geqslant d \geqslant
  \frac{K}{c}$,
  \[ \left\| r_1^{1 + \sigma'} 2 c d \frac{x_1^2 - d^2 -
     x_2^2}{r_{1^{\nosymbol \nosymbol}}^2 r_{- 1}^2} \right\|_{L^{\infty} (\{
     2 \leqslant r_1 \leqslant r_{- 1} \})} \leqslant K \left\| \frac{r_1^{1 +
     \sigma'}}{r_1 r_{- 1}} \right\|_{L^{\infty} (\{ 2 \leqslant r_1 \leqslant
     r_{- 1} \})}, \]
  and since we have
  \[ \frac{r_1^{1 + \sigma'}}{r_1 r_{- 1}} = \left( \frac{r_1}{r_{- 1}}
     \right)^{\sigma'} \frac{1}{r_{- 1}^{1 - \sigma'}} \leqslant K (\sigma')
     c^{1 - \sigma \prime}, \]
  we infer
  \[ \left\| 2 r_1^{1 + \sigma'} c d \frac{x_1^2 - d^2 -
     x_2^2}{r_{1^{\nosymbol \nosymbol}}^2 r_{- 1}^2} \right\|_{L^{\infty} (\{
     2 \leqslant r_1 \leqslant r_{- 1} \})} \leqslant K (\sigma') c^{1 -
     \sigma'} . \]
  It is easy to check that in the right half-plane
  \[ r_1^{2 + \sigma'} \left( \frac{c}{r_1^3} + \frac{c}{r_{- 1}^3} \right)
     \leqslant K c \leqslant o^{\sigma'}_{c \rightarrow 0} (1) c^{1 - \sigma},
  \]
  and therefore by symmetry for the left half-plane,
  \begin{eqnarray*}
    &  & \left\| V \left( i c \frac{\partial_{x_2} V}{V} \right)
    \right\|_{C^1 (\{ \tilde{r} \leqslant 3 \})} + \left\| \tilde{r}^{1 +
    \sigma'} \mathfrak{R}\mathfrak{e} \left( i c \frac{\partial_{x_2} V}{V}
    \right) \right\|_{L^{\infty} (\{ \tilde{r} \geqslant 2 \})}\\
    & + & \left\| \tilde{r}^{2 + \sigma'} \mathfrak{I}\mathfrak{m} \left( i c
    \frac{\partial_{x_2} V}{V} \right) \right\|_{L^{\infty} (\{ \tilde{r}
    \geqslant 2 \})}\\
    & \leqslant & K (\sigma') c^{1 - \sigma'} .
  \end{eqnarray*}
  From the proof of Lemma \ref{fineest}, we check (using Lemma \ref{nonmodV})
  that, if $\tilde{r} \geqslant 1$,
  \[ \left| \nabla \left( i c \frac{\partial_{x_2} V}{V} - 2 c d \frac{x_1^2 -
     d^2 - x_2^2}{r_{1^{\nosymbol \nosymbol}}^2 r_{- 1}^2} \right) \right|
     \leqslant K \left( \frac{c}{r_1^3} + \frac{c}{r_{- 1}^3} \right) . \]
  With $\left| \nabla \left( \frac{1}{r_{\pm 1}} \right) \right| \leqslant
  \frac{K}{r_{\pm 1}^2}$ if $\tilde{r} \geqslant 1$ and similar computations
  as previously, we check that
  \[ \left| \nabla \left( 2 c d \frac{x_1^2 - d^2 - x_2^2}{r_{1^{\nosymbol
     \nosymbol}}^2 r_{- 1}^2} \right) \right| \leqslant K (\sigma') c^{1 -
     \sigma'} . \]
  Therefore, there exists $C_1'' (\sigma') > 0$ such that
  \[ \left\| i c \frac{\partial_{x_2} V}{V} \right\|_{\ast \ast, \sigma', d}
     \leqslant C''_1 (\sigma') c^{1 - \sigma'} . \]
  We conclude by taking $C_1 (\sigma') = \max (C'_1 (\sigma'), C_1''
  (\sigma'))$.
\end{proof}

\begin{lemma}
  \label{L2141904}For $0 < \sigma < \sigma' < 1$, for $\Phi = V \Psi, \Phi' =
  V \Psi' \in \mathcal{E}_{\ast, \sigma, d}$ such that $\| \Psi \|_{\ast,
  \sigma, d}, \| \Psi' \|_{\ast, \sigma, d} \leqslant C_0$ with $C_0$ defined
  in Lemma \ref{lemma7}, if there exists $K_{} (\sigma, \sigma') > 0$ such
  that $\| \Psi \|_{\ast, \sigma, d}, \| \Psi' \|_{\ast, \sigma, d} \leqslant
  K (\sigma, \sigma') c^{1 - \sigma'}$, then
  \[ \left\| \frac{R (\Psi)}{V} \right\|_{\ast \ast, \sigma', d} \leqslant
     o^{\sigma'}_{c \rightarrow 0} (1) c^{1 - \sigma'} \]
  and
  \[ \left\| \frac{R (\Psi') - R (\Psi)}{V} \right\|_{\ast \ast, \sigma', d}
     \leqslant o^{\sigma'}_{c \rightarrow 0} (1) \| \Psi' - \Psi \|_{\ast,
     \sigma, d}, \]
  where the $o^{\sigma, \sigma'}_{c \rightarrow 0} (1)$ is a quantity that,
  for fixed $\sigma$ and $\sigma'$, goes to $0$ when $c \rightarrow 0$.
\end{lemma}

\begin{proof}
  Since $\eta \neq 0$ only in the domain where $\| . \|_{\ast \ast, \sigma',
  d} = \| V. \|_{C^1 (\{ \tilde{r} \leqslant 3 \})}$ and $\| . \|_{\ast,
  \sigma, d} = \| V . \|_{C^2 (\{ \tilde{r} \leqslant 3 \})}$, we will work
  only with these two norms. Recall from Lemma \ref{lemma7} that $R (\Psi)$ is
  supported in $\{ \eta \neq 0 \}$ and
  \[ | R (\Psi) | + | \nabla R (\Psi) | \leqslant C \| \Phi \|^2_{C^2 (\{
     \nobracket \tilde{r} \leqslant 2 \} \nobracket)} \]
  since $\| \Psi \|_{\ast, \sigma, d} \leqslant C_0$. We deduce
  \[ \left\| \frac{R (\Psi)}{V} \right\|_{\ast \ast, \sigma', d} = \| R (\Psi)
     \|_{C^1 (\{ \tilde{r} \leqslant 3 \})} \leqslant K (\sigma') c^{2 - 2
     \sigma'} \leqslant o^{\sigma'}_{c \rightarrow 0} (1) c^{1 - \sigma'} . \]
  Furthermore, using the definition of $R (\Psi)$ in the proof of Lemma
  \ref{lemma7} we check that every term is at least quadratic in $\Psi$ (or
  its real or imaginary part), therefore, with $\| \Psi \|_{\ast, \sigma, d},
  \| \Psi' \|_{\ast, \sigma, d} \leqslant C_0$, $R (\Psi') - R (\Psi)$ can be
  estimated by
  \begin{eqnarray*}
    \left\| \frac{R (\Psi') - R (\Psi)}{V} \right\|_{\ast \ast, \sigma', d} &
    = & \| R (\Psi') - R (\Psi) \|_{C^1 (\{ \tilde{r} \leqslant 3 \})}\\
    & \leqslant & K (\| \Psi \|_{\ast, \sigma, d} + \| \Psi' \|_{\ast,
    \sigma, d}) \| \Psi' - \Psi \|_{\ast, \sigma, d}\\
    & \leqslant & o_{c \rightarrow 0} (1) \| \Psi' - \Psi \|_{\ast, \sigma,
    d} .
  \end{eqnarray*}
\end{proof}

\begin{lemma}
  \label{L2151904}For $0 < \sigma < \sigma' < 1$, for $\Phi = V \Psi, \Phi' =
  V \Psi' \in \mathcal{E}_{\ast, \sigma, d}$ such that $\| \Psi \|_{\ast,
  \sigma, d}, \| \Psi' \|_{\ast, \sigma, d} \leqslant C_0$ with $C_0$ defined
  in Lemma \ref{lemma7}, if there exists $K_{} (\sigma, \sigma') > 0$ such
  that $\| \Psi \|_{\ast, \sigma, d}, \| \Psi' \|_{\ast, \sigma, d} \leqslant
  K (\sigma, \sigma') c^{1 - \sigma'}$, then
  \[ \| (1 - \eta) (- \nabla \Psi . \nabla \Psi + | V |^2 S (\Psi)) \|_{\ast
     \ast, \sigma', d} \leqslant o^{\sigma, \sigma'}_{c \rightarrow 0} (1)
     c^{1 - \sigma'}, \]
  \[ \| (1 - \eta) (- \nabla \Psi' . \nabla \Psi' + \nabla \Psi . \nabla \Psi
     + | V |^2 (S (\Psi') - S (\Psi))) \|_{\ast \ast, \sigma', d} \leqslant
     o^{\sigma, \sigma'}_{c \rightarrow 0} (1) \| \Psi' - \Psi \|_{\ast,
     \sigma, d} . \]
\end{lemma}

\begin{proof}
  As done in Lemma \ref{L2141904}, we check easily that
  \[ \| (1 - \eta) (\nabla \Psi . \nabla \Psi + | V |^2 S (\Psi)) V \|_{C^1
     (\{ \tilde{r} \leqslant 3 \})} \leqslant K (\sigma, \sigma') c^{1 -
     \sigma'} \| \Phi \|_{C^2 (\{ \tilde{r} \leqslant 3 \})}, \]
  since in the area where $(1 - \eta) \neq 0$, $C_1 \leqslant | V | \leqslant
  1$ for a universal constant $C_1 > 0$, $\Phi = V \Psi$ and using $\| V \Psi
  \|_{C^1 (\{ \tilde{r} \leqslant 3 \})} \leqslant K (\sigma, \sigma') c^{1 -
  \sigma'}$.
  
  \
  
  We then estimate (with $\eta = 0$ in $\{ \tilde{r} \geqslant 2 \}$)
  \[ \| \tilde{r}^{1 + \sigma'} \mathfrak{R}\mathfrak{e} (\nabla \Psi . \nabla
     \Psi) \|_{L^{\infty} (\{ \tilde{r} \geqslant 2 \})} \leqslant K \| \Psi
     \|^2_{\ast, \sigma, d} \left\| \frac{\tilde{r}^{1 +
     \sigma'}}{\tilde{r}^{2 + 2 \sigma}} \right\|_{L^{\infty} (\{ \tilde{r}
     \geqslant 2 \})} \leqslant K (\sigma, \sigma') c^{2 - 2 \sigma'}
     \leqslant o^{\sigma, \sigma'}_{c \rightarrow 0} (1) c^{1 - \sigma'}, \]
  and
  \begin{eqnarray*}
    \| \tilde{r}^{2 + \sigma'} \mathfrak{I}\mathfrak{m} (\nabla \Psi . \nabla
    \Psi) \|_{L^{\infty} (\{ \tilde{r} \geqslant 2 \})} & \leqslant & 2 \|
    \tilde{r}^{2 + \sigma'} \mathfrak{I}\mathfrak{m} (\nabla \Psi)
    .\mathfrak{R}\mathfrak{e} (\nabla \Psi) \|_{L^{\infty} (\{ \tilde{r}
    \geqslant 2 \})}\\
    & \leqslant & K \| \Psi \|^2_{\ast, \sigma, d} \left\| \frac{\tilde{r}^{2
    + \sigma'}}{\tilde{r}^{3 + 2 \sigma}} \right\|_{L^{\infty} (\{ \tilde{r}
    \geqslant 2 \})}\\
    & \leqslant & o^{\sigma, \sigma'}_{c \rightarrow 0} (1) c^{1 - \sigma'},
  \end{eqnarray*}
  and we check that with similar computations, that
  \[ \| \tilde{r}^{2 + \sigma'} \nabla (\nabla \Psi . \nabla \Psi)
     \|_{L^{\infty} (\{ \tilde{r} \geqslant 2 \})} \leqslant o^{\sigma,
     \sigma'}_{c \rightarrow 0} (1) c^{1 - \sigma'}, \]
  thus
  \[ \| (1 - \eta) (- \nabla \Psi . \nabla \Psi) \|_{\ast \ast, \sigma', d}
     \leqslant o^{\sigma, \sigma'}_{c \rightarrow 0} (1) c^{1 - \sigma'} . \]
  Now, since $(1 - \eta) (- \nabla \Psi' . \nabla \Psi' + \nabla \Psi . \nabla
  \Psi) = - (1 - \eta) (\nabla (\Psi' - \Psi) . \nabla (\Psi' + \Psi))$, with
  similar computations (and $\| \Psi' + \Psi \|_{\ast, \sigma, d} \leqslant 2
  K (\sigma, \sigma') c^{1 - \sigma'}$), we have
  \[ \| (1 - \eta) (- \nabla \Psi' . \nabla \Psi' + \nabla \Psi . \nabla \Psi)
     \|_{\ast \ast, \sigma', d} \leqslant o^{\sigma, \sigma'}_{c \rightarrow
     0} (1) \| \Psi' - \Psi \|_{\ast, \sigma, d} . \]
  Finally, recall that
  \[ S (\Psi) = e^{2\mathfrak{R}\mathfrak{e} (\Psi)} - 1 -
     2\mathfrak{R}\mathfrak{e} (\Psi) . \]
  Moreover, $e^{2\mathfrak{R}\mathfrak{e} (\Psi)} - 1 -
  2\mathfrak{R}\mathfrak{e} (\Psi)$ is real-valued and for $\tilde{r}
  \geqslant 2$, if $\| \Psi \|_{\ast, \sigma, d} \leqslant C_0$,
  \[ | \tilde{r}^{1 + \sigma'} | V |^2 (e^{2\mathfrak{R}\mathfrak{e} (\Psi)} -
     1 - 2\mathfrak{R}\mathfrak{e} (\Psi)) | \leqslant K | \tilde{r}^{1 +
     \sigma'} \mathfrak{R}\mathfrak{e}^2 (\Psi) | \leqslant K (\sigma,
     \sigma') \| \Psi \|_{\ast, \sigma, d}^2 \leqslant o^{\sigma, \sigma'}_{c
     \rightarrow 0} (1) c^{1 - \sigma'}, \]
  and with Lemma \ref{nonmodV},
  \begin{eqnarray*}
    &  & | \tilde{r}^{2 + \sigma'} \nabla (| V |^2
    (e^{2\mathfrak{R}\mathfrak{e} (\Psi)} - 1 - 2\mathfrak{R}\mathfrak{e}
    (\Psi))) |\\
    & \leqslant & 2 | \tilde{r}^{2 + \sigma'} \nabla \mathfrak{R}\mathfrak{e}
    (\Psi) (e^{2\mathfrak{R}\mathfrak{e} (\Psi)} - 1) | + 2 | \tilde{r}^{2 +
    \sigma'} \nabla (| V |^2) \nobracket (e^{2\mathfrak{R}\mathfrak{e} (\Psi)}
    - 1 - 2\mathfrak{R}\mathfrak{e} (\Psi) \nobracket) |\\
    & \leqslant & K \left( | \tilde{r}^{2 + \sigma'} \nabla
    \mathfrak{R}\mathfrak{e} (\Psi) \mathfrak{R}\mathfrak{e} (\Psi) | + \left|
    \frac{\tilde{r}^{2 + \sigma'}}{\tilde{r}^3} \mathfrak{R}\mathfrak{e}^2
    (\Psi) \right| \right)\\
    & \leqslant & K (\sigma, \sigma') \| \Psi \|_{\ast, \sigma, d}^2 \left\|
    \frac{\tilde{r}^{2 + \sigma'}}{\tilde{r}^{3 + 2 \sigma}}
    \right\|_{L^{\infty} (\{ \tilde{r} \geqslant 2 \})}\\
    & \leqslant & o^{\sigma, \sigma'}_{c \rightarrow 0} (1) c^{1 - \sigma'},
  \end{eqnarray*}
  hence
  \[ \| (1 - \eta) | V |^2 S (\Psi) \|_{\ast \ast, \sigma', d} \leqslant
     o^{\sigma, \sigma'}_{c \rightarrow 0} (1) c^{1 - \sigma'} . \]
  With similar comutations on
  \[ | V |^2 (S (\Psi') - S (\Psi)) = 2 | V |^2 (\mathfrak{R}\mathfrak{e}
     (\Psi') -\mathfrak{R}\mathfrak{e} (\Psi)) \sum_{n = 2}^{+ \infty} 2^{n -
     1} \sum_{k = 0}^{n - 1} \frac{\mathfrak{R}\mathfrak{e} (\Psi)^{n - 1 - k}
     \mathfrak{R}\mathfrak{e} (\Psi')^k}{n!}, \]
  we conclude with
  \[ \| (1 - \eta) (| V |^2 (S (\Psi') - S (\Psi))) \|_{\ast \ast, \sigma', d}
     \leqslant o^{\sigma, \sigma'}_{c \rightarrow 0} (1) \| \Psi' - \Psi
     \|_{\ast, \sigma, d} . \]
\end{proof}

Now, we end the proof of Proposition \ref{contractionest}

\

\begin{proof}[of Proposition \ref{contractionest}]
  We take the constants $C (\sigma, \sigma')$ defined in Proposition
  \ref{invertop} and $C_1 (\sigma')$ from Lemma \ref{fineestimate}. We then
  define $K_0 (\sigma, \sigma') \assign C (\sigma, \sigma') (C_1 (\sigma') +
  1)$.
  
  To apply the contraction mapping theorem, we need to show that for $\Phi = V
  \Psi, \Phi' = V \Psi' \in \mathcal{E}_{\ast, \sigma, d}$ with
  \[ \| \Psi \|_{\ast, \sigma, d}, \| \Psi' \|_{\ast, \sigma, d} \leqslant K_0
     (\sigma, \sigma') c^{1 - \sigma'}, \]
  we have for small $c > 0$,
  \begin{equation}
    \left\| \frac{F (\Psi)}{V} \right\|_{\ast \ast, \sigma', d} \leqslant
    \frac{K_0 (\sigma, \sigma')}{C (\sigma, \sigma')} c^{1 - \sigma'}
    \label{zest1}
  \end{equation}
  and
  \begin{equation}
    \left\| \frac{F (\Psi') - F (\Psi)}{V} \right\|_{\ast \ast, \sigma', d}
    \leqslant o^{\sigma, \sigma'}_{c \rightarrow 0} (1) \| \Psi' - \Psi
    \|_{\ast, \sigma, d} \label{zest2} .
  \end{equation}
  If these estimates hold, using Proposition \ref{invertop}, we have that the
  closed ball $B_{\| . \|_{\ast, \sigma, d}} (0, K_0 (\sigma, \sigma') c^{1 -
  \sigma'})$ is stable by $\Phi \mapsto V (\eta L (V.) + (1 - \eta) V L'
  (.))^{- 1} (\Pi_d^{\bot} (- F (\Phi / V)))$ and this operator is a
  contraction in the ball (for $c$ small enough, depending on $\sigma,
  \sigma'$), hence we can apply the contraction mapping theorem.
  
  From Lemma \ref{lemma7}, we have
  \[ F (\Psi) = E - i c \partial_{x_2} V + V (1 - \eta) (- \nabla \Psi .
     \nabla \Psi + | V |^2 S (\Psi)) + R (\Psi) . \]
  By Lemmas \ref{fineestimate} to \ref{L2151904}, we have, given that $c$ is
  small enough (depending only on $\sigma, \sigma'$), that both $\left(
  \ref{zest1} \right)$ and $\left( \ref{zest2} \right)$ hold. Therefore,
  defining $c_0 (\sigma, \sigma') > 0$ small enough such that all the required
  conditions on $c$ are satisfied if $c < c_0 (\sigma, \sigma')$, we end the
  proof of Proposition \ref{contractionest}.
  
  We have therefore constructed a function $\Phi = V \Psi \in
  \mathcal{E}_{\ast, \sigma, d}$ such that
  \[ \Phi = (\eta L (.) + (1 - \eta) V L' (. / V))^{- 1} (\Pi_d^{\bot} (- F
     (\Phi / V))) . \]
  Therefore, by definition of the operator $(\eta L (.) + (1 - \eta) V L' (. /
  V))^{- 1}$, we have, in the distribution sense,
  \[ \eta L (\Phi) + (1 - \eta) V L' (\Psi) = \Pi_d^{\bot} (- F (\Phi / V)),
  \]
  and thus, there exists $\lambda (c, d) \in \mathbbm{R}$ such that
  \[ \eta L (\Phi) + (1 - \eta) V L' (\Psi) + F (\Psi) = \lambda (c, d) Z_d .
  \]
\end{proof}

At this point, we have the existence of a function $\Phi = V \Psi \in
\mathcal{E}_{\ast, \sigma, d}$ depending on $c, d$ and a priori $\sigma,
\sigma'$, such that $\| \Psi \|_{\ast, \sigma, d} \leqslant K (\sigma,
\sigma') c^{1 - \sigma'}$ and
\begin{equation}
  \eta L (\Phi) + (1 - \eta) V L' (\Psi) + F (\Psi) = \lambda (c, d) Z_d
  \label{azz220}
\end{equation}
in the distribution sense for some $\lambda (c, d) \in \mathbbm{R}$. By using
elliptic regularity, we show easily that $\Phi \in C^{\infty} (\mathbbm{R}^2,
\mathbbm{C})$ and that (\ref{azz220}) is verified in the strong sense. The
goal is now to show that we can take $\lambda (c, d) = 0$ for a good choice of
$d$, but first we need a better estimate on $\Phi$ using the parameters
$\sigma$ and $\sigma'$. We denote by $\Phi_{\sigma, \sigma'} = V \Psi_{\sigma,
\sigma'}$ the solution obtained by Proposition \ref{contractionest} for the
values $\sigma < \sigma'$.

\begin{corollary}
  \label{sigmarem}For $0 < \sigma_1 < \sigma_1' < 1$, $0 < \sigma_2 <
  \sigma_2' < 1$, there exists $c_0 (\sigma_1, \sigma_1', \sigma_2, \sigma_2')
  > 0$ such that for $0 < c < c_0 (\sigma_1, \sigma_1', \sigma_2, \sigma_2')$,
  $\Phi_{\sigma_1, \sigma'_1} = V \Psi_{\sigma_1, \sigma_1'} = V
  \Psi_{\sigma_2, \sigma_2'} = \Phi_{\sigma_2, \sigma_2'}$. We can thus take
  any values of $\sigma, \sigma'$ with $\sigma < \sigma'$ and the estimate
  \[ \| \Psi \|_{\ast, \sigma, d} \leqslant K (\sigma, \sigma') c^{1 -
     \sigma'} \]
  holds for $0 < c < c_0 (\sigma, \sigma')$. In particular, for $c$ small
  enough,
  \[ \| \Phi \|_{C^2 (\{ \tilde{r} \leqslant 3 \})} \leqslant K c^{3 / 4} . \]
\end{corollary}

\begin{proof}
  This is because for $\sigma_1 < \sigma_2$, $\mathcal{E}_{\ast, \sigma_2}
  \subset \mathcal{E}_{\ast, \sigma_1}$ hence the fixed point for $\sigma_2$
  (for any $\sigma_2' > \sigma_2$) yields the same value of $\Psi$ as the
  fixed point for $\sigma_1$ for $c$ small enough (for any $\sigma_1' >
  \sigma_1$). In particular, this implies also that $\lambda (c, d)$ is
  independent of $\sigma, \sigma'$ (for $c$ small enough).
\end{proof}

\subsection{Estimation on the Lagrange multiplier $\lambda (c,
d)$}\label{orthocond}

To finish the construction of a solution of $(\tmop{TW}_c)$, we need to find a
link between $d$ and $c$ such that $\lambda (c, d) = 0$ in (\ref{azz220}).
Here, we give an estimate of $\lambda (c, d)$ for small values of $c$.

\begin{proposition}
  \label{dinc}For $\lambda (c, d), \Phi = V \Psi$ defined in the equation of
  Proposition \ref{contractionest}, namely
  \[ \eta L (\Phi) + (1 - \eta) V L' (\Psi) + F (\Psi) = \lambda (c, d) Z_d,
  \]
  we have, for any $0 < \sigma < 1$,
  \[ \lambda (c, d) \int_{\mathbbm{R}^2} | \partial_d V |^2 \eta = \pi \left(
     \frac{1}{d} - c \right) + O^{\sigma}_{c \rightarrow 0} (c^{2 - \sigma}) .
  \]
\end{proposition}

We will take the scalar product of $\eta L (\Phi) + (1 - \eta) V L' (\Psi) + F
(\Psi) - \lambda (c, d) Z_d$ with $\partial_d V$. We will show in the proof
that in the term $\langle \eta L (\Phi) + (1 - \eta) V L' (\Psi) + F (\Psi),
\partial_d V \rangle$, the largest contribution come from the source term $E -
i c \partial_{x_2} V$in $F (\Psi)$. We will show that $\langle E, \partial_d V
\rangle \simeq \frac{\pi}{d}$ and $\langle - i c \partial_{x_2} V, \partial_d
V \rangle \simeq - \pi c$, so that, at the leading order, $\lambda (c, d)
\thicksim K \left( \frac{1}{d} - c \right)$. In the proof, steps 1, 2 and 7
show that the terms other than $E - i c \partial_{x_2} V$ are of lower order,
and steps 3-6 compute exactly the contribution of these leading order terms.

\

\begin{proof}
  Recall from Lemma \ref{lemma7} that $L (\Phi) = (E - i c \partial_{x_2} V)
  \Psi + V L' (\Psi)$, hence we write the equation under the form
  \[ L (\Phi) - (1 - \eta)  (E - i c \partial_{x_2} V) \Psi + F (\Psi) =
     \lambda (c, d) Z_d . \]
  We want to take the scalar product with $\partial_d V$. We will compute the
  terms $(1 - \eta) E \Psi$ (step 1), $F (\Psi)$ (steps 2 to 6) and in step 7
  we will show that we can do an integration by parts for $\langle L (\Phi),
  Z_d \rangle$ and compute its contribution.
  
  \
  
  We have by definition $Z_d = \eta \partial_d V$, hence
  \[ \langle Z_d, \partial_d V \rangle = \int_{\mathbbm{R}^2} | \partial_d V
     |^2 \eta \]
  which is finite and independent of $d$ since $\eta = 0$ outside $\{
  \tilde{r} \leqslant 2 \}$. Recall that $\| \Psi \|_{\ast, \sigma} \leqslant
  K (\sigma, \sigma') c^{1 - \sigma'}$ where
  \begin{eqnarray*}
    \| \Psi \|_{\ast, \sigma} & = & \| V \Psi \|_{C^2 (\{ \tilde{r} \leqslant
    3 \})} + \| \tilde{r}^{1 + \sigma} \Psi_1 \|_{L^{\infty} (\{ \tilde{r}
    \geqslant 2 \})} + \| \tilde{r}^{2 + \sigma} \nabla \Psi_1 \|_{L^{\infty}
    (\{ \tilde{r} \geqslant 2 \})}\\
    & + & \| \tilde{r}^{\sigma} \Psi_2 \|_{L^{\infty} (\{ \tilde{r} \geqslant
    2 \})} + \| \tilde{r}^{1 + \sigma} \nabla \Psi_2 \|_{L^{\infty} (\{
    \tilde{r} \geqslant 2 \})} + \| \tilde{r}^{2 + \sigma} \nabla^2 \Psi
    \|_{L^{\infty} (\{ \tilde{r} \geqslant 2 \})},
  \end{eqnarray*}
  which we will heavily use with several values of $\sigma, \sigma'$ in the
  following computations, in particular for $\sigma \in] 0, 1 [$, the estimate
  \[ \| \Psi \|_{\ast, \sigma / 2, d} \leqslant K (\sigma) c^{1 - \sigma} . \]
  
\ 
 
  \begin{tmindent}
    Step 1.  We have $\langle (1 - \eta) (E - i c \partial_{x_2} V) \Psi,
    \partial_d V \rangle = O_{c \rightarrow 0}^{\sigma} (c^{2 - \sigma})$.
  \end{tmindent}
  
\
  
  From Lemma \ref{ddVest}, we have
  \begin{equation}
    | \partial_d V | \leqslant \frac{K}{1 + \tilde{r}} . \label{ddv}
  \end{equation}
  In (\ref{E2}), we showed that
  \[ E = - 2 \nabla V_1 . \nabla V_{- 1} + (1 - | V_1 |^2) (1 - | V_{- 1} |^2)
     V_1 V_{- 1}, \]
  hence, with Lemmas \ref{lemme3} and \ref{fineest} (estimating $i c
  \partial_{x_2} V$ as in step 2 of the proof of Lemma \ref{fineestimate}), we
  have
  \[ | E - i c \partial_{x_2} V | \leqslant \frac{K c}{1 + \tilde{r}} \]
  by using $| \nabla V_1 | \leqslant \frac{K}{1 + \tilde{r}}$, $| \nabla V_{-
  1} | \leqslant \frac{K}{d} \leqslant K c$ and $| 1 - | V_{- 1} |^2 |
  \leqslant K c^2$ in the right half-plane and the symmetric estimate in the
  other one. We also have, in $\{ 1 - \eta \neq 0 \}$,
  \[ | \Psi | \leqslant K \frac{\| \Psi \|_{\ast, \sigma / 2, d}}{(1 +
     \tilde{r})^{\sigma / 2}} \leqslant \frac{K (\sigma) c^{1 - \sigma}}{(1 +
     \tilde{r})^{\sigma / 2}}, \]
  hence
  \[ | \langle (1 - \eta) (E - i c \partial_{x_2} V) \Psi, \partial_d V
     \rangle | \leqslant K (\sigma)  \int_{\mathbbm{R}^2} \frac{c^{2 -
     \sigma}}{(1 + \tilde{r})^{2 + \sigma / 2}} = O_{c \rightarrow 0}^{\sigma}
     (c^{2 - \sigma}) . \]
  
\  
  
  \begin{tmindent}
    Step 2.  We have $\langle F (\Psi), \partial_d V \rangle = \langle E - i c
    \partial_{x_2} V, \partial_d V \rangle + O_{c \rightarrow 0}^{\sigma}
    (c^{2 - \sigma})$.
  \end{tmindent}
  
\

  In this step, we want to show that the nonlinear terms in $F (\Psi)$ are
  negligible. Recall that
  \[ F (\Psi) = E - i c \partial_{x_2} V + R (\Psi) + V (1 - \eta) (- \nabla
     \Psi . \nabla \Psi + | V |^2 S (\Psi)) . \]
  We first show that
  \[ \langle R (\Psi), \partial_d V \rangle = O_{c \rightarrow 0}^{\sigma}
     (c^{2 - \sigma}) . \]
  Indeed, $R (\Psi)$ is localized in $\{ \tilde{r} \leqslant 2 \}$ and $| R
  (\Psi) | \leqslant C \| \Phi \|_{C^1 (\{ \tilde{r} \leqslant 3 \})}^2$
  (since $\| \Psi \|_{\ast, \sigma, d} \leqslant C_0$, see Lemma
  \ref{lemma7}), and using that in $\{ \tilde{r} \leqslant 3 \}$, $| \Phi | +
  | \nabla \Phi | \leqslant K (\sigma) c^{1 - \sigma / 2}$ yields
  \[ | R (\Psi) | \leqslant c \| \partial_{x_2} \Phi \|_{C^0 (\{ \tilde{r}
     \leqslant 3 \})} + C \| \Phi \|^2_{C^1 (\{ \tilde{r} \leqslant 3 \})} =
     O_{c \rightarrow 0}^{\sigma} (c^{2 - \sigma}) . \]
  Now, we use $\| \Psi \|_{\ast, \sigma / 2, d} \leqslant K (\sigma) c^{1 -
  \sigma}$ to estimate, in $\{ 1 - \eta \neq 0 \}$,
  \[ | \nabla \Psi . \nabla \Psi | \leqslant \frac{K \| \Psi \|_{\ast, \sigma,
     d}^2}{(1 + \tilde{r})^{2 + \sigma}} \leqslant \frac{K (\sigma) c^{2 -
     \sigma}}{(1 + \tilde{r})^{2 + \sigma}}, \]
  therefore
  \[ | \langle - \nabla \Psi . \nabla \Psi V (1 - \eta), \partial_d V \rangle
     | \leqslant K c^{2 - \sigma} \int_{\mathbbm{R}^2} \frac{1}{(1 +
     \tilde{r})^{3 + \sigma}} = O_{c \rightarrow 0}^{\sigma} (c^{2 - \sigma})
     . \]
  The same argument can be made for
  \[ | \langle - | V |^2 S (\Psi) V (1 - \eta), \partial_d V \rangle | = O_{c
     \rightarrow 0}^{\sigma} (c^{2 - \sigma}) \]
  by using $S (\Psi) = e^{2\mathfrak{R}\mathfrak{e} (\Psi)} - 1 -
  2\mathfrak{R}\mathfrak{e} (\Psi)$ and the fact that it is real-valued.
  
\
  
  \begin{tmindent}
    Step 3.  We have $\langle E - i c \partial_{x_2} V, \partial_d V \rangle =
    - 2 \int_{\{ x_1 \geqslant 0 \}} \mathfrak{R}\mathfrak{e} \left( (E - i c
    \partial_{x_2} V) \overline{\partial_{x_1} V_1 V_{- 1}} \right) + O_{c
    \rightarrow 0}^{\sigma} (c^{2 - \sigma})$.
  \end{tmindent}
 
 \ 
  
  The goal of this step is to simplify the computation by using the symmetry.
  By symmetry, we can only look in the right half-plane:
  \[ \langle E - i c \partial_{x_2} V, \partial_d V \rangle = 2 \int_{\{ x_1
     \geqslant 0 \}} \mathfrak{R}\mathfrak{e} ((E - i c \partial_{x_2} V)
     \overline{\partial_d V}) . \]
  Recall that $\partial_d V = - \partial_{x_1} V_1 V_{- 1} + \partial_{x_1}
  V_{- 1} V_1$, hence we need to show that
  \[ \int_{\{ x_1 \geqslant 0 \}} \mathfrak{R}\mathfrak{e} \left( (E - i c
     \partial_{x_2} V) \overline{\partial_{x_1} V_{- 1} V_1} \right) = O_{c
     \rightarrow 0}^{\sigma} (c^{2 - \sigma}) . \]
  We compute
  \begin{eqnarray*}
    \int_{\{ x_1 \geqslant 0 \}} \mathfrak{R}\mathfrak{e} \left( (E - i c
    \partial_{x_2} V) \overline{\partial_{x_1} V_{- 1} V_1} \right) & = &
    \int_{\{ x_1 \geqslant 0 \}} \mathfrak{R}\mathfrak{e} \left( \left(
    \frac{E - i c \partial_{x_2} V}{V} | V |^2 \right)
    \overline{\frac{\partial_{x_1} V_{- 1}}{V_{- 1}}} \right)\\
    & = & \int_{\{ x_1 \geqslant 0 \}} \mathfrak{R}\mathfrak{e} \left(
    \frac{E - i c \partial_{x_2} V}{V} | V |^2 \right)
    \mathfrak{R}\mathfrak{e} \left( \frac{\partial_{x_1} V_{- 1}}{V_{- 1}}
    \right)\\
    & + & \int_{\{ x_1 \geqslant 0 \}} \mathfrak{I}\mathfrak{m} \left(
    \frac{E - i c \partial_{x_2} V}{V} | V |^2 \right)
    \mathfrak{I}\mathfrak{m} \left( \frac{\partial_{x_1} V_{- 1}}{V_{- 1}}
    \right) .
  \end{eqnarray*}
  In the right half-plane, we have $d \leqslant r_{- 1}$ and $\tilde{r}
  \leqslant r_1$, hence
  \[ \left| \mathfrak{R}\mathfrak{e} \left( \frac{\partial_{x_1} V_{- 1}}{V_{-
     1}} \right) \right| \leqslant \frac{K}{r^3_{- 1}} \leqslant \frac{K c^{1
     - \sigma / 2}}{(1 + \tilde{r})^{2 + \sigma / 2}}, \]
  \[ \left| \mathfrak{I}\mathfrak{m} \left( \frac{\partial_{x_1} V_{- 1}}{V_{-
     1}} \right) \right| \leqslant \frac{K}{r_{- 1}} \leqslant \frac{K c^{1 -
     \sigma / 2}}{(1 + \tilde{r})^{\sigma / 2}}, \]
  from Lemma \ref{lemme3}. Moreover,
  \[ \left| \mathfrak{R}\mathfrak{e} \left( \frac{E - i c \partial_{x_2} V}{V}
     | V |^2 \right) \right| \leqslant \frac{K c^{1 - \sigma / 2}}{(1 +
     \tilde{r})^{1 + \sigma / 2}}, \]
  \[ \left| \mathfrak{I}\mathfrak{m} \left( \frac{E - i c \partial_{x_2} V}{V}
     | V |^2 \right) \right| \leqslant \frac{K c^{1 - \sigma / 2}}{(1 +
     \tilde{r})^{2 + \sigma / 2}}, \]
  from Lemma \ref{fineestimate}. We thus deduce that
  \[ \left| \int_{\{ x_1 \geqslant 0 \}} \mathfrak{R}\mathfrak{e} \left( (E -
     i c \partial_{x_2} V) \overline{\partial_{x_1} V_{- 1} V_1} \right)
     \right| \leqslant K c^{1 - \sigma / 2}  \int_{\mathbbm{R}^2} \frac{c^{1 -
     \sigma / 2}}{(1 + \tilde{r})^{2 + \sigma}} = O_{c \rightarrow 0}^{\sigma}
     (c^{2 - \sigma}) . \]

\
  
  \begin{tmindent}
    Step 4.  We have
    \[ \int_{\{ x_1 \geqslant 0 \}} \mathfrak{R}\mathfrak{e} \left( E
       \overline{\partial_{x_1} V_1 V_{- 1}} \right) = - 2 \int_{\{ x_1
       \geqslant 0 \}} \mathfrak{R}\mathfrak{e} \left( \partial_{x_2} V_1
       \overline{\partial_{x_1} V_1} \partial_{x_2} V_{- 1} \overline{V_{- 1}}
       \right) + O_{c \rightarrow 0}^{\sigma} (c^{2 - \sigma}) . \]
  \end{tmindent}
  
\ 
  
  The goal of this step is to compute the part of $E$ that produces the higher
  order term. Recall from (\ref{E2}) that
  \[ E = - 2 \nabla V_1 . \nabla V_{- 1} + (1 - | V_1 |^2) (1 - | V_{- 1} |^2)
     V_1 V_{- 1} \]
  and since
  \[ | (1 - | V_1 |^2) (1 - | V_{- 1} |^2) | \leqslant \frac{K c^2}{(1 +
     \tilde{r})^2} \]
  by Lemma \ref{lemme3}, we deduce
  \[ \int_{\{ x_1 \geqslant 0 \}} \mathfrak{R}\mathfrak{e} \left( (1 - | V_1
     |^2) (1 - | V_{- 1} |^2) V_1 V_{- 1} \overline{\partial_{x_1} V_1 V_{-
     1}} \right) = O_{c \rightarrow 0}^{\sigma} (c^{2 - \sigma}) . \]
  Now we compute the first contribution from $- 2 \nabla V_1 . \nabla V_{- 1}
  = - 2 \partial_{x_1} V_1 \partial_{x_1} V_{- 1} - 2 \partial_{x_2} V_1
  \partial_{x_2} V_{- 1}$,
  \[ \int_{\{ x_1 \geqslant 0 \}} \mathfrak{R}\mathfrak{e} \left( (- 2
     \partial_{x_1} V_1 \partial_{x_1} V_{- 1}) \overline{\partial_{x_1} V_1
     V_{- 1}} \right) = - 2 \int_{\{ x_1 \geqslant 0 \}} | \partial_{x_1} V_1
     |^2 \mathfrak{R}\mathfrak{e} (\partial_{x_1} V_{- 1} \overline{V_{- 1}})
     . \]
  From Lemma \ref{lemme3} we have
  \[ \mathfrak{R}\mathfrak{e} (\partial_{x_1} V_{- 1} \overline{V_{- 1}}) = O
     \left( \frac{1}{r_{- 1}^3} \right) \]
  since the main part in $\partial_{x_1} V_{- 1} \overline{V_{- 1}}$ is purely
  imaginary. Using $r_1 \leqslant r_{- 1} \tmop{and} r_{- 1} \geqslant d
  \geqslant \frac{K}{c}$ in the right half-plane, we have $\frac{1}{r_{- 1}^3}
  \leqslant \frac{K c^{2 - \sigma}}{(1 + \tilde{r})^{1 + \sigma}}$ and, noting
  that $| \partial_{x_1} V_1 |^2 \leqslant \frac{K}{(1 + \tilde{r})^2}$, we
  obtain
  \[ \int_{\{ x_1 \geqslant 0 \}} | \partial_{x_1} V_1 |^2  |
     \mathfrak{R}\mathfrak{e} (\partial_{x_1} V_{- 1} \overline{V_{- 1}}) |
     \leqslant K c^{2 - \sigma} \int_{\{ x_1 \geqslant 0 \}} \frac{1}{(1 +
     \tilde{r})^{3 + \sigma}} = O_{c \rightarrow 0} (c^{5 / 4}) . \]
  Finally, the second contribution from $- 2 \nabla V_1 . \nabla V_{- 1}$ is
  \[ \int_{\{ x_1 \geqslant 0 \}} \mathfrak{R}\mathfrak{e} \left( (- 2
     \partial_{x_2} V_1 \partial_{x_2} V_{- 1}) \overline{\partial_{x_1} V_1
     V_{- 1}} \right) = - 2 \int_{\{ x_1 \geqslant 0 \}}
     \mathfrak{R}\mathfrak{e} \left( \partial_{x_2} V_1
     \overline{\partial_{x_1} V_1} \partial_{x_2} V_{- 1} \overline{V_{- 1}}
     \right) \]
  which concludes the proof of this step.
  
\ 
  
  \begin{tmindent}
    Step 5.  We have $\int_{\{ x_1 \geqslant 0 \}} \mathfrak{R}\mathfrak{e}
    \left( E \overline{\partial_{x_1} V_1 V_{- 1}} \right) = \frac{\pi}{d} +
    O_{c \rightarrow 0}^{\sigma} (c^{2 - \sigma})$.
  \end{tmindent}
  
\  
  
  By Lemma \ref{lemme3}, we have
  \[ \partial_{x_2} V_{- 1} \overline{V_{- 1}} = - i | V_{- 1} |^2 \frac{y_1 +
     2 d}{r_{- 1}^2} + O \left( \frac{1}{r_{- 1}^3} \right) . \]
  The $O \left( \frac{1}{r_{- 1}^3} \right)$ yielding a term which is a $O_{c
  \rightarrow 0}^{\sigma} (c^{2 - \sigma})$ as in step 4, therefore
  \[ \int_{\{ x_1 \geqslant 0 \}} \mathfrak{R}\mathfrak{e} \left( (- 2
     \partial_{x_2} V_1 \partial_{x_2} V_{- 1}) \overline{\partial_{x_1} V_1
     V_{- 1}} \right) = 2 \int_{\{ x_1 \geqslant 0 \}}
     \mathfrak{R}\mathfrak{e} \left( i \partial_{x_2} V_1
     \overline{\partial_{x_1} V_1} \right) | V_{- 1} |^2 \frac{y_1 + 2 d}{r_{-
     1}^2} + O_{c \rightarrow 0}^{\sigma} (c^{2 - \sigma}) . \]
  Now we compute in polar coordinate around $d \overrightarrow{e_1}$, writing
  $V_1 = \rho_1 (r_1) e^{i \theta_1}$. From Lemma \ref{dervor}, we have
  \[ \partial_{x_1} V_1 = \left( \cos (\theta_1) \frac{\rho'_1 (r_1)}{\rho_1
     (r_1)} - \frac{i}{r_1} \sin (\theta_1) \right) V_1, \]
  \[ \partial_{x_2} V_1 = \left( \sin (\theta_1) \frac{\rho'_1 (r_1)}{\rho_1
     (r_1)} + \frac{i}{r_1} \cos (\theta_1) \right) V_1 . \]
  We then compute
  \[ \mathfrak{R}\mathfrak{e} \left( i \partial_{x_2} V_1
     \overline{\partial_{x_1} V_1} \right) = - | V_1 |^2 \left( \cos^2
     (\theta_1) \frac{\rho'_1}{r_1 \rho_1} + \sin^2 (\theta_1)
     \frac{\rho'_1}{r_1 \rho_1} \right) = - | V_1 |^2 \frac{\rho'_1}{r_1
     \rho_1} . \]
  From Lemma \ref{lemme3}, we have $\rho'_1 (r_1) = O_{r_1 \rightarrow \infty}
  \left( \frac{1}{r_1^3} \right)$. As a consequence
  \begin{eqnarray*}
    &  & \left| \int_{\{ x_1 \geqslant 0 \}} | V_1 |^2 \frac{\rho'_1}{r_1
    \rho_1} | V_{- 1} |^2 \frac{y_1 + 2 d}{r_{- 1}^2} - \int_{\{ r_1 \leqslant
    d^{1 / 2} \}} | V_1 |^2 \frac{\rho'_1}{r_1 \rho_1} | V_{- 1} |^2 \frac{y_1
    + 2 d}{r_{- 1}^2} \right|\\
    & \leqslant & K c^{2 - \sigma} \int_{\{ r_1 \geqslant d^{1 / 2} \}}
    \frac{1}{(1 + \tilde{r})^{2 + 2 \sigma}}
  \end{eqnarray*}
  because when $x_1 \geqslant 0 \tmop{and} r_1 \geqslant d^{1 / 2}$, we have
  $\left| | V_1 |^2 \frac{\rho'_1}{r_1 \rho_1} | V_{- 1} |^2 \frac{y_1 + 2
  d}{r_{- 1}^2} \right| \leqslant \frac{K c^{2 - \sigma}}{(1 + \tilde{r})^{2 +
  2 \sigma}}$. We deduce that
  \[ \int_{\{ x_1 \geqslant 0 \}} | V_1 |^2 \frac{\rho'_1}{r_1 \rho_1} | V_{-
     1} |^2 \frac{y_1 + 2 d}{r_{- 1}^2} = \int_{\{ r_1 \leqslant d^{1 / 2} \}}
     | V_1 |^2 \frac{\rho'_1}{r_1 \rho_1} | V_{- 1} |^2 \frac{y_1 + 2 d}{r_{-
     1}^2} + O_{c \rightarrow 0}^{\sigma} (c^{2 - \sigma}) . \]
  In the ball $\{ r_1 \leqslant d^{1 / 2} \}$, we have
  \[ r_{- 1}^2 = 4 d^2 \left( 1 + O_{d \rightarrow \infty} \left( \frac{1}{d}
     \right) \right) \quad \tmop{and} \quad | V_{- 1} |^2 = 1 + O \left(
     \frac{1}{d^2} \right) \]
  therefore
  \[ \int_{\{ x_1 \geqslant 0 \}} | V_1 |^2 \frac{\rho'_1}{r_1 \rho_1} | V_{-
     1} |^2 \frac{y_1 + 2 d}{r_{- 1}^2} = \frac{1}{4 d^2}  \int_{\{ r_1
     \leqslant d^{1 / 2} \}} | V_1 |^2 \frac{\rho'_1}{r_1 \rho_1} (y_1 + 2 d)
     + O_{c \rightarrow 0}^{\sigma} (c^{2 - \sigma}) . \]
  Since $y_1 = r_1 \cos (\theta_1)$, by integration in polar coordinates we
  have
  \[ \int_{\{ r_1 \leqslant d^{1 / 2} \}} | V_1 |^2 \frac{\rho'_1}{r_1 \rho_1}
     y_1 = 0 \]
  hence
  \[ \int_{\{ x_1 \geqslant 0 \}} \mathfrak{R}\mathfrak{e} \left( E
     \overline{\partial_{x_1} V_1 V_{- 1}} \right) = \frac{1}{d}  \int_{\{ r_1
     \leqslant d^{1 / 2} \}} | V_1 |^2 \frac{\rho'_1}{r_1 \rho_1} + O_{c
     \rightarrow 0}^{\sigma} (c^{2 - \sigma}) . \]
  Remark that $| V_1 |^2 = \rho_1^2$ hence
  \[ \int_{\{ r_1 \leqslant d^{1 / 2} \}} | V_1 |^2 \frac{\rho'_1}{r_1 \rho_1}
     = 2 \pi \int_0^{d^{1 / 2}} \rho_1 \rho_1' d r_1 = \pi [\rho_1^2]_0^{d^{1
     / 2}} = \pi + O_{d \rightarrow \infty} \left( \frac{1}{d^{}} \right) \]
  Since $\rho_1 = 1 + O \left( \frac{1}{r_1^2} \right)$ when $r_1 \rightarrow
  \infty$ and $\rho_1 (0) = 0$ by Lemma \ref{lemme3}. Therefore, as claimed,
  \[ \int_{\{ x_1 \geqslant 0 \}} \mathfrak{R}\mathfrak{e} \left( E
     \overline{\partial_{x_1} V_1 V_{- 1}} \right) = \frac{\pi}{d} + O_{c
     \rightarrow 0}^{\sigma} (c^{2 - \sigma}) . \]
  Notice that we have shown in particular that
  \begin{equation}
    \int_{\mathbbm{R}^2} \mathfrak{R}\mathfrak{e} \left( i \partial_{x_2} V_1
    \overline{\partial_{x_1} V_1} \right) | V_{- 1} |^2 = - \pi +
    O^{\sigma}_{c \rightarrow 0} (c^{1 - \sigma}) . \label{bleu}
  \end{equation}
  
\  
  
  \begin{tmindent}
    Step 6.  We have $\int_{\{ x_1 \geqslant 0 \}} \mathfrak{R}\mathfrak{e}
    \left( - i c \partial_{x_2} V \overline{\partial_{x_1} V_1 V_{- 1}}
    \right) = - \pi c + O_{c \rightarrow 0}^{\sigma} (c^{2 - \sigma})$.
  \end{tmindent}
  
\  
  
  We are left with the computation of
  \[ \int_{\{ x_1 \geqslant 0 \}} \mathfrak{R}\mathfrak{e} \left( - i c
     \partial_{x_2} V \overline{\partial_{x_1} V_1 V_{- 1}} \right) = \]
  \begin{equation}
    \int_{\{ x_1 \geqslant 0 \}} \mathfrak{R}\mathfrak{e} \left( - i c
    \partial_{x_2} V_1 \overline{\partial_{x_1} V_1} \right) | V_{- 1} |^2 +
    \int_{\{ x_1 \geqslant 0 \}} \mathfrak{R}\mathfrak{e} \left( - i c
    \partial_{x_2} V_{- 1} V_1 \overline{\partial_{x_1} V_1 V_{- 1}} \right)
    \label{noir}
  \end{equation}
  since $\partial_{x_2} V = \partial_{x_2} V_1 V_{- 1} + \partial_{x_2} V_{-
  1} V_1$. For the second term in (\ref{noir}), we compute
  \[ - c \int_{\{ x_1 \geqslant 0 \}} \mathfrak{R}\mathfrak{e} \left( i
     \partial_{x_2} V_{- 1} V_1 \overline{\partial_{x_1} V_1 V_{- 1}} \right)
     = c \int_{\{ x_1 \geqslant 0 \}} \mathfrak{R}\mathfrak{e} \left(
     \overline{\partial_{x_1} V_1} V_1 \right) | V_{- 1} |^2 \frac{y_1 + 2
     d}{r_{- 1}^2} + O_{c \rightarrow 0}^{\sigma} (c^{2 - s}) \]
  in view of the relation
  \[ i \partial_{x_2} V_{- 1} \overline{V_{- 1}} = - | V_{- 1} |^2 \frac{y_1 +
     2 d}{r_{- 1}^2} + O \left( \frac{1}{r_{- 1}^3} \right) \]
  from Lemma \ref{lemme3} and the fact that \ $\int_{\{ x_1 \geqslant 0 \}}
  c^{} O \left( \frac{1}{r_{- 1}^3} \right) = O_{c \rightarrow 0}^{\sigma}
  (c^{2 - \sigma})$ (as in step 4). Now recall from Lemma \ref{dervor} that
  \[ \partial_{x_1} V_1 = \left( \cos (\theta_1) \frac{\rho'_1 (r_1)}{\rho_1
     (r_1)} - \frac{i}{r_1} \sin (\theta_1) \right) V_1 \]
  therefore
  \[ \mathfrak{R}\mathfrak{e} \left( \overline{\partial_{x_1} V_1} V_1
     \right) = \cos (\theta_1) \frac{\rho'_1}{\rho_1} | V_1 |^2 . \]
  In particular, $\left| \mathfrak{R}\mathfrak{e} \left(
  \overline{\partial_{x_1} V_1} V_1 \right) \right| \leqslant \frac{K}{1 +
  r_1^3}$ is integrable. Furthermore,$\left| | V_{- 1} |^2 \frac{y_1 + 2
  d}{r_{- 1}^2} \right| = O_{c \rightarrow 0} (c)$ in the right half-plane,
  therefore
  \[ - c \int_{\{ x_1 \geqslant 0 \}} \mathfrak{R}\mathfrak{e} \left( i
     \partial_{x_2} V_{- 1} V_1 \overline{\partial_{x_1} V_1 V_{- 1}} \right)
     = O_{c \rightarrow 0} (c^2) = O_{c \rightarrow 0}^{\sigma} (c^{2 -
     \sigma}) . \]
  The first contribution in (\ref{noir}) is
  \[ c \int_{\{ x_1 \geqslant 0 \}} \mathfrak{R}\mathfrak{e} \left( i
     \partial_{x_2} V_1 \overline{\partial_{x_1} V_1} \right) | V_{- 1} |^2 =
     c \int_{\{ x_1 \geqslant 0 \}} \mathfrak{R}\mathfrak{e} \left( i
     \partial_{x_2} V_1 \overline{\partial_{x_1} V_1} \right) + O_{c
     \rightarrow 0}^{\sigma} (c^{2 - \sigma}) \]
  using that $| V_{- 1} |^2 = 1 + O \left( \frac{1}{r_{- 1}^2} \right)$. From
  (\ref{bleu}), we have
  \[ \int_{\{ x_1 \geqslant 0 \}} \mathfrak{R}\mathfrak{e} \left( i
     \partial_{x_2} V_1 \overline{\partial_{x_1} V_1} \right) = - \pi +
     O^{\sigma}_{c \rightarrow 0} (c^{1 - \sigma}) . \]
  This conclude the proof of step 6, and combining step 4, 5 and 6 we deduce
  \[ \int_{\{ x_1 \geqslant 0 \}} \mathfrak{R}\mathfrak{e} \left( (E - i c
     \partial_{x_2} V) \overline{\partial_{x_1} V_1 V_{- 1}} \right) = \pi
     \left( \frac{1}{d} - c \right) + O_{c \rightarrow 0}^{\sigma} (c^{2 -
     \sigma}) . \]
  
\ 
 
  \begin{tmindent}
    Step 7.  We have $\langle L (\Phi), \partial_d V \rangle = O_{c
    \rightarrow 0}^{\sigma} (c^{2 - \sigma})$.
  \end{tmindent}
 
 \ 
  
  We want to compute, by integration by parts, that
  \[ \langle L (\Phi), \partial_d V \rangle = \langle \Phi, L (\partial_d V)
     \rangle . \]
  First, we recall that the left hand side is well defined, because we showed
  in the previous steps that all the other terms are bounded, therefore this
  one is also bounded. We have
  \[ \int_{B (0, R)} \mathfrak{R}\mathfrak{e} (\Delta \Phi
     \overline{\partial_d V}) = \int_{\partial B (0, R)}
     \mathfrak{R}\mathfrak{e} (\nabla \Phi \overline{\partial_d V}) . \vec{n}
     -\mathfrak{R}\mathfrak{e} (\Phi \overline{\nabla \partial_d V}) . \vec{n}
     + \int_{B (0, R)} \mathfrak{R}\mathfrak{e} (\Phi \overline{\Delta
     \partial_d V}), \]
  and
  \[ | \mathfrak{R}\mathfrak{e} (\nabla \Phi \overline{\partial_d V}) | + |
     \mathfrak{R}\mathfrak{e} (\Phi \overline{\nabla \partial_d V}) |
     \leqslant \frac{K}{(1 + \tilde{r})^{2 + 1 / 2}}, \]
  therefore
  \[ \int_{\partial B (0, R)} \mathfrak{R}\mathfrak{e} (\nabla \Phi
     \overline{\partial_d V}) . \vec{n} -\mathfrak{R}\mathfrak{e} (\Phi
     \overline{\nabla \partial_d V}) . \vec{n} = o_{R \rightarrow \infty} (1)
  \]
  and the integration by parts holds.
  
  \
  
  Recall that
  \[ L (h) = - \Delta h - (1 - | V |^2) h + 2\mathfrak{R}\mathfrak{e} (\bar{V}
     h) V - i c \partial_{x_2} h \]
  and
  \[ L_{V_1} (h) = - \Delta h - (1 - | V_1 |^2) h + 2\mathfrak{R}\mathfrak{e}
     (\overline{V_1} h) V_1 . \]
  From Lemma \ref{ddVest} and $\| \Psi \|_{\ast, \sigma / 2} \leqslant K
  (\sigma) c^{1 - \sigma}$, we check easily that
  \[ | \langle \Phi, - i c \partial_{x_2} \partial_d V \rangle | \leqslant
     \int_{\mathbbm{R}^2} \frac{K (\sigma) c^{2 - \sigma}}{(1 + \tilde{r})^{2
     + \sigma / 2}} = O_{c \rightarrow 0}^{\sigma} (c^{2 - \sigma}) . \]
  We therefore focus on the remaining part, with the operator
  \[ \tilde{L} (h) \assign - \Delta h - (1 - | V |^2) h +
     2\mathfrak{R}\mathfrak{e} (\bar{V} h) V - i c \partial_{x_2} h. \]
  We remark that we have $L_{V_1} (\partial_{x_1} V_1) = 0$, since
  $\partial_{x_1} (- \Delta V_1 - (1 - | V_1 |^2) V_1) = 0$. Recall that
  $\partial_d V = - \partial_{x_1} V_1 V_{- 1} + \partial_{x_1} V_{- 1} V_1$
  and let us compute
  \[ \tilde{L} (V_{- 1} h) = L_{V_1} (h) V_{- 1} - \Delta (V_{- 1} h) + \Delta
     h V_{- 1} + (| V |^2 - | V_1 |^2) h V_{- 1} + 2\mathfrak{R}\mathfrak{e}
     (\overline{V_1} h) (1 - | V_{- 1} |^2) V, \]
  therefore, using the equation or $V_{- 1}$,
  \[ \tilde{L} (V_{- 1} h) = L_{V_1} (h) V_{- 1} - 2 \nabla V_{- 1} . \nabla h
     + (1 - | V_{- 1} |^2) (1 - | V_1 |^2) V_{- 1} h +
     2\mathfrak{R}\mathfrak{e} (\overline{V_1} h) (1 - | V_{- 1} |^2) V. \]
  Taking $h = \partial_{x_1} V_1$ then yields
  \[ \tilde{L} (V_{- 1} \partial_{x_1} V_1) = - 2 \nabla V_{- 1} . \nabla
     \partial_{x_1} V_1 + (1 - | V_{- 1} |^2) (1 - | V_1 |^2) V_{- 1}
     \partial_{x_1} V_1 + 2\mathfrak{R}\mathfrak{e} (\overline{V_1}
     \partial_{x_1} V_1) (1 - | V_{- 1} |^2) V. \]
  Remark that $| \nabla V_{- 1} . \nabla \partial_{x_1} V_1 | \leqslant
  \frac{K}{(1 + r_1) (1 + r_{- 1})^2}$, $| (1 - | V_{- 1} |^2) (1 - | V_1 |^2)
  V_{- 1} \partial_{x_1} V_1 | \leqslant \frac{K}{(1 + r_1)^3 (1 + r_{-
  1})^2}$ and $| 2\mathfrak{R}\mathfrak{e} (\overline{V_1} \partial_{x_1} V_1)
  (1 - | V_{- 1} |^2) V | \leqslant \frac{K}{(1 + r_1)^3 (1 + r_{- 1})^2}$ for
  a universal constant $K > 0$ by Lemma \ref{lemme3}, therefore
  \[ \langle \Phi, \tilde{L} (\partial_{x_1} V_1 V_{- 1}) \rangle = O_{c
     \rightarrow 0}^{\sigma} (c^{2 - \sigma}) . \]
  Exchanging the roles of $V_1$ and $V_{- 1}$, we have similarly
  \[ \tilde{L} (V_1 \partial_{x_1} V_{- 1}) = - 2 \nabla V_1 . \nabla
     \partial_{x_1} V_{- 1} + (1 - | V_{- 1} |^2) (1 - | V_1 |^2) V_1
     \partial_{x_1} V_{- 1} . \]
  We then conclude that
  \[ \langle \tilde{L} (\Phi), \partial_d V \rangle = O_{c \rightarrow
     0}^{\sigma} (c^{2 - \sigma}), \]
  which end the proof of this step. Notice that we have shown
  \begin{equation}
    | L (\partial_d V) | \leqslant \frac{K c}{(1 + \tilde{r})^2} \label{bleu2}
  \end{equation}
  because $\frac{1}{(1 + r_1) (1 + r_{- 1})} \leqslant \frac{K c}{(1 +
  \tilde{r})}$ in the whole space.

 \ 
  
  \begin{tmindent}
    Step 8.  Conclusion.
  \end{tmindent}
 
 \ 
  
  Adding all the results obtained in steps 1 to 7, we deduce
  \[ \lambda (c, d) \int_{\mathbbm{R}^2} | \partial_d V |^2 \eta = \pi \left(
     \frac{1}{d} - c \right) + O_{c \rightarrow 0}^{\sigma} (c^{2 - \sigma}) .
  \]
\end{proof}

At this point, we cannot conclude that there exists $d$ such that $\lambda (c,
d) = 0$. For that, we need to show that the $O_{c \rightarrow 0}^{\sigma}
(c^{2 - \sigma})$ is continuous with respect to $c$ and $d$. This will be
shown in section \ref{n3}.

\section{Construction and properties of the travelling wave}\label{n3}

Given $0 < \sigma < \sigma' < 1$, $d, c > 0$ satisfying $\frac{1}{2 c} < d <
\frac{2}{c}$ and $c < c_0 (\sigma, \sigma')$ defined in Proposition
\ref{contractionest}, we define $\Phi_{c, d} = V \Psi_{c, d} \in
\mathcal{E}_{\ast, \sigma, d}$ the fonction constructed by the contraction
mapping theorem in Proposition \ref{contractionest}. From Corollary
\ref{sigmarem}, for any $0 < \sigma < \sigma' < 1$, this function satisfies,
for $c < c_0 (\sigma, \sigma')$, that
\[ \| \Psi_{c, d} \|_{\ast, \sigma_1, d} \leqslant K (\sigma, \sigma') c^{1 -
   \sigma'} . \]
With equation (\ref{azz220}) and Proposition \ref{dinc}, if we show that
$\Phi_{c, d}$ is a continuous function of $c$ and $d$, then there exists $c_0
> 0$ such that, for any $0 < c < c_0$, by the intermediate value theorem,
there exists $d_c > 0$ such that $\lambda (c, d_c) = 0$. This would conclude
the construction of the travelling wave, and is done in subsection \ref{31}.
In subsection \ref{th1proof}, we compute some estimates on $Q_c$ which will be
usefull for understanding the linearized operator around $Q_c$. We also show
there that $Q_c$ is a travelling wave solution with finite energy.

\subsection{Proof that $\Phi_{c, d}$ is a $C^1$ function of $c$ and
$d$}\label{31}

To end the construction of the travelling wave, we only need the continuity of
$\Phi_{c, d}$ with respect to $c$ and $d$. But for the construction of the
$C^1$ branch of travelling wave in section \ref{qcc1}, we need its
differentiability.

\subsubsection{Setup of the problem}

From Proposition \ref{contractionest}, the function $\Phi_{c, d}$ is defined
by the implicit equation on $\mathcal{E}_{\ast, \sigma, d}$
\[ (\eta L (.) + (1 - \eta) V L' (. / V))^{- 1} (\Pi_d^{\bot} (- F (\Phi_{c,
   d} / V))) + \Phi_{c, d} = 0, \]
where $(\eta L (.) + (1 - \eta) V L' (. / V))^{- 1}$ is the linear operator
from $\mathcal{E}_{\ast \ast, \sigma', d} \cap \{ \langle ., Z_d \rangle = 0
\}$ to $\mathcal{E}_{\ast, \sigma, d}$, that, for a function $V h \in
\mathcal{E}_{\ast \ast, \sigma', d}$ with $\langle V h, Z_d \rangle = 0$,
yields the unique function $\Phi = V \Psi \in \mathcal{E}_{\ast, \sigma, d}$
such that
\[ \eta L (\Phi) + (1 - \eta) V L' (\Psi) = V h \]
in the distribution sense. We recall the quantity $Z_d (x) = \partial_d V (x)
(\tilde{\eta} (4 r_1) + \tilde{\eta} (4 r_{- 1}))$ defined in subsection
\ref{normsetup} and we have defined the projection
\[ \Pi_d^{\bot} (\Phi) = \Phi - \langle \Phi, Z_d \rangle \frac{Z_d}{\| Z_d
   \|_{L^2 (\mathbbm{R}^2)}^2} . \]
We want to show that $(c, d) \mapsto \Phi_{c, d}$ is of class $C^1$ from
values of $c, d$ such that $0 < c < c_0 (\sigma)$ and $\frac{1}{2 d} < c <
\frac{2}{d}$ to $\mathcal{E}_{\ast, \sigma, d}$. The first obstacle is that
$\mathcal{E}_{\ast, \sigma, d}$ depends on $d$ (through $\tilde{r}$), both in
the weights in $\| . \|_{\ast, \sigma, d}$ and in the orthogonality required:
$\langle \Phi, Z_d \rangle = 0$. To be able to use the implicit function
theorem, we first need to write an equation on $\Phi$ in a space that does not
depend on $d$. The norm $\| . \|_{\ast, \sigma, d}$ depends on $d$ (through
$\tilde{r}$):
\begin{eqnarray*}
  \| \Psi \|_{\ast, \sigma, d} & = & \| V \Psi \|_{C^2 (\{ \tilde{r} \leqslant
  3 \})} + \| \tilde{r}^{1 + \sigma} \Psi_1 \|_{L^{\infty} (\{ \tilde{r}
  \geqslant 2 \})} + \| \tilde{r}^{2 + \sigma} \nabla \Psi_1 \|_{L^{\infty}
  (\{ \tilde{r} \geqslant 2 \})}\\
  & + & \| \tilde{r}^{\sigma} \Psi_2 \|_{L^{\infty} (\{ \tilde{r} \geqslant 2
  \})} + \| \tilde{r}^{1 + \sigma} \nabla \Psi_2 \|_{L^{\infty} (\{ \tilde{r}
  \geqslant 2 \})} + \| \tilde{r}^{2 + \sigma} \nabla^2 \Psi \|_{L^{\infty}
  (\{ \tilde{r} \geqslant 2 \})} .
\end{eqnarray*}
For $d_{\circledast} \in \mathbbm{R}, d_{\circledast} \geqslant 10$ and $d \in
\mathbbm{R}$ such that $| d - d_{\circledast} | < \delta$ for some small
$\delta > 0$ (that we will fix later on), we define
\begin{eqnarray*}
  \| \Phi \|_{\circledast, \sigma, d_{\circledast}} & \assign & \| \Phi
  \|_{C^2 (\{ \tilde{r}_{\circledast} \leqslant 3 \})} + \left\|
  \tilde{r}_{\circledast}^{1 + \sigma} \mathfrak{R}\mathfrak{e} \left(
  \frac{\Phi}{V_{\circledast \nosymbol}} \right) \right\|_{L^{\infty} (\{
  \tilde{r}_{\circledast} \geqslant 2 \})} + \left\|
  \tilde{r}_{\circledast}^{2 + \sigma} \nabla \mathfrak{R}\mathfrak{e} \left(
  \frac{\Phi}{V_{\circledast \nosymbol}} \right) \right\|_{L^{\infty} (\{
  \tilde{r}_{\circledast} \geqslant 2 \})}\\
  & + & \left\| \tilde{r}_{\circledast}^{\sigma} \mathfrak{I}\mathfrak{m}
  \left( \frac{\Phi}{V_{\circledast \nosymbol}} \right) \right\|_{L^{\infty}
  (\{ \tilde{r}_{\circledast} \geqslant 2 \})} + \left\|
  \tilde{r}_{\circledast}^{1 + \sigma} \nabla \mathfrak{I}\mathfrak{m} \left(
  \frac{\Phi}{V_{\circledast \nosymbol}} \right) \right\|_{L^{\infty} (\{
  \tilde{r}_{\circledast} \geqslant 2 \})}\\
  & + & \left\| \tilde{r}_{\circledast}^{2 + \sigma} \nabla^2 \left(
  \frac{\Phi}{V_{\circledast \nosymbol}} \right) \right\|_{L^{\infty} (\{
  \tilde{r}_{\circledast} \geqslant 2 \})},
\end{eqnarray*}
where $V_{\circledast} = V_1 (x - d_{\circledast} \overrightarrow{e_1}) V_{-
1} (x + d_{\circledast} \overrightarrow{e_1})$ and $\tilde{r}_{\circledast} =
\min (r_{1, \circledast}, r_{- 1, \circledast})$ with $r_{1, \circledast} = |
x - d_{\circledast} \overrightarrow{e_1} |, r_{- 1, \circledast} = | x +
d_{\circledast} \overrightarrow{e_1} |$. Then, for $\Phi = V \Psi \in
\mathcal{E}_{\ast, \sigma, d}$ ($V$ taken in $d$),
\begin{equation}
  K_1 \| \Psi \|_{\ast, \sigma, d} \leqslant \| \Phi_{\nosymbol}
  \|_{\circledast, \sigma, d_{\circledast}} \leqslant K_2 \| \Psi \|_{\ast,
  \sigma, d} \label{322504}
\end{equation}
where $K_{1, 2} > 0$ are absolute when $| d - d_{\circledast} | < \delta$.
Indeed, we check with simple geometric arguments that if
$\tilde{r}_{\circledast} \geqslant 1$, $V$ taken in $d$, then $\tilde{r}
\geqslant 1 / 2$ and we have
\begin{equation}
  \left| \frac{V}{V_{\circledast}} - 1 \right| \leqslant \frac{K}{(1 +
  \tilde{r})} \quad \tmop{and} \quad \left| \nabla \left(
  \frac{V}{V_{\circledast}} \right) \right| \leqslant \frac{K}{(1 +
  \tilde{r})^2} \label{312504}
\end{equation}
for a universal constant $K > 0$. Moreover, we have, for instance, if
$\tilde{r}_{\circledast} \geqslant 2$ (hence $\tilde{r}_{\circledast}
\leqslant 2 \tilde{r}$),
\begin{eqnarray*}
  \left| \tilde{r}_{\circledast}^{1 + \sigma} \mathfrak{R}\mathfrak{e} \left(
  \frac{\Phi}{V_{\circledast \nosymbol}} \right) \right| & \leqslant & \left|
  \tilde{r}_{\circledast}^{1 + \sigma} \mathfrak{R}\mathfrak{e} \left(
  \frac{\Phi}{V_{\nosymbol}} \right) \right| + \left|
  \tilde{r}_{\circledast}^{1 + \sigma} \mathfrak{R}\mathfrak{e} \left(
  \frac{\Phi}{V_{\nosymbol}} \left( \frac{V}{V_{\circledast}} - 1 \right)
  \right) \right|\\
  & \leqslant & K \| \Psi \|_{\ast, \sigma, d} + K \left|
  r_{\circledast}^{\sigma} \frac{\Phi}{V} \right| \leqslant K \| \Psi
  \|_{\ast, \sigma, d} .
\end{eqnarray*}
Using (\ref{312504}), we can estimate similarly all the terms in
(\ref{322504}).

We define similarly, for $g = V_{\circledast} (g_1 + i g_2) \in C^1
(\mathbbm{R}^2)$, $\sigma' > 0$
\[ \| g \|_{\circledast \circledast, \sigma', d_{\circledast}} \assign \| g
   \|_{C^1 (\{ \tilde{r}_{\circledast} \leqslant 3 \})} + \|
   \tilde{r}_{\circledast}^{1 + \sigma'} g_1 \|_{L^{\infty} (\{
   \tilde{r}_{\circledast} \geqslant 2 \})} + \| \tilde{r}_{\circledast}^{2 +
   \sigma'} g_2 \|_{L^{\infty} (\{ \tilde{r}_{\circledast} \geqslant 2 \})} +
   \| \tilde{r}_{\circledast}^{2 + \sigma'} \nabla g \|_{L^{\infty} (\{
   \tilde{r}_{\circledast} \geqslant 2 \})} . \]
We have that there exist $C_1, C_2 > 0$ universal constants such that, for $0
< \sigma' < 1$ and any $d, d_{\circledast} \geqslant 10$ with $| d -
d_{\circledast} | < \delta$, for any $V h \in \mathcal{E}_{\ast \ast, \sigma',
d}$, $g = V h$,
\[ C_1 \| h \|_{\ast \ast, \sigma', d} \leqslant \| g_{\nosymbol}
   \|_{\circledast \circledast, \sigma', d_{\circledast}} \leqslant C_2 \| h
   \|_{\ast \ast, \sigma', d} . \]
We define the spaces, for $\sigma, \sigma' > 0$,
\[ \mathcal{E}_{\circledast, \sigma, d_{\circledast}} \assign \]
\[ \{ \Phi \in C^2 (\mathbbm{R}^2, \mathbbm{C}), \| \Phi \|_{\circledast,
   \sigma, d_{\circledast}} < + \infty, \langle \Phi, Z_{d_{\circledast}}
   \rangle = 0, \forall x \in \mathbbm{R}^2, \Phi (x_1, x_2) = \overline{\Phi
   (x_1, - x_2)} = \Phi (- x_1, x_2) \} \]
and
\[ \mathcal{E}_{\circledast \circledast, \sigma', d_{\circledast}} \assign \{
   g \in C^1 (\mathbbm{R}^2, \mathbbm{C}), \| g \|_{\circledast \circledast,
   \sigma', d_{\circledast}} < + \infty \} . \]
We infer that, from Proposition \ref{invertop}, that the operator
\[ (\eta L (.) + (1 - \eta) V L' (. / V))^{- 1} o \Pi_d^{\bot} \]
goes from $\mathcal{E}_{\circledast \circledast, \sigma', d_{\circledast} }$
to $\mathcal{E}_{\circledast, \sigma, d_{\circledast}}$, and that (for $0 <
\sigma < \sigma' < 1$)
\[ \interleave (\eta L (.) + (1 - \eta) V L' (. / V))^{- 1} o \Pi_d^{\bot}
   \interleave_{\mathcal{E}_{\circledast \circledast, \sigma', d_{\circledast}
   } \rightarrow \mathcal{E}_{\circledast, \sigma, d_{\circledast}}} \]
is bounded independently of $c, d$ and $d_{\circledast}$ if $| d -
d_{\circledast} | < \delta$. Indeed, the norms $\| . \|_{\ast, \sigma, d}$ and
$\| . \|_{\circledast, \sigma, d_{\circledast}}$ are equivalent, as well as
the norms $\| . \|_{\ast \ast, \sigma', d}$ and $\| . \|_{\circledast
\circledast, \sigma', d_{\circledast}}$ for any $\sigma, \sigma' > 0$. About
the orthogonality, we replaced $\langle \Phi, Z_d \rangle = 0$ by $\langle
\Phi, Z_{d_{\circledast}} \rangle = 0$. This does not change the proof of
Proposition \ref{invertop}, since when we argue by contradiction, if for a
universal constant $| \lambda | \leqslant \delta$ we took the orthogonality
$\langle \Phi, Z_{d + \lambda} \rangle = 0$ instead of $\langle \Phi, Z_d
\rangle = 0$, the proof does not change, given that $\delta$ is small enough
(independently of $d$). To be specific, we have to take $\delta$ small enough
such that $\langle \partial_{x_1} V_1, \partial_{x_1} V_1 (. + \lambda)
\rangle > 0$ for all $\lambda \in] - \delta, \delta [$.

Therefore, we take a sequence $\mathcal{D}^{(n)} > 0$ going to infinity such
that $| \mathcal{D}^{(n + 1)} -\mathcal{D}^{(n)} | < \delta / 2$, and for any
given $d$ large enough, there exists $k (d)$ such that $d \in] D^{(k (d))} -
\delta / 2, D^{(k (d))} + \delta / 2 [$, and the proof of Proposition
\ref{invertop} holds with the orthogonality $\langle \Phi, Z_{D^{(k (d))}}
\rangle = 0$ for any value of $d$ in $] D^{(k (d))} - \delta / 2, D^{(k (d))}
+ \delta / 2 [$. We denote $D^{(k (d))} = d_{\circledast}$. The inversion of
the linearized operator then holds for $d \in] D^{(n)} - \delta / 2, D^{(n)} +
\delta / 2 [$ with $D^{(n)} = d_{\circledast}$, for all $n \in \mathbbm{N}$
large enough.

Furthermore, the contraction arguments given in the proof of Proposition
\ref{contractionest} still hold (because the norms are equivalent), hence we
can define $\Phi_{c, d}$ by a fixed point argument if $\frac{1}{2 d} \leqslant
c \leqslant \frac{2}{d}$ and $| d - d_{\circledast} | < \delta$ in the space
$\mathcal{E}_{\circledast, \sigma, d_{\circledast}}$ that does not depend on
$d$.

We want to emphasize the fact that we change a little the definition of the
spaces compared to section \ref{sec2}. In particular, for $\Phi = V \Psi$, the
norm $\| . \|_{\circledast, \sigma, d_{\circledast}}$ is on the function
$\Phi$, and before, for $\| . \|_{\ast, \sigma, d}$, it was on $\Psi$. This is
because $V$ depends on $d$, and we want to remove any dependence on $d$. The
same remark holds for $\| . \|_{\circledast \circledast, \sigma',
d_{\circledast}}$ and $\| . \|_{\circledast \circledast, \sigma', d}$ (with $g
= V h$).

\

We continue, and we define
\[ H (\Phi, c, d) \assign (\eta L (.) + (1 - \eta) V L' (. / V))^{- 1} (-
   \Pi_d^{\bot} (F (\Phi / V))) + \Phi . \]
The function $\Phi_{c, d} \in \mathcal{E}_{\circledast, \sigma,
d_{\circledast}}$ is defined, for $\frac{1}{2 d} < c < \frac{2}{d}$ and $| d -
d_{\circledast} | < \delta$, by being the only solution in a ball of
$\mathcal{E}_{\circledast, \sigma, d_{\circledast}}$ (with a radius depending
on $\sigma, \sigma', d_{\circledast}$ and $c$ but not $d$) to the implicit
equation on $\Phi$: $H (\Phi, c, d) = 0$. This means that we shall be able to
use the implicit function theorem in the space $\mathcal{E}_{\circledast,
\sigma, d_{\circledast}}$ on the equation $H (\Phi, c, d) = 0$ to show that
$\Phi_{c, d}$ is a $C^1$ function of $d$ in $\mathcal{E}_{\circledast, \sigma,
d_{\circledast}}$ (for values of $d$ such that $\frac{1}{2 d} < c <
\frac{2}{d}$ and $| d - d_{\circledast} | < \delta$). We want to differentiate
this equation with respect to $\Phi$ at a fixed $c$ and $d$, and show that we
can invert the operator obtained when we take $\Phi$ close to $\Phi_{c, d}$.
Since $(\eta L (.) + (1 - \eta) V L' (. / V))^{- 1}$ and $\Pi_d^{\bot}$ are
linear operators that do not depend on $\Phi$, it is easy to check that $H
(\Phi, c, d)$ is differentiable with respect to $\Phi$, and we compute
\[ d_{\Phi} H (\Phi, c, d) (\varphi) = (\eta L (.) + (1 - \eta) V L' (. /
   V))^{- 1} (\Pi_d^{\bot} (- d_{\Psi} F (\varphi / V))) + \varphi . \]
To show that $d_{\Phi} H (\Phi, c, d) : \mathcal{E}_{\circledast, \sigma,
d_{\circledast}} \rightarrow \mathcal{E}_{\circledast, \sigma,
d_{\circledast}}$ and that it is invertible, it is enough to check that
\begin{equation}
  \interleave (\eta L (.) + (1 - \eta) V L' (. / V))^{- 1} (\Pi_d^{\bot}
  (d_{\Psi} F (. / V)))_{} \interleave_{\mathcal{E}_{\circledast, \sigma,
  d_{\circledast}} \rightarrow \mathcal{E}_{\circledast, \sigma,
  d_{\circledast}}} = o^{\sigma}_{c \rightarrow 0} (1), \label{roug}
\end{equation}
which implies that $d_{\Phi} H (\Phi, c, d)$ is a small perturbation of
$\tmop{Id}$ for small values of $c$ (at fixed $\sigma$). From Proposition
\ref{invertop}, we have that $\interleave (\eta L (.) + (1 - \eta) V L' (. /
V))^{- 1} o \Pi_d^{\bot} \interleave_{\mathcal{E}_{\circledast
\circledast, \sigma', d_{\circledast}} \rightarrow \mathcal{E}_{\circledast,
\sigma, d_{\circledast}}}$ is bounded independently of $d$ and
$d_{\circledast}$ if $| d - d_{\circledast} | < \delta$, thus it is enough to
check that, for some $\sigma' > \sigma$ (we will take $\sigma' = \frac{1 +
\sigma}{2} > \sigma$),
\[ \interleave d_{\Psi} F (.) \interleave_{\mathcal{E}_{\circledast, \sigma,
   d_{\circledast}} \rightarrow \mathcal{E}_{\circledast \circledast, \sigma',
   d_{\circledast}}} = o^{\sigma, \sigma'}_{c_{} \rightarrow 0} (1) . \]
This is a consequence of the following lemma (for functions $\Phi = V \Psi$
such that $\| \Psi \|_{\ast, \sigma, d} = o^{\sigma}_{c_{} \rightarrow 0}
(1)$, which is the case if $\Phi$ is near $\Phi_{c, d}$ since $\| \Psi_{c, d}
\|_{\ast, \sigma, d} \leqslant K (\sigma, \sigma') c^{1 - \sigma'}$), where we
do the computations with the $\ast -$norms since they are equivalent, with
uniform constants, to the $\circledast$-norms. We define
\[ \gamma (\sigma) \assign \frac{1 + \sigma}{2} > \sigma . \]
\begin{lemma}
  \label{lemme20}There exists $C > 0$ such that, for $0 < \sigma < 1$ and
  functions $\Phi = V \Psi, \varphi = V \psi \in \mathcal{E}_{\ast, \sigma,
  d}$, if $\frac{1}{2 d} < c < \frac{2}{d}$ and $\| \Psi \|_{\ast, \sigma, d}
  \leqslant 1$, then
  \[ \| d_{\Psi} F (\psi)_{} \|_{\ast \ast, \gamma (\sigma), d_{}} \leqslant C
     \| \Psi \|_{\ast, \sigma, d} \| \psi \|_{\ast, \sigma, d} . \]
\end{lemma}

\begin{proof}
  Recall from Lemma \ref{lemma7} that
  \[ F (\Psi) = E - i c \partial_{x_2} V + V (1 - \eta) (- \nabla \Psi .
     \nabla \Psi + | V |^2 S (\Psi)) + R (\Psi) \]
  with $S (\Psi) = e^{2\mathfrak{R}\mathfrak{e} (\Psi)} - 1 -
  2\mathfrak{R}\mathfrak{e} (\Psi)$ and $R (\Psi)$ at least quadratic in
  $\Phi$ and supported in $\{ \tilde{r} \leqslant 2 \}$. We compute
  \[ d_{\Psi} F (\psi) = V (1 - \eta) (- 2 \nabla \Psi . \nabla \psi + | V |^2
     d S (\psi)) + d_{\Psi} R (\psi) . \]
  We recall the condition $\frac{1}{2 d} < c < \frac{2}{d}$. For the term
  $d_{\Psi} R (\psi)$, since $R$ is a sum of terms at least quadratic in
  $\Phi$ and is supported in $\{ \tilde{r} \leqslant 2 \}$ (see the proof of
  Lemma \ref{lemma7}), when we differentiate, every term has $\Psi$ or $\nabla
  \Psi$ as a factor. Therefore,
  \begin{eqnarray*}
    \| d_{\Psi} R (\psi)_{} \|_{\ast \ast, \gamma (\sigma), d} & \leqslant & K
    \| \Phi \|_{C^2 (\{ \widetilde{r_{}} \leqslant 2 \})} \| V \psi \|_{C^2
    (\{ \widetilde{r_{}} \leqslant 2 \})}\\
    & \leqslant & K \| \Psi \|_{\ast, \sigma, d} \| \psi \|_{\ast, \sigma, d}
    .
  \end{eqnarray*}
  Now, for $\mathfrak{R}\mathfrak{e} (\nabla \Psi . \nabla \psi)$, since
  $\sigma > 0, \gamma (\sigma) < 1$, we estimate
  \begin{eqnarray*}
    \| \tilde{r}^{1 + \gamma (\sigma)} \mathfrak{R}\mathfrak{e} (\nabla \Psi .
    \nabla \psi) \|_{L^{\infty} (\{ \tilde{r} \geqslant 2 \})} & \leqslant &
    \| \tilde{r}^{1 + \gamma (\sigma)} | \nabla \Psi | \times | \nabla \psi |
    \|_{L^{\infty} (\{ \tilde{r} \geqslant 2 \})}\\
    & \leqslant & K \| \Psi \|_{\ast, \sigma, d} \| \psi \|_{\ast, \sigma, d}
    \left\| \frac{\tilde{r}^{1 + \gamma (\sigma)}}{\tilde{r}^{2 + 2 \sigma}}
    \right\|_{L^{\infty} (\{ \tilde{r} \geqslant 2 \})}\\
    & \leqslant & K \| \Psi \|_{\ast, \sigma, d} \| \psi \|_{\ast, \sigma, d}
    .
  \end{eqnarray*}
  Similarly,
  \begin{eqnarray*}
    \| \tilde{r}^{2 + \gamma (\sigma)} \mathfrak{I}\mathfrak{m} (\nabla \Psi .
    \nabla \psi) \|_{L^{\infty} (\{ \tilde{r} \geqslant 2 \})} & \leqslant &
    \| \tilde{r}^{2 + \gamma (\sigma)} \nabla \mathfrak{R}\mathfrak{e} \Psi .
    \nabla \mathfrak{I}\mathfrak{m} \psi \|_{L^{\infty} (\{ \tilde{r}
    \geqslant 2 \})}\\
    & + & \| \tilde{r}^{2 + \gamma (\sigma)} \nabla \mathfrak{I}\mathfrak{m}
    \Psi . \nabla \mathfrak{R}\mathfrak{e} \psi \|_{L^{\infty} (\{ \tilde{r}
    \geqslant 2 \})}\\
    & \leqslant & K \| \Psi \|_{\ast, \sigma, d} \| \psi \|_{\ast, \sigma, d}
    \left\| \frac{\tilde{r}^{2 + \gamma (\sigma)}}{\tilde{r}^{3 + 2 \sigma}}
    \right\|_{L^{\infty} (\{ \tilde{r} \geqslant 2 \})}\\
    & \leqslant & K \| \Psi \|_{\ast, \sigma, d} \| \psi \|_{\ast, \sigma, d}
    .
  \end{eqnarray*}
  With similar computation, we check that
  \[ \| \tilde{r}^{2 + \gamma (\sigma)} \nabla (\nabla \Psi . \nabla \psi)
     \|_{L^{\infty} (\{ \tilde{r} \geqslant 2 \})} \leqslant K \| \Psi
     \|_{\ast, \sigma, d} \| \psi \|_{\ast, \sigma, d} . \]
  Finally, we have
  \[ d_{\Psi} S (\psi) = 2\mathfrak{R}\mathfrak{e} (\psi)
     (e^{2\mathfrak{R}\mathfrak{e} (\Psi)} - 1), \]
  a real-valued term, and since $\| \Psi \|_{\ast, \sigma, d} \leqslant 1$, we
  estimate
  \begin{eqnarray*}
    \| \tilde{r}^{1 + \gamma (\sigma)} \mathfrak{R}\mathfrak{e} (\psi)
    (e^{2\mathfrak{R}\mathfrak{e} (\Psi)} - 1) \|_{L^{\infty} (\{ \tilde{r}
    \geqslant 2 \})} & \leqslant & K \| \tilde{r}^{1 + \gamma (\sigma)}
    \mathfrak{R}\mathfrak{e} (\psi) \mathfrak{R}\mathfrak{e} (\Psi)
    \|_{L^{\infty} (\{ \tilde{r} \geqslant 2 \})}\\
    & \leqslant & K \| \psi \|_{\ast, \sigma, d} \| \Psi \|_{\ast, \sigma, d}
    \left\| \frac{\tilde{r}^{1 + \gamma (\sigma)}}{\tilde{r}^{2 + 2 \sigma}}
    \right\|_{L^{\infty} (\{ \tilde{r} \geqslant 2 \})}\\
    & \leqslant & K \| \Psi_{c, d} \|_{\ast, \sigma, d} \| \psi \|_{\ast,
    \sigma, d},
  \end{eqnarray*}
  as well as
  \begin{eqnarray*}
    \| \tilde{r}^{2 + \gamma (\sigma)} \nabla (\mathfrak{R}\mathfrak{e} (\psi)
    (e^{2\mathfrak{R}\mathfrak{e} (\Psi)} - 1)) \|_{L^{\infty} (\{ \tilde{r}
    \geqslant 2 \})} & \leqslant & K \| \tilde{r}^{2 + \gamma (\sigma)}
    \mathfrak{R}\mathfrak{e} (\nabla \psi) \mathfrak{R}\mathfrak{e} (\Psi)
    \|_{L^{\infty} (\{ \tilde{r} \geqslant 2 \})}\\
    & + & K \| \tilde{r}^{2 + \gamma (\sigma)} \mathfrak{R}\mathfrak{e}
    (\psi) \mathfrak{R}\mathfrak{e} (\nabla \Psi) \|_{L^{\infty} (\{ \tilde{r}
    \geqslant 2 \})}\\
    & \leqslant & K \| \Psi_{c, d} \|_{\ast, \sigma, d} \| \psi \|_{\ast,
    \sigma, d} \left\| \frac{\tilde{r}^{2 + \gamma (\sigma)}}{\tilde{r}^{3 + 2
    \sigma}} \right\|_{L^{\infty} (\{ \tilde{r} \geqslant 2 \})}\\
    & \leqslant & K \| \Psi_{c, d} \|_{\ast, \sigma, d} \| \psi \|_{\ast,
    \sigma, d} .
  \end{eqnarray*}
  These estimates imply
  \[ \| d_{\Psi} F (\psi)_{} \|_{\ast \ast, \gamma (\sigma), d_{}} \leqslant C
     \| \Psi_{c, d} \|_{\ast, \sigma, d} \| \psi \|_{\ast, \sigma, d} . \]
\end{proof}

\subsubsection{Proof of the differentiabilities of $\Phi_{c, d}$ with respect
of $c$ and $d$}

We shall now show that $c \mapsto \Phi_{c, d}$ is $C^1$ and compute estimates
on $\partial_c \Psi_{c, d}$ at fixed $d$, and then show that $d \mapsto
\Phi_{c, d}$ is $C^1$ at fixed $c$ and estimate $\partial_d \Phi_{c, d}$.
These estimates will be usefull in subsection \ref{37}. For $d \mapsto
\Phi_{c, d}$, we will use the implicit function theorem (see Lemma
\ref{dderpsy}), but we start here with the derivation with respect to $c$.

\begin{lemma}
  \label{rest31}For $0 < \sigma < 1$, there exists $c_0 (\sigma) > 0$ such
  that, at fixed $d > \frac{1}{2 c_0 (\sigma)}$,
  \[ c \mapsto \Phi_{c, d} \in C^1 \left( \left] \frac{1}{2 d}, \frac{2}{d}
     \right[ \cap] 0, c_0 (\sigma) [, \mathcal{E}_{\ast, \sigma, d} \right) .
  \]
\end{lemma}

Remark that, at fixed $d$, $\partial_c \Phi_{c, d} = V \partial_c \Psi_{c,
d}$.

\begin{proof}
  In this proof, we consider a fixed $d > \frac{1}{2 c_0 (\sigma)}$. We
  define, for $c \in \mathbbm{R}$ such that $\frac{1}{2 d} < c < \frac{2}{d}$
  and $0 < c < c_0 (\sigma)$, the operator
  \[ \mathbbm{H}_c : \Phi \mapsto (\eta L (.) + (1 - \eta) V L' (. / V))^{- 1}
     (\Pi_d^{\bot} (F (\Phi / V))) \]
  from $\mathcal{E}_{\circledast, \sigma, d_{\circledast}}$ to
  $\mathcal{E}_{\circledast, \sigma, d_{\circledast}}$. The dependency on $c$
  is coming from both $F$ and $(\eta L (.) + (1 - \eta) V L' (. / V))^{- 1}$,
  and in this proof, we will add a subscript on these functions giving the
  value of $c$ where it is taken. Take $c' \in \mathbbm{R}$ such that
  $\frac{1}{2 d} < c' < \frac{2}{d}$ and $0 < c' < c_0 (\sigma)$, and let us
  show that
  \[ \| \mathbbm{H}_{c + \varepsilon} (\Phi_{c', d}) -\mathbbm{H}_c (\Phi_{c',
     d}) \|_{\circledast, \sigma, d_{\circledast}} = o_{\varepsilon
     \rightarrow 0}^{\sigma, c} (1) . \]
  In particular, remark that we look for a convergence uniform in $c'$. By
  definition of the operator $(\eta L (.) + (1 - \eta) V L' (. / V))^{- 1}$,
  the function $\mathbbm{H}_{c + \varepsilon} (\Phi_{c', d})$ (in
  $\mathcal{E}_{\circledast, \sigma, d_{\circledast}}$) is such that, in the
  distribution sense,
  \[ \left( \eta L (.) + (1 - \eta) V L' \left( \frac{.}{V} \right) \right)_{c
     + \varepsilon} (\mathbbm{H}_{c + \varepsilon} (\Phi_{c', d})) =
     \Pi_d^{\bot} (F_{c + \varepsilon} (\Phi_{c', d} / V)) . \]
  Since $\Phi_{c', d} \in C^{\infty} (\mathbbm{R}^2)$, we have that
  $\mathbbm{H}_{c + \varepsilon} (\Phi_{c', d}) \in C^{\infty}
  (\mathbbm{R}^2)$ and the equation is satisfied in the strong sense.
  Furtheremore, since $\Pi_d^{\bot} (F_{c + \varepsilon} (\Phi_{c', d} / V))
  \in \mathcal{E}_{\circledast \circledast, \frac{2 + \sigma}{3},
  d_{\circledast} }$ by Lemmas \ref{fineestimate} to \ref{L2151904} with $\|
  \Pi_d^{\bot} (F_{c + \varepsilon} (\Phi_{c', d} / V)) \|_{\circledast
  \circledast, \frac{2 + \sigma}{3}, d_{\circledast}} \leqslant K (\sigma)$
  (since $\Phi_{c', d} \in \mathcal{E}_{\circledast, \frac{2 + \sigma}{3},
  d_{\circledast} }$ with $\| \Phi_{c', d} \|_{\circledast, \frac{2 +
  \sigma}{3}, d_{\circledast}} \leqslant K (\sigma)$), we have, by Lemma
  \ref{funky}, that $\mathbbm{H}_{c + \varepsilon} (\Phi_{c', d}) \in
  \mathcal{E}_{\circledast, \gamma (\sigma), d_{\circledast}}$ (since $\gamma
  (\sigma) < \frac{2 + \sigma}{3}$) with, fom Proposition \ref{invertop}, $\|
  \mathbbm{H}_{c + \varepsilon} (\Phi_{c', d}) \|_{\circledast, \gamma
  (\sigma), d_{\circledast}} \leqslant K (\sigma)$. We check similarly that
  \[ \left( \eta L (.) + (1 - \eta) V L' \left( \frac{.}{V} \right) \right)_c
     (\mathbbm{H}_c (\Phi_{c', d})) = \Pi_d^{\bot} (F_c (\Phi_{c', d} / V)) .
  \]
  Now, from the definitions of $L$ and $L'$ from Lemma \ref{lemma7}, we have
  \begin{eqnarray*}
    \left( \eta L (.) + (1 - \eta) V L' \left( \frac{.}{V} \right) \right)_{c
    + \varepsilon} (\mathbbm{H}_{c + \varepsilon} (\Phi_{c', d})) & = & \left(
    \eta L (.) + (1 - \eta) V L' \left( \frac{.}{V} \right) \right)_c
    (\mathbbm{H}_{c + \varepsilon} (\Phi_{c', d}))\\
    & - & i \varepsilon \eta \partial_{x_2} \mathbbm{H}_{c + \varepsilon}
    (\Phi_{c', d})\\
    & - & i \varepsilon (1 - \eta) V \partial_{x_2} \left(
    \frac{\mathbbm{H}_{c + \varepsilon} (\Phi_{c', d})}{V} \right),
  \end{eqnarray*}
  and therefore
  \begin{eqnarray*}
    &  & \left( \eta L (.) + (1 - \eta) V L' \left( \frac{.}{V} \right)
    \right)_c (\mathbbm{H}_{c + \varepsilon} (\Phi_{c', d}) -\mathbbm{H}_c
    (\Phi_{c', d}))\\
    & = & - (\Pi_d^{\bot} (F_{c + \varepsilon} (\Phi_{c', d} / V) - F_c
    (\Phi_{c', d} / V)))\\
    & - & i \varepsilon \left( \eta \partial_{x_2} \mathbbm{H}_{c +
    \varepsilon} (\Phi_{c', d}) + (1 - \eta) V \partial_{x_2} \left(
    \frac{\mathbbm{H}_{c + \varepsilon} (\Phi_{c', d})}{V} \right) \right) .
  \end{eqnarray*}
  We check, using $\mathbbm{H}_{c + \varepsilon} (\Phi_{c', d}) \in
  \mathcal{E}_{\circledast, \gamma (\sigma), d_{\circledast}}$, $\|
  \mathbbm{H}_{c + \varepsilon} (\Phi_{c', d}) \|_{\circledast, \gamma
  (\sigma), d_{\circledast}} \leqslant K (\sigma)$ that
  \[ i \varepsilon \left( \eta \partial_{x_2} \mathbbm{H}_{c + \varepsilon}
     (\Phi_{c', d}) + (1 - \eta) V \partial_{x_2} \left( \frac{\mathbbm{H}_{c
     + \varepsilon} (\Phi_{c', d})}{V} \right) \right) \in
     \mathcal{E}_{\circledast \circledast, \gamma (\sigma), d_{\circledast} },
  \]
  with
  \[ \left\| i \varepsilon \left( \eta \partial_{x_2} \mathbbm{H}_{c +
     \varepsilon} (\Phi_{c', d}) + (1 - \eta) V \partial_{x_2} \left(
     \frac{\mathbbm{H}_{c + \varepsilon} (\Phi_{c', d})}{V} \right) \right)
     \right\|_{\circledast \circledast, \gamma (\sigma), d_{\circledast}}
     \leqslant K (\sigma) \varepsilon . \]
  In particular, by Proposition \ref{invertop} (from $\mathcal{E}_{\circledast
  \circledast, \gamma (\sigma), d_{\circledast} }$ to
  $\mathcal{E}_{\circledast, \sigma, d_{\circledast}}$), we have
  \begin{eqnarray*}
    &  & \| \mathbbm{H}_{c + \varepsilon} (\Phi_{c', d}) -\mathbbm{H}_c
    (\Phi_{c', d}) \|_{\circledast, \sigma, d_{\circledast}}\\
    & \leqslant & K (\sigma) \| \Pi_d^{\bot} (F_{c + \varepsilon} (\Phi_{c',
    d} / V) - F_c (\Phi_{c', d} / V)) \|_{\circledast \circledast, \gamma
    (\sigma), d_{\circledast}}\\
    & + & K (\sigma) \varepsilon .
  \end{eqnarray*}
  We recall that
  \[ F_c (\Psi) = E - i c \partial_{x_2} V + V (1 - \eta) (- \nabla \Psi .
     \nabla \Psi + | V |^2 S (\Psi)) + R_c (\Psi), \]
  therefore
  \[ F_{c + \varepsilon} (\Phi_{c', d} / V) - F_c (\Phi_{c', d} / V) = - i
     \varepsilon \partial_{x_2} V + R_{c + \varepsilon} (\Phi_{c', d} / V) -
     R_c (\Phi_{c', d} / V) . \]
  By Lemma \ref{fineest} (for $i \partial_{x_2} V$) and the definition of
  $R_c$ (in the proof of Lemma \ref{lemma7}), we check that, for any $0 <
  \sigma < 1$, since $\| \Psi_{c', d} \|_{\ast, \sigma, d} \leqslant K
  (\sigma) c_0 (\sigma)^{1 - \gamma (\sigma)} \leqslant K (\sigma)$,
  \[ \| \Pi_d^{\bot} (F_{c + \varepsilon} (\Phi_{c', d} / V) - F_c (\Phi_{c',
     d} / V)) \|_{\circledast \circledast, \sigma, d_{\circledast}} \leqslant
     K (\sigma) \frac{\varepsilon}{c} . \]
  We conclude that
  \[ \| \mathbbm{H}_{c + \varepsilon} (\Phi_{c', d}) -\mathbbm{H}_c (\Phi_{c',
     d}) \|_{\circledast, \sigma, d_{\circledast}} = o_{\varepsilon
     \rightarrow 0}^{\sigma, c} (1), \]
  thus $\mathbbm{H}_{c + \varepsilon} (\Phi_{c', d}) \rightarrow \mathbbm{H}_c
  (\Phi_{c', d})$ when $\varepsilon \rightarrow 0$ in
  $\mathcal{E}_{\circledast, \sigma, d_{\circledast}}$ uniformly in $c'$. We
  remark that it is also uniform in $d$ in any compact set of $] 0, c_0
  (\sigma) [$.
  
  The next step is to show that $c \mapsto \Phi_{c, d}$ is a continuous
  function in $\mathcal{E}_{\ast, \sigma, d}$. Take $\varepsilon_n$ a sequence
  such that $\varepsilon_n \rightarrow 0$ when $n \rightarrow \infty$, then
  $\| \Phi_{c + \varepsilon_n, d} \|_{\ast, \sigma, d} \leqslant K_0 (\sigma,
  \sigma') (c + \varepsilon_n)^{1 - \sigma'}$ (for $K_0 (\sigma, \sigma')$ the
  constant in Proposition \ref{contractionest}), and (in the strong sense)
  \[ \left( \eta L (.) + (1 - \eta) V L' \left( \frac{.}{V} \right) \right)_{c
     + \varepsilon_n} (\Phi_{c + \varepsilon_n, d}) + \Pi_d^{\bot} (F_{c +
     \varepsilon_n} (\Phi_{c + \varepsilon_n, d} / V)) = 0. \]
  With the same arguments as in step 1 of the proof of Proposition
  \ref{invertop}, we check that, up to a subsequence, $\Phi_{c +
  \varepsilon_n} \rightarrow \Phi$ locally uniformly in $\mathbbm{R}^2$ for
  some function $\Phi \in \mathcal{E}_{\ast, \sigma, d}$ such that $\| \Phi
  \|_{\ast, \sigma, d} \leqslant K_0 (\sigma, \sigma') c^{1 - \sigma'}$. Then,
  since
  \[ \mathbbm{H}_{c + \varepsilon_n} (\Phi_{c + \varepsilon_n, d}) + \Phi_{c +
     \varepsilon_n, d} = 0, \]
  by taking the limit when $n \rightarrow \infty$, up to a subsequence, since
  $\mathbbm{H}_{c + \varepsilon} (\Phi_{c', d}) \rightarrow \mathbbm{H}_c
  (\Phi_{c', d})$ when $\varepsilon \rightarrow 0$ in $\mathcal{E}_{\ast,
  \sigma, d}$ (the norm is equivalent to the one of $\mathcal{E}_{\circledast,
  \sigma, d_{\circledast}}$) uniformly in $c'$, we have
  \[ \mathbbm{H}_c (\Phi) + \Phi = 0. \]
  But then, $\Phi \in \mathcal{E}_{\ast, \sigma, d}$, $\| \Phi \|_{\ast,
  \sigma, d} \leqslant K_0 (\sigma, \sigma') c^{1 - \sigma'}$ and
  $\mathbbm{H}_c (\Phi) + \Phi = H (\Phi, c, d) = 0$. By Proposition
  \ref{contractionest}, this implies that $\Phi = \Phi_{c, d}$, therefore
  $\Phi_{c, d}$ is an accumulation point of $\Phi_{c + \varepsilon_n, d}$. It
  is the only accumulation point, since any other will also satisfy $\Phi \in
  \mathcal{E}_{\ast, \sigma, d}$, $\| \Phi \|_{\ast, \sigma, d} \leqslant K_0
  (\sigma, \sigma') c^{1 - \sigma'}$ and $H (\Phi, c, d) = 0$. Therefore,
  $\Phi_{c + \varepsilon_n, d} \rightarrow \Phi_{c, d}$ in $\mathcal{E}_{\ast,
  \sigma, d}$, hence $c \mapsto \Phi_{c, d}$ is a continuous function in
  $\mathcal{E}_{\ast, \sigma, d}$.
  
  Now, let us show that it is a $C^1$ function in $\mathcal{E}_{\ast, \sigma,
  d}$. Since $\mathbbm{H}_c (\Phi_{c, d}) + \Phi_{c, d} = 0$, we have
  \[ \begin{array}{lll}
       &  & \left( \eta L (.) + (1 - \eta) V L' \left( \frac{.}{V} \right)
       \right)_c (\Phi_{c + \varepsilon, d} - \Phi_{c, d})\\
       & = & - (\Pi_d^{\bot} (F_{c + \varepsilon} (\Phi_{c + \varepsilon, d}
       / V) - F_c (\Phi_{c, d} / V)))\\
       & - & i \varepsilon \left( \eta \partial_{x_2} \Phi_{c + \varepsilon,
       d} + (1 - \eta) V \partial_{x_2} \left( \frac{\Phi_{c + \varepsilon,
       d}}{V} \right) \right) .
     \end{array} \]
  Furthermore, from $\| \Pi_d^{\bot} (F_{c + \varepsilon} (\Phi_{c', d} / V) -
  F_c (\Phi_{c', d} / V)) \|_{\circledast \circledast, \sigma,
  d_{\circledast}} \leqslant K (\sigma, c) \varepsilon$ and
  \[ \left\| i \varepsilon \left( \eta \partial_{x_2} \Phi_{c + \varepsilon,
     d} + (1 - \eta) V \partial_{x_2} \left( \frac{\Phi_{c + \varepsilon,
     d}}{V} \right) \right) \right\|_{\circledast \circledast, \sigma,
     d_{\circledast}} \leqslant K (\sigma, c) \varepsilon, \]
  we deduce that $\| \Phi_{c + \varepsilon, d} - \Phi_{c, d} \|_{\circledast,
  \sigma, d_{\circledast}} \leqslant K (\sigma, c) \varepsilon$.
  
  From the definition of $F$, we infer that
  \begin{eqnarray*}
    \nobracket F_{c + \varepsilon} (\Phi_{c + \varepsilon, d} / V) - F_c
    (\Phi_{c, d} / V \nobracket) & = & - i \varepsilon \partial_{x_2} V\\
    & + & V (1 - \eta) (- \nabla \Psi_{c + \varepsilon, d} . \nabla \Psi_{c +
    \varepsilon, d} + \nabla \Psi_{c, d} . \nabla \Psi_{c, d})\\
    & + & V (1 - \eta) | V |^2 (S (\Psi_{c + \varepsilon, d}) - S (\Psi_{c,
    d}))\\
    & + & R_{c + \varepsilon} (\Psi_{c + \varepsilon, d}) - R_c (\Psi_{c, d})
    .
  \end{eqnarray*}
  Now, regrouping the terms of $\Pi_d^{\bot} (d_{\Psi} F_c ((\Phi_{c +
  \varepsilon, d} \nobracket - \nobracket \Phi_{c, d}) / V))$ and using $\|
  \Phi_{c + \varepsilon, d} - \Phi_{c, d} \|_{\circledast, \sigma,
  d_{\circledast}} \leqslant K (\sigma, c) \varepsilon$ for the remaining
  nonlinear terms (which will be at least quadratic in $\Phi_{c + \varepsilon,
  d} - \Phi_{c, d}$, since $F$ is $C^{\infty}$ with respect to $\Psi$), as
  well as the fact that $c \mapsto R_c \in C^{\infty} (] 0, c_0 (\sigma) [,
  C^1 (\mathbbm{R}^2))$, for any $0 < \sigma < 1$,
  \begin{eqnarray*}
    \Pi_d^{\bot} (F_{c + \varepsilon} (\Phi_{c + \varepsilon, d} / V) - F_c
    (\Phi_{c, d} / V)) & = & \Pi_d^{\bot} (d_{\Psi} F_c ((\Phi_{c +
    \varepsilon, d} \nobracket - \nobracket \Phi_{c, d}) / V))\\
    & + & \varepsilon \Pi_d^{\bot} (- i \partial_{x_2} V)\\
    & + & O^{\sigma, c}_{\| . \|_{\ast \ast, \sigma, d}} (\varepsilon^2),
  \end{eqnarray*}
  where $O^{\sigma, c}_{\| . \|_{\ast \ast, \sigma, d}} (\varepsilon^2)$ is a
  quantity going to $0$ as $\varepsilon^2$ when $\varepsilon \rightarrow 0$ in
  the norm $\| . \|_{\ast \ast, \sigma, d}$ at fixed $\sigma, c$. We deduce
  that
  \[ \begin{array}{lll}
       &  & \left( \tmop{Id} + \left( \eta L (.) + (1 - \eta) V L' \left(
       \frac{.}{V} \right) \right)_c^{- 1} (\Pi_d^{\bot} (d_{\Psi} F_c (. /
       V))) \right) ((\Phi_{c + \varepsilon, d} \nobracket - \nobracket
       \Phi_{c, d}))\\
       & = & \left( \eta L (.) + (1 - \eta) V L' \left( \frac{.}{V} \right)
       \right)_c^{- 1} \left( - \varepsilon \Pi_d^{\bot} (- i \partial_{x_2}
       V) - i \varepsilon \left( \eta \partial_{x_2} \Phi_{c + \varepsilon, d}
       + (1 - \eta) V \partial_{x_2} \left( \frac{\Phi_{c + \varepsilon,
       d}}{V} \right) \right) \right)\\
       & + & \left( \eta L (.) + (1 - \eta) V L' \left( \frac{.}{V} \right)
       \right)_c^{- 1} (O^{\sigma, c}_{\| . \|_{\ast \ast, \sigma, d}}
       (\varepsilon^2)),
     \end{array} \]
  and we have shown that $\left( \tmop{Id} + \left( \eta L (.) + (1 - \eta) V
  L' \left( \frac{.}{V} \right) \right)_c^{- 1} \left( \Pi_d^{\bot} \left(
  \frac{1}{V} d_{\Phi} F_c (. / V) \right) \right) \right)$ is invertible from
  $\mathcal{E}_{\circledast, \sigma, d_{\circledast}}$ to
  $\mathcal{E}_{\circledast, \sigma, d_{\circledast}}$ (with an operator norm
  equal to $1 + o_{c \rightarrow 0}^{\sigma} (1)$ if taken in $\Phi = \Phi_{c,
  d}$, see Lemma \ref{lemme20}). Furthermore, $\Phi_{c, d}$ is continuous with
  respect to $c$ in $\mathcal{E}_{\circledast, \gamma (\sigma),
  d_{\circledast}}$ (with the same computations as previously, replacing
  $\sigma$ by $\gamma (\sigma)$), therefore
  \[ \eta \partial_{x_2} \Phi_{c + \varepsilon, d} + (1 - \eta) V
     \partial_{x_2} \left( \frac{\Phi_{c + \varepsilon, d}}{V} \right)
     \rightarrow \eta \partial_{x_2} \Phi_{c, d} + (1 - \eta) V \partial_{x_2}
     \left( \frac{\Phi_{c, d}}{V} \right) \]
  in $\mathcal{E}_{\circledast \circledast, \gamma (\sigma), d_{\circledast}}$
  when $\varepsilon \rightarrow 0$. We deduce that $c \mapsto \Phi_{c, d}$ is
  $C^1$ in $\mathcal{E}_{\circledast, \sigma, d_{\circledast}}$ (and therefore
  in $\mathcal{E}_{\ast, \sigma, d}$).
\end{proof}

Now, we show the differentiablity of $\Phi_{c, d}$ with respect to $d$.

\begin{lemma}
  \label{dderpsy}For $0 < \sigma < 1$, there exists $c_0 (\sigma) > 0$ such
  that, for $0 < c < c_0 (\sigma)$,
  \[ d \mapsto \Phi_{c, d} \in C^1 \left( \left] \frac{1}{2 c}, \frac{2}{c}
     \right[ \cap \left] d_{\circledast} - \frac{\delta}{2}, d_{\circledast} +
     \frac{\delta}{2} \right[, \mathcal{E}_{\circledast, \sigma,
     d_{\circledast}} \right) . \]
\end{lemma}

We recall that $\delta > 0$ is defined at the beginning of this subsection.

\

We check easily by standard elliptic regularity arguments that $\partial_c
\Phi_{c, d} \in C^{\infty} (\mathbbm{R}^2, \mathbbm{C})$. Furthermore, $c
\mapsto \Phi_{c, d}$ is $C^1$ with values in $\mathcal{E}_{\ast, \sigma, d}$,
therefore $\partial_c \nabla \Phi_{c, d}$ is well defined (in $C^0
(\mathbbm{R}^2, \mathbbm{C})$). Let us show that it is equal to $\nabla
\partial_c \Phi_{c, d}$. For $\varphi \in C^{\infty}_c (\mathbbm{R}^2,
\mathbbm{C})$, we have, by derivation under an integral, that
\begin{eqnarray*}
  \int_{\mathbbm{R}^2} \partial_c \nabla \Phi_{c, d} \varphi & = & \partial_c
  \int_{\mathbbm{R}^2} \nabla \Phi_{c, d} \varphi\\
  & = & - \partial_c \int_{\mathbbm{R}^2} \Phi_{c, d} \nabla \varphi\\
  & = & - \int_{\mathbbm{R}^2} \partial_c \Phi_{c, d} \nabla \varphi\\
  & = & \int_{\mathbbm{R}^2} \nabla \partial_c \Phi_{c, d} \varphi .
\end{eqnarray*}
Therefore $\partial_c \nabla \Phi_{c, d} = \nabla \partial_c \Phi_{c, d}$ in
the distribution sense, and thus in the strong sense. Furthermore, thanks to
the equation $\eta L (\Phi_{c, d}) + (1 - \eta) V L' (\Psi_{c, d}) + F
(\Psi_{c, d}) = \lambda (c, d) Z_d$, we can isolate $\Delta \Phi_{c, d}$ as in
(\ref{jaune}), and show in particular that it is a $C^1$ function of $c$. By
similar arguments as for the gradient, we can show that $\partial_c \Delta
\Phi_{c, d} = \Delta \partial_c \Phi_{c, d}$. Furthermore, the same proof
holds if we differentiate $\Phi_{c, d}$ with respect to $d$. We can therefore
inverse derivatives in position and derivatives with respect to $c$ or $d$ on
$\Phi_{c, d}$.

Let us also show that $(c, d) \mapsto \partial_c \Phi_{c, d}$ is a continuous
function from $\Omega \assign \left\{ (c, d) \in \mathbbm{R}^2, 0 < c < c_0
(\sigma), \frac{1}{2 c} < d < \frac{2}{c} \right\}$ to $\mathcal{E}_{\ast,
\sigma, d}$. With the same compactness argument used in the proof of the
continuity of $c \mapsto \Phi_{c, d}$, we can show that $(c, d) \mapsto
\Phi_{c, d}$ is continuous from $\Omega$ to $\mathcal{E}_{\ast, \sigma, d}$.
From the proof of Lemma \ref{rest31}, we have that
\begin{eqnarray*}
  &  & \left( \tmop{Id} + \left( \eta L (.) + (1 - \eta) V L' \left(
  \frac{.}{V} \right) \right)^{- 1} (\Pi_d^{\bot} (d_{\Psi} F (. / V)))
  \right) (\partial_c \Phi_c)\\
  & = & \Pi_d^{\bot} (\partial_c F (\Phi_{c, d} / V)) - i \eta \partial_{x_2}
  \Phi_{c, d} + (1 - \eta) V \partial_{x_2} \left( \frac{\Phi_{c, d}}{V}
  \right) .
\end{eqnarray*}
Since $(c, d) \mapsto \Phi_{c, d}$ is continuous from $\Omega$ to
$\mathcal{E}_{\ast, \sigma, d}$, and that the dependence on $(c, d)$ of the
other terms of the right-hand side is explicit, we check that $\Pi_d^{\bot}
(\partial_c F (\Phi_{c, d} / V)) - i \eta \partial_{x_2} \Phi_{c, d} + (1 -
\eta) V \partial_{x_2} \left( \frac{\Phi_{c, d}}{V} \right)$ is continuous
from $\Omega$ to $\mathcal{E}_{\ast \ast, \gamma (\sigma), d}$. We check also
that $(c, d) \mapsto \left( \tmop{Id} + \left( \eta L (.) + (1 - \eta) V L'
\left( \frac{.}{V} \right) \right)^{- 1} (\Pi_d^{\bot} (d_{\Psi} F (. / V)))
\right)$ is continuous from $\Omega$ to $\mathcal{E}_{\ast \ast, \gamma
(\sigma), d} \rightarrow \mathcal{E}_{\ast, \sigma, d}$, and thus $(c, d)
\mapsto \partial_c \Phi_{c, d}$ is a continuous function from $\Omega$ to
$\mathcal{E}_{\ast, \sigma, d}$. The same proof holds for $(c, d) \mapsto
\partial_d \Phi_{c, d}$.

\

We end this subsection with the symmetries of $\partial_d \Phi_{c, d}$.

\begin{lemma}
  \label{l381805}The function $\partial_d \Phi_{c, d}$ satisfies the
  symmetries: for $x = (x_1, x_2) \in \mathbbm{R}^2$,
  \[ \partial_d \Phi_{c, d} (x_1, x_2) = \partial_d \Phi_{c, d_c} (- x_1, x_2)
     = \overline{\partial_d \Phi_{c, d} (x_1, - x_2)} . \]
\end{lemma}

\begin{proof}
  From subsection \ref{normsetup},
  \[ \forall x = (x_1, x_2) \in \mathbbm{R}^2, \Psi_{c, d} (x_1, x_2) =
     \overline{\Psi_{c, d} (x_1, - x_2)} = \Psi_{c, d} (- x_1, x_2) \]
  and $V$ enjoys the same symmetries, therefore for all $d \in \mathbbm{R}$
  such that $\frac{1}{2 c} < d < \frac{2}{c}$,
  \[ \Phi_{c, d} (x_1, x_2) = \Phi_{c, d} (- x_1, x_2) = \overline{\Phi_{c, d}
     (x_1, - x_2)} . \]
  Since
  \[ \partial_d \Phi_{c, d} = \lim_{\varepsilon \rightarrow 0} \frac{\Phi_{c,
     d + \varepsilon} - \Phi_{c, d}}{\varepsilon}, \]
  these symmetries also hold for $\partial_d \Phi_{c, d}$.
\end{proof}

\subsection{End of the construction and properties of $Q_c$\label{th1proof}}

A consequence of equation (\ref{azz220}) and Proposition \ref{dinc} is that,
for $0 < \sigma < 1$, there exists $c_0 (\sigma) > 0$ such that, for $0 < c <
c_0 (\sigma)$,
\[ \eta L (\Phi_{c, d}) + (1 - \eta) V L' (\Psi_{c, d}) + F (\Psi_{c, d}) =
   \lambda (c, d) Z_d \]
with
\[ \lambda (c, d) \int_{\mathbbm{R}^2} | \partial_d V |^2 \eta = \pi \left(
   \frac{1}{d} - c \right) + O^{\sigma}_{c \rightarrow 0} (c^{2 - \sigma}) .
\]
Following the proof of Proposition \ref{dinc}, with Lemmas \ref{rest31} and
\ref{dderpsy}, we can check that the $O^{\sigma}_{c \rightarrow 0} (c^{2 -
\sigma})$ is continuous with respect of $c$ and $d$. Therefore, by the
intermediate value theorem, there exists $d_c > 0$ such that $\lambda (c, d_c)
= 0$, with
\[ d_c = \frac{1}{c} + O^{\sigma}_{c \rightarrow 0} (c^{- \sigma}), \]
for $c > 0$ small enough. Then, for the function $\Phi_{c, d_c} = V \Psi_{c,
d_c}$ with $\| \Psi_{c, d_c} \|_{\ast, \sigma, d_c} \leqslant K (\sigma \comma
\sigma') c^{1 - \sigma'}$, we have
\[ \eta L (\Phi_{c, d_c}) + (1 - \eta) V L' (\Psi_{c, d_c}) + F (\Psi_{c,
   d_c}) = 0, \]
meaning that if we define
\[ Q_c \assign \eta V (1 + \Psi_{c, d_c}) + (1 - \eta) V e^{\Psi_{c, d_c}}, \]
then $Q_c$ solves $(\tmop{TW}_c)$.

\subsubsection{Behaviour at infinity and energy estimation}

\begin{lemma}
  The function $Q_c$ satisfies $Q_c (x) \rightarrow 1$ when $| x | \rightarrow
  \infty$.
\end{lemma}

\begin{proof}
  From $\| \Psi_{c, d_c} \|_{\ast, \sigma, d_c} \leqslant K (\sigma, \sigma')
  c^{1 - \sigma'}$ we have $\Psi_{c, d_c} (x) \rightarrow 0$ when $| x |
  \rightarrow \infty$. Furthermore $| 1 - V |^2 \leqslant \frac{C (d_c)}{1 +
  r^2}$ by Lemma \ref{nonmodV} and $Q_c = V e^{\Psi_{c, d_c}}$ for large
  values of $| x |$, hence $Q_c (x) \rightarrow 1$ when $| x | \rightarrow
  \infty$.
\end{proof}

In the statement of Theorem \ref{th1}, we have set $Q_c = V + \Gamma_{c,
d_c}$, we therefore define
\begin{equation}
  \Gamma_{c, d_c} \assign \eta V \Psi_{c, d_c} + (1 - \eta) V \left(
  e^{\Psi_{c, d_c}} - 1 \right) . \label{gammadef}
\end{equation}
We compute that
\[ \left\| \frac{\Gamma_{c, d_c}}{V} \right\|_{\ast, \sigma, d_c} \leqslant K
   \| \Psi_{c, d_c} \|_{\ast, \sigma, d_c} + \left\| (1 - \eta) \left(
   e^{\Psi_{c, d_c}} - 1 - \Psi_{c, d_c} \right) \right\|_{\ast, \sigma, d_c},
\]
and since $\| \Psi_{c, d_c} \|_{\ast, \sigma, d_c} \leqslant 1$ for $c$ small
enough (depending on $\sigma$), we have
\[ \left\| (1 - \eta) \left( e^{\Psi_{c, d_c}} - 1 - \Psi_{c, d_c} \right)
   \right\|_{\ast, \sigma, d_c} \leqslant K \left\| (1 - \eta) \Psi^2_{c, d_c}
   \sum_{n = 2}^{+ \infty} \frac{\Psi_{c, d_c}^{n - 2}}{n!} \right\|_{\ast,
   \sigma, d_c} . \]
Now, for $0 < \sigma < \sigma' < 1$, we have $\frac{1 + \sigma'}{2} > \frac{1
+ \sigma}{2}$, hence
\[ | \Psi_{c, d_c} | \leqslant K (\sigma, \sigma') \frac{c^{1 - \frac{1 +
   \sigma'}{2}}}{(1 + \tilde{r})^{\frac{1 + \sigma}{2}}} \quad \tmop{and}
   \quad | \nabla \Psi_{c, d_c} | \leqslant K (\sigma, \sigma') \frac{c^{1 -
   \frac{1 + \sigma'}{2}}}{(1 + \tilde{r})^{1 + \frac{1 + \sigma}{2}}}, \]
therefore
\[ | \Psi_{c, d_c} |^2 \leqslant K (\sigma, \sigma') \frac{c^{1 - \sigma'}}{(1
   + \tilde{r})^{1 + \sigma}} \quad \tmop{and} \quad | \nabla \Psi_{c, d_c}
   |^2 \leqslant K (\sigma, \sigma') \frac{c^{1 - \sigma'}}{(1 + \tilde{r})^{2
   + \sigma}} . \]
Thus, with $| \nabla^2 \Psi_{c, d_c} | \leqslant K (\sigma, \sigma')
\frac{c^{1 - \sigma'}}{(1 + \tilde{r})^{1 + \sigma}}$, we check that, for any
$0 < \sigma < \sigma' < 1$,
\[ \left\| (1 - \eta) \Psi_{c, d_c}^2 \sum_{n = 2}^{+ \infty} \frac{\Psi_{c,
   d_c}^{n - 2}}{n!} \right\|_{\ast, \sigma, d_c} \leqslant K (\sigma,
   \sigma') c^{1 - \sigma'} . \]
Combining this result with $\| \Psi_{c, d_c} \|_{\ast, \sigma, d_c} \leqslant
K (\sigma, \sigma') c^{1 - \sigma'}$, we deduce that
\begin{equation}
  \left\| \frac{\Gamma_{c, d_c}}{V} \right\|_{\ast, \sigma, d_c} \leqslant K
  (\sigma, \sigma') c^{1 - \sigma'} . \label{gammaetoile}
\end{equation}
In particular, we have, for any $0 < \sigma < \sigma' < 1$, $0 < c < c_0
(\sigma, \sigma')$, that
\begin{equation}
  | \Gamma_{c, d_c} | \leqslant \frac{K (\sigma, \sigma') c^{1 - \sigma'}}{(1
  + \tilde{r})^{\sigma}}, \label{gammaetoile1904}
\end{equation}
\begin{equation}
  \left| \mathfrak{R}\mathfrak{e} \left( \frac{\Gamma_{c, d_c}}{V} \right)
  \right| \leqslant \frac{K (\sigma, \sigma') c^{1 - \sigma'}}{(1 +
  \tilde{r})^{1 + \sigma}},
\end{equation}
and, if $\tilde{r} \geqslant 2$,
\[ | \nabla \Gamma_{c, d_c} | \leqslant \left| \nabla \left( \frac{\Gamma_{c,
   d_c}}{V} \right) \right| + \left| \frac{\nabla V}{V} \right| \times \left|
   \frac{\Gamma_{c, d_c}}{V} \right|, \]
therefore, using $| \nabla V | \leqslant \frac{K}{(1 + \tilde{r})}$ from Lemma
\ref{lemme3}, we have
\begin{equation}
  | \nabla \Gamma_{c, d_c} | \leqslant \frac{K (\sigma, \sigma') c^{1 -
  \sigma'}}{(1 + \tilde{r})^{1 + \sigma}} . \label{gradgamma}
\end{equation}
Estimate (\ref{gradgamma}) remains true in $\{ \tilde{r} \leqslant 2 \}$ since
$\| \Gamma_{c, d_c} \|_{C^1 (\{ \tilde{r} \leqslant 2 \})} \leqslant \left\|
\frac{\Gamma}{V} \right\|_{\ast, \sigma, d_c} \leqslant K (\sigma, \sigma')
c^{1 - \sigma'}$. We now show the estimates on $\Gamma_{c, d_c}$ of Theorem
\ref{th1}.

\begin{lemma}
  \label{lema217}For $+ \infty \geqslant p > 2$, there exists $c_0 (p) > 0$
  such that if $0 < c < c_0 (p)$, we have $\Gamma_{c, d_c} \in L^p
  (\mathbbm{R}^2), \nabla \Gamma_{c, d_c} \in L^{p - 1} (\mathbbm{R}^2)$ and
  \[ \| \Gamma_{c, d_c} \|_{L^p (\mathbbm{R}^2)} + \| \nabla \Gamma_{c, d_c}
     \|_{L^{p - 1} (\mathbbm{R}^2)} = o_{c \rightarrow 0} (1) . \]
\end{lemma}

\begin{proof}
  If $p = + \infty$, using (\ref{gammaetoile1904}) and (\ref{gradgamma}), we
  infer
  \[ \| \Gamma_{c, d_c} \|_{L^{\infty} (\mathbbm{R}^2)} \leqslant K (\sigma)
     c^{1 - \sigma}, \]
  \[ \| \nabla \Gamma_{c, d_c} \|_{L^{\infty} (\mathbbm{R}^2)} \leqslant K
     (\sigma) c^{1 - \sigma}, \]
  hence the result holds. If $2 < p < + \infty$ then, by
  (\ref{gammaetoile1904}),
  \[ \int_{\mathbbm{R}^2} | \Gamma_{c, d_c} |^p \leqslant \int_{\mathbbm{R}^2}
     \frac{\| \Gamma_{c, d_c} \|_{\ast, \sigma, d_c}^p}{(1 + \tilde{r})^{p
     \sigma}} d x \leqslant \int_{\mathbbm{R}^2} \frac{K (\sigma, \sigma')
     c^{(1 - \sigma') p}}{(1 + \tilde{r})^{p \sigma}} d x. \]
  Taking $0 < \sigma < \sigma' < 1$ such that $p \sigma > 2$ then gives the
  result. Furthermore, by (\ref{gradgamma}),
  \[ \int_{\mathbbm{R}^2} | \nabla \Gamma_{c, d_c} |^p \leqslant
     \int_{\mathbbm{R}^2} \frac{K (\sigma, \sigma') c^{(1 - \sigma') p}}{(1 +
     \tilde{r})^{p (\sigma + 1)}} d x, \]
  so for $p > 1$ we can take $0 < \sigma < \sigma' < 1$ such that $p (\sigma +
  1) > 2$ and we have the result.
\end{proof}

Remark that we can have better estimates on $\Gamma_{c, d_c}$, in particular
if we look at real and imaginary parts of $\frac{\Gamma_{c, d_c}}{V}$. For
instance it is possible to show that
\[ \left\| \mathfrak{R}\mathfrak{e} \left( \frac{\Gamma_{c, d_c}}{V} \right)
   \right\|_{L^p (\{ \tilde{r} \geqslant 1 \})} = o_{c \rightarrow 0} (1) \]
for $p > 1$ instead of $p > 2$. This estimate does not hold for $\tilde{r}$
small since it is not clear that $\Psi_{c, d_c}$ is bounded there (but
$\Phi_{c, d_c}$ is). This is due to the fact that the zeros of $Q_c$ are not
exactly those of $V$.

\begin{lemma}
  \label{finiteE}The travelling wave $Q_c$ has finite energy, that is:
  \[ E (Q_c) = \frac{1}{2}  \int_{\mathbbm{R}^2} | \nabla Q_c |^2 +
     \frac{1}{4}  \int_{\mathbbm{R}^2} (1 - | Q_c |^2)^2 < + \infty . \]
\end{lemma}

\begin{proof}
  Far from the vortices, $\nabla Q_c = \nabla (V_1 V_{- 1}) e^{\Psi_{c, d_c}}
  + \nabla \Psi_{c, d_c} V_1 V_{- 1} e^{\Psi_{c, d_c}}$. We know that, for
  $\tilde{r} \geqslant 1$,
  \[ | \nabla \Psi_{c, d_c} | \leqslant \frac{K (\sigma)}{\tilde{r}^{1 +
     \sigma}} \]
  and (by Lemma \ref{nonmodV})
  \[ | \nabla (V_1 V_{- 1}) | \leqslant \frac{K (c)}{\tilde{r}^2}, \]
  hence
  \[ | \nabla Q_c |^2 \leqslant \frac{K (c, \sigma)}{\tilde{r}^{2 + 2 \sigma}}
  \]
  and is therefore integrable. On the other hand,
  \[ | 1 - | Q_c |^2 | = \left| 1 - | V_1 V_{- 1} |^2
     e^{2\mathfrak{R}\mathfrak{e} (\Psi_{c, d_c})} \right| \leqslant K (1 - |
     V_1 V_{- 1} |^2 + | V_1 V_{- 1} |^2  | \mathfrak{R}\mathfrak{e} (\Psi_{c,
     d_c}) |), \]
  and we have
  \[ 1 - | V_1 V_{- 1} |^2 = O \left( \frac{1}{\tilde{r}^2} \right) \quad
     \tmop{and} \quad \mathfrak{R}\mathfrak{e} (\Psi_{c, d_c}) = O^{\sigma}
     \left( \frac{1}{\tilde{r}^{1 + \sigma}} \right), \]
  therefore
  \[ (1 - | Q_c |^2)^2 = O \left( \frac{1}{\tilde{r}^{2 + 2 \sigma}} \right)
  \]
  and is integrable.
\end{proof}

At this point, we have finished the proof of the construction of $Q_c$. In the
next two subsection, we add some estimates on $Q_c$ that will be usefull for
the differentiability of the branch, and others that are interesting in
themselves.

\subsubsection{A set of estimations on $Q_c$}

The next Lemma gives additional estimates on $Q_c$ which are more precise but
more technical than the ones in Theorem \ref{th1}.

\begin{lemma}
  \label{CP1L328}For any $0 < \sigma < \sigma' < 1$, there exists $c_0
  (\sigma, \sigma'), K (\sigma, \sigma') > 0$ such that for $0 < c < c_0
  (\sigma, \sigma')$ we have\label{Qcfest}
  \begin{equation}
    \| \Psi_{c, d_c} \|_{\ast, \sigma, d_c} \leqslant K (\sigma, \sigma') c^{1
    - \sigma'} . \label{1605235}
  \end{equation}
  Furthermore, for any $0 < \sigma < 1$, there exist $c_0 (\sigma), K (\sigma)
  > 0$ such that for $0 < c < c_0 (\sigma)$ we have
  \begin{eqnarray}
    &  & \| V \Psi_{c, d_c} \|_{C^1 (\tilde{r} \leqslant 3)} + \|
    \tilde{r}^{\sigma} \mathfrak{I}\mathfrak{m} (\Psi_{c, d_c}) \|_{L^{\infty}
    (\tilde{r} \geqslant 2)} + \| \tilde{r}^{1 + \sigma}
    \mathfrak{R}\mathfrak{e} (\Psi_{c, d_c}) \|_{L^{\infty} (\tilde{r}
    \geqslant 2)} \nonumber\\
    & + & \| \tilde{r}^{1 + \sigma} \mathfrak{I}\mathfrak{m} (\nabla \Psi_{c,
    d_c}) \|_{L^{\infty} (\tilde{r} \geqslant 2)} + \| \tilde{r}^{2 + \sigma}
    \mathfrak{R}\mathfrak{e} (\nabla \Psi_{c, d_c}) \|_{L^{\infty} (\tilde{r}
    \geqslant 2)} \nonumber\\
    & \leqslant & K (\sigma) c^{1 - \sigma},  \label{impro1}
  \end{eqnarray}
  \begin{equation}
    | 1 - | Q_c | | \leqslant \frac{K (\sigma)}{(1 + \tilde{r})^{1 + \sigma}},
    \label{217}
  \end{equation}
  \begin{equation}
    | Q_c - V | \leqslant \frac{K (\sigma) c^{1 - \sigma}}{(1 +
    \tilde{r})^{\sigma}}, \label{218}
  \end{equation}
  \begin{equation}
    | | Q_c |^2 - | V |^2 | \leqslant \frac{K (\sigma) c^{1 - \sigma}}{(1 +
    \tilde{r})^{1 + \sigma}}, \label{219}
  \end{equation}
  \begin{equation}
    | \mathfrak{R}\mathfrak{e} (\nabla Q_c \overline{Q_c}) | \leqslant \frac{K
    (\sigma)}{(1 + \tilde{r})^{1 + \sigma}}, \label{220}
  \end{equation}
  \begin{equation}
    | \mathfrak{I}\mathfrak{m} (\nabla Q_c \overline{Q_c}) | \leqslant
    \frac{K}{1 + \tilde{r}} \label{221}
  \end{equation}
\end{lemma}

Equation (\ref{impro1}) is a slight improvements of (\ref{1605235}). It is,
except for the second derivatives, the estimate in the case $\sigma' =
\sigma$.

\begin{proof}
  The first estimate (\ref{1605235}) comes from the construction of the
  solution.
  
  We now take $\chi$ a cutoff function with value $1$ in $\{ \tilde{r}
  \geqslant 2 \}$ and $0$ in $\{ \tilde{r} \leqslant 1 \}$, we write
  $\tilde{\Psi} = \chi \Psi_{c, d_c}$ and $\tilde{h} = \chi h$, where $h$
  contains the nonlinear and source terms. We recall from (\ref{systnew}) that
  $\tilde{\Psi} = \tilde{\Psi}_1 + i \tilde{\Psi}_2$ and $\tilde{h} =
  \tilde{h}_1 + i \tilde{h}_2$ satisfy the system

  \[ \left\{ \begin{array}{l}
       \Delta \tilde{\Psi}_1 - 2 \tilde{\Psi}_1 = - \tilde{h}_1 -
       2\mathfrak{R}\mathfrak{e} \left( \frac{\nabla V}{V} . \nabla
       \tilde{\Psi} \right) - 2 (1 - | V |^2) \tilde{\Psi}_1 + c
       \partial_{x_2} \tilde{\Psi}_2 + \tmop{Loc}_1 (\Psi)\\
       \Delta \tilde{\Psi}_2 = - \tilde{h}_2 - 2\mathfrak{I}\mathfrak{m}
       \left( \frac{\nabla V}{V} . \nabla \tilde{\Psi} \right) + \tmop{Loc}_2
       (\Psi) - c \partial_{x_2} \tilde{\Psi}_1,
     \end{array} \right. \]
  where $\tmop{Loc}_1 (\Psi), \tmop{Loc}_2 (\Psi)$ are localized terms. From
  Lemmas \ref{fineestimate} to \ref{L2151904}, we check that for any $0 <
  \sigma < 1$,
  \[ \| \tilde{h} \|_{\ast \ast, \sigma, d} \leqslant K (\sigma) c^{1 -
     \sigma} . \]
  Furthermore, as in the proof of Proposition \ref{invertop}, we check that
  (using $\| \tilde{\Psi} \|_{\ast, \sigma / 2, d} \leqslant K (\sigma) c^{1 -
  \sigma}$)
  \[ \left\| \frac{\nabla V}{V} . \nabla \tilde{\Psi} - 2 (1 - | V |^2)
     \mathfrak{R}\mathfrak{e} (\tilde{\Psi}) + \tmop{Loc} (\Psi)
     \right\|_{\ast \ast, \sigma, d} \leqslant K (\sigma) c^{1 - \sigma} . \]
  Finally, with (\ref{1605235}), for $\sigma' = \frac{1 + \sigma}{2} >
  \sigma$,
  \[ \| c \partial_{x_2} \tilde{\Psi} \|_{\ast \ast, \sigma, d} \leqslant K
     (\sigma) c \| \tilde{\Psi} \|_{\ast, \sigma, d} \leqslant K (\sigma) c^{1
     + 1 - \frac{1 + \sigma}{2}} \leqslant K (\sigma) c^{1 - \sigma} . \]
  With Lemma \ref{see} for $\alpha = 1 + \sigma > 0$, we deduce from the first
  equation of the system that
  \begin{eqnarray*}
    &  & \| (1 + \tilde{r})^{1 + \sigma} \tilde{\Psi}_1 \|_{L^{\infty}
    (\mathbbm{R}^2)}\\
    & \leqslant & K (\sigma) \left\| (1 + \tilde{r})^{1 + \sigma} \left( -
    \tilde{h}_1 - 2\mathfrak{R}\mathfrak{e} \left( \frac{\nabla V}{V} . \nabla
    \tilde{\Psi} \right) - 2 (1 - | V |^2) \tilde{\Psi}_1 + c \partial_{x_2}
    \tilde{\Psi}_2 + \tmop{Loc}_1 (\Psi) \right) \right\|_{L^{\infty}
    (\mathbbm{R}^2)}\\
    & \leqslant & K (\sigma) c^{1 - \sigma},
  \end{eqnarray*}
  and, by differentiating the equation, by Lemma \ref{see} for $\alpha = 2 +
  \sigma > 0$
  \begin{eqnarray*}
    &  & \| (1 + \tilde{r})^{2 + \sigma} \nabla \tilde{\Psi}_1 \|_{L^{\infty}
    (\mathbbm{R}^2)}\\
    & \leqslant & K (\sigma) \left\| (1 + \tilde{r})^{2 + \sigma} \nabla
    \left( - \tilde{h}_1 - 2\mathfrak{R}\mathfrak{e} \left( \frac{\nabla V}{V}
    . \nabla \tilde{\Psi} \right) - 2 (1 - | V |^2) \tilde{\Psi}_1 + c
    \partial_{x_2} \tilde{\Psi}_2 + \tmop{Loc}_1 (\Psi) \right)
    \right\|_{L^{\infty} (\mathbbm{R}^2)}\\
    & \leqslant & K (\sigma) c^{1 - \sigma} .
  \end{eqnarray*}
  Now, using Lemma \ref{lapzeta} and $\| (1 + \tilde{r})^{2 + \sigma} \nabla
  \tilde{\Psi}_1 \|_{L^{\infty} (\mathbbm{R}^2)} \leqslant K (\sigma) c^{1 -
  \sigma}$, we infer that
  \begin{eqnarray*}
    &  & \| (1 + \tilde{r})^{\sigma} \tilde{\Psi}_2 \|_{L^{\infty}
    (\mathbbm{R}^2)} + \| (1 + \tilde{r})^{1 + \sigma} \nabla \tilde{\Psi}_2
    \|_{L^{\infty} (\mathbbm{R}^2)}\\
    & \leqslant & K (\sigma) \left\| (1 + \tilde{r})^{2 + \sigma} \left( -
    \tilde{h}_2 - 2\mathfrak{I}\mathfrak{m} \left( \frac{\nabla V}{V} . \nabla
    \tilde{\Psi} \right) + \tmop{Loc}_2 (\Psi) - c \partial_{x_2}
    \tilde{\Psi}_1 \right) \right\|_{L^{\infty} (\mathbbm{R}^2)}\\
    & \leqslant & K (\sigma) c^{1 - \sigma},
  \end{eqnarray*}
  which concludes the proof of (\ref{impro1}).
  
  \
  
  The estimate (\ref{217}) is clear if $\tilde{r} \leqslant 3$. If $\tilde{r}
  \geqslant 3$, then $Q_c = V e^{\Psi_{c, d_c}}$ and, for $c$ small enough
  (depending on $\sigma$), $| \mathfrak{R}\mathfrak{e} (\Psi_{c, d_c}) |
  \leqslant 1$, thus
  \begin{eqnarray*}
    | 1 - | Q_c | | & = & \left| 1 - | V | - | V | \left(
    e^{\mathfrak{R}\mathfrak{e} (\Psi_{c, d_c})} - 1 \right) \right|\\
    & \leqslant & | 1 - | V | | + K | \mathfrak{R}\mathfrak{e} (\Psi_{c,
    d_c}) |\\
    & \leqslant & \frac{K}{(1 + \tilde{r})^2} + \frac{K (\sigma) c^{1 -
    \sigma}}{(1 + \tilde{r})^{1 + \sigma}}\\
    & \leqslant & \frac{K (\sigma)}{(1 + \tilde{r})^{1 + \sigma}}
  \end{eqnarray*}
  by Lemma \ref{nonmodV} and (\ref{impro1}). For (\ref{218}), if $\tilde{r}
  \geqslant 3$, we compute
  \[ | Q_c - V | = | V | \times \left| e^{\Psi_{c, d_c}} - 1 \right| \leqslant
     C | \Psi_{c, d_c} | \leqslant \frac{K (\sigma) c^{1 - \sigma}}{(1 +
     \tilde{r})^{\sigma}} \]
  and if $\tilde{r} \leqslant 3$, $| Q_c - V | \leqslant C \| \Psi_{c, d_c}
  \|_{\ast, \sigma, d_c}$ and the estimate (\ref{218}) holds. Similarly, for
  $\tilde{r} \geqslant 3$,
  \[ | | Q_c |^2 - | V |^2 | \leqslant | V |^2 \left|
     e^{2\mathfrak{R}\mathfrak{e} (\Psi_{c, d_c})} - 1 \right| \leqslant
     \frac{K (\sigma) c^{1 - \sigma}}{(1 + \tilde{r})^{1 + \sigma}} \]
  and for the same reason if $\tilde{r} \leqslant 3$ the estimate (\ref{219})
  holds. Inequalities (\ref{220}) and (\ref{221}) are clear if $\tilde{r}
  \leqslant 3$ and we compute, for $\tilde{r} \geqslant 3$,
  \[ \nabla Q_c \overline{Q_c} = \nabla \left( V e^{\Psi_{c, d_c}} \right)
     \bar{V} e^{\bar{\Psi}_{c, d_c}} = \nabla V \bar{V}
     e^{2\mathfrak{R}\mathfrak{e} (\Psi c, d_c)} + | V |^2 \nabla \Psi_{c,
     d_c} e^{2\mathfrak{R}\mathfrak{e} (\Psi_{c, d_c})} . \]
  We have $\left| e^{2\mathfrak{R}\mathfrak{e} (\Psi_{c, d_c})} \right|
  \leqslant 1$ for $c$ small enough and by Lemma \ref{lemme3} we have $|
  \mathfrak{I}\mathfrak{m} (\nabla V \bar{V}) | \leqslant \frac{K}{1 +
  \tilde{r}}$ and $| \mathfrak{R}\mathfrak{e} (\nabla V \bar{V}) | \leqslant
  \frac{K}{(1 + \tilde{r})^3} .$ Combining it with $| \nabla \Psi_{c, d_c} |
  \leqslant \frac{K (\sigma) c^{1 - \sigma}}{(1 + \tilde{r})^{1 + \sigma}}$
  from (\ref{impro1}), estimates (\ref{220}) and (\ref{221}) hold.
\end{proof}

\subsubsection{Estimations on derivatives of $\Phi_{c, d}$ with respect to $c$
and $d$ at $d = d_c$.}

We cannot easily compute $\partial_d \Psi_{c, d | d = d_c \nobracket}$ because
of issues locally around the vortices (due to the fact that $\Psi_{c, d}$ is
unbounded near $\tilde{r} = 0$, and changing $d$ change the position of the
vortices). We shall prove instead an estimate on $\partial_d \Phi_{c, d | d =
d_c \nobracket}$, as well as an estimate on $\partial_c \Psi_{c, d | d = d_c
\nobracket}$.

\begin{lemma}
  \label{nl310}For any $0 < \sigma < \sigma' < 1, c \in \mathbbm{R}$ such that
  $\frac{1}{2 d} < c < \frac{2}{d}$ and $0 < c < c_0 (\sigma, \sigma')$, we
  have
  \[ \| \partial_c \Psi_{c, d | d = d_c \nobracket} \|_{\ast, \sigma, d}
     \leqslant K (\sigma, \sigma') c^{- \sigma'} \]
  and
  \[ \left\| \frac{\partial_d \Phi_{c, d}}{V}_{| d = d_c \nobracket}
     \right\|_{\ast, \sigma, d_c} \leqslant K (\sigma, \sigma') c^{1 -
     \sigma'}, \]
  with $K (\sigma, \sigma') > 0$ depending only on $\sigma, \sigma'$.
\end{lemma}

See Appendix \ref{AC2} for the proof of this result.

\section{Differentiability of the branch $c \mapsto Q_c$}\label{qcc1}

The goal of this section is to prove that the constructed branch is $C^1$, and
to give the leading order term of $\partial_c Q_c$ as $c \rightarrow 0$. The
result is the following one.

\begin{proposition}
  \label{prop16}For any $+ \infty \geqslant p > 2$, there exists $c_0 (p) > 0$
  such that
  \[ c \mapsto Q_c - 1 \in C^1 (] 0, c_0 (p) [, X_p), \]
  with the estimate
  \[ \left\| \partial_c Q_c + \left( \frac{1 + o_{c \rightarrow 0} (1)}{c^2}
     \right) \partial_d (V_1 (. - d \overrightarrow{e_1}) V_{- 1} (. + d
     \overrightarrow{e_1}))_{| d = d_c \nobracket} \right\|_{X_{^p}} = o_{c
     \rightarrow 0} \left( \frac{1}{c^2} \right) . \]
\end{proposition}

Proposition \ref{prop16}, together with subsection \ref{th1proof}, ends the
proof of Theorem \ref{th1}. Subsections \ref{32} to \ref{38} are devoted to
the proof of Proposition \ref{prop16}.

In this section, to make the dependances on $c$ and $d$ clear, we use the
following notations. We denote $\Phi_{c, d}, \Psi_{c, d}$ and $\Gamma_{c, d}$
in order to emphasize the dependence of $\Phi, \Psi$ and $\Gamma$ in
Proposition \ref{contractionest} on $c$ and $d$. A value of $d$ that makes
$\lambda (c, d) = 0$ in Proposition \ref{dinc} is written $d_c$. We will show
later on that there exist one and only one value of $d_c$ satisfying this in
$\left] \frac{c}{2}, 2 c \right[$. With these notations, $Q_c = V_1 (. - d_c
\overrightarrow{e_1}) V_{- 1} (. + d_c \overrightarrow{e_1}) + \Gamma_{c,
d_c}$ is the solution of $(\tmop{TW}_c)$ we constructed in section \ref{n3}.

In subsection \ref{31} we showed that $\Phi_{c, d}$ is a $C^1$ function of
both $c$ and $d$. We also have computed estimates for the derivatives of
$\Phi_{c, d}$ with respect to $c$ and $d$ in Lemma \ref{nl310}, that will be
usefull here.

The goal is to show that $d_c$ is a $C^1$ function of $c$. We will do this by
the implicit function theorem, but this requires a lot of computations. In
particular, in Proposition \ref{dinc}, $d_c$ was choosen so that
\[ \langle L (\Phi_{c, d}) - (1 - \eta) (E - i c \partial_{x_2} V) \Psi_{c, d}
   + F (\Psi_{c, d}), \partial_d V \rangle = 0, \]
but we may equivalently define it by the implicit equation
\[ \int_{B (d \overrightarrow{e_1}, d^{\varepsilon'}) \cup B (- d
   \overrightarrow{e_1}, d^{\varepsilon'})} \mathfrak{R}\mathfrak{e} ((L
   (\Phi_{c, d}) - (1 - \eta) E \Psi_{c, d} + F (\Psi_{c, d}))
   \overline{\partial_d V}) = 0. \]
This is the same equation but the scalar product is not taken on the whole
space but only on $B (d \overrightarrow{e_1}, d^{\varepsilon'}) \cup B (- d
\overrightarrow{e_1}, d^{\varepsilon'})$ for some $0 < \varepsilon' < 1$ (we
will take $\varepsilon' = 13 / 24$ but this value is purely technical, other
values are possible). The only reason why we take it in the whole space in
Lemma \ref{dinc} was because of the boundary terms that will appear in the
integration by parts when we write
\[ \langle L (\Phi), \partial_d V \rangle = \langle \Phi, L (\partial_d V)
   \rangle . \]
With the boundary terms on the boundary of $B (\pm d \overrightarrow{e_1},
d^{\varepsilon'})$, $\varepsilon' > 0$, we are far enough from the vortices to
make them small enough for our estimations. Thanks to this we can separate
what happens near the vortex $V_1$ from what happens near the vortex $V_{- 1}$
because now the integrals are in two well separated domain, one around each
vortex. We use this in subsection \ref{32}. We need to differentiate the
equation with respect to $d$. If we write $Q_{c, d} = V + \Gamma_{c, d}$, then
$\partial_d Q_c = \partial_d V + \partial_d (\Gamma_{c, d})$. The term
$\partial_d V$ is easy to compute and to understand: we just move both
vortices in opposite direction. But $\partial_d \Gamma_{c, d}$ is very
difficult to understand, and our estimations on $\Gamma_{c, d_c}$ are not
enough to compute easily what happens with sufficient precision to control its
contribution. We would rather write $Q_{c, d}$ under the form
\[ Q_{c, d} (x) = (V_1 (x - d \overrightarrow{e_1}) + \tilde{\Gamma}_1 (x - d
   \overrightarrow{e_1})) + (V_{- 1} (x + d \overrightarrow{e_1}) +
   \tilde{\Gamma}_{- 1} (x + d \overrightarrow{e_1})) + \tmop{Err} \]
where $\tilde{\Gamma}_1 (x - d \overrightarrow{e_1})$ is centered near $V_1$,
is small and is here because of the existence of $V_{- 1}$ far away. Then the
term we understand is
\[ \partial_{x_1 + d} (V_1 (x - d \overrightarrow{e_1}) + \tilde{\Gamma}_1 (x
   - d \overrightarrow{e_1})) \]
which is what changes near the center of $V_1$ when we move only the other
vortex. This can be computed more easily and that is what we do in subsection
\ref{34}. This term is easy to compute only near the vortex $V_1$, and that is
one of the reasons we work only on $B (d \overrightarrow{e_1},
d^{\varepsilon'})$. The main contribution to the variation of the position of
$V_{- 1}$ is as expected from the source term $E - i c \partial_{x_2} V$. This
is the computation of subsection \ref{35}.

Furthermore, most estimations boils down to what happen near each vortex, see
for instance the contribution of $E$ in step 5 of the proof of Proposition
\ref{dinc}, where we separate the contribution far from both vortices and
close to them. By integrating only on $B (d \overrightarrow{e_1},
d^{\varepsilon'})$ we reduce the number of estimations we need to do.
Moreover, in such a ball the contribution of the vortex $V_{- 1}$ and its
derivatives are easy to compute, see subsection \ref{33}.

Subsection \ref{36} gathers all the estimations needed to show that only the
contribution from the source term is of leading order. Subsection \ref{37} and
\ref{38} are easy computations using previous subsections to compute the first
order term of $\partial_c Q_c$.

The main and most difficult part is subsection \ref{34}. We want to show that
$\partial_{x_1 + d} (\tilde{\Gamma}_1 (x - d \overrightarrow{e_1}))$ is much
smaller than $\tilde{\Gamma}_1 (x - d \overrightarrow{e_1})$, i.e. that the
derivative with respect to $x_1 + d$ gives us additional smallness in $c$. For
this we do a proof by contradiction which follows closely what was done in the
proof of Proposition \ref{invertop}.

\

We define the following differential operators:
\[ \partial_{y_1} \assign \partial_{x_1} - \partial_d, \]
\[ \partial_{z_1} \assign \partial_{x_1} + \partial_d . \]
These notations follow the definitions of $y_1 = x_1 - d$ and $z_1 = x_1 + d$
from (\ref{notation}). The derivative in $d$ is taken at fixed $c$. The
function $\partial_d \Phi_{c, d}$ is the derivative of $\Phi$ with respect to
$d$ at fixed $c$ and we shall use the notation
\[ \partial_d \Phi_{c, d_c} \assign \partial_d \Phi_{c, d | d = d_c
   \nobracket}, \]
and similarly for $\partial_d \Gamma_{c, d_c}$ and $\partial_d \Psi_{c, d_c}$.
The derivatives $\partial_{y_1}$ and $\partial_{z_1}$ behave naturally on
function depending on $x$ and $d$ only through $y$ or $z$, as shown in the
following lemma.

\begin{lemma}
  \label{deryz}For any $\mathfrak{F} \in C^1 (\mathbbm{R}^2, \mathbbm{C}),$ we
  have
  \[ \partial_{y_1} (\mathfrak{F} (z)) = \partial_{z_1} (\mathfrak{F} (y)) = 0
  \]
  and
  \[ \partial_{y_1} (\mathfrak{F} (y)) = 2 \partial_{x_1} \mathfrak{F} (y), \]
  \[ \partial_{z_1} (\mathfrak{F} (z)) = 2 \partial_{x_1} \mathfrak{F} (z) .
  \]
\end{lemma}

\begin{proof}
  We compute
  \[ \partial_{y_1} (\mathfrak{F} (z)) = \partial_{x_1} (\mathfrak{F} (x_1 +
     d, x_2)) - \partial_d (\mathfrak{F} (x_1 + d, x_2)) = \partial_{x_1}
     \mathfrak{F} (z) - \partial_{x_1} \mathfrak{F} (z) = 0. \]
  Similarly we have $\partial_{z_1} (\mathfrak{F} (y)) = 0.$ Moreover,
  \[ \partial_{y_1} (\mathfrak{F} (y)) = \partial_{x_1} (\mathfrak{F} (x_1 -
     d, x_2)) - \partial_d (\mathfrak{F} (x_1 - d, x_2)) = \partial_{x_1}
     \mathfrak{F} (y) + \partial_{x_1} \mathfrak{F} (y) = 2 \partial_{x_1}
     \mathfrak{F} (y) \]
  and similarly, $\partial_{z_1} (\mathfrak{F} (z)) = 2 \partial_{x_1}
  \mathfrak{F} (z)$.
\end{proof}

We have an estimate on $\partial_d \Phi_{c, d_{} | d = d_c \nobracket}$, but
it is not enough to show that $d_c$ is a $C^1$ function of $c$. The main idea
of the proof is to compute an estimate on $\partial_{z_1} \Phi_{c, d_c} =
\partial_{x_1} \Phi_{c, d_c} + \partial_d \Phi_{c, d_c}$ near the vortex $V_1$
which is better than the ones on $\partial_{x_1} \Phi_{c, d_c}$ and
$\partial_d \Phi_{c, d_c}$. In particular we will have $\partial_{z_1}
\Phi_{c, d_c} = o_{c \rightarrow 0} (c^{1 + \lambda})$ for some $\lambda > 0$
instead of $o_{c \rightarrow 0} (c^{1 - \sigma})$ for $\sigma > 0$. This
estimate is done in Proposition \ref{derg}. First, we compute a first rough
estimate on $\partial_{z_1} \Psi_{c, d}$ which is a corollary of Lemma
\ref{dderpsy}.

\begin{corollary}
  \label{cor361405}For $\chi$ a smooth cutoff function with value 1 in $\{
  r_{- 1} \geqslant 3 \}$ and $0$ in $\{ r_{- 1} \leqslant 2 \}$, for $0 <
  \sigma < \sigma' < 1$, there exist $c_0 (\sigma, \sigma') > 0$ such that,
  for $0 < c < c_0 (\sigma, \sigma')$, we have
  \begin{eqnarray*}
    &  & \| V \chi \partial_{z_1} \Psi_{c, d | d = d_c \nobracket} \|_{C^1
    (\{ \tilde{r} \leqslant 3 \})}\\
    & + & \| \tilde{r}^{1 + \sigma} \mathfrak{R}\mathfrak{e} (\partial_{z_1}
    \Psi_{c, d | d = d_c \nobracket}) \|_{L^{\infty} (\{ \tilde{r} \geqslant 2
    \})} + \| \tilde{r}^{2 + \sigma} \nabla \mathfrak{R}\mathfrak{e}
    (\partial_{z_1} \Psi_{c, d | d = d_c \nobracket}) \|_{L^{\infty} (\{
    \tilde{r} \geqslant 2 \})}\\
    & + & \| \tilde{r}^{\sigma} \mathfrak{I}\mathfrak{m} (\partial_{z_1}
    \Psi_{c, d | d = d_c \nobracket}) \|_{L^{\infty} (\{ \tilde{r} \geqslant 2
    \})} + \| \tilde{r}^{1 + \sigma} \nabla \mathfrak{I}\mathfrak{m}
    (\partial_{z_1} \Psi_{c, d | d = d_c \nobracket}) \|_{L^{\infty} (\{
    \tilde{r} \geqslant 2 \})}\\
    & \leqslant & K (\sigma, \sigma') c^{1 - \sigma'} .
  \end{eqnarray*}
\end{corollary}

\begin{proof}
  Remark that $V_1 \partial_d \Psi_{c, d}$ might not be bounded near $d
  \overrightarrow{e_1}$, but $V_1 \partial_{z_1} \Psi_{c, d}$ is, since, by
  Lemma \ref{deryz}, $\partial_{z_1} V_1 = 0$ hence
  \[ V_1 \partial_{z_1} \Psi_{c, d} = \partial_{z_1} \Phi_{c, d} = \partial_d
     \Phi_{c, d} + \partial_{x_1} \Phi_{c, d}, \]
  with $\partial_d \Phi_{c, d}$ bounded by Lemma \ref{dderpsy}. We take a
  cutoff $\chi$ to avoid the fact that $V_{- 1} \partial_{z_1} \Psi_{c, d}$ is
  not necessary bounded near $- d \overrightarrow{e_1}$. In particular, with
  these remarks, we easily check, with Lemma \ref{dderpsy}, that
  \[ \| V \chi \partial_{z_1} \Psi_{c, d | d = d_c \nobracket} \|_{C^1 (\{
     \tilde{r} \leqslant 3 \})} \leqslant K (\sigma, \sigma') c^{1 - \sigma'}
     . \]
  We now focus on the region $\{ \tilde{r} \geqslant 2 \}$. From the
  definition of $\partial_{z_1}$, we have that
  \[ \partial_{z_1} \Psi_{c, d_{| d = d_c \nobracket}} = \partial_d \Psi_{c,
     d_c} + \partial_{x_1} \Psi_{c, d_c} . \]
  We compute
  \[ \partial_d \Psi_{c, d_c} = \frac{\partial_d \Phi_{c, d_c}}{V} +
     \frac{\partial_d V}{V} \Psi_{c, d_c}, \]
  and from Lemma \ref{dderpsy}, we have
  \[ \left\| \frac{\partial_d \Phi_{c, d_c}}{V} \right\|_{\ast, \sigma, d_c}
     \leqslant K (\sigma, \sigma') c^{1 - \sigma'} . \]
  From Lemma \ref{ddVest}, we have
  \[ | \partial_d V | \leqslant \frac{K}{(1 + \tilde{r})} \]
  and
  \[ | \nabla \partial_d V | \leqslant \frac{K}{(1 + \tilde{r})^2}, \]
  and together with $\| \Psi_{c, d_c} \|_{\ast, \sigma, d_c} \leqslant K
  (\sigma, \sigma') c^{1 - \sigma'}$, we check that
  \[ \begin{array}{lll}
       &  & \left\| \tilde{r}^{1 + \sigma} \mathfrak{R}\mathfrak{e} \left(
       \frac{\partial_d V}{V} \Psi_{c, d_c} \right) \right\|_{L^{\infty} (\{
       \tilde{r} \geqslant 2 \})} + \left\| \tilde{r}^{2 + \sigma} \nabla
       \mathfrak{R}\mathfrak{e} \left( \frac{\partial_d V}{V} \Psi_{c, d_c}
       \right) \right\|_{L^{\infty} (\{ \tilde{r} \geqslant 2 \})}\\
       & + & \left\| \tilde{r}^{\sigma} \mathfrak{I}\mathfrak{m} \left(
       \frac{\partial_d V}{V} \Psi_{c, d_c} \right) \right\|_{L^{\infty} (\{
       \tilde{r} \geqslant 2 \})} + \left\| \tilde{r}^{1 + \sigma} \nabla
       \mathfrak{I}\mathfrak{m} \left( \frac{\partial_d V}{V} \Psi_{c, d_c}
       \right) \right\|_{L^{\infty} (\{ \tilde{r} \geqslant 2 \})}\\
       & \leqslant & K (\sigma, \sigma') c^{1 - \sigma'} .
     \end{array} \]
  Finally, for the contribution of $\partial_{x_1} \Psi_{c, d_c}$, using $\|
  \Psi_{c, d_c} \|_{\ast, \sigma, d_c} \leqslant K (\sigma, \sigma') c^{1 -
  \sigma'}$, we show that, with some margin,
  \[ \begin{array}{lll}
       &  & \| \tilde{r}^{1 + \sigma} \mathfrak{R}\mathfrak{e}
       (\partial_{x_1} \Psi_{c, d_c}) \|_{L^{\infty} (\{ \tilde{r} \geqslant 2
       \})} + \| \tilde{r}^{2 + \sigma} \nabla \mathfrak{R}\mathfrak{e}
       (\partial_{x_1} \Psi_{c, d_c}) \|_{L^{\infty} (\{ \tilde{r} \geqslant 2
       \})}\\
       & + & \| \tilde{r}^{\sigma} \mathfrak{I}\mathfrak{m} (\partial_{x_1}
       \Psi_{c, d_c}) \|_{L^{\infty} (\{ \tilde{r} \geqslant 2 \})} + \|
       \tilde{r}^{1 + \sigma} \nabla \mathfrak{I}\mathfrak{m} (\partial_{x_1}
       \Psi_{c, d_c}) \|_{L^{\infty} (\{ \tilde{r} \geqslant 2 \})}\\
       & \leqslant & K (\sigma, \sigma') c^{1 - \sigma'},
     \end{array} \]
  which ends the proof of this corollary.
\end{proof}

\subsection{Recasting the implicit equation defining $d_c$}\label{32}

At this point, we do not know if $d_c$ is uniquely defined for $c > 0$. We
denote by $d_c$ a value defined by the implicit equation on $d$:
\[ \langle \tmop{TW}_c (Q_{c, d}), \partial_d V \rangle = 0, \]
where
\[ Q_{c, d} \assign V + \Gamma_{c, d}, \]
with $\Gamma_{c, d} = \eta V \Psi_{c, d} + (1 - \eta) V (e^{\Psi_{c, d}} -
1)$, which is a $C^1$ function of $d$ and $c$ in $\mathcal{E}_{\ast, \sigma,
d}$ thanks to subsection \ref{31}. Remark that $d_c$ is also defined by the
implicit equation for $0 < \varepsilon' < 1$:
\[ \int_{B (d \overrightarrow{e_1}, d^{\varepsilon'}) \cup B (- d
   \overrightarrow{e_1}, d^{\varepsilon'})} \mathfrak{R}\mathfrak{e}
   (\overline{\partial_d V} \tmop{TW}_c (Q_{c, d})) = 0, \]
that we will use instead because of the reasons explained at the begining of
section \ref{qcc1}. We can check easily that $\partial_d Q_{c, d}, \partial_c
Q_{c, d} \in C^{\infty} (\mathbbm{R}^2)$ (by looking at the equations they
satisfy in the distribution sense and using standard elliptic regularity
arguments), and furthermore, that $d \mapsto \partial_d Q_{c, d}$ and $c
\mapsto \partial_c Q_c$ are continuous functions (on their domain of
definition in $C^{\infty}_{\tmop{loc}} (\mathbbm{R}^2)$ for instance). From
now on, we take any $0 < \varepsilon' < 1$, but we will fix its value later
on. We want to differentiate this quantity with respect to $d$ and take the
result at a value $d_c$ such that $\tmop{TW}_c (Q_{c, d_c}) = 0$ in
$\mathbbm{R}^2$. In particular, we have
\[ \partial_d \int_{B (d \overrightarrow{e_1}, d^{\varepsilon'}) \cup B (- d
   \overrightarrow{e_1}, d^{\varepsilon'})} \mathfrak{R}\mathfrak{e}
   (\overline{\partial_d V} \tmop{TW}_c (Q_{c, d}))_{| d = d_c \nobracket} =
\]
\[ \int_{B (d \overrightarrow{e_1}, d^{\varepsilon'}) \cup B (- d
   \overrightarrow{e_1}, d^{\varepsilon'})} \mathfrak{R}\mathfrak{e}
   (\overline{\partial_d V} \partial_d (\tmop{TW}_c (Q_{c, d})))_{| \nobracket
   d = d_c} . \]
Now, by symmetry, we remark that
\[ \int_{B (d \overrightarrow{e_1}, d^{\varepsilon'}) \cup B (- d
   \overrightarrow{e_1}, d^{\varepsilon'})} \mathfrak{R}\mathfrak{e}
   (\overline{\partial_d V} \partial_d (\tmop{TW}_c (Q_{c, d}))) = 2 \int_{B
   (d \overrightarrow{e_1}, d^{\varepsilon'})} \mathfrak{R}\mathfrak{e}
   (\overline{\partial_d V} \partial_d (\tmop{TW}_c (Q_{c, d}))) . \]
We will use the two operators we have already defined:
\[ \partial_{y_1} = \partial_{x_1} - \partial_d \quad \tmop{and} \quad
   \partial_{z_1} = \partial_{x_1} + \partial_d . \]
Since $\tmop{TW}_c (Q_{c, d_c}) = 0$ everywhere in $\mathbbm{R}^2$, we
therefore have $\partial_{x_1} (\tmop{TW}_c (Q_{c, d_c})) = 0$, hence, at $d =
d_c$,
\[ \partial_d (\tmop{TW}_c (Q_{c, d})) = \partial_{z_1} (\tmop{TW}_c (Q_{c,
   d})) . \]
We write
\[ \tmop{TW}_c (Q_{c, d}) = \tmop{TW}_c (V) + L_{} (\Gamma_{c, d}) +
   \tmop{NL}_V (\Gamma_{c, d}), \]
with
\[ L_{} (\Gamma_{c, d}) = - \Delta \Gamma_{c, d} - i c \partial_{x_2}
   \Gamma_{c, d} - (1 - | V |^2) \Gamma_{c, d} + 2\mathfrak{R}\mathfrak{e}
   (\bar{V} \Gamma_{c, d}) V \]
and
\[ \tmop{NL}_V (\Gamma_{c, d}) \assign 2\mathfrak{R}\mathfrak{e} (\bar{V}
   \Gamma_{c, d}) \Gamma_{c, d} + | \Gamma_{c, d} |^2 (V + \Gamma_{c, d}) . \]
We compute
\[ \partial_{z_1} (\tmop{TW}_c (Q_{c, d})) = \partial_{z_1} (\tmop{TW}_c (V))
   + L (\partial_{z_1} \Gamma_{c, d}) + (\partial_{z_1} L) (\Gamma_{c, d}) +
   \partial_{z_1} (\tmop{NL}_V (\Gamma_{c, d})), \]
therefore, at $d = d_c$,
\[ \partial_d \int_{B (d \overrightarrow{e_1}, d^{\varepsilon'})}
   \mathfrak{R}\mathfrak{e} (\overline{\partial_d V} \tmop{TW}_c (Q_{c, d})) =
   \int_{B (d \overrightarrow{e_1}, d^{\varepsilon'})}
   \mathfrak{R}\mathfrak{e} (\overline{\partial_d V} \partial_{z_1}
   (\tmop{TW}_c (V))) \]
\[ + \int_{B (d \overrightarrow{e_1}, d^{\varepsilon'})}
   \mathfrak{R}\mathfrak{e} (\overline{\partial_d V} L (\partial_{z_1}
   \Gamma_{c, d})) + \int_{B (d \overrightarrow{e_1}, d^{\varepsilon'})}
   \mathfrak{R}\mathfrak{e} (\overline{\partial_d V} (\partial_{z_1} L)
   (\Gamma_{c, d})) \]
\begin{equation}
  + \int_{B (d \overrightarrow{e_1}, d^{\varepsilon'})}
  \mathfrak{R}\mathfrak{e} (\overline{\partial_d V} \partial_{z_1}
  (\tmop{NL}_V (\Gamma_{c, d}))) \label{pouet3}
\end{equation}
since the boundary term is $0$ (when the differentiation is on the $d$ in $B
(d \overrightarrow{e_1}, d^{\varepsilon'})$) because $\tmop{TW}_c (Q_{c, d_c})
= 0$. We need to estimate those four terms at $d = d_c$, and that is the goal
of the next subsections. Subsections \ref{33} and \ref{34} yield estimates on
the derivatives of $V_{- 1}$ and $\partial_{z_1} \Psi_{c, d}$ respectively in
$B'_d \assign B (d \overrightarrow{e_1}, d^{\varepsilon'})$. Subsection
\ref{35} is about the estimation of
\[ \int_{B (d \overrightarrow{e_1}, d^{\varepsilon'})}
   \mathfrak{R}\mathfrak{e} (\overline{\partial_d V} \partial_{z_1}
   (\tmop{TW}_c (V))) \]
which will be the leading order term, and subsection \ref{36} shows that all
the other terms are smaller for $d_c$ large enough.

\subsection{Estimates on the derivatives of $V_{- 1}$ in $B (d
\protect\overrightarrow{e_1}, d^{\varepsilon})$}\label{33}

\begin{lemma}
  \label{V-1est}For $0 < \varepsilon < 1$, in $B (d \overrightarrow{e_1},
  d^{\varepsilon})$, with the $O (.)$ being always real valued, we have
  \[ \partial_{x_1} V_{- 1} = \left( O_{d \rightarrow \infty} \left(
     \frac{1}{d^3} \right) + i O_{d \rightarrow \infty} \left( \frac{1}{d^{2 -
     \varepsilon}} \right) \right) V_{- 1}, \]
  \[ \partial_{x_2} V_{- 1} = \left( O_{d \rightarrow \infty} \left(
     \frac{1}{d^{4 - \varepsilon}} \right) + i O_{d \rightarrow \infty} \left(
     \frac{1}{d} \right) \right) V_{- 1}, \]
  \[ \partial_{x_1 x_1} V_{- 1} = \left( O_{d \rightarrow \infty} \left(
     \frac{1}{d^{4 - 2 \varepsilon}} \right) + i O_{d \rightarrow \infty}
     \left( \frac{1}{d^{3 - \varepsilon}} \right) \right) V_{- 1}, \]
  \[ \partial_{x_1 x_2} V_{- 1} = \left( O_{d \rightarrow \infty} \left(
     \frac{1}{d^{3 - \varepsilon}} \right) + \frac{i}{4 d^2}  \left( 1 + O_{d
     \rightarrow \infty} \left( \frac{1}{d^{1 - \varepsilon}} \right) \right)
     \right) V_{- 1} . \]
\end{lemma}

\begin{proof}
  Recall from Lemma \ref{dervor} that, with $u = \frac{\rho_{- 1}' (r_{-
  1})}{\rho_{- 1} (r_{- 1})}$,
  \[ \partial_{x_1} V_{- 1} = \left( \cos (\theta_{- 1}) u + \frac{i}{r_{- 1}}
     \sin (\theta_{- 1}) \right) V_{- 1}, \]
  \[ \partial_{x_2} V_{- 1} = \left( \sin (\theta_{- 1}) u - \frac{i}{r_{- 1}}
     \cos (\theta_{- 1}) \right) V_{- 1}, \]
  \[ \partial_{x_1 x_1} V_{- 1} = \left( \cos^2 (\theta_{- 1}) (u^2 + u') +
     \sin^2 (\theta_{- 1}) \left( \frac{u}{r_{- 1}} - \frac{1}{r_{- 1}^2}
     \right) - 2 i \sin (\theta_{- 1}) \cos (\theta_{- 1}) \left(
     \frac{1}{r_{- 1}^2} - \frac{u}{r_{- 1}} \right) \right) V_{- 1} \]
  and
  \[ \partial_{x_1 x_2} V_{- 1} = \left( \sin (\theta_{- 1}) \cos^{\nosymbol}
     (\theta_{- 1}) \left( u^2 + u' + \frac{1}{r_{- 1}^2} - \frac{u}{r_{- 1}}
     \right) + i \cos (2 \theta_{- 1}) \left( \frac{1}{r_{- 1}^2} -
     \frac{u}{r_{- 1}} \right) \right) V_{- 1} . \]
  In the ball $B (d \overrightarrow{e_1}, d^{\varepsilon})$, we have, by Lemma
  \ref{lemme3}, that $\frac{1}{r_{- 1}} \leqslant \frac{K}{d}$,
  \[ u = O_{d \rightarrow \infty} \left( \frac{1}{d^3} \right) \quad
     \tmop{and} \quad \sin (\theta_{- 1}) = O_{d \rightarrow \infty} \left(
     \frac{1}{d^{1 - \varepsilon}} \right), \]
  the last one is because for $(y_1, y_2) \in B (d \overrightarrow{e_1},
  d^{\varepsilon})$, we have $| y_2 | \leqslant d^{\varepsilon}$ hence
  \[ | \sin (\theta_{- 1}) | = \frac{| y_2 |}{r_{- 1}} \leqslant \frac{K}{d^{1
     - \varepsilon}} . \]
  We also compute in the same way that
  \[ \cos (\theta_{- 1}) = \sqrt{1 - \sin^2 (\theta_{- 1})} = 1 + O_{d
     \rightarrow \infty} \left( \frac{1}{d^{2 - 2 \varepsilon}} \right) . \]
  With the equation on $\rho_{- 1}$ coming fom $- \Delta V_{- 1} - (1 - | V_{-
  1} |^2) V_{- 1} = 0$, we check easily that
  \[ u' = O_{d \rightarrow \infty} \left( \frac{1}{d^4} \right) \]
  as well (or see {\cite{HH}}). Finally, we estimate
  \[ \cos (2 \theta_{- 1}) = 1 - 2 \sin^2 (\theta_{- 1}) = 1 + O_{d
     \rightarrow \infty} \left( \frac{1}{d^{2 - 2 \varepsilon}} \right) \]
  and
  \[ \frac{1}{r_{- 1}^2} = (2 d + O_{d \rightarrow \infty}
     (d^{\varepsilon}))^{- 2} = \frac{1}{4 d^2} + O_{d \rightarrow \infty}
     \left( \frac{1}{d^{3 - \varepsilon}} \right) . \]
  With this estimations, we end the proof of the lemma.
\end{proof}

\subsection{Estimate on $\partial_{z_1} \Psi_{c, d}$ in $B (d
\protect\overrightarrow{e_1}, d^{\varepsilon'})$}\label{34}

We define the following norms for $\Psi = \Psi_1 + i \Psi_2$ and $h = h_1 + i
h_2$, $0 < \alpha < 1, 0 < \varepsilon' < \varepsilon < 1$:
\begin{eqnarray*}
  \| \Psi \|_{\ast, B'_d} & \assign & \| V \Psi \|_{C^1 (\{ r_1 \leqslant 2
  \})}\\
  & + & \| r_1^{1 - \alpha} \Psi_1 \|_{L^{\infty} (\{ d^{\varepsilon'}
  \geqslant r_1 \geqslant 2 \})} + \| r_1^{1 - \alpha} \nabla \Psi_1
  \|_{L^{\infty} (\{ d^{\varepsilon'} \geqslant r_1 \geqslant 2 \})}\\
  & + & \| r_1^{- \alpha} \Psi_2 \|_{L^{\infty} (\{ d^{\varepsilon'}
  \geqslant r_1 \geqslant 2 \})} + \| r_1^{1 - \alpha} \nabla \Psi_2
  \|_{L^{\infty} (\{ d^{\varepsilon'} \geqslant r_1 \geqslant 2 \})}
\end{eqnarray*}
and
\begin{eqnarray*}
  \| h \|_{\ast \ast, B_d} & \assign & \| V h \|_{C^0 (\{ r_1 \leqslant 3
  \})}\\
  & + & \| r_1^{1 - \alpha} h_1 \|_{L^{\infty} (\{ d^{\varepsilon} \geqslant
  r_1 \geqslant 2 \})} + \| r_1^{2 - \alpha} h_2 \|_{L^{\infty} (\{
  d^{\varepsilon} \geqslant r_1 \geqslant 2 \})} .
\end{eqnarray*}
They are the norms $\| . \|_{\ast, - \alpha, d}$ and $\| . \|_{\ast \ast, -
\alpha, d}$ of subsection \ref{normsetup}, but without the second derivatives,
less decay on the gradient of the real part for $\| . \|_{\ast, B'_d}$, and
only on $B'_d = B (d \overrightarrow{e_1}, d^{\varepsilon'})$ for $\| .
\|_{\ast, B'_d}$ and \ on $B_d \assign B (d \overrightarrow{e_1},
d^{\varepsilon})$ for $\| . \|_{\ast \ast, B_d}$. The other main difference
with the previous norms is that we require less decay (we take $- \alpha < 0$
instead of $\sigma > 0$ in the decay) in space, which here, since the norms
are only in $\{ r_1 \leqslant d^{\varepsilon} \}$, can be compensated by some
smallness in $c$.

\

From Corollary \ref{cor361405}, we have that $\| \partial_{z_1} \Psi_{c, d_c}
\|_{\ast, B'_{d_c}} < + \infty$. We want to show the following proposition.

\begin{proposition}
  \label{derg}For $0 < \alpha < 1, 0 < \varepsilon' < \varepsilon < 1, 0 <
  \lambda < 1$, if
  \[ \lambda < (1 + \alpha) \varepsilon', \]
  \[ \lambda + (1 - \alpha) \varepsilon' < 2 \varepsilon - \varepsilon' \]
  and
  \[ \lambda < 2 - \varepsilon (2 - \alpha), \]
  we have
  \[ \| \partial_{z_1} \Psi_{c, d | d = d_c \nobracket} \|_{\ast, B'_{d_c}} =
     o_{c \rightarrow 0} (c^{1 + \lambda_{}}) . \]
\end{proposition}

Such a choice of parameters $(\lambda, \alpha, \varepsilon, \varepsilon')$
exists, we can take for instance $\alpha = 1 / 2, \lambda = 3 / 4, \varepsilon
= 19 / 24$ and $\varepsilon' = 13 / 24$. Furthermore, with this particular
choice of parameters, we also have
\begin{equation}
  \lambda + (1 - \alpha) \varepsilon' > 1, \label{3142305}
\end{equation}
which will be usefull later on. These conditions are bounds on how much
additional smallness we can have on $\partial_{z_1} \Psi_{c, d}$ near $d_c
\overrightarrow{e_1}$.

The main goal of this proposition is to have a decay in $c$ better than $O_{c
\rightarrow 0} (c)$, which is not obvious from the estimates we have done
until now. The estimate on $\partial_{z_1} \Psi_{c, d | d = d_c \nobracket}$
from Corollary \ref{cor361405} will not be enough in the computation of
$\partial_c d_c$ for the nonlinear terms. The proof of Proposition \ref{derg}
follows closely the proof of the inversibility of the linearized operator in
Proposition \ref{invertop}. We want to invert the same linearized operator,
but with a different norm, which is better locally around the vortex $V_1$.

The reason why we take $B_d$ a little bigger than $B'_d$ is to make the
elliptic estimates of step 2 in Proposition \ref{invertop} work here too. The
main idea of this proposition is to show that if we move $V_{- 1}$ a little,
then locally around $V_1$ the change is very small. We now start the proof of
Proposition \ref{derg}.

\

\begin{proof}
  First, we remark that in $B_d$, since $\varepsilon < 1$, $\tilde{r} = r_1$.

\  
  
  \begin{tmindent}
    Step 1.  Computation of the equation on $\partial_{z_1} \Psi_{c, d}$.
  \end{tmindent}
  
\ 
  
  Recall that $\Phi_{c, d}$ solves the equation (with $\Phi_{c, d} = V
  \Psi_{c, d}$)
  \[ \eta L (\Phi_{c, d}) + (1 - \eta) V L' (\Psi_{c, d}) + F (\Psi_{c, d}) =
     \lambda (c, d) Z_d, \]
  and we recall that $\lambda (c, d) = \frac{\langle F (\Psi_{c, d}), Z_d
  \rangle}{\| Z_d \|_{L^2 (\mathbbm{R}^2)}^2}$, and we check easily, with
  Lemma \ref{dderpsy}, that it is a $C^1$ function of $d$. The equation on
  $\Phi_{c, d}$ holds for any $x \in \mathbbm{R}^2$ and any $d \in
  \mathbbm{R}, \frac{1}{2 d} < c < \frac{2}{d}$, hence
  \[ \partial_{z_1} (\eta L (\Phi_{c, d}) + (1 - \eta) V L' (\Psi_{c, d}) +
     \Pi_d^{\bot} (F (\Psi_{c, d})) - \lambda (c, d) Z_d) = 0. \]
  We compute
  \begin{eqnarray*}
    \partial_{z_1} (\lambda (c, d) Z_d) & = & (\partial_{x_1} + \partial_d)
    (\lambda (c, d) Z_d)\\
    & = & \partial_d \lambda (c, d) Z_d + \lambda (c, d_{}) \partial_{z_1}
    Z_d,
  \end{eqnarray*}
  and we recall, from the proof of Proposition \ref{dinc} that
  \[ \lambda (c, d) \int_{\mathbbm{R}^2} | \partial_d V |^2 \eta^2 = \pi
     \left( \frac{1}{d} - c \right) + O_{c \rightarrow 0}^{\sigma} (c^{2 -
     \sigma}) . \]
  With Lemma \ref{dderpsy} and Corollary \ref{cor361405}, as well as Lemma
  \ref{ddVest}, we infer that the terms contributing to the $O_{c \rightarrow
  0}^{\sigma} (c^{2 - \sigma})$ are such that, when differentiated with
  respect to $d$, their contributions are still a $O_{c \rightarrow
  0}^{\sigma} (c^{2 - \sigma})$. Indeed, if the derivative with respect to $d$
  fall on a $\Psi_{c, d}$, then by Lemma \ref{dderpsy} and Corollary
  \ref{cor361405}, the same estimates used in the proof of Proposition
  \ref{dinc} still hold. If the derivative fall on a term depending on $V$, by
  Lemma \ref{ddVest}, we gain some decay in the integrals. We deduce that,
  since $\lambda (c, d_c) = 0$,
  \[ \partial_d \lambda (c, d)_{| d = d_c \nobracket} = \frac{- \pi}{d_c^2} +
     O_{c \rightarrow 0}^{\sigma} (c^{2 - \sigma}) = O_{c \rightarrow
     0}^{\sigma} (c^{2 - \sigma}) . \]
  Here, we see why the fact that $d$ is differentiable with respect to $c$ is
  not obvious. The main contribution is at this point not enough to beat the
  error terms. Therefore, showing that $\partial_d \lambda (c, d) \neq 0$ is
  not simple here. This is why we need improved estimations on $\partial_{z_1}
  \Psi_{c, d_c}$, that will give us the fact that the error terms are a
  $O^{\varepsilon}_{c \rightarrow 0} (c^{2 + \varepsilon})$ for some
  $\varepsilon > 0$.
  
  \
  
  Now, writing
  \[ \tmop{TW}_c (Q_{c, d}) = \eta L (\Phi_{c, d}) + (1 - \eta) V L' (\Psi_{c,
     d}) + F (\Psi_{c, d}), \]
  (with the notations of Lemma \ref{lemma7}), we have (since $\lambda (c, d_c)
  = 0$)
  \[ (\partial_{z_1} (\tmop{TW}_c (Q_{c, d}))_{} - \partial_d \lambda (c, d)
     Z_d)_{| d = d_c \nobracket} = 0. \]
  We recall that
  \[ F (\Psi_{c, d}) = E - i c \partial_{x_2} V + V (1 - \eta) (- \nabla
     \Psi_{c, d} . \nabla \Psi_{c, d} + | V |^2 S (\Psi_{c, d})) + R (\Psi_{c,
     d}), \]
  where $R (\Psi_{c, d})$ is a sum of terms at least quadratic in $\Psi_{c,
  d}$ or $\Phi_{c, d}$ localized in the area where $\eta \neq 0$.
  
  We compute
  \begin{eqnarray*}
    \partial_{z_1} (\tmop{TW}_c (Q_{c, d})) & = & \eta L (V \partial_{z_1}
    \Psi_{c, d}) + (1 - \eta) V L' (\partial_{z_1} \Psi_{c, d})\\
    & + & \eta \partial_{z_1} L (\Phi_{c, d}) + (1 - \eta) V \partial_{z_1}
    L' (\Psi_{c, d}) + \partial_{z_1} (E - i c \partial_{x_2} V)\\
    & + & \eta L (\partial_{z_1} V \Psi_{c, d}) + (1 - \eta) \partial_{z_1} V
    L' (\Psi_{c, d})\\
    & + & \partial_{z_1} \eta (L (\Phi_{c, d}) - V L' (\Psi_{c, d}) - i c
    \partial_{x_2} \Phi_{c, d})\\
    & - & \partial_{z_1} \eta V (- i c \partial_{x_2} \Psi_{c, d} - \nabla
    \Psi_{c, d} . \nabla \Psi_{c, d} + | V |^2 S (\Psi_{c, d}))\\
    & + & \partial_{z_1} (R (\Psi_{c, d}))\\
    & + & \partial_{z_1} V (1 - \eta) (- i c \partial_{x_2} \Psi_{c, d} -
    \nabla \Psi_{c, d} . \nabla \Psi_{c, d} + | V |^2 S (\Psi_{c, d}))\\
    & + & V (1 - \eta) \partial_{z_1} (- i c \partial_{x_2} \Psi_{c, d} -
    \nabla \Psi_{c, d} . \nabla \Psi_{c, d} + | V |^2 S (\Psi_{c, d})) .
  \end{eqnarray*}
  We regroup the terms in the following way. We define
  \[ \mathcal{L} (\partial_{z_1} \Psi_{c, d}) \assign \eta L (V \partial_{z_1}
     \Psi_{c, d}) + (1 - \eta) V L' (\partial_{z_1} \Psi_{c, d}), \]
  which is the same linearized operator we have inverted in Proposition
  \ref{invertop} (taken in $\partial_{z_1} \Psi_{c, d}$), and we define the
  operator
  \[ \mathcal{L}_{\partial_{z_1}} (\Psi_{c, d}) \assign \eta \partial_{z_1} L
     (\Phi_{c, d}) + (1 - \eta) V \partial_{z_1} L' (\Psi_{c, d}) + \eta L
     (\partial_{z_1} V \Psi_{c, d}) + (1 - \eta) \partial_{z_1} V L' (\Psi_{c,
     d}) . \]
  We already have shown that $\tmop{TW}_c (V) = E - i c \partial_{x_2} V$,
  therefore
  \[ \partial_{z_1} (\tmop{TW}_c (V)) = \partial_{z_1} (E - i c \partial_{x_2}
     V) . \]
  We define the local error
  \[ \tmop{Err}_{\tmop{loc}} \assign \partial_{z_1} (R (\Psi_{c, d})) -
     \partial_d \lambda (c, d) Z_d, \]
  the far away error
  \[ \tmop{Err}_{\tmop{far}} \assign \partial_{z_1} V (1 - \eta) (- \nabla
     \Psi_{c, d} . \nabla \Psi_{c, d} + | V |^2 S (\Psi)) \]
  and the nonlinear terms
  \[ \tmop{NL}_{\partial_{z_1}} (\Psi_{c, d}) \assign V (1 - \eta)
     \partial_{z_1} (- \nabla \Psi_{c, d} . \nabla \Psi_{c, d} + | V |^2 S
     (\Psi_{c, d})) . \]
  Finally, we write the cutoff error
  \[ \tmop{Err}_{\tmop{cut}} \assign \partial_{z_1} \eta (L (\Phi_{c, d}) - V
     L' (\Psi_{c, d}) + i c \partial_{x_2} \Psi_{c, d} + \nabla \Psi_{c, d} .
     \nabla \Psi_{c, d} - | V |^2 S (\Psi_{c, d})) \]
  which is supported in the area $\{ 2 \leqslant r_{- 1} \leqslant 3 \}$, and
  in particular is zero in $B (d_c \overrightarrow{e_1}, d_c^{\varepsilon})$.
  With these definitions, we have, at $d = d_c$,
  \begin{eqnarray*}
    &  & (\partial_{z_1} (\eta L (\Phi_{c, d}) + (1 - \eta) V L' (\Psi_{c,
    d}) + F (\Psi_{c, d})) - \partial_d \lambda (c, d) Z_d)_{| d = d_c
    \nobracket}\\
    & = & \mathcal{L} (\partial_{z_1} \Psi_{c, d})_{| d = d_c \nobracket}\\
    & + & \left( \partial_{z_1} (\tmop{TW}_c (V))
    +\mathcal{L}_{\partial_{z_1}} (\Psi_{c, d}) + \tmop{NL}_{\partial_{z_1}}
    (\Psi_{c, d}) \right)_{| d = d_c \nobracket}\\
    & + & (\tmop{Err}_{\tmop{loc}} + \tmop{Err}_{\tmop{far}} +
    \tmop{Err}_{\tmop{cut}})_{| d = d_c \nobracket} .
  \end{eqnarray*}
  The equation satisfied by $\partial_{z_1} \Psi_{c, d}$ at $d = d_c$ is
  therefore
  \[ \left( \mathcal{L} (\partial_{z_1} \Psi_{c, d}) + \partial_{z_1}
     (\tmop{TW}_c (V)) +\mathcal{L}_{\partial z_1} (\Psi_{c, d}) +
     \tmop{NL}_{\partial_{z_1}} (\Psi_{c, d}) + \tmop{Err}_{\tmop{loc}} +
     \tmop{Err}_{\tmop{far}} + \tmop{Err}_{\tmop{cut}} \right)_{| d = d_c
     \nobracket} = 0. \]

\
  
  \begin{tmindent}
    Step 2.  Beginning of the contradiction argument.
  \end{tmindent}
  
\ 
  
  Now, suppose that the result of Proposition \ref{derg} is false. The scheme
  of this proof is the same as in Proposition \ref{invertop}. Then, there
  exist an absolute constant $\delta > 0$ and sequences $\partial_{z_1}
  \Psi_n$, $c_n \rightarrow 0$, $d_n \rightarrow \infty$ such that
  \[ d_n^{1 + \lambda_{}} \| \partial_{z_1} \Psi_{n | d = d_n \nobracket}
     \|_{\ast, B'_{d_n}} \geqslant \delta, \]
  where we write $d_n = d_{c_n}$ (a value such that $\lambda (c_n, d_n) = 0$
  in Proposition \ref{dinc}). We have just shown that $\Psi_n$ (where we omit
  the subscripts in $d_n, c_n$) satisfies
  \begin{eqnarray*}
    \mathcal{L} (\partial_{z_1} \Psi_n) + \partial_{z_1} (\tmop{TW}_{c_n} (V))
    +\mathcal{L}_{\partial z_1} (\Psi_n) + \tmop{NL}_{\partial_{z_1}} (\Psi_n)
    + \tmop{Err}_{\tmop{loc}} + \tmop{Err}_{\tmop{far}} +
    \tmop{Err}_{\tmop{cut}} & = & 0.
  \end{eqnarray*}
  The function
  \[ \frac{(V \partial_{z_1} \Psi_n) (. - d_n \overrightarrow{e_1})}{\|
     \partial_{z_1} \Psi_n \|_{\ast, B'_{d_n}}} \]
  converges locally uniformly up to a subsequence to a limit $\mathfrak{G}$,
  since it is bounded in $\| . \|_{\ast, B'_{\lambda}}$ for any $\lambda > 0$
  (for the same reasons that $\Psi_n \rightarrow \Psi$ locally uniformly in
  the beginning of the proof of Proposition \ref{invertop}).
  
  \
  
  The equation on $\partial_{z_1} \Psi_n$ is
  \begin{equation}
    \mathcal{L} (\partial_{z_1} \Psi_n) + V h_n = 0, \label{crapaud}
  \end{equation}
  with
  \begin{eqnarray*}
    V h_n & \assign & \partial_{z_1} (\tmop{TW}_{c_n} (V))
    +\mathcal{L}_{\partial z_1} (\Psi_n) + \tmop{NL}_{\partial_{z_1}} (\Psi_n)
    + \tmop{Err}_{\tmop{loc}} + \tmop{Err}_{\tmop{far}} +
    \tmop{Err}_{\tmop{cut}} .
  \end{eqnarray*}
  The goal of Proposition \ref{invertop} was to estimate $\| \Psi \|_{\ast,
  \sigma, d}$ with $\| h \|_{\ast \ast, \sigma', d}$ for the equation
  $\mathcal{L} (\Psi) = h$ if $d$ is large enough (given an orthogonality
  condition on $\Psi$). Here we do the same thing, but localized in space, and
  with a very particular $h_n$ that we will estimate. To continue as in the
  proof of Proposition \ref{invertop}, we want to show that
  \[ \frac{h_n (. - d_n \overrightarrow{e_1})}{\| \partial_{z_1} \Psi_n
     \|_{\ast, B'_{d_n}}} \rightarrow 0 \]
  in $\mathcal{C}^0_{\tmop{loc}}$ so that we get at the limit (following the
  $+ 1$ vortex) in (\ref{crapaud})
  \[ L_{V_1} (\mathfrak{G}) = 0, \]
  using the sames techniques as in the proof of Proposition \ref{invertop}. It
  will be enough for that to show that
  \begin{equation}
    \left\| \frac{h_n}{\| \partial_{z_1} \Psi_n \|_{\ast, B'_{d_n}}}
    \right\|_{\ast \ast, B_{d_n}} \rightarrow 0 \label{e31}
  \end{equation}
  and we will also use this estimate later on. Remark that here, the problem
  is no longer symmetric in $x_1$, in particular, we cannot use the same
  argument near the $- 1$ vortex, but it is not needed.
  
\
  
  \begin{tmindent}
    Step 3.  Proof of (\ref{e31}).
  \end{tmindent}
  
\ 
  
  Recall the definition of $\| . \|_{\ast \ast, B_{d_n}}$:
  \[ \begin{array}{lll}
       \| h \|_{\ast \ast, B_{d_n}} & = & \| V h \|_{C^0 (\{ r_1 \leqslant 3
       \})}\\
       & + & \| r_1^{1 - \alpha} h_1 \|_{L^{\infty} (\{ d_n^{\varepsilon}
       \geqslant r_1 \geqslant 2 \})} + \| r_1^{2 - \alpha} h_2 \|_{L^{\infty}
       (\{ d_n^{\varepsilon} \geqslant r_1 \geqslant 2 \})} .
     \end{array} \]

  Since
  \[ d_n^{1 + \lambda_{}} \| \partial_{z_1} \Psi_{n | d = d_n \nobracket}
     \|_{\ast, B'_{d_n}} \geqslant \delta, \]
  we have
  \[ \frac{1}{\| \partial_{z_1} \Psi_n \|_{\ast, B'_{d_n}}} \leqslant
     \frac{1}{\delta c_n^{1 + \lambda_{}}}, \]
  therefore it is enough to show that
  \begin{equation}
    \| h_n \|_{\ast \ast, B_{d_n}} = o_{c_n \rightarrow 0} (c_n^{1 + \lambda})
    \label{3133001}
  \end{equation}
  to have (\ref{e31}). We recall that
  \begin{eqnarray*}
    V h_n & = & \partial_{z_1} (\tmop{TW}_{c_n} (V)) +\mathcal{L}_{\partial
    z_1} (\Psi_n) + \tmop{NL}_{\partial_{z_1}} (\Psi_n) +
    \tmop{Err}_{\tmop{loc}} + \tmop{Err}_{\tmop{far}} +
    \tmop{Err}_{\tmop{cut}} .
  \end{eqnarray*}
  The contribution of $\partial_{z_1} (\tmop{TW}_{c_n} (V))$ will be
  established in step 3.1, $\mathcal{L}_{\partial_{z_1}} (\Psi_n)$ in step
  3.2, $\tmop{NL}_{\partial_{z_1}} (\Psi_n)$ in step 3.3, and finally,
  $\tmop{Err}_{\tmop{loc}} + \tmop{Err}_{\tmop{far}} +
  \tmop{Err}_{\tmop{cut}}$ in step 3.4.
  
\
  
  \begin{tmindent}
    Step 3.1.  Proof of $\left\| \frac{\partial_{z_1} \tmop{TW}_{c_n} (V)}{V}
    \right\|_{\ast \ast, B_{d_n}} = o_{c_n \rightarrow 0} (c_n^{1 +
    \lambda})$.
  \end{tmindent}
 
 \ 
  
  Recall from (\ref{E2}) that
  \[ \tmop{TW}_c (V) = E - i c \partial_{x_2} V = - 2 \nabla V_1 . \nabla V_{-
     1} + (1 - | V_1 |^2) (1 - | V_{- 1} |^2) V_1 V_{- 1} - i c \partial_{x_2}
     V, \]
  therefore, with Lemma \ref{deryz}, we have
  \[ \partial_{z_1} (\tmop{TW}_c (V)) = - 4 \nabla V_1 . \nabla \partial_{x_1}
     V_{- 1} + 2 (1 - | V_1 |^2) V_1 \partial_{x_1} ((1 - | V_{- 1} |^2) V_{-
     1}) - 2 i c \partial_{x_2} (V_1 \partial_{x_1} V_{- 1}) . \]
  We now estimate this quantity at $d = d_n$. We have, in $\{ r_1 \leqslant
  d_n^{\varepsilon} \}$,
  \[ | (1 - | V_1 |^2) V_1 \partial_{x_1} ((1 - | V_{- 1} |^2) V_{- 1}) |
     \leqslant \frac{K}{1 + r_1^2} \times \frac{1}{d_n^3}, \]
  and using $\lambda < 1$, $\alpha > 0$, we deduce
  \[ \left\| \frac{(1 - | V_1 |^2) V_1 \partial_{x_1} ((1 - | V_{- 1} |^2)
     V_{- 1})}{V} \right\|_{\ast \ast, B_{d_n}} = o_{c_n \rightarrow 0}
     (c_n^{1 + \lambda}) . \]
  We compute with Lemmas \ref{dervor} and \ref{V-1est} that
  \[ \mathfrak{R}\mathfrak{e} \left( \frac{4 \nabla V_1 . \nabla
     \partial_{x_1} V_{- 1}}{V} \right) = 4\mathfrak{R}\mathfrak{e} \left(
     \frac{\nabla V_1}{V_1} \right) .\mathfrak{R}\mathfrak{e} \left(
     \frac{\nabla \partial_{x_1} V_{- 1}}{V_{- 1}} \right) -
     4\mathfrak{I}\mathfrak{m} \left( \frac{\nabla V_1}{V_1} \right)
     .\mathfrak{I}\mathfrak{m} \left( \frac{\nabla \partial_{x_1} V_{-
     1}}{V_{- 1}} \right), \]
  leading to
  \[ | V | \left| \mathfrak{R}\mathfrak{e} \left( \frac{4 \nabla V_1 . \nabla
     \partial_{x_1} V_{- 1}}{V} \right) \right| \leqslant \frac{K}{(1 + r_1^3)
     d_n^{3 - \varepsilon}} + \frac{K}{(1 + r_1) d_n^2} \]
  for a universal constant $K$. Since $\lambda < 1$ and $\alpha > 0$, we have
  \[ \left\| \mathfrak{R}\mathfrak{e} \left( \frac{4 \nabla V_1 . \nabla
     \partial_{x_1} V_{- 1}}{V} \right) \right\|_{\ast \ast, B_{d_n}} = o_{c_n
     \rightarrow 0} (c_n^{1 + \lambda}) . \]
  Similarly, we have, in $\{ r_1 \leqslant d_n^{\varepsilon} \}$,
  \[ | V | \left| \mathfrak{I}\mathfrak{m} \left( \frac{4 \nabla V_1 . \nabla
     \partial_{x_1} V_{- 1}}{V} \right) \right| \leqslant \frac{K}{(1 + r_1^3)
     d_n^2} + \frac{K}{(1 + r_1) d_n^{3 - \varepsilon}} . \]
  Therefore, using
  \[ \frac{1}{d_n} \leqslant \frac{K}{(1 + r_1)^{1 / \varepsilon}}, \]
  since we are in $B_{d_n} = B (d_n \overrightarrow{e_1}, d_n^{\varepsilon})$,
  and
  \[ \lambda < 2 - \varepsilon (2 - \alpha), \]
  which is one of the hypothesis of the lemma, we have
  \[ \left\| i\mathfrak{I}\mathfrak{m} \left( \frac{4 \nabla V_1 . \nabla
     \partial_{x_1} V_{- 1}}{V} \right) \right\|_{\ast \ast, B_{d_n}} = o_{c_n
     \rightarrow 0} (c_n^{1 + \lambda}) . \]

  Now, for 2$i c_n \partial_{x_2} (V_1 \partial_{x_1} V_{- 1}) = 2 i c_n
  \partial_{x_2} V_1 \partial_{x_1} V_{- 1} + 2 i c_n \partial_{x_1 x_2} V_{-
  1} V_1$, we estimate similarly (still using Lemma \ref{dervor} and
  \ref{V-1est})
  \[ \left| \mathfrak{R}\mathfrak{e} \left( \frac{i c_n \partial_{x_2} V_1
     \partial_{x_1} V_{- 1}}{V} \right) \right| \leqslant \frac{K}{(1 + r_1^3)
     d_n^{3 - \varepsilon}} + \frac{K}{(1 + r_1) d_n^4}, \]
  \[ \left| \mathfrak{I}\mathfrak{m} \left( \frac{i c_n \partial_{x_2} V_1
     \partial_{x_1} V_{- 1}}{V} \right) \right| \leqslant \frac{K}{(1 + r_1^3)
     d_n^4} + \frac{K}{(1 + r_1) d_n^{3 - \varepsilon}}, \]
  therefore, using $\frac{1}{d_n} \leqslant \frac{K}{(1 + r_1)^{1 /
  \varepsilon}}$, we have, under the condition
  \[ \lambda < 2 - \varepsilon (2 - \alpha) \]
  for the imaginary part (as for the previous term) and with no condition for
  the real part (since $\alpha > 0, \lambda < 1$), that
  \[ \left\| \frac{2 i c_n \partial_{x_2} V_1 \partial_{x_1} V_{- 1}}{V}
     \right\|_{\ast \ast, B_{d_n}} = o_{c_n \rightarrow 0} (c_n^{1 + \lambda})
     . \]
  We then compute (still using Lemma \ref{dervor} and \ref{V-1est})
  \[ | V | \left| \mathfrak{R}\mathfrak{e} \left( \frac{i c_n \partial_{x_1
     x_2} V_{- 1} V_1}{V} \right) \right| \leqslant \frac{K}{d_n^3}, \]
  \[ | V | \left| \mathfrak{I}\mathfrak{m} \left( \frac{i c_n \partial_{x_1
     x_2} V_{- 1} V_1}{V} \right) \right| \leqslant \frac{K}{d_n^{4 -
     \varepsilon}}, \]
  therefore, using $\frac{1}{d_n} \leqslant \frac{K}{(1 + r_1)^{1 /
  \varepsilon}},$ we have, under the conditions
  \[ \lambda < 2 - \varepsilon (1 - \alpha) \quad \tmop{and} \quad \lambda < 3
     - \varepsilon (3 - \alpha), \]
  which are met since
  \[ \lambda < 2 - \varepsilon (2 - \alpha) = 2 - \varepsilon (1 - \alpha) -
     \varepsilon < 2 - \varepsilon (1 - \alpha), \]
  and $\lambda < 2 - \varepsilon (2 - \alpha) = 3 - \varepsilon (3 - \alpha) -
  1 + \varepsilon < 3 - \varepsilon (3 - \alpha)$, that
  \[ \left\| \frac{i c \partial_{x_2} (V_1 \partial_{x_1 x_2} V_{- 1})}{V}
     \right\|_{\ast \ast, B_{d_n}} = o_{c_n \rightarrow 0} (c_n^{1 + \lambda})
     . \]
  This concludes the proof of step 3.1.
  
\
  
  \begin{tmindent}
    Step 3.2.  Proof of $\left\| \frac{\mathcal{L}_{\partial_{z_1}}
    (\Psi_n)}{V} \right\|_{\ast \ast, B_{d_n}} = o_{c_n \rightarrow 0} (c_n^{1
    + \lambda})$.
  \end{tmindent}
  
\
  
  We have defined
  \[ \mathcal{L}_{\partial_{z_1}} (\Psi_n) = \eta (\partial_{z_1} L) (\Phi_n)
     + (1 - \eta) V (\partial_{z_1} L') (\Psi_n) + \eta L ((\partial_{z_1} V)
     \Psi_n) + (1 - \eta) \partial_{z_1} V L' (\Psi_n) . \]
  We recall from Lemma \ref{lemma7} that
  \[ L' (\Psi_n) = - \Delta \Psi_n - 2 \frac{\nabla V}{V} . \nabla \Psi_n + 2
     | V |^2 \mathfrak{R}\mathfrak{e} (\Psi_n) - i c_n \partial_{x_2} \Psi_n,
  \]
  \[ L (\Phi_n) = - \Delta \Phi_n - (1 - | V |^2) \Phi_n +
     2\mathfrak{R}\mathfrak{e} (\bar{V} \Phi_n) V - i c_n \partial_{x_2}
     \Phi_n, \]
  hence
  \[ (\partial_{z_1} L) (\Phi_n) = 4\mathfrak{R}\mathfrak{e} (\overline{V_{-
     1}} \partial_{x_1} V_{- 1}) \Phi_n + 4\mathfrak{R}\mathfrak{e} \left(
     \overline{\partial_{x_1} V_{- 1} V_1} \Phi_n \right) V +
     4\mathfrak{R}\mathfrak{e} (\bar{V} \Phi_n) V_1 \partial_{x_1} V_{- 1} .
  \]
  We shall now estimate all these terms one by one.
  
  Since $\eta \partial_{z_1} L (\Phi_n)$ is compactly supported in $\{
  \tilde{r} \leqslant 2 \}$ and $\| . \|_{\ast \ast, B_{d_n}}$ looks at the
  function only on $\{ r_1 \leqslant d^{\varepsilon} \}$, using Lemma
  \ref{V-1est} ($\nabla V_{- 1} = O_{c \rightarrow 0} (c)$) and $\| \Psi_n
  \|_{\ast, \frac{1 - \lambda}{4}, d_n} \leqslant K (\lambda) c^{\frac{1 +
  \lambda}{2}}$, we check that
  \[ \left\| \frac{\eta \partial_{z_1} L (\Phi_n)}{V} \right\|_{\ast \ast,
     B_{d_n}} = o_{c_n \rightarrow 0} (c_n^{1 + \lambda}) . \]
  With the same arguments, we also check that
  \[ \left\| \frac{\eta L (\partial_{z_1} V \Psi_n)}{V} \right\|_{\ast \ast,
     B_{d_n}} = o_{c_n \rightarrow 0} (c_n^{1 + \lambda}) . \]
  Now, with $\| \Psi_n \|_{\ast, \sigma, d_n} \leqslant K (\sigma, \sigma')
  c_n^{1 - \sigma'}$, we check that for any $0 < \sigma < \sigma' < 1$,
  \[ | L' (\Psi_n) | \leqslant \frac{K (\sigma, \sigma')}{(1 + r_1)^{1 +
     \sigma} d_n^{1 - \sigma'}}, \]
  therefore, with Lemma \ref{V-1est}, we have
  \[ | (1 - \eta) \partial_{z_1} V L' (\Psi_n) | \leqslant \frac{K (\sigma,
     \sigma')}{(1 + r_1)^{1 + \sigma} d_n^{3 - \varepsilon - \sigma'}} . \]
  In particular, we check that if
  \[ \lambda < 2 - \varepsilon (2 - \alpha), \]
  we can take $\sigma, \sigma'$ such that $0 < \sigma < \sigma' < \frac{2 -
  \varepsilon (2 - \alpha) - \lambda}{1 - \varepsilon}$, hence
  \[ \left\| \frac{(1 - \eta) \partial_{z_1} V L' (\Psi_n)}{V} \right\|_{\ast
     \ast, B_{d_n}} = o_{c_n \rightarrow 0} (c_n^{1 + \lambda}) . \]
  Finally, we estimate
  \[ | \partial_{z_1} L' (\Psi_n) | \leqslant K \left| \partial_{x_1}
     \frac{\nabla V_{- 1}}{V_{- 1}} . \nabla \Psi_n \right| + K |
     \mathfrak{R}\mathfrak{e} (\partial_{x_1} V_{- 1} \overline{V_{- 1}})
     \mathfrak{R}\mathfrak{e} (\Psi_n) | . \]
  With Lemma \ref{V-1est} and $\| \Psi_n \|_{\ast, \sigma, d_n} \leqslant K
  (\sigma, \sigma') c^{1 - \sigma}_{n'}$ (from (\ref{1605235})), we check that
  \[ | (1 - \eta) V \partial_{z_1} L' (\Psi_n) | \leqslant \frac{K (\sigma,
     \sigma') (1 - \eta)}{r_1^{1 + \sigma} d_n^{4 - \varepsilon - \sigma'}},
  \]
  therefore, with the same condition as for the previous term, namely
  \[ \lambda < 2 - (2 - \alpha) \varepsilon, \]
  we infer, taking $\sigma < \sigma'$ small enough,
  \[ \left\| \frac{(1 - \eta) V \partial_{z_1} L' (\Psi_n)}{V} \right\|_{\ast
     \ast, B_{d_n}} = o_{c_n \rightarrow 0} (c_n^{1 + \lambda}) . \]
  This concludes the proof of step 3.2.
  
\ 
  
  \begin{tmindent}
    Step 3.3.  Proof of $\left\| \frac{\tmop{NL}_{\partial_{z_1}} (\Psi_{n
    \nosymbol})}{V} \right\|_{\ast \ast, B_{d_n}} = o_{c_n \rightarrow 0}
    (c_n^{1 + \lambda})$.
  \end{tmindent}
 
 \ 
  
  We recall
  \[ \tmop{NL}_{\partial_{z_1}} (\Psi_n) = V (1 - \eta) \partial_{z_1} (-
     \nabla \Psi_n . \nabla \Psi_n + | V |^2 S (\Psi_n)), \]
  with $S (\Psi_n) = e^{2\mathfrak{R}\mathfrak{e} (\Psi_n)} - 1 -
  2\mathfrak{R}\mathfrak{e} (\Psi_n)$. We compute
  \begin{eqnarray*}
    \partial_{z_1} (- \nabla \Psi_n . \nabla \Psi_n + | V |^2 S (\Psi_n)) & =
    & - 2 \nabla \partial_{z_1} \Psi_n . \nabla \Psi_n\\
    & + & 4\mathfrak{R}\mathfrak{e} (\partial_{x_1} V_{- 1} \overline{V_{-
    1}}) S (\Psi_n)\\
    & + & | V |^2 \partial_{z_1} S (\Psi_n) .
  \end{eqnarray*}
  Now, with Corollary \ref{cor361405} and (\ref{1605235}), we check that, for
  any $0 < \sigma < \sigma' < 1$, $r_1 \geqslant 2$,
  \[ | \nabla \partial_{z_1} \Psi_n . \nabla \Psi_n | \leqslant \frac{K
     (\sigma, \sigma')}{r_1^{2 + 2 \sigma} d_n^{2 - 2 \sigma'}}, \]
  \[ | 4\mathfrak{R}\mathfrak{e} (\partial_{x_1} V_{- 1} \overline{V_{- 1}}) S
     (\Psi_n) + | V |^2 \partial_{z_1} S (\Psi_n) | \leqslant \frac{K (\sigma,
     \sigma')}{r_1^{2 + 2 \sigma} d_n^{2 - 2 \sigma'}}, \]
  therefore, taking $\sigma < \sigma' < \frac{1 - \lambda}{2}$, we check that
  \[ \| (1 - \eta) (- 2 \nabla \partial_{z_1} \Psi_n . \nabla \Psi_n +
     4\mathfrak{R}\mathfrak{e} (\partial_{x_1} V_{- 1} \overline{V_{- 1}}) S
     (\Psi_n) + | V |^2 \partial_{z_1} S (\Psi_n)) \|_{\ast \ast, B_{d_n}} =
     o_{c_n \rightarrow 0} (c_n^{1 + \lambda}) . \]
  The proof of step 3.3 is complete.
 
 \ 
  
  \begin{tmindent}
    Step 3.4.  Proof of \ $\left\| \frac{\tmop{Err}_{\tmop{loc}} +
    \tmop{Err}_{\tmop{far}} + \tmop{Err}_{\tmop{cut}}}{V} \right\|_{\ast \ast,
    B_{d_n}} = o_{c_n \rightarrow 0} (c_n^{1 + \lambda})$.
  \end{tmindent}
  
\
  
  We recall
  \[ \tmop{Err}_{\tmop{cut}} = \partial_{z_1} \eta (L (\Phi_n) - V L' (\Psi_n)
     + i c \partial_{x_2} \Psi_n + \nabla \Psi_n . \nabla \Psi_n - | V |^2 S
     (\Psi_n)), \]
  \[ \tmop{Err}_{\tmop{loc}} = \partial_{z_1} (R (\Psi_n)) - \partial_d
     \lambda (c_n, d_n) Z_{d_n}, \]
  \[ \tmop{Err}_{\tmop{far}} = \partial_{z_1} V (1 - \eta) (- \nabla \Psi_n .
     \nabla \Psi_n + | V |^2 S (\Psi_n)) . \]
  $\tmop{Err}_{\tmop{cut}}$ is compactly supported in $\{ r_{- 1} \leqslant 2
  \}$, therefore $\tmop{Err}_{\tmop{cut}} = 0$ in $B_{d_n}$, hence
  \[ \left\| \frac{\tmop{Err}_{\tmop{cut}}}{V} \right\|_{\ast \ast, B_{d_n}} =
     0. \]
  Now, $\tmop{Err}_{\tmop{loc}}$ is supported in $\{ r_1 \leqslant 2 \}$, and
  from Lemma \ref{lemma7}, we know that $R (\Psi_n)$ is a sum of terms at
  least quadratic in $\Psi_n$ or $\Phi_n$ localized in the area where $\eta
  \neq 0$. Therefore, from Corollary \ref{cor361405} and (\ref{impro1}), we
  check that
  \[ | \partial_{z_1} (R (\Psi_n)) | \leqslant \frac{K (\sigma)}{d_n^{2 - 2
     \sigma}}, \]
  and we have check in step 1 that $| \partial_d \lambda (c_n, d_n) | = O_{c_n
  \rightarrow 0}^{\sigma} (c_n^{2 - \sigma})$. Thus, taking $\sigma < \frac{1
  - \lambda}{2}$,
  \[ \left\| \frac{\tmop{Err}_{\tmop{loc}}}{V} \right\|_{\ast \ast, B_{d_n}} =
     o_{c_n \rightarrow 0} (c_n^{1 + \lambda}) . \]
  From (\ref{1605235}), we check that, for any $1 > \sigma' > \sigma > 0$, in
  $\{ r_1 \leqslant d_n^{\varepsilon} \}$,
  \[ | - \nabla \Psi_n . \nabla \Psi_n + | V |^2 S (\Psi_n) | \leqslant
     \frac{K (\sigma, \sigma')}{(1 + r_1)^{2 + 2 \sigma} d_n^{2 - 2 \sigma'}},
  \]
  and from Lemma \ref{V-1est}, we have there
  \[ | \partial_{z_1} V | \leqslant \frac{K}{d_n^{2 - \varepsilon}}, \]
  therefore, choosing $\sigma < \sigma'$ small enough, we have
  \[ \left\| \frac{\partial_{z_1} V}{V} (1 - \eta) (- \nabla \Psi_n . \nabla
     \Psi_n + | V |^2 S (\Psi_n)) \right\|_{\ast \ast, B_{d_n}} = o_{c_n
     \rightarrow 0} (c_n^{1 + \lambda}) . \]
  This ends the proof of step 3.4 and hence of (\ref{e31}).

\  
  
  \begin{tmindent}
    Step 4.  Three additional estimates on $h_n$.
  \end{tmindent}

\  
  
  This step is devoted to the proof of the following three estimates:
  \begin{equation}
    \| V h_n \|_{L^{\infty} (\{ \tilde{r} \leqslant 3 \})} + \| \tilde{r}^{1 +
    \sigma} \mathfrak{R}\mathfrak{e} (h_n) \|_{L^{\infty} (\{ \tilde{r}
    \geqslant 2 \})} + \| \tilde{r}^{2 + \sigma} \mathfrak{I}\mathfrak{m}
    (h_n) \|_{L^{\infty} (\{ \tilde{r} \geqslant 2 \})} \leqslant K (\sigma,
    \sigma') c_n^{1 - \sigma'} \label{3130310} .
  \end{equation}
  In the right half-plane, we want to show that
  \begin{equation}
    | h_n | \leqslant \frac{K (\sigma) c_n^{1 + \sigma}}{(1 + r_1)}
    \label{3140410},
  \end{equation}
  and, in the left half-plane,
  \begin{equation}
    | h_n | \leqslant \frac{K (\sigma) c_n^{1 - \sigma}}{(1 + r_{- 1})^2} .
    \label{3173101}
  \end{equation}
  Observe that $h_n$ is not symmetrical with respect to $x_1$ because of the
  cutoff. Recall that
  \[ V h_n = \partial_{z_1} (\tmop{TW}_{c_n} (V)) +\mathcal{L}_{\partial z_1}
     (\Psi_n) + \tmop{NL}_{\partial_{z_1}} (\Psi_n) + \tmop{Err}_{\tmop{loc}}
     + \tmop{Err}_{\tmop{far}} + \tmop{Err}_{\tmop{cut}} . \]
  We complete estimates done in the previous step to show that
  (\ref{3130310}), (\ref{3140410}) and (\ref{3173101}) hold.

\
  
  \begin{tmindent}
    Step 4.1.  Estimates for $\partial_{z_1} (\tmop{TW}_{c_n} (V))$.
  \end{tmindent}

\  
  
  From Step 3.1, we have
  \[ \partial_{z_1} (\tmop{TW}_c (V)) = - 4 \nabla V_1 . \nabla \partial_{x_1}
     V_{- 1} + 2 (1 - | V_1 |^2) V_1 \partial_{x_1} ((1 - | V_{- 1} |^2) V_{-
     1}) - 2 i c \partial_{x_2} (V_1 \partial_{x_1} V_{- 1}) . \]
  In view of Lemma \ref{lemme3}, equation (\ref{230210}) and the estimate $(1
  + r_1) (1 + r_{- 1}) \geqslant d_n (1 + \tilde{r})$, we have
  \[ \left\| \frac{\partial_{z_1} (\tmop{TW}_c (V))}{V} \right\|_{\ast \ast,
     \sigma, d_n} \leqslant K (\sigma) c^{1 - \sigma}_n . \]
  Furthermore, in the left half-plane, with Lemma \ref{lemme3} and equation
  (\ref{230210}), we check easily that
  \[ | \partial_{z_1} (\tmop{TW}_c (V)) | \leqslant \frac{K c_n^{}}{(1 +
     r_1)^2} . \]
  Furthermore, in the right half-plane, we have $\frac{1}{(1 + r_{- 1})}
  \leqslant K c_n$, therefore, still using Lemma \ref{lemme3} and equation
  (\ref{230210}), we check that
  \[ | \partial_{z_1} (\tmop{TW}_c (V)) | \leqslant \frac{K c_n^2}{(1 + r_1)}
     . \]

\
  
  \begin{tmindent}
    Step 4.2.  Estimates for $\mathcal{L}_{\partial_{z_1}} (\Psi_n)$.
  \end{tmindent}
  
\
  
  We have, from Step 3.2, that
  \[ \mathcal{L}_{\partial_{z_1}} (\Psi_n) = \eta \partial_{z_1} L (\Phi_n) +
     (1 - \eta) V \partial_{z_1} L' (\Psi_n) + \eta L (\partial_{z_1} V
     \Psi_n) + (1 - \eta) \partial_{z_1} V L' (\Psi_n), \]
  with
  \[ (\partial_{z_1} L) (\Phi_n) = 4\mathfrak{R}\mathfrak{e} (\overline{V_{-
     1}} \partial_{x_1} V_{- 1}) \Phi_n + 4\mathfrak{R}\mathfrak{e} \left(
     \overline{\partial_{x_1} V_{- 1} V_1} \Phi_n \right) V +
     4\mathfrak{R}\mathfrak{e} (\bar{V} \Phi_n) V_1 \partial_{x_1} V_{- 1}, \]
  \[ L' (\Psi_n) = - \Delta \Psi_n - 2 \frac{\nabla V}{V} . \nabla \Psi_n + 2
     | V |^2 \mathfrak{R}\mathfrak{e} (\Psi_n) - i c_n \partial_{x_2} \Psi_n
  \]
  and
  \[ | \partial_{z_1} L' (\Psi_n) | \leqslant K \left| \partial_{x_1}
     \frac{\nabla V_{- 1}}{V_{- 1}} . \nabla \Psi_n \right| + K |
     \mathfrak{R}\mathfrak{e} (\partial_{x_1} V_{- 1} \overline{V_{- 1}})
     \mathfrak{R}\mathfrak{e} (\Psi_n) | . \]
  Similarly as in Step 4.1, every local term (in the area $\{ \eta \neq 0 \}$)
  satisfies the two estimates, using $\| \Psi_n \|_{\ast, \frac{1 -
  \sigma}{2}, d_n} \leqslant K (\sigma) c_n^{\sigma}$. The two nonlocal terms
  are $(1 - \eta) V \partial_{z_1} L' (\Psi_n)$ and $(1 - \eta) \partial_{z_1}
  V L' (\Psi_n)$. For the first term, in view of Lemma \ref{lemme3}, the
  previous estimate and equations (\ref{230210}), (\ref{impro1}), we check
  that
  \begin{eqnarray*}
    &  & \| V (1 - \eta) \partial_{z_1} L' (\Psi_n) \|_{L^{\infty} (\{
    \tilde{r} \leqslant 3 \})}\\
    & + & \| \tilde{r}^{1 + \sigma} \mathfrak{R}\mathfrak{e} ((1 - \eta)
    \partial_{z_1} L' (\Psi_n)) \|_{L^{\infty} (\{ \tilde{r} \geqslant 2 \})}
    + \| \tilde{r}^{2 + \sigma} \mathfrak{I}\mathfrak{m} ((1 - \eta)
    \partial_{z_1} L' (\Psi_n)) \|_{L^{\infty} (\{ \tilde{r} \geqslant 2
    \})}\\
    & \leqslant & K (\sigma, \sigma') c^{1 - \sigma'}_n
  \end{eqnarray*}
  and, in the left-half plane,
  \[ | (1 - \eta) V \partial_{z_1} L' (\Psi_n) | \leqslant \frac{K (\sigma)
     c_n^{1 - \sigma}}{(1 + r_{- 1})^2} \]
  Furthermore, using now $\| \Psi_n \|_{\ast, \frac{1 - \sigma}{2}, d_n}
  \leqslant K (\sigma) c^{\sigma}_n$, we check that, in the right half-plane,
  \[ | (1 - \eta) V \partial_{z_1} L' (\Psi_n) | \leqslant \frac{K (\sigma)
     c_n^{1 + \sigma}}{(1 + r_1)} . \]
  Finally, for the term $(1 - \eta) \partial_{z_1} V L' (\Psi_n)$, we use $\|
  \Psi_n \|_{\ast, \sigma, d_n} \leqslant K (\sigma, \sigma') c_n^{1 -
  \sigma'}$ and (\ref{impro1}) to check that
  \[ | L' (\Psi_n) | \leqslant \frac{K (\sigma) c^{1 - \sigma'}_n}{(1 +
     \tilde{r})^{1 + \sigma}} . \]
  Combining this estimate with $| \partial_{z_1} V | \leqslant \frac{K}{(1 +
  \tilde{r})}$, we show that
  \[ \left\| (1 + \tilde{r})^{2 + \sigma} \left( (1 - \eta)
     \frac{\partial_{z_1} V}{V} L' (\Psi_n) \right) \right\|_{L^{\infty}
     (\mathbbm{R}^2)} \leqslant K (\sigma, \sigma') c^{1 - \sigma'}_n, \]
  and, in the left half-plane,
  \[ | (1 - \eta) \partial_{z_1} V L' (\Psi_n) | \leqslant \frac{K (\sigma)
     c_n^{1 - \sigma}}{(1 + r_{- 1})^2} . \]
  Furtherore, using $\| \Psi_n \|_{\ast, \frac{1 - \sigma}{2}, d_n} \leqslant
  K (\sigma) c_n^{\sigma}$ and (\ref{impro1}), we also have the estimate
  \[ | L' (\Psi_n) | \leqslant \frac{K (\sigma) c^{\sigma}_n}{(1 +
     \tilde{r})}, \]
  and using $| \partial_{z_1} V | \leqslant K c_n$ in the right half-plane, we
  estimate in this same area that
  \[ | (1 - \eta) \partial_{z_1} V L' (\Psi_n) | \leqslant \frac{K (\sigma)
     c^{1 + \sigma}_n}{(1 + \tilde{r})} . \]

\ 
 
  \begin{tmindent}
    Step 4.3.  Estimates for $\tmop{NL}_{\partial_{z_1}} (\Psi_{n
    \nosymbol})$.
  \end{tmindent}
 
 \ 
  
  From Step 3.3,
  \[ \tmop{NL}_{\partial_{z_1}} (\Psi_n) = V (1 - \eta) \partial_{z_1} (-
     \nabla \Psi_n . \nabla \Psi_n + | V |^2 S (\Psi_n)) . \]
  Using equation (\ref{impro1}) for $\frac{1 + \sigma}{2}$ and Corollary
  \ref{cor361405} (also for $\frac{1 + \sigma}{2}$), we check without
  difficulties that
  \begin{eqnarray*}
    &  & \left\| \tmop{NL}_{\partial_{z_1}} (\Psi_n) \right\|_{L^{\infty} (\{
    \tilde{r} \leqslant 3 \})}\\
    & + & \left\| \tilde{r}^{1 + \sigma} \mathfrak{R}\mathfrak{e} \left(
    \tmop{NL}_{\partial_{z_1}} (\Psi_n) / V \right) \right\|_{L^{\infty} (\{
    \tilde{r} \geqslant 2 \})} + \left\| \tilde{r}^{2 + \sigma}
    \mathfrak{I}\mathfrak{m} \left( \tmop{NL}_{\partial_{z_1}} (\Psi_n) / V
    \right) \right\|_{L^{\infty} (\{ \tilde{r} \geqslant 2 \})}\\
    & \leqslant & K (\sigma) c^{1 - \sigma}_n,
  \end{eqnarray*}
  and, with, some margin, that in the left half-plane,
  \[ \left| \tmop{NL}_{\partial_{z_1}} (\Psi_n) \right| \leqslant \frac{K
     (\sigma) c_n^{1 - \sigma}}{(1 + r_{- 1})^2} . \]
  Now, using $\| \Psi_n \|_{\ast, \frac{1 - \sigma}{4}, d_n} \leqslant K
  (\sigma) c^{\frac{1 + \sigma}{2}}_n$ and Corollary \ref{cor361405} (for
  $\frac{1 - \sigma}{2}$), we have, in the right half-plane,
  \[ \left| \tmop{NL}_{\partial_{z_1}} (\Psi_n) \right| \leqslant \frac{K
     (\sigma) c^{1 + \sigma}_n}{(1 + \tilde{r})} . \]

\
  
  \begin{tmindent}
    Step 4.4.  Estimates for \ $\tmop{Err}_{\tmop{loc}} +
    \tmop{Err}_{\tmop{far}} + \tmop{Err}_{\tmop{cut}}$.
  \end{tmindent}
  
\
  
  For $\tmop{Err}_{\tmop{loc}} = \partial_{z_1} (R (\Psi_n)) - \partial_d
  \lambda (c_n, d_n) Z_{d_n}$, the same computations as in Step 4.3 yield the
  estimates (because this term is compactly supported in the area $\{ \eta
  \neq 0 \}$) needed for (\ref{3130310}) to (\ref{3173101}).
  
  For $\tmop{Err}_{\tmop{cut}} = \partial_{z_1} \eta (L (\Phi_n) - V L'
  (\Psi_n) + i c \partial_{x_2} \Psi_n + \nabla \Psi_n . \nabla \Psi_n - | V
  |^2 S (\Psi_n))$, this term is compactly supported near the vortex $- 1$,
  hence is $0$ in the right half-plane. Furthermore, using $\| \Psi_n
  \|_{\ast, \sigma / 2, d_n} \leqslant K (\sigma) c_n^{1 - \sigma}$, we check
  easily that
  \[ \| \tmop{Err}_{\tmop{cut}} / V \|_{\ast \ast, \sigma, d_n} \leqslant K
     (\sigma) c_n^{1 - \sigma} \]
  and, since it is compactly supported, in the left half-plane,
  \[ | \tmop{Err}_{\tmop{cut}} | \leqslant \frac{K (\sigma) c_n^{1 -
     \sigma}}{(1 + r_{- 1})^2} \]
  Finally, for $\tmop{Err}_{\tmop{far}} = \partial_{z_1} V (1 - \eta) (-
  \nabla \Psi_n . \nabla \Psi_n + | V |^2 S (\Psi_n))$, from (\ref{impro1}) we
  have
  \[ | (1 - \eta) (- \nabla \Psi_n . \nabla \Psi_n + | V |^2 S (\Psi_n)) |
     \leqslant \frac{K (\sigma) c_n^{1 - \sigma}}{(1 + \tilde{r})^{2 +
     \sigma}}, \]
  and we conclude as in Step 4.2.
  
  \
  
  This concludes the proof of estimates (\ref{3130310}), (\ref{3140410}) and
  (\ref{3173101}).
  
\
  
  \begin{tmindent}
    Step 5.  Inner estimates.
  \end{tmindent}
  
\
  
  By the estimation we have just proved, we have in particular
  \[ \frac{h_n (. - d_n \overrightarrow{e_1})}{\| \partial_{z_1} \Psi_n
     \|_{\ast, B'_{d_n}}} \rightarrow 0 \]
  in $\mathcal{C}^0_{\tmop{loc}}$ (which corresponds to follow the $+ 1$
  vortex). Therefore, at the limit, in the distribution sense,
  \[ L_{V_1} (\mathfrak{G}) = 0 \]
  in all $\mathbbm{R}^2$. If we show that $\langle \mathfrak{G}, \chi
  \partial_{x_1} V_1 \rangle = 0$ for $\chi$ a cutoff near $0$, we can then
  use Theorem \ref{delpino} to show, similarly as in the proof of Proposition
  \ref{invertop}, that $\mathfrak{G}= 0$ since
  \[ \left\| \frac{(V \partial_{z_1} \Psi_n) (. - d_n
     \overrightarrow{e_1})}{\| \partial_{z_1} \Psi_n \|_{\ast, B'_{d_n}}}
     \right\|_{\ast, B_{d_n}} = 1, \]
  hence $\| \mathfrak{G} \|_{H_{V_1}} < + \infty$. We recall that, by
  construction, we have $\langle \Phi_{c, d}, Z_d \rangle = 0$. By symmetry,
  this implies that $\langle \Phi_{c, d}, \eta (y) \partial_d V \rangle = 0$.
  Both $\Phi_{c, d}$ and $\eta (y) \partial_d V$ are $C^1$ with respect to
  $d$, and therefore
  \[ 0 = \partial_d \langle \Phi_{c, d}, \eta (y) \partial_d V \rangle =
     \langle \partial_d \Phi_{c, d}, \eta (y) \partial_d V \rangle + \langle
     \Phi_{c, d}, \partial_d (\eta (y) \partial_d V) \rangle . \]
  Furthermore, $\langle \partial_{x_1} \Phi_{c, d}, \eta (y) \partial_d V
  \rangle = - \langle \Phi_{c, d}, \partial_{x_1} \eta (y) \partial_d V
  \rangle$, thus
  \[ \langle \partial_{z_1} \Phi_{c, d}, \eta (y) \partial_d V \rangle = -
     \langle \Phi_{c, d}, \eta (y) \partial_{z_1} \partial_d V \rangle, \]
  and we check easily that $| \eta (y) \partial_{z_1} \partial_d V | \leqslant
  K c \eta (y)$, therefore, since $\| \Psi_{c, d} \|_{\ast, \sigma, d}
  \leqslant K (\sigma, \sigma') c^{1 - \sigma'}$, we have $| \langle
  \partial_{z_1} \Phi_{c, d}, \eta (y) \partial_d V \rangle | \leqslant K
  (\sigma, \sigma') c^{2 - \sigma'}$, and thus, taking $0 < \sigma' < 1 -
  \lambda$, for $c_n$ and $d_n$, $n \rightarrow \infty$, we infer that
  $\langle \mathfrak{G}, \eta \partial_{x_1} V_1 \rangle = 0$.
  
  We continue as in the proof of Proposition \ref{invertop}. The fact that
  $\mathfrak{G}= 0$ gives us that for any $R > 0$, we have
  \[ \frac{\| V \partial_{z_1} \Psi_n \|_{L^{\infty} (\{ r_1 \leqslant R \})}
     + \| \nabla (V \partial_{z_1} \Psi_n) \|_{L^{\infty} (\{ r_1 \leqslant R
     \})}}{\| \partial_{z_1} \Psi_n \|_{\ast, B'_{d_n}}}
     \rightarrow_{\nosymbol} 0. \]

\
  
  \begin{tmindent}
    Step 6.  Outer computations.
  \end{tmindent}
  
\
  
  We have the same outer computations as in step 2 of the proof of Proposition
  \ref{invertop}, but with $\mathcal{Y}_n = \frac{\partial_{z_1} \Psi_n}{\|
  \partial_{z_1} \Psi_n \|_{\ast, B_{d_n}}}$ playing the role of $\Psi_n$ and
  $\mathcal{H}_n = \frac{h_n}{\| \partial_{z_1} \Psi_n \|_{\ast, B_{d_n}}}$
  playing the role of $h_n$, since they satisfy the same equation. We showed
  in (\ref{e31}) that
  \[ \| \mathcal{H}_n \|_{\ast \ast, B_{d_n}} = o_{n \rightarrow \infty} (1),
  \]
  and the system of equation is, with $\mathcal{Y}_n =\mathcal{Y}_1 +
  i\mathcal{Y}_2$ and $\mathcal{H}_n =\mathcal{H}_1 + i\mathcal{H}_2$,
  \[ \left\{ \begin{array}{l}
       \Delta \mathcal{Y}_1 - 2 | V |^2 \mathcal{Y}_1 = -\mathcal{H}_1 -
       2\mathfrak{R}\mathfrak{e} \left( \frac{\nabla V}{V} . \nabla
       \mathcal{Y}_n \right) + c \partial_{x_2} \mathcal{Y}_2\\
       \Delta \mathcal{Y}_2 + c \partial_{x_2} \mathcal{Y}_1 = -\mathcal{H}_2
       - 2\mathfrak{I}\mathfrak{m} \left( \frac{\nabla V}{V} . \nabla
       \mathcal{Y}_n \right) .
     \end{array} \right. \]
  Recall the two balls $B_{d_n} = B (d_n \overrightarrow{e_1},
  d_n^{\varepsilon})$ and $B'_{d_n} = B (d_n \overrightarrow{e_1},
  d_n^{\varepsilon'})$. We have, as in the proof of Proposition
  \ref{invertop}, outside $\{ r_1 \leqslant R \}$ but in $B'_{d_n}$, that $\|
  \mathcal{Y}_n \|_{\ast, B'_{d_n}} = 1$ and $\| \mathcal{H}_n \|_{\ast \ast,
  B_{d_n}} = o_{n \rightarrow \infty} (1)$, therefore
  \begin{equation}
    | \Delta \mathcal{Y}_1 - 2\mathcal{Y}_1 | \leqslant \frac{o_{R \rightarrow
    \infty} (1) + o^R_{n \rightarrow \infty} (1)}{(1 + r_1)^{1 - \alpha}}
    \label{azz36}
  \end{equation}
  and
  \begin{equation}
    | \Delta \mathcal{Y}_2 + c \partial_{x_2} \mathcal{Y}_1 | \leqslant
    \frac{o_{R \rightarrow \infty} (1) + o^R_{n \rightarrow \infty} (1)}{(1 +
    r_1)^{2 - \alpha}} . \label{azz37}
  \end{equation}
  We want to extend these estimates in $B_{d_n} = B (d_n \overrightarrow{e_1},
  d_n^{\varepsilon})$ and not only on $B'_{d_n} = (d_n \overrightarrow{e_1},
  d_n^{\varepsilon'})$. Since $\| \mathcal{H}_n \|_{\ast \ast, B_{d_n}} = o_{n
  \rightarrow \infty} (1)$ from (\ref{e31}), the estimates on $\mathcal{H}_1$
  and $\mathcal{H}_2$ are already on $B_{d_n}$, leaving $c \partial_{x_2}
  \mathcal{Y}_2$ and the real and imaginary parts of $\frac{\nabla V}{V} .
  \nabla \mathcal{Y}_n$ to estimate.
  
  First, we check that, in $B_{d_n} \backslash B'_{d_n}$,
  \[ | c_n \partial_{x_2} \mathcal{Y}_2 | \leqslant \frac{d_n^{1 + \lambda}
     c_n^{2 - \sigma}}{(1 + r_1)^{1 + \sigma}} = \frac{o_{n \rightarrow
     \infty} (1)}{(1 + r_1)^{1 - \sigma}} \]
  taking $\sigma > 0$ small enough. We use $\mathcal{Y}_n =
  \frac{\partial_{z_1} \Psi_n}{\| \partial_{z_1} \Psi_n \|_{\ast, B'_{d_n}}}$,
  $\frac{1}{\| \partial_{z_1} \Psi_n \|_{\ast, B'_{d_n}}} \leqslant K d_n^{1 +
  \lambda}$ and Corollary \ref{cor361405} to compute, for any $1 > \sigma >
  0$,
  \[ \left| \mathfrak{R}\mathfrak{e} \left( \frac{\nabla V}{V} . \nabla
     \mathcal{Y}_n \right) \right| \leqslant \left| \frac{\nabla V}{V} \right|
     \times | \nabla \mathcal{Y}_n | \leqslant \frac{K (\sigma) d_n^{1 +
     \lambda}}{r_1^{2 + \sigma} d_n^{1 - \sigma}} \leqslant \frac{K
     (\sigma)}{r_1^{2 + \sigma} d_n^{- \sigma - \lambda}} . \]
  In $B_{d_n} \backslash B'_{d_n}$, we have $r_1 \geqslant
  d_n^{\varepsilon'}$, therefore
  \[ \left| \mathfrak{R}\mathfrak{e} \left( \frac{\nabla V}{V} . \nabla
     \mathcal{Y}_n \right) \right| \leqslant \frac{K (\sigma)}{r_1^{1 -
     \alpha} d_n^{- \sigma - \lambda + (1 + \alpha + \sigma) \varepsilon'}} .
  \]
  Since we assume
  \[ \lambda < (1 + \alpha) \varepsilon', \]
  then we can choose $\sigma > 0$ small such that $- \sigma - \lambda + (1 +
  \alpha + \sigma) \varepsilon' > 0$ and deduce, in $B_{d_n} \backslash
  B'_{d_n}$, that
  \[ \left| \mathfrak{R}\mathfrak{e} \left( \frac{\nabla V}{V} . \nabla
     \mathcal{Y}_n \right) \right| \leqslant \frac{o_{n \rightarrow \infty}
     (1)}{r_1^{1 - \alpha}} . \]
  This result shows that (\ref{azz36}) holds on $B_{d_n}$. Now, we compute
  \[ \left| \mathfrak{I}\mathfrak{m} \left( \frac{\nabla V}{V} . \nabla
     \mathcal{Y}_n \right) \right| \leqslant \left| \mathfrak{R}\mathfrak{e}
     \left( \frac{\nabla V}{V} \right) .\mathfrak{I}\mathfrak{m} (\nabla
     \mathcal{Y}_n) \right| + \left| \mathfrak{R}\mathfrak{e} (\nabla
     \mathcal{Y}_n) .\mathfrak{I}\mathfrak{m} \left( \frac{\nabla V}{V}
     \right) \right|, \]
  and with Corollary \ref{cor361405}, Lemma \ref{dervor} and \ref{V-1est}, we
  estimate
  \[ \left| \mathfrak{R}\mathfrak{e} \left( \frac{\nabla V}{V} \right)
     .\mathfrak{I}\mathfrak{m} (\nabla \mathcal{Y}_n) \right| \leqslant K
     (\sigma) \left( \frac{1}{d_n^3} + \frac{1}{r_1^3} \right) \frac{d_n^{1 +
     \lambda}}{r_1^{1 + \sigma} d_n^{1 - \sigma}} \]
  and
  \[ \left| \mathfrak{R}\mathfrak{e} (\nabla \mathcal{Y}_n)
     .\mathfrak{I}\mathfrak{m} \left( \frac{\nabla V}{V} \right) \right|
     \leqslant K (\sigma) \frac{d_n^{1 + \lambda} }{r_1^{2 + \sigma} d_n^{1 -
     \sigma}} \left( \frac{1}{d_n^{2 - \varepsilon}} + \frac{1}{r_1} \right) .
  \]
  In $B_{d_n} \backslash B'_{d_n}$, we have $d_n^{\varepsilon} \geqslant r_1
  \geqslant d_n^{\varepsilon'}$, and with similar estimates as for the
  previous term, we check that, since $\lambda < (1 + \alpha) \varepsilon'$,
  we have
  \[ \lambda < (2 + \alpha) \varepsilon', \]
  for the first term, and
  \[ \lambda < (1 + \alpha) \varepsilon' \]
  for the second one. We can find $\sigma > 0$ such that
  \[ \left| \mathfrak{I}\mathfrak{m} \left( \frac{\nabla V}{V} . \nabla
     \mathcal{Y}_n \right) \right| \leqslant \frac{o_{n \rightarrow \infty}
     (1)}{(1 + r_1)^{2 - \alpha}} \]
  in $B_{d_n} \backslash B'_{d_n}$. We deduce that (\ref{azz37}) holds on
  $B_{d_n}$. Additionally, we will use (from Lemma \ref{dderpsy}) for $0 <
  \sigma < \sigma' < 1$,
  \begin{eqnarray}
    &  & \| V \chi \mathcal{Y}_n \|_{C^1 (\{ \tilde{r} \leqslant 3 \})}
    \nonumber\\
    & + & \| \tilde{r}^{1 + \sigma} \mathfrak{R}\mathfrak{e} (\mathcal{Y}_n)
    \|_{L^{\infty} (\{ \tilde{r} \geqslant 2 \})} + \| \tilde{r}^{1 + \sigma}
    \nabla \mathfrak{R}\mathfrak{e} (\mathcal{Y}_n) \|_{L^{\infty} (\{
    \tilde{r} \geqslant 2 \})} \nonumber\\
    & + & \| \tilde{r}^{\sigma} \mathfrak{I}\mathfrak{m} (\mathcal{Y}_n)
    \|_{L^{\infty} (\{ \tilde{r} \geqslant 2 \})} + \| \tilde{r}^{1 + \sigma}
    \nabla \mathfrak{I}\mathfrak{m} (\mathcal{Y}_n) \|_{L^{\infty} (\{
    \tilde{r} \geqslant 2 \})} \nonumber\\
    & \leqslant & K (\sigma, \sigma') c_n^{1 - \sigma'} d^{1 + \lambda}_n
    \nonumber\\
    & \leqslant & K (\sigma, \sigma') d_n^{\lambda + \sigma''}
    \label{3150310} 
  \end{eqnarray}
  and from (\ref{3130310}),
  \begin{equation}
    \| V\mathcal{H}_n \|_{L^{\infty} (\{ \tilde{r} \leqslant 3 \})} + \|
    \tilde{r}^{1 + \sigma} \mathfrak{R}\mathfrak{e} (\mathcal{H}_n)
    \|_{L^{\infty} (\{ \tilde{r} \geqslant 2 \})} + \| \tilde{r}^{2 + \sigma}
    \mathfrak{I}\mathfrak{m} (\mathcal{H}_n) \|_{L^{\infty} (\{ \tilde{r}
    \geqslant 2 \})} \leqslant K (\sigma, \sigma') d_n^{\lambda + \sigma''}
    \label{3160310}
  \end{equation}
  to do estimates outside of $B_{d_n}$. These estimates are not optimal (in
  particular in the smallness in $c_n$) but we will only use them on parts far
  away from the center of $V_1$. Thanks to (\ref{3140410}), we have a slightly
  better estimate in the right half-plane, that is, for $0 < \sigma < 1$,
  \begin{equation}
    | \mathcal{H}_n | \leqslant K | h_n | d^{1 + \lambda}_n \leqslant \frac{K
    (\sigma) d_n^{\lambda - \sigma}}{(1 + r_1)} . \label{3200410}
  \end{equation}
 
\ 
  
  \begin{tmindent}
    Step 7.  Elliptic estimates.
  \end{tmindent}
 
 \ 
  
  We follow the proof of Proposition \ref{invertop}. At this point, we have on
  $\mathcal{Y}_n$ that $\| \mathcal{Y}_n \|_{\ast, B'_{d_n}} = 1$, $\|
  V\mathcal{Y}_n \|_{L^{\infty} (\{ r_{1 \nosymbol} \leqslant R \})} + \|
  \nabla (V\mathcal{Y}_n) \|_{L^{\infty} (\{ r_1 \leqslant R \})}
  \rightarrow_{\nosymbol} 0$ as $n \rightarrow \infty$ for any $R > 1$, and
  with $\mathcal{Y}_n = \mathcal{Y}_1 + i\mathcal{Y}_2$,
  \[ | \Delta \mathcal{Y}_2 + c \partial_{x_2} \mathcal{Y}_1 | \leqslant
     \frac{o_{R \rightarrow \infty} (1) + o^R_{n \rightarrow \infty} (1)}{(1 +
     r_1)^{2 - \alpha}}, \]
  \[ | \Delta \mathcal{Y}_1 - 2 | V |^2 \mathcal{Y}_1 | \leqslant \frac{o_{R
     \rightarrow \infty} (1) + o^R_{n \rightarrow \infty} (1)}{(1 + r_1)^{1 -
     \alpha}} . \]
  We want to show that $\| \mathcal{Y}_n \|_{\ast, B'_{d_n}} = o_{R
  \rightarrow \infty} (1) + o^R_{n \rightarrow \infty} (1)$. We want to use
  similar elliptic estimates as in the proof of Proposition \ref{invertop},
  but we have to show that they still work if we only have the estimate in
  $B_{d_n} = B (d_n \overrightarrow{e_1}, d_n^{\varepsilon})$ and we want the
  final estimates in $B'_{d_n} = B (d_n \overrightarrow{e_1},
  d_n^{\varepsilon'})$, with $\varepsilon' < \varepsilon$.
 
 \ 
  
  \begin{tmindent}
    Step 7.1.  Elliptic estimate for \ensuremath{\mathcal{Y}}\tmrsub{2}.
  \end{tmindent}
  
 \
  
  We start by solving the following problem in $\mathbbm{R}^2$:
  \[ \Delta \zeta = f, \]
  with
  \[ f \assign -\mathcal{H}_2 - 2\mathfrak{I}\mathfrak{m} \left( \frac{\nabla
     V}{V} . \nabla \mathcal{Y}_n \right), \]
  which is odd in $x_2$ (the derivation with respect to $z_1$ breaks the
  symmetry on $x_1$, but not on $x_2$) and satisfies
  \[ | f | \leqslant \frac{o_{R \rightarrow \infty} (1) + o^R_{n \rightarrow
     \infty} (1)}{(1 + r_1)^{2 - \alpha}} \]
  in $B_{d_n} = B (d_n \overrightarrow{e_1}, d_n^{\varepsilon})$, and, from
  (\ref{3150310}) and (\ref{3160310}),

  \begin{equation}
    | f | \leqslant \frac{K (\sigma, \sigma') d_n^{\lambda + \sigma'}}{(1 +
    \tilde{r})^{2 + \sigma}} \label{3192305}
  \end{equation}
  in $\mathbbm{R}^2$, for any $1 > \sigma' > \sigma > 0$. Similarly as in the
  proof of Lemma \ref{lapzeta}, we write, for $x \in B (d_n
  \overrightarrow{e_1}, d_n^{\varepsilon'})$,
  \begin{eqnarray}
    \nabla \zeta (x) & = & \frac{1}{2 \pi} \int_{\mathbbm{R}^2} \frac{x - Y}{|
    x - Y |^2} f (Y) d Y \label{3201705} . 
  \end{eqnarray}
  By symmetry (see in particular Lemma \ref{l381805}), we have
  \[ \int_{B (d_n \overrightarrow{e_1}, 2 | x - d_n \overrightarrow{e_1} |)} f
     (Y) d Y = 0, \]
  hence
  \begin{eqnarray*}
    &  & \int_{B (d_n \overrightarrow{e_1}, d_n^{\varepsilon})} \frac{x -
    Y}{| x - Y |^2} f (Y) d Y\\
    & = & \int_{B (d_n \overrightarrow{e_1}, d_n^{\varepsilon})} f (Y) \left(
    \frac{x - Y}{| x - Y |^2} -\tmmathbf{1}_{\{ | Y - d_n \overrightarrow{e_1}
    | \leqslant 2 | x - d_n \overrightarrow{e_1} | \}} \frac{x - d_n
    \overrightarrow{e_1}}{| x - d_n \overrightarrow{e_1} |^2} \right) d Y,
  \end{eqnarray*}
  and then, we infer
  \begin{eqnarray*}
    & \left| \frac{1}{2 \pi} \int_{B (d_n \overrightarrow{e_1},
    d_n^{\varepsilon})} f (Y) \left( \frac{x - Y}{| x - Y |^2}
    -\tmmathbf{1}_{\{ | Y - d_n \overrightarrow{e_1} | \leqslant 2 | x - d_n
    \overrightarrow{e_1} | \}} \frac{x - d_n \overrightarrow{e_1}}{| x - d_n
    \overrightarrow{e_1} |^2} \right) d Y \right| & \\
    \leqslant & \int_{B (d_n \overrightarrow{e_1}, d_n^{\varepsilon})}
    \frac{(o_{R \rightarrow \infty} (1) + o^R_{n \rightarrow \infty} (1))}{(1
    + | Y |)^{2 - \alpha}} \left| \frac{x - Y}{| x - Y |^2} -\tmmathbf{1}_{\{
    | Y - d_n \overrightarrow{e_1} | \leqslant 2 | x - d_n
    \overrightarrow{e_1} | \}} \frac{x - d_n \overrightarrow{e_1}}{| x - d_n
    \overrightarrow{e_1} |^2} \right| d Y. & 
  \end{eqnarray*}
  We do the same change of variable $Z = Y - d_n \overrightarrow{e_1}$ as in
  the proof of lemma \ref{lapzeta}, and we are now at
  \[ \begin{array}{lll}
       & \left| \frac{1}{2 \pi} \int_{B (d_n \overrightarrow{e_1},
       d_n^{\varepsilon})} f (Y) \left( \frac{x - Y}{| x - Y |^2}
       -\tmmathbf{1}_{\{ | Y - d_n \overrightarrow{e_1} | \leqslant 2 | x -
       d_n \overrightarrow{e_1} | \}} \frac{x - d_n \overrightarrow{e_1}}{| x
       - d_n \overrightarrow{e_1} |^2} \right) d Y \right| & \\
       \leqslant & \int_{B (0, d_n^{\varepsilon})} \frac{(o_{R \rightarrow
       \infty} (1) + o^R_{n \rightarrow \infty} (1))}{(1 + | Z |)^{2 -
       \alpha}} \left| \frac{x - d_n \overrightarrow{e_1} - Z}{| x - d_n
       \overrightarrow{e_1} - Z |^2} -\tmmathbf{1}_{\{ | Z | \leqslant 2 | x -
       d_n \overrightarrow{e_1} | \}} \frac{x - d_n \overrightarrow{e_1}}{| x
       - d_n \overrightarrow{e_1} |^2} \right| d Z. & 
     \end{array} \]
  We want to follow the same computations as in the proof of Lemma
  \ref{lapzeta}, but now $\frac{1}{(1 + | Z |)^{2 - \alpha}}$ is no longer
  integrable, and this is why we added the function $\tmmathbf{1}_{\{ | Z |
  \leqslant 2 | x - d_n \overrightarrow{e_1} | \}}$. If $| Z | \geqslant 2 | x
  - d_n \overrightarrow{e_1} |$, then $| x - d_n \overrightarrow{e_1} - Z |
  \geqslant | Z | / 2$ and
  \begin{eqnarray*}
    &  & \int_{B (0, d_n^{\varepsilon}) \cap \{ | Z | \geqslant 2 | x - d_n
    \overrightarrow{e_1} | \}} \frac{(o_{R \rightarrow \infty} (1) + o^R_{n
    \rightarrow \infty} (1))}{(1 + | Z |)^{2 - \alpha}} \left| \frac{x - d_n
    \overrightarrow{e_1} - Z}{| x - d_n \overrightarrow{e_1} - Z |^2} \right|
    d Z\\
    & \leqslant & \int_{B (0, d_n^{\varepsilon}) \cap \{ | Z | \geqslant 2 |
    x - d_n \overrightarrow{e_1} | \}} \frac{(o_{R \rightarrow \infty} (1) +
    o^R_{n \rightarrow \infty} (1))}{(1 + | Z |)^{2 - \alpha} | Z |} d Z\\
    & \leqslant & \frac{\nobracket o_{R \rightarrow \infty} (1) + o^R_{n
    \rightarrow \infty} (1 \nobracket)}{(1 + | x - d_n \overrightarrow{e_1}
    |)^{1 - \alpha}} .
  \end{eqnarray*}
  Then, in $\{ | Z | \leqslant 2 | x - d_n \overrightarrow{e_1} | \}$, we
  follow exactly the same computation as in the proof of the proof of Lemma
  \ref{lapzeta} for the remaining part of the integral, and we conclude that
  \begin{eqnarray*}
    &  & \left| \frac{1}{2 \pi} \int_{B (d_n \overrightarrow{e_1},
    d_n^{\varepsilon})} f (Y) \left( \frac{x - Y}{| x - Y |^2}
    -\tmmathbf{1}_{\{ | Y - d_n \overrightarrow{e_1} | \leqslant 2 | x - d_n
    \overrightarrow{e_1} | \}} \frac{x - d_n \overrightarrow{e_1}}{| x - d_n
    \overrightarrow{e_1} |^2} \right) d Y \right|\\
    & \leqslant & \frac{\nobracket o_{R \rightarrow \infty} (1) + o^R_{n
    \rightarrow \infty} (1 \nobracket)}{(1 + | x - d_n \overrightarrow{e_1}
    |)^{1 - \alpha}} .
  \end{eqnarray*}

  We are left with the estimation of (after a translation)
  \[ \int_{\mathbbm{R}^2 \backslash B (0, d_n^{\varepsilon})} \frac{| f (Z +
     d_n \overrightarrow{e_1}) |}{| Z - (x - d_n \overrightarrow{e_1}) |} d Z.
  \]
  By symmetry (see Lemma \ref{l381805}), we have
  
  \
  \[ \int_{\mathbbm{R}^2 \backslash B (0, d_n^{\varepsilon})} \frac{f (Z +
     d_n \overrightarrow{e_1})}{| Z |} d Z = 0, \]
  therefore
  \begin{eqnarray*}
    \left| \int_{\mathbbm{R}^2 \backslash B (0, d_n^{\varepsilon})} \frac{f (Z
    + d_n \overrightarrow{e_1})}{| Z - (x - d_n \overrightarrow{e_1}) |} d Z
    \right| & = & \left| \int_{\mathbbm{R}^2 \backslash B (0,
    d_n^{\varepsilon})} f (Z + d_n \overrightarrow{e_1}) \left( \frac{1}{| Z -
    (x - d_n \overrightarrow{e_1}) |} - \frac{1}{| Z |} \right) d Z \right| .
  \end{eqnarray*}
  Since $| x - d_n \overrightarrow{e_1} | \leqslant d_n^{\varepsilon'} \ll
  d_n^{\varepsilon} \leqslant | Z |$, we have, for $Z \in \mathbbm{R}^2
  \backslash B (0, d_n^{\varepsilon})$,
  \[ \left| \frac{1}{| Z - (x - d_n \overrightarrow{e_1}) |} - \frac{1}{| Z |}
     \right| \leqslant \frac{K | x - d_n \overrightarrow{e_1} |}{| Z |^2}
     \leqslant \frac{K d_n^{\varepsilon'}}{d_n^{2 \varepsilon}}, \]
  thus, with (\ref{3192305}),
  \begin{eqnarray*}
    &  & \left| \int_{\mathbbm{R}^2 \backslash B (0, d_n^{\varepsilon})} f (Z
    + d_n \overrightarrow{e_1}) \left( \frac{1}{| Z - (x - d_n
    \overrightarrow{e_1}) |} - \frac{1}{| Z |} \right) d Z \right|\\
    & \leqslant & \frac{K (\sigma, \sigma') d_n^{\varepsilon' + \lambda +
    \sigma'}}{d_n^{2 \varepsilon}} \left| \int_{\mathbbm{R}^2 \backslash B
    (d_n \overrightarrow{e_1}, d_n^{\varepsilon})} \frac{1}{(1 + \tilde{r})^{2
    + \sigma}} \right|\\
    & \leqslant & K (\sigma, \sigma') d_n^{\varepsilon' + \lambda - 2
    \varepsilon + \sigma'} .
  \end{eqnarray*}
  In particular, we have
  \[ \left| \int_{\mathbbm{R}^2 \backslash B (0, d_n^{\varepsilon})} f (Y +
     d_n \overrightarrow{e_1}) \left( \frac{1}{| Y - (x - d_n
     \overrightarrow{e_1}) |} - \frac{1}{| Y |} \right) d Y \right| \leqslant
     \frac{o_{n \rightarrow \infty} (1)}{(1 + | x - d_n \overrightarrow{e_1}
     |)^{1 - \alpha}} \]
  if, since $| x - d_n \overrightarrow{e_1} | \leqslant d_n^{\varepsilon'}$,
  \[ K (\sigma, \sigma') d_n^{\varepsilon' + \lambda - 2 \varepsilon +
     \sigma'} \leqslant \frac{o_{n \rightarrow \infty} (1)}{d_n^{\varepsilon'
     (1 - \alpha)}}, \]
  hence, since we make the assumption
  \[ \lambda + \varepsilon' (1 - \alpha) < 2 \varepsilon - \varepsilon', \]
  we can find $\sigma' > 0$ such that, for $x \in B (d_n \overrightarrow{e_1},
  d_n^{\varepsilon})$,
  \begin{equation}
    | \nabla \zeta (x) | \leqslant \frac{\nobracket o_{R \rightarrow \infty}
    (1) + o^R_{n \rightarrow \infty} (1 \nobracket)}{(1 + | x - d_n
    \overrightarrow{e_1} |)^{1 - \alpha}} . \label{aend}
  \end{equation}
  Using Lemma \ref{lapzeta} and (\ref{3192305}), we also have, in all
  $\mathbbm{R}^2$ this time, that
  \begin{equation}
    | \nabla \zeta (x) | \leqslant \frac{K (\sigma, \sigma') d_n^{\lambda +
    \sigma'}}{(1 + \tilde{r})^{1 + \sigma}} . \label{3230410}
  \end{equation}
  Here, we cannot integrate from infinity (since the estimate is only on a
  ball) to get an estimation on $\zeta$, but this will be dealt with later on.
  
  \
  
  Now, we define $\mathcal{Y}'_2 \assign \mathcal{Y}_2 - \zeta$, and we have,
  for $j \in \{ 1, 2 \}$,
  \[ \partial_{x_j} \mathcal{Y}'_2 = K_j \ast f', \]
  where
  \[ f' \assign -\mathcal{H}_1 - 2\mathfrak{R}\mathfrak{e} \left( \frac{\nabla
     V}{V} . \nabla \mathcal{Y}_n \right) - (1 - | V |^2) \mathcal{Y}_1 - c
     \partial_{x_2} \zeta . \]
  We first estimate the convolution in $B (d_n \overrightarrow{e_1},
  d_n^{\varepsilon})$. With $\| \mathcal{Y}_n \|_{\ast, B'_{d_n}} = 1$, we
  check that, with some margin in $B (d_n \overrightarrow{e_1},
  d_n^{\varepsilon'})$,
  \[ \left| 2\mathfrak{R}\mathfrak{e} \left( \frac{\nabla V}{V} . \nabla
     \mathcal{Y}_n \right) - (1 - | V |^2) \mathcal{Y}_1 \right| \leqslant
     \frac{o_{R \rightarrow \infty} (1)}{(1 + r_1)^{3 / 2 - \alpha}} . \]
  Now, we have shown in step 6 that
  \[ \left| \mathfrak{R}\mathfrak{e} \left( \frac{\nabla V}{V} . \nabla
     \mathcal{Y}_n \right) \right| \leqslant \frac{o_{n \rightarrow \infty}
     (1)}{r_1^{1 - \alpha + \sigma''}} \]
  for some $\sigma'' > 0$. In $B (d_n \overrightarrow{e_1}, d_n^{\varepsilon})
  \backslash B (d_n \overrightarrow{e_1}, d_n^{\varepsilon'})$, we have
  \[ | (1 - | V |^2) \mathcal{Y}_1 | \leqslant \frac{d_n^{\lambda +
     \sigma'}}{r_1^{3 + \sigma}} \leqslant \frac{d_n^{\lambda + \sigma' - (2 +
     \alpha - \sigma'') \varepsilon'}}{r_1^{1 - \alpha + \sigma''}} =
     \frac{o_{n \rightarrow \infty} (1)}{r_1^{1 - \alpha + \sigma''}} \]
  given that $\sigma'$ and $\sigma''$ are small enough since $\lambda - (2 +
  \alpha) \varepsilon' < 0$. Therefore, following the proof of Lemma
  \ref{P3kerest1} (only changing the integral from $\mathbbm{R}^2$ to $B (d_n
  \overrightarrow{e_1}, d_n^{\varepsilon})$), we check with the same
  computations (since we have some margin $\sigma'' > 0$ on the decay) that
  \[ \left| \int_{B (d_n \overrightarrow{e_1}, d_n^{\varepsilon})} K_j (x - Y)
     \left( 2\mathfrak{R}\mathfrak{e} \left( \frac{\nabla V}{V} . \nabla
     \mathcal{Y}_n \right) - (1 - | V |^2) \mathcal{Y}_1 \right) (Y) d Y
     \right| \leqslant \frac{o_{R \rightarrow \infty} (1)}{(1 + r_1)^{1 -
     \alpha}} . \]
  Now, using (\ref{aend}), we check that, following the proof of Lemma
  \ref{P3kerest1} (using H{\"o}lder inequality instead of Cauchy-Schwarz in
  the last estimate to make sur that the two integrals are well defined, this
  does not change the final estimate),
  \[ \left| \int_{B (d_n \overrightarrow{e_1}, d_n^{\varepsilon})} K_j (x - Y)
     (c \partial_{x_2} \zeta) (Y) d Y \right| \leqslant \frac{c (\nobracket
     o_{R \rightarrow \infty} (1) + o^R_{n \rightarrow \infty} (1
     \nobracket))}{(1 + r_1)^{1 - \alpha - 1 / 10}} . \]
  And, since $x \in B (d_n \overrightarrow{e_1}, d_n^{\varepsilon'})$, $c (1 +
  r_1)^{1 / 10} \leqslant K$, therefore
  \[ \left| \int_{B (d_n \overrightarrow{e_1}, d_n^{\varepsilon})} K_j (x -
     Y) (c \partial_{x_2} \zeta) (Y) d Y \right| \leqslant \frac{\nobracket
     o_{R \rightarrow \infty} (1) + o^R_{n \rightarrow \infty} (1
     \nobracket)}{(1 + r_1)^{1 - \alpha}} . \]
  For the last remaining term, we use (\ref{3140410}) with $\sigma =
  \frac{\lambda + 1}{2}$ to estimate
  \[ | \mathcal{H}_1 | \leqslant \frac{o^R_{n \rightarrow 0} (1)}{(1 + r_1)},
  \]
  and then, from Lemma \ref{P3kerest1} (only changing the integral from
  $\mathbbm{R}^2$ to $B (d_n \overrightarrow{e_1}, d_n^{\varepsilon})$ in the
  proof), we infer
  \[ \left| \int_{B (d_n \overrightarrow{e_1}, d_n^{\varepsilon})} K_j (x - Y)
     \mathcal{H}_1 (Y) d Y \right| \leqslant \frac{\nobracket o_{R \rightarrow
     \infty} (1) + o^R_{n \rightarrow \infty} (1 \nobracket)}{(1 + r_1)^{1 -
     \alpha}} . \]
  Combining these estimates, we have shown that
  \[ \left| \int_{B (d_n \overrightarrow{e_1}, d_n^{\varepsilon})} K_j (x - Y)
     f' (Y) d Y \right| \leqslant \frac{\nobracket o_{R \rightarrow \infty}
     (1) + o^R_{n \rightarrow \infty} (1 \nobracket)}{(1 + r_1)^{1 - \alpha}}
     . \]

  Now, we focus on the left half-plane. From (\ref{3173101}), we have
  \[ | \mathcal{H}_1 | \leqslant \frac{K (\sigma) c_n^{1 - \sigma} d^{1 +
     \lambda}_n}{(1 + r_{- 1})^2} . \]
  Furthermore, we check, using (\ref{3150310}) and (\ref{3230410}) that, in
  the left half-plane,
  \[ \left| - 2\mathfrak{R}\mathfrak{e} \left( \frac{\nabla V}{V} . \nabla
     \mathcal{Y}_n \right) - (1 - | V |^2) \mathcal{Y}_1 \right| \leqslant
     \frac{K (\sigma, \sigma') d_n^{1 + \lambda} c_n^{1 - \sigma'}}{(1 + r_{-
     1})^{2 + \sigma}} \]
  and
  \[ | c_n \partial_{x_2} \zeta | \leqslant \frac{K (\sigma, \sigma')
     d_n^{\lambda + \sigma'} c_n}{(1 + r_{- 1})^{1 + \sigma}} . \]

  We have by Theorem \ref{P3gravejatftw} (since $x \in B (d_n
  \overrightarrow{e_1}, d_n^{\varepsilon'})$) that $| K_j (x - Y) | \leqslant
  \frac{K}{d^{\beta}_n (1 + \tilde{r} (Y))^{2 - \beta}}$ for $Y$ in the left
  half-plane, for any $0 \leqslant \beta \leqslant 2$. Therefore, taking
  $\beta = 2 - \sigma$, we have
  \[ \left| \int_{\{ y_1 \leqslant 0 \}} K_j (x - Y) \mathcal{H}_1 (Y) d Y
     \right| \leqslant \int_{\mathbbm{R}^2} \frac{K (\sigma, \sigma')
     d_n^{\lambda + \sigma + \sigma' - 2}}{(1 + \tilde{r})^{2 + \sigma}}
     \leqslant \frac{K (\sigma, \sigma') d_n^{\lambda + \sigma + \sigma' - 2 +
     (1 - \alpha) \varepsilon'}}{(1 + | x - d_n \overrightarrow{e_1} |)^{1 -
     \alpha}} . \]
  Taking $\beta = 2$, we have
  \begin{eqnarray*}
    &  & \left| \int_{\{ Y_1 \leqslant 0 \}} K_j (x - Y) \left( -
    2\mathfrak{R}\mathfrak{e} \left( \frac{\nabla V}{V} . \nabla \mathcal{Y}_n
    \right) - (1 - | V |^2) \mathcal{Y}_1 \right) (Y) d Y \right|\\
    & \leqslant & \int_{\mathbbm{R}^2} \frac{K (\sigma, \sigma') d_n^{\lambda
    + \sigma' - 2}}{(1 + \tilde{r})^{2 + \sigma}}\\
    & \leqslant & \frac{K (\sigma, \sigma') d_n^{\lambda + \sigma' - 2 + (1 -
    \alpha) \varepsilon'}}{(1 + | x - d_n \overrightarrow{e_1} |)^{1 -
    \alpha}},
  \end{eqnarray*}
  and finally, taking $\beta = 1$, we estimate
  \[ \left| \int_{\{ Y_1 \leqslant 0 \}} K_j (x - Y) c_n \partial_{x_2} \zeta
     (Y) d Y \right| \leqslant \int_{\mathbbm{R}^2} \frac{K (\sigma, \sigma')
     d_n^{\lambda + \sigma' - 2}}{(1 + \tilde{r})^{2 + \sigma}} \leqslant
     \frac{K (\sigma, \sigma') d_n^{\lambda + \sigma' - 2 + (1 - \alpha)
     \varepsilon'}}{(1 + | x - d_n \overrightarrow{e_1} |)^{1 - \alpha}} . \]
  Thus, taking $\sigma' > \sigma > 0$ small enough, since $\lambda - 2 + (1 -
  \alpha) \varepsilon' < 0$, we have
  \[ \left| \int_{\{ Y_1 \leqslant 0 \}} K_j (x - Y) f' (Y) d Y \right|
     \leqslant \frac{o_{n \rightarrow 0} (1)}{(1 + | x - d_n
     \overrightarrow{e_1} |)^{1 - \alpha}} . \]

  We are left with the estimation in $\Omega \assign \{ Y_1 \geqslant 0 \}
  \backslash B (d_n \overrightarrow{e_1}, d_n^{\varepsilon})$. We infer that,
  in $\Omega$, we have, for $0 < \sigma < \sigma' < 1$
  \[ | f' | \leqslant \frac{K (\sigma') d_n^{\lambda - \sigma'}}{(1 + r_1)} +
     \frac{K (\sigma) d_n^{\lambda + \sigma}}{(1 + r_1)^{2 + \sigma}} . \]
  Indeed, from equation (\ref{3200410}) and (\ref{3230410}), we have $|
  \mathcal{H}_1 - c \partial_{x_2} \zeta | \leqslant \frac{K (\sigma)
  d_n^{\lambda + \sigma}}{(1 + r_1)}$, and using (\ref{3150310}), we check
  that
  \[ \left| 2\mathfrak{R}\mathfrak{e} \left( \frac{\nabla V}{V} . \nabla
     \mathcal{Y}_n \right) - (1 - | V |^2) \mathcal{Y}_1 \right| \leqslant
     \frac{K (\sigma) d_n^{\lambda + \sigma}}{(1 + \tilde{r})^{2 + \sigma}} .
  \]
  Now, for $y \in \Omega$, $x \in B (d_n \overrightarrow{e_1},
  d_n^{\varepsilon'})$, we have from Theorem \ref{P3gravejatftw} that
  \[ | K_j (x - Y) | \leqslant \frac{K}{d_n^{2 \varepsilon}} \]
  and
  \[ | K_j (x - Y) | \leqslant \frac{K}{(1 + \tilde{r} (Y))^{3 / 2}
     d_n^{\varepsilon / 2}} . \]
  We deduce that, for $x \in B (d_n \overrightarrow{e_1},
  d_n^{\varepsilon'})$,
  \begin{eqnarray*}
    \int_{\Omega} | K_j (x - Y) | \frac{K (\sigma') d_n^{\lambda -
    \sigma'}}{(1 + r_1 (Y))} d Y & \leqslant & K (\sigma') d_n^{\lambda -
    \sigma' - \varepsilon / 2} \int_{\mathbbm{R}^2} \frac{K}{(1 + \tilde{r}
    (Y))^{5 / 2}} d Y\\
    & \leqslant & \frac{K (\sigma') d_n^{\lambda - \sigma' + (1 - \alpha)
    \varepsilon' - \varepsilon / 2}}{(1 + | x - d_n \overrightarrow{e_1} |)^{1
    - \alpha}} = \frac{o_{n \rightarrow 0} (1)}{(1 + | x - d_n
    \overrightarrow{e_1} |)^{1 - \alpha}}
  \end{eqnarray*}
  taking $\sigma' < 1$ large enough (since $\lambda + (1 - \alpha)
  \varepsilon' - 1 - \varepsilon / 2 < 0$), and
  \[ \begin{array}{lll}
       \int_{\Omega} | K_j (x - Y) | \frac{K (\sigma) d_n^{\lambda +
       \sigma}}{(1 + \tilde{r} (Y))^{2 + \sigma}} d Y & \leqslant & K (\sigma)
       d_n^{\lambda + \sigma - 2 \varepsilon} \int_{\mathbbm{R}^2} \frac{1}{(1
       + \tilde{r} (Y))^{2 + \sigma}} d Y\\
       & \leqslant & \frac{K (\sigma) d_n^{\lambda + \sigma + (1 - \alpha)
       \varepsilon' - 2 \varepsilon}}{(1 + | x - d_n \overrightarrow{e_1}
       |)^{1 - \alpha}} = \frac{o_{n \rightarrow 0} (1)}{(1 + | x - d_n
       \overrightarrow{e_1} |)^{1 - \alpha}}
     \end{array} \]
  taking $\sigma > 0$ small enough (since $\lambda + (1 - \alpha) \varepsilon'
  - 2 \varepsilon < 0$). We deduce that, for $x \in B (d_n
  \overrightarrow{e_1}, d_n^{\varepsilon'})$,
  \[ | \partial_{x_j} \mathcal{Y}'_2 | = | K_j \ast f' | \leqslant \frac{o_{n
     \rightarrow 0} (1) + o_{R \rightarrow \infty} (1)}{(1 + | x - d_n
     \overrightarrow{e_1} |)^{1 - \alpha}} . \]
  With (\ref{aend}), we have shown that
  \[ | \partial_{x_j} \mathcal{Y}_2 | \leqslant \frac{o_{n \rightarrow 0} (1)
     + o_{R \rightarrow \infty} (1)}{(1 + | x - d_n \overrightarrow{e_1} |)^{1
     - \alpha}} . \]
  Now, since $| \mathcal{Y}_2 | + | \nabla \mathcal{Y}_2 | = o_{R \rightarrow
  \infty} (1)$ in $B (d_n \overrightarrow{e_1}, 10)$, by integration from $d_n
  \overrightarrow{e_1}$, we check that, since $\alpha > 0$,
  \[ | \mathcal{Y}_2 | \leqslant \frac{o_{n \rightarrow 0} (1) + o_{R
     \rightarrow \infty} (1)}{(1 + | x - d_n \overrightarrow{e_1} |)^{-
     \alpha}} . \]

\
  
  \begin{tmindent}
    Step 7.2.  Elliptic estimate for $\mathcal{Y}_1$.
  \end{tmindent}
  
\
  
  For $\mathcal{Y}_1$ we also use the function $K_0$ and we have
  \[ \mathcal{Y}_1 = \frac{1}{2 \pi} K_0 \left( \sqrt{2} | . | \right) \ast
     (- \Delta \mathcal{Y}_1 + 2\mathcal{Y}_1), \]
  therefore
  \begin{eqnarray*}
    | \mathcal{Y}_1 | (x) & \leqslant & \int_{\tilde{B}_{d_n} (x)} \frac{1}{2
    \pi} K_0 \left( \sqrt{2} | x - Y | \right) | (- \Delta \mathcal{Y}_1 +
    2\mathcal{Y}_1) (Y) | d Y\\
    & + & \int_{\mathbbm{R} \backslash \tilde{B}_{d_n} (x)} \frac{1}{2 \pi}
    K_0 \left( \sqrt{2} | x - Y | \right) | (- \Delta \mathcal{Y}_1 +
    2\mathcal{Y}_1) (Y) | d Y,
  \end{eqnarray*}
  where $B_{d_n} (x) = B (x - d_n \overrightarrow{e_1}, d_n^{\varepsilon})$.
  The first term can be computed as in the proof of Lemma \ref{see}, and for
  the second term, in $\mathbbm{R} \backslash B_{d_n}$, we have
  \[ K_0 \left( \sqrt{2} | x | \right) \leqslant K e^{- d_n^{\varepsilon / 2}}
     e^{- | x |^{1 / 2}} \]
  from Lemma \ref{Kzero}, which, with (\ref{3150310}) and (\ref{3160310}),
  make the term integrable and a $o_{d_n \rightarrow \infty} (e^{-
  d_n^{\varepsilon / 4}})$, which is enough to show that
  \[ | \nabla \mathcal{Y}_1 | + | \mathcal{Y}_1 | \leqslant \frac{o_{n
     \rightarrow \infty} (1) + o_{R \rightarrow \infty} (1)}{(1 + r_1)^{1 -
     \alpha}} . \]

\
  
  \begin{tmindent}
    Step 8.  Conclusion.
  \end{tmindent}

\  
  
  We conclude that there is a contradiction, as in the end of the proof of
  Proposition \ref{invertop}. This ends the proof of Proposition \ref{derg}.
\end{proof}

In the rest of this chapter, we take $\alpha, \varepsilon, \varepsilon',
\lambda$ such that they satisfy the conditions of Proposition \ref{derg}, and
\[ \lambda + (1 - \alpha) \varepsilon' > 1. \]
\subsection{Proof of $\int_{B (d \protect\overrightarrow{e_1}, d^{\varepsilon'})}
\mathfrak{R}\mathfrak{e} (\overline{\partial_d V} \partial_{z_1} (\tmop{TW}_c
(V)))_{| d = d_c \nobracket} = \frac{- \pi}{d_c^2} + o_{d_c \rightarrow
\infty} \left( \frac{1}{d_c^2} \right)$}\label{35}

From (\ref{E2}), the equation on $V$ is
\[ \tmop{TW}_c (V) = E - i c \partial_{x_2} V = - 2 \nabla V_1 . \nabla V_{-
   1} + (1 - | V_1 |^2) (1 - | V_{- 1} |^2) V_1 V_{- 1} - i c \partial_{x_2}
   (V_1 V_{- 1}) . \]
We use Lemma \ref{deryz} to compute
\[ \partial_{z_1} V = \partial_{x_1} V_1 V_{- 1} + \partial_{x_1} V_{- 1} V_1
   - (- \partial_{x_1} V_1 V_{- 1} + \partial_{x_1} V_{- 1} V_1) = 2
   \partial_{x_1} V_1 V_{- 1} . \]
Therefore
\[ \partial_{z_1} \tmop{TW}_c (V) = - 4 \nabla V_1 . \nabla \partial_{x_1}
   V_{- 1} + 2 (1 - | V_1 |^2) V_1 \partial_{x_1} ((1 - | V_{- 1} |^2) V_{-
   1}) - 2 i c \partial_{x_2} (V_1 \partial_{x_1} V_{- 1}), \]
and then
\begin{eqnarray*}
  \int_{B (d \overrightarrow{e_1}, d^{\varepsilon'})} \mathfrak{R}\mathfrak{e}
  (\overline{\partial_d V} \partial_{z_1} (\tmop{TW}_c (V))) & = & - 4 \int_{B
  (d \overrightarrow{e_1}, d^{\varepsilon'})} \mathfrak{R}\mathfrak{e}
  (\overline{\partial_d V} \nabla V_1 . \nabla \partial_{x_1} V_{- 1})\\
  & + & 2 \int_{B (d \overrightarrow{e_1}, d^{\varepsilon'})}
  \mathfrak{R}\mathfrak{e} (\overline{\partial_d V} (1 - | V_1 |^2) V_1
  \partial_{x_1} ((1 - | V_{- 1} |^2) V_{- 1}))\\
  & - & 2 \int_{B (d \overrightarrow{e_1}, d^{\varepsilon'})}
  \mathfrak{R}\mathfrak{e} (\overline{\partial_d V} i c \partial_{x_2} (V_1
  \partial_{x_1} V_{- 1})) .
\end{eqnarray*}
We want to compute this quantity at $d = d_c$. We omit the subscript and use
only $d$ in this proof. In fact, it works for any $d$ such that $\frac{1}{2 d}
\leqslant c \leqslant \frac{2}{d}$.

\

\begin{tmindent}
  Step 1.  Proof of $\int_{B (d \overrightarrow{e_1}, d^{\varepsilon'})}
  \mathfrak{R}\mathfrak{e} (\overline{\partial_d V} (1 - | V_1 |^2) V_1
  \partial_{x_1} ((1 - | V_{- 1} |^2) V_{- 1})) = o_{d \rightarrow \infty}
  \left( \frac{1}{d^2} \right)$.
\end{tmindent}

\

First remark that $\partial_{x_1} ((1 - | V_{- 1} |^2) V_{- 1}) = O_{d
\rightarrow \infty} \left( \frac{1}{d^3} \right)$ in $B (d
\overrightarrow{e_1}, d^{\varepsilon'})$ by Lemma \ref{V-1est} and
\[ (1 - | V_1 |^2) V_1 \overline{\partial_d V} = O_{r_1 \rightarrow \infty}
   \left( \frac{1}{r_1^3} \right) \]
therefore
\[ \int_{B (d \overrightarrow{e_1}, d^{\varepsilon'})}
   \mathfrak{R}\mathfrak{e} (\overline{\partial_d V} (1 - | V_1 |^2) V_1
   \partial_{x_1} ((1 - | V_{- 1} |^2) V_{- 1})) = o_{d \rightarrow \infty}
   \left( \frac{1}{d^2} \right) . \]

\

\begin{tmindent}
  Step 2.  Proof of $\int_{B (d \overrightarrow{e_1}, d^{\varepsilon'})}
  \mathfrak{R}\mathfrak{e} (\overline{\partial_d V} i c \partial_{x_2} (V_1
  \partial_{x_1} V_{- 1})) = o_{d \rightarrow \infty} \left( \frac{1}{d^2}
  \right)$.
\end{tmindent}

\

Now we compute
\[ i c \partial_{x_2} (V_1 \partial_{x_1} V_{- 1}) = i c \partial_{x_2} V_1
   \partial_{x_1} V_{- 1} + i c \partial_{x_1 x_2} V_{- 1} V_1, \]
hence the equality
\begin{eqnarray*}
  \int_{B (d \overrightarrow{e_1}, d^{\varepsilon'})} \mathfrak{R}\mathfrak{e}
  (\overline{\partial_d V} i c \partial_{x_2} (V_1 \partial_{x_1} V_{- 1})) &
  = & - c \int_{B (d \overrightarrow{e_1}, d^{\varepsilon'})}
  \mathfrak{R}\mathfrak{e} \left( \overline{\partial_{x_1} V_1 V_{- 1}} i
  \partial_{x_2} V_1 \partial_{x_1} V_{- 1} \right)\\
  & - & c \int_{B (d \overrightarrow{e_1}, d^{\varepsilon'})}
  \mathfrak{R}\mathfrak{e} \left( \overline{\partial_{x_1} V_1 V_{- 1}} i
  \partial_{x_1 x_2} V_{- 1} V_1 \right)\\
  & + & c \int_{B (d \overrightarrow{e_1}, d^{\varepsilon'})}
  \mathfrak{R}\mathfrak{e} \left( \overline{\partial_{x_1} V_{- 1} V_1} i
  \partial_{x_2} V_1 \partial_{x_1} V_{- 1} \right)\\
  & + & c \int_{B (d \overrightarrow{e_1}, d^{\varepsilon'})}
  \mathfrak{R}\mathfrak{e} \left( \overline{\partial_{x_1} V_{- 1} V_1} i
  \partial_{x_1 x_2} V_{- 1} V_1 \right) . \label{zert}
\end{eqnarray*}
Now, using Lemma \ref{V-1est}, we estimate the first term of this equality,
\[ c \left| \int_{B (d \overrightarrow{e_1}, d^{\varepsilon'})}
   \mathfrak{R}\mathfrak{e} \left( \overline{\partial_{x_1} V_1 V_{- 1}} i
   \partial_{x_2} V_1 \partial_{x_1} V_{- 1} \right) \right| \leqslant \]
\[ c \int_{B (d \overrightarrow{e_1}, d^{\varepsilon'})} \left|
   \overline{\partial_{x_1} V_1} \partial_{x_2} V_1 \right| \times |
   \overline{V_{- 1}} \partial_{x_1} V_{- 1} | \leqslant K \int_{B (d
   \overrightarrow{e_1}, d^{\varepsilon'})} \frac{1}{(1 + r_1^2)} 
   \frac{1}{d^{3 - \varepsilon'}} \leqslant \frac{K \ln
   (d^{\varepsilon'})}{d^{3 - \varepsilon'}} . \]
Since $\varepsilon' > 0$, we have
\[ c \int_{B (d \overrightarrow{e_1}, d^{\varepsilon'})}
   \mathfrak{R}\mathfrak{e} \left( \overline{\partial_{x_1} V_1 V_{- 1}} i
   \partial_{x_2} V_1 \partial_{x_1} V_{- 1} \right) = o_{d \rightarrow
   \infty} \left( \frac{1}{d^2} \right) . \label{pouet} \]
Using Lemma \ref{V-1est}, for the second term of the equality, we have
\begin{eqnarray*}
  \left| c \int_{B (d \overrightarrow{e_1}, d^{\varepsilon'})}
  \mathfrak{R}\mathfrak{e} \left( \overline{\partial_{x_1} V_1 V_{- 1}} i
  \partial_{x_1 x_2} V_{- 1} V_1 \right) \right| & \leqslant & \left| c
  \int_{B (d \overrightarrow{e_1}, d^{\varepsilon'})} \mathfrak{I}\mathfrak{m}
  \left( \overline{\partial_{x_1} V_1} V_1 \right) \mathfrak{R}\mathfrak{e}
  (\partial_{x_1 x_2} V_{- 1} \overline{V_{- 1}}) \right|\\
  & + & \left| c \int_{B (d \overrightarrow{e_1}, d^{\varepsilon'})}
  \mathfrak{R}\mathfrak{e} \left( \overline{\partial_{x_1} V_1} V_1 \right)
  \mathfrak{I}\mathfrak{m} (\partial_{x_1 x_2} V_{- 1} \overline{V_{- 1}})
  \right|\\
  & \leqslant & \int_{B (d \overrightarrow{e_1}, d^{\varepsilon'})}
  \frac{K}{(1 + r_1) d^{4 - \varepsilon'}} \leqslant \frac{K}{d^{4 - 2
  \varepsilon'}} = o_{d \rightarrow \infty} \left( \frac{1}{d^2} \right)
\end{eqnarray*}
since $c \leqslant \frac{2}{d}$ and $\varepsilon' < 1$. For the third term of
the equality, we obtain similarly
\begin{eqnarray*}
  \left| c \int_{B (d \overrightarrow{e_1}, d^{\varepsilon'})}
  \mathfrak{R}\mathfrak{e} \left( \overline{\partial_{x_1} V_{- 1} V_1} i
  \partial_{x_2} V_1 \partial_{x_1} V_{- 1} \right) \right| & \leqslant &
  \left| c \int_{B (d \overrightarrow{e_1}, d^{\varepsilon'})}
  \mathfrak{I}\mathfrak{m} (\overline{V_1} \partial_{x_2} V_1)
  \mathfrak{R}\mathfrak{e} \left( \partial_{x_1} V_{- 1}
  \overline{\partial_{x_1} V_{- 1}} \right) \right|\\
  & + & \left| c \int_{B (d \overrightarrow{e_1}, d^{\varepsilon'})}
  \mathfrak{R}\mathfrak{e} (\overline{V_1} \partial_{x_2} V_1)
  \mathfrak{I}\mathfrak{m} \left( \partial_{x_1} V_{- 1}
  \overline{\partial_{x_1} V_{- 1}} \right) \right|\\
  & \leqslant & \int_{B (d \overrightarrow{e_1}, d^{\varepsilon'})}
  \frac{K}{(1 + r_1) d^{5 - 2 \varepsilon'}} = o_{d \rightarrow \infty} \left(
  \frac{1}{d^2} \right) .
\end{eqnarray*}
Finally, for the last term of the equality,
\begin{eqnarray*}
  \left| c \int_{B (d \overrightarrow{e_1}, d^{\varepsilon'})}
  \mathfrak{R}\mathfrak{e} \left( \overline{\partial_{x_1} V_{- 1} V_1} i
  \partial_{x_1 x_2} V_{- 1} V_1 \right) \right| & \leqslant & \left| c
  \int_{B (d \overrightarrow{e_1}, d^{\varepsilon'})} \mathfrak{I}\mathfrak{m}
  (\overline{V_1} V_1) \mathfrak{R}\mathfrak{e} \left( \partial_{x_1 x_2} V_{-
  1} \overline{\partial_{x_1} V_{- 1}} \right) \right|\\
  & + & \left| c \int_{B (d \overrightarrow{e_1}, d^{\varepsilon'})}
  \mathfrak{R}\mathfrak{e} (\overline{V_1} V_1) \mathfrak{I}\mathfrak{m}
  \left( \partial_{x_1 x_2} V_{- 1} \overline{\partial_{x_1} V_{- 1}} \right)
  \right|\\
  & \leqslant & \int_{B (d \overrightarrow{e_1}, d^{\varepsilon'})}
  \frac{K}{d^{5 - \varepsilon'}} \leqslant \frac{K}{d^{5 - 3 \varepsilon'}} =
  o_{d \rightarrow \infty} \left( \frac{1}{d^2} \right) .
\end{eqnarray*}
This conclude the proof of step 2.

\

\begin{tmindent}
  Step 3.  Proof of $\int_{B (d \overrightarrow{e_1}, d^{\varepsilon'})}
  \mathfrak{R}\mathfrak{e} (\overline{\partial_d V} (- 4 \nabla V_1 . \nabla
  \partial_{x_1} V_{- 1})) = - \frac{\pi}{d^2} + o_{d \rightarrow \infty}
  \left( \frac{1}{d^2} \right)$.
\end{tmindent}

\

We have
\[ - 4 \nabla V_1 . \nabla \partial_{x_1} V_{- 1} = - 4 \partial_{x_1} V_1
   \partial_{x_1 x_1} V_{- 1} - 4 \partial_{x_2} V_1 \partial_{x_1 x_2} V_{-
   1} . \]
Remark that using $| \partial_d V | \leqslant \frac{K}{(1 + r_1)}$ and Lemma
\ref{V-1est} once again,
\[ \left| \int_{B (d \overrightarrow{e_1}, d^{\varepsilon'})}
   \mathfrak{R}\mathfrak{e} (\overline{\partial_d V} \partial_{x_1} V_1
   \partial_{x_1 x_1} V_{- 1}) \right| \leqslant \int_{B (d
   \overrightarrow{e_1}, d^{\varepsilon'})} \frac{K}{(1 + r_1^2) d^{3 -
   \varepsilon'}} = o_{d \rightarrow \infty} \left( \frac{1}{d^2} \right) . \]
Moreover,
\[ - 4 \int_{B (d \overrightarrow{e_1}, d^{\varepsilon'})}
   \mathfrak{R}\mathfrak{e} (\overline{\partial_d V} \partial_{x_2} V_1
   \partial_{x_1 x_2} V_{- 1}) = \]
\begin{equation}
  4 \int_{B (d \overrightarrow{e_1}, d^{\varepsilon'})}
  \mathfrak{R}\mathfrak{e} \left( \overline{\partial_{x_1} V_1 V_{- 1}}
  \partial_{x_2} V_1 \partial_{x_1 x_2} V_{- 1} \right) - 4 \int_{B (d
  \overrightarrow{e_1}, d^{\varepsilon'})} \mathfrak{R}\mathfrak{e} \left(
  \overline{\partial_{x_{1 \nosymbol}} V_{- 1} V_1} \partial_{x_2} V_1
  \partial_{x_1 x_2} V_{- 1} \right) . \label{pouet2}
\end{equation}
For the first integral in (\ref{pouet2}), we write
\[ 4 \int_{B (d \overrightarrow{e_1}, d^{\varepsilon'})}
   \mathfrak{R}\mathfrak{e} \left( \overline{\partial_{x_1} V_1 V_{- 1}}
   \partial_{x_2} V_1 \partial_{x_1 x_2} V_{- 1} \right) = \]
\[ 4 \int_{B (d \overrightarrow{e_1}, d^{\varepsilon'})}
   \mathfrak{R}\mathfrak{e} \left( \overline{\partial_{x_1} V_1}
   \partial_{x_2} V_1 \right) \mathfrak{R}\mathfrak{e} (\overline{V_{- 1}}
   \partial_{x_1 x_2} V_{- 1}) -\mathfrak{I}\mathfrak{m} \left(
   \overline{\partial_{x_1} V_1} \partial_{x_2} V_1 \right)
   \mathfrak{I}\mathfrak{m} (\overline{V_{- 1}} \partial_{x_1 x_2} V_{- 1}) .
\]
For the first contribution, we have
\[ \left| \int_{B (d \overrightarrow{e_1}, d^{\varepsilon'})}
   \mathfrak{R}\mathfrak{e} \left( \overline{\partial_{x_1} V_1}
   \partial_{x_2} V_1 \right) \mathfrak{R}\mathfrak{e} (\overline{V_{- 1}}
   \partial_{x_1 x_2} V_{- 1}) \right| \leqslant \int_{B (d
   \overrightarrow{e_1}, d^{\varepsilon'})} \frac{K}{(1 + r_1^2) d^{3 -
   \varepsilon'}} = o_{d \rightarrow \infty} \left( \frac{1}{d^2} \right) . \]
For the second contribution, recall from Lemma \ref{dervor} that
\[ \partial_{x_1} V_1 = \left( \cos (\theta_1) u - \frac{i}{r_1} \sin
   (\theta_1) \right) V_1 \quad \tmop{and} \quad \partial_{x_2} V_1 = \left(
   \sin (\theta_1) u + \frac{i}{r_1} \cos (\theta_1) \right) V_1, \]
therefore
\[ \mathfrak{I}\mathfrak{m} \left( \overline{\partial_{x_1} V_1}
   \partial_{x_2} V_1 \right) = \frac{u}{r_1} | V_1 |^2, \]
and then, by Lemma \ref{V-1est},
\[ - 4 \int_{B (d \overrightarrow{e_1}, d^{\varepsilon'})}
   \mathfrak{I}\mathfrak{m} \left( \overline{\partial_{x_1} V_1}
   \partial_{x_2} V_1 \right) \mathfrak{I}\mathfrak{m} (\overline{V_{- 1}}
   \partial_{x_1 x_2} V_{- 1}) = - 4 \int_{B (d \overrightarrow{e_1},
   d^{\varepsilon'})} \frac{u}{r_1}  \frac{1}{4 d^2} | V_1 |^2 d r_1 + o_{d
   \rightarrow \infty} \left( \frac{1}{d^2} \right) \]
since
\[ \int_{B (d \overrightarrow{e_1}, d^{\varepsilon'})} \frac{u}{r_1} 
   \frac{1}{4 d^{2 + 1 / 4}} | V_1 |^2 d r_1 = o_{d \rightarrow \infty} \left(
   \frac{1}{d^2} \right) . \]
We compute, using $| V_1 |^2 = \rho_1^2$, $u = \frac{\rho_1'}{\rho_1}$ and
Lemma \ref{lemme3},
\[ - 4 \int_{B (d \overrightarrow{e_1}, d^{\varepsilon'})} \frac{u}{r_1} 
   \frac{| V_1 |^2}{4 d^2} d r_1 = \frac{- 2 \pi}{d^2} 
   \int_0^{d^{\varepsilon'}} \rho_1' (r_1) \rho (r_1) d r_1 = \frac{-
   \pi}{d^2} [\rho_1^2]_0^{d^{\varepsilon'}} = \frac{- \pi}{d^2} + o_{d
   \rightarrow \infty} \left( \frac{1}{d^2} \right) . \]
We obtain the estimate for the first integral in (\ref{pouet2}):
\[ 4 \int_{B (d \overrightarrow{e_1}, d^{\varepsilon'})}
   \mathfrak{R}\mathfrak{e} \left( \overline{\partial_{x_1} V_1 V_{- 1}}
   \partial_{x_2} V_1 \partial_{x_1 x_2} V_{- 1} \right) = \frac{- \pi}{d^2} +
   o_{d \rightarrow \infty} \left( \frac{1}{d^2} \right) . \]
For the second integral in (\ref{pouet2}), we estimate
\[ \left| \int_{B (d \overrightarrow{e_1}, d^{\varepsilon'})}
   \mathfrak{R}\mathfrak{e} \left( \overline{\partial_{x_{1 \nosymbol}} V_{-
   1} V_1} \partial_{x_2} V_1 \partial_{x_1 x_2} V_{- 1} \right) \right|
   \leqslant \int_{B (d \overrightarrow{e_1}, d^{\varepsilon'})} \frac{K}{(1 +
   r_1) d^{4 - \varepsilon'}} = o_{d \rightarrow \infty} \left( \frac{1}{d^2}
   \right) . \]
This ends the proof of this subsection.

\subsection{Proof of $\partial_d \int_{B (d \protect\overrightarrow{e_1},
d^{\varepsilon'}) \cup B (- d \protect\overrightarrow{e_1}, d^{\varepsilon'})}
\mathfrak{R}\mathfrak{e} (\overline{\partial_d V} \tmop{TW}_c (Q_{c, d}))_{| d
= d_c \nobracket} = \frac{- 2 \pi}{d_c^2} + o_{d_c \rightarrow \infty} \left(
\frac{1}{d_c^2} \right)$}\label{36}

In order to prove the result of this subsection, by using (\ref{pouet3}) and
the result of subsection \ref{35} we just have to show that at $d = d_c$,
\[ \int_{B (d \overrightarrow{e_1}, \varepsilon')} \mathfrak{R}\mathfrak{e}
   (\overline{\partial_d V} L (\partial_{z_1} \Gamma_{c, d_c})) + \int_{B (d
   \overrightarrow{e_1}, d^{\varepsilon'})} \mathfrak{R}\mathfrak{e}
   (\overline{\partial_d V} (\partial_{z_1} L) (\Gamma_{c, d_c})) \]
\[ + \int_{B (d \overrightarrow{e_1}, d^{\varepsilon'})}
   \mathfrak{R}\mathfrak{e} (\overline{\partial_d V} \partial_{z_1}
   (\tmop{NL}_V (\Gamma_{c, d_c}))) = o_{d_c \rightarrow \infty} \left(
   \frac{1}{d_c^2} \right) . \]
Similarly to subsection \ref{35}, we omit the subscript on $d_c$ in the proof.

\

\begin{tmindent}
  Step 1.  Proof of $\int_{B (d \overrightarrow{e_1}, d^{3 / 4})}
  \mathfrak{R}\mathfrak{e} (\overline{\partial_d V} L (\partial_{z_1}
  \Gamma_{c, d})) = o_{d \rightarrow \infty} \left( \frac{1}{d^2} \right)$.
\end{tmindent}

\

For this term, we want to do integration by parts and use that $L (\partial_d
V)$ is very small, but since the integral is not on the whole space, there are
the two boundary terms:
\[ \left| \int_{B (d \overrightarrow{e_1}, d^{\varepsilon'})}
   \mathfrak{R}\mathfrak{e} (\overline{\partial_d V} L (\partial_{z_1}
   \Gamma_{c, d})) \right| \leqslant \left| \int_{B (d \overrightarrow{e_1},
   d^{\varepsilon'})} \mathfrak{R}\mathfrak{e} \left( L (\partial_d V)
   \overline{\partial_{z_1} \Gamma_{c, d}} \right) \right| \]
\[ + \left| \int_{\partial B (d \overrightarrow{e_1}, d^{\varepsilon'})}
   \mathfrak{R}\mathfrak{e} (\overline{\partial_d V} \nabla \partial_{z_1}
   \Gamma_{c, d}) \right| + \left| \int_{\partial B (d \overrightarrow{e_1},
   d^{\varepsilon'})} \mathfrak{R}\mathfrak{e} (\overline{\nabla \partial_d V}
   \partial_{z_1} \Gamma_{c, d}) \right|, \]
where $\partial B (d \overrightarrow{e_1}, d^{\varepsilon'})$ is the boundary
of $B (d \overrightarrow{e_1}, d^{\varepsilon'})$. On $\partial B (d
\overrightarrow{e_1}, d^{\varepsilon'})$, we have
\[ \Gamma_{c, d} = V (e^{\Psi_{c, d}} - 1), \]
hence
\begin{equation}
  \partial_{z_1} \Gamma_{c, d} = 2 V_1 \partial_{x_1} V_{- 1} (e^{\Psi_{c, d}}
  - 1) + V \partial_{z_1} \Psi_{c, d} e^{\Psi_{c, d}} \label{99}
\end{equation}
and
\begin{eqnarray}
  \nabla \partial_{z_1} \Gamma_{c, d} & = & 2 \nabla V_1 \partial_{x_1} V_{-
  1} (e^{\Psi_{c, d}} - 1) + 2 V_1 \nabla \partial_{x_1} V_{- 1} (e^{\Psi_{c,
  d}} - 1) + 2 V_1 \partial_{x_1} V_{- 1} \nabla \Psi_{c, d} e^{\Psi_{c, d}}
  \nonumber\\
  & + & \nabla V \partial_{z_1} \Psi_{c, d} e^{\Psi_{c, d}} + V \nabla
  \partial_{z_1} \Psi_{c, d} e^{\Psi_{c, d}} + V \partial_{z_1} \Psi_{c, d}
  \nabla \Psi_{c, d} e^{\Psi_{c, d}} .  \label{3252305}
\end{eqnarray}
By Lemmas \ref{dervor} and \ref{V-1est}, Proposition \ref{derg} and
(\ref{impro1}), we infer on $\partial B (d \overrightarrow{e_1},
d^{\varepsilon'})$ that, for any $1 > \sigma > 0$,
\begin{equation}
  | \partial_{z_1} \Gamma_{c, d} | \leqslant \frac{K (\sigma)}{d^{2 -
  \varepsilon'} d^{1 - \sigma}} + \frac{K}{d^{1 + \lambda - \alpha
  \varepsilon'}} . \label{326235}
\end{equation}
Thus, still on $\partial B (d \overrightarrow{e_1}, d^{\varepsilon'})$, from
Lemma \ref{ddVest} we compute

\[ \left| \int_{\partial B (d \overrightarrow{e_1}, d^{\varepsilon'})}
   \mathfrak{R}\mathfrak{e} (\overline{\nabla \partial_d V} \partial_{z_1}
   \Gamma_{c, d}) \right| \leqslant \frac{K}{d^{\varepsilon'}} \left( \frac{K
   (\sigma)}{d^{2 - \varepsilon'} d^{1 - \sigma}} + \frac{K}{d^{1 + \lambda -
   \alpha \varepsilon'}} \right) \leqslant \frac{K (\sigma)}{d^{3 - \sigma}} +
   \frac{K}{d^{1 + \lambda + (1 - \alpha) \varepsilon'}} . \]
Since $3 - \sigma > 2$ and $\lambda + (1 - \alpha) \varepsilon' > 1$ by
(\ref{3142305}), we have

\[ \left| \int_{\partial B (d \overrightarrow{e_1}, d^{\varepsilon'})}
   \mathfrak{R}\mathfrak{e} (\overline{\nabla \partial_d V} \partial_{z_1}
   \Gamma_{c, d}) \right| = o_{d \rightarrow \infty} \left( \frac{1}{d^2}
   \right) . \]
For (\ref{3252305}), we estimate on $\partial B (d \overrightarrow{e_1},
d^{\varepsilon'})$, still using Lemmas \ref{dervor} and \ref{V-1est},
Proposition \ref{derg} and (\ref{impro1}), for any $1 > \sigma > 0$,
\[ | 2 \nabla V_1 \partial_{x_1} V_{- 1} (e^{\Psi_{c, d}} - 1) + 2 V_1 \nabla
   \partial_{x_1} V_{- 1} (e^{\Psi_{c, d}} - 1) + 2 V_1 \partial_{x_1} V_{- 1}
   \nabla \Psi_{c, d} e^{\Psi_{c, d}} | \leqslant \frac{K (\sigma)}{d^{3 -
   \sigma}}, \]
and
\[ | \nabla V \partial_{z_1} \Psi_{c, d} e^{\Psi_{c, d}} + V \nabla
   \partial_{z_1} \Psi_{c, d} e^{\Psi_{c, d}} + V \partial_{z_1} \Psi_{c, d}
   \nabla \Psi_{c, d} e^{\Psi_{c, d}} | \leqslant \frac{K}{d^{1 + \lambda + (1
   - \alpha) \varepsilon'}} + \frac{K (\sigma)}{e^{2 + \lambda + (1 - \alpha)
   \varepsilon' - \sigma}} . \]
In particular, from (\ref{3252305}), we can find $1 > \sigma > 0$ such that,
on $\partial B (d \overrightarrow{e_1}, d^{\varepsilon'})$,
\[ | \nabla \partial_{z_1} \Gamma_{c, d} | = o_{d \rightarrow \infty} \left(
   \frac{1}{d^2} \right), \]
thus
\[ \left| \int_{\partial B (d \overrightarrow{e_1}, d^{\varepsilon'})}
   \mathfrak{R}\mathfrak{e} (\overline{\partial_d V} \nabla \partial_{z_1}
   \Gamma_{c, d}) \right| = o_{d \rightarrow \infty} \left( \frac{1}{d^2}
   \right) . \]
From (\ref{bleu2}), we know that
\[ | L (\partial_d V) | \leqslant \frac{K}{(1 + \tilde{r}^2) d} . \]
Moreover, by Proposition \ref{derg}, we have $| \partial_{z_1} \Gamma_{c, d} |
\leqslant \frac{K}{d^{1 + \lambda - \alpha \varepsilon'}}$ in $B (d
\overrightarrow{e_1}, d^{\varepsilon'})$, which is enough to show that
\[ \int_{B (d \overrightarrow{e_1}, d^{\varepsilon'})}
   \mathfrak{R}\mathfrak{e} (L (\overline{\partial_d V}) \partial_{z_1}
   \Gamma_{c, d}) = o_{d \rightarrow \infty} \left( \frac{1}{d^2} \right) . \]

\

\begin{tmindent}
  Step 2.  Proof of $\int_{B (d \overrightarrow{e_1}, d^{\varepsilon'})}
  \mathfrak{R}\mathfrak{e} (\overline{\partial_d V} (\partial_{z_1} L)
  (\Gamma_{c, d})) = o_{d \rightarrow \infty} \left( \frac{1}{d^2} \right)$.
\end{tmindent}

\

We have
\[ (\partial_{z_1} L) (\Gamma_{c, d}) = 4\mathfrak{R}\mathfrak{e}
   (\overline{V_{- 1}} \partial_{x_1} V_{- 1}) \Gamma_{c, d} +
   4\mathfrak{R}\mathfrak{e} \left( \overline{\partial_{x_1} V_{- 1} V_1}
   \Gamma_{c, d} \right) V + 4\mathfrak{R}\mathfrak{e} (\bar{V} \Gamma_{c, d})
   V_1 \partial_{x_1} V_{- 1}, \]
thus
\begin{eqnarray}
  \int_{B (d \overrightarrow{e_1}, d^{\varepsilon'})} \mathfrak{R}\mathfrak{e}
  (\overline{\partial_d V} (\partial_{z_1} L) (\Gamma_{c, d})) & = & 4 \int_{B
  (d \overrightarrow{e_1}, d^{\varepsilon'})} \mathfrak{R}\mathfrak{e}
  (\overline{\partial_d V} \Gamma_{c, d}) \mathfrak{R}\mathfrak{e}
  (\overline{V_{- 1}} \partial_{x_1} V_{- 1}) \nonumber\\
  & + & 4 \int_{B (d \overrightarrow{e_1}, d^{\varepsilon'})}
  \mathfrak{R}\mathfrak{e} \left( \overline{\partial_{x_1} V_{- 1} V_1}
  \Gamma_{c, d} \right) \mathfrak{R}\mathfrak{e} (\overline{\partial_d V} V)
  \nonumber\\
  & + & 4 \int_{B (d \overrightarrow{e_1}, d^{\varepsilon'})}
  \mathfrak{R}\mathfrak{e} (\overline{\partial_d V} V_1 \partial_{x_1} V_{-
  1}) \mathfrak{R}\mathfrak{e} (\bar{V} \Gamma_{c, d}) .  \label{pouet4}
\end{eqnarray}
Using $| \partial_d V | \leqslant \frac{K}{1 + r_1}$,
\[ \mathfrak{R}\mathfrak{e} (\overline{V_{- 1}} \partial_{x_1} V_{- 1}) = O_{d
   \rightarrow \infty} \left( \frac{1}{d^3} \right) \quad \tmop{and} \quad |
   \Gamma_{c, d} | \leqslant \frac{K}{(1 + r_1)^{1 / 2} d^{1 / 2}} \]
from Lemma \ref{ddVest}, Lemma \ref{dervor} and (\ref{gammaetoile})
respectively, we may bound
\[ \left| \int_{B (d \overrightarrow{e_1}, d^{\varepsilon'})}
   \mathfrak{R}\mathfrak{e} (\overline{\partial_d V} \Gamma_{c, d})
   \mathfrak{R}\mathfrak{e} (\overline{V_{- 1}} \partial_{x_1} V_{- 1})
   \right| \leqslant \int_{B (d \overrightarrow{e_1}, d^{\varepsilon'})}
   \frac{K}{(1 + r_1^{\nosymbol})^{1 + 1 / 2} d^{3 + 1 / 2}} = o_{d
   \rightarrow \infty} \left( \frac{1}{d^2} \right) . \]
The second term of (\ref{pouet4}) is
\[ 4 \int_{B (d \overrightarrow{e_1}, d^{\varepsilon'})}
   \mathfrak{R}\mathfrak{e} \left( \overline{\partial_{x_1} V_{- 1} V_1}
   \Gamma_{c, d} \right) \mathfrak{R}\mathfrak{e} (\overline{\partial_d V} V)
   . \]
We compute that
\[ \left| \mathfrak{R}\mathfrak{e} \left( \overline{\partial_{x_1} V_{- 1}
   V_1} \Gamma_{c, d} \right) \right| \leqslant \frac{K}{(1 + r_1)^{1 / 8}
   d^{17 / 8}} \quad \tmop{and} \quad | \mathfrak{R}\mathfrak{e}
   (\overline{\partial_d V} V) | \leqslant \frac{K}{(1 + r_1)^3} \]
in $B (d \overrightarrow{e_1}, d^{\varepsilon'})$ using
\[ | \Gamma_{c, d} | \leqslant \frac{K}{(1 + r_1)^{1 / 8} d^{7 / 8}} \]
by (\ref{impro1}) and the definition of $\Gamma_{c, d}$. Therefore, since $17
/ 8 > 2$,
\[ \int_{B (d \overrightarrow{e_1}, d^{\varepsilon'})}
   4\mathfrak{R}\mathfrak{e} \left( \overline{\partial_{x_1} V_{- 1} V_1}
   \Gamma_{c, d} \right) \mathfrak{R}\mathfrak{e} (\overline{\partial_d V} V)
   = o_{d \rightarrow \infty} \left( \frac{1}{d^2} \right) . \]
The last term of (\ref{pouet4}) is
\[ \int_{B (d \overrightarrow{e_1}, d^{\varepsilon'})}
   4\mathfrak{R}\mathfrak{e} (\bar{V} \Gamma_{c, d}) \mathfrak{R}\mathfrak{e}
   (V_1 \partial_{x_1} V_{- 1} \overline{\partial_d V}) . \]
Recalling that
\[ | \mathfrak{R}\mathfrak{e} (\bar{V} \Gamma_{c, d}) | \leqslant K |
   \mathfrak{R}\mathfrak{e} (\Psi) | \leqslant \frac{K}{(1 +
   r^{\nosymbol}_1)^{1 + 1 / 8} d^{7 / 8}} \]
and
\[ | \mathfrak{R}\mathfrak{e} (V_1 \partial_{x_1} V_{- 1} \overline{\partial_d
   V}) | \leqslant \frac{K}{d^{5 / 4} (1 + r_1)}, \]
we deduce
\[ \int_{B (d \overrightarrow{e_1}, d^{\varepsilon'})}
   4\mathfrak{R}\mathfrak{e} (\bar{V} \Gamma_{c, d}) \mathfrak{R}\mathfrak{e}
   (V_1 \partial_{x_1} V_{- 1} \overline{\partial_d V}) = o_{d \rightarrow
   \infty} \left( \frac{1}{d^2} \right) . \]

\begin{tmindent}
  Step 3.  Proof of $\int_{B (d \overrightarrow{e_1}, d^{\varepsilon'})}
  \mathfrak{R}\mathfrak{e} (\overline{\partial_d V} \partial_{z_1}
  (\tmop{NL}_V (\Gamma_{c, d}))) = o_{d \rightarrow \infty} \left(
  \frac{1}{d^2} \right)$.
\end{tmindent}

Recall that
\begin{eqnarray*}
  \partial_{z_1} \tmop{NL}_V (\Gamma_{c, d \nosymbol \nosymbol}) & = &
  4\mathfrak{R}\mathfrak{e} \left( \overline{\partial_{x_1} V_{- 1} V_1}
  \Gamma_{c, d \nosymbol} \right) \Gamma_{c, d \nosymbol} +
  2\mathfrak{R}\mathfrak{e} (\bar{V} \partial_{z_1} \Gamma_{c, d \nosymbol})
  \Gamma_{c, d \nosymbol} + 2\mathfrak{R}\mathfrak{e} (\bar{V} \Gamma_{c, d
  \nosymbol}) \partial_{z_1} \Gamma_{\nosymbol c, d}\\
  &  & + 2\mathfrak{R}\mathfrak{e} (\overline{\Gamma_{c, d \nosymbol}}
  \partial_{z_1} \Gamma_{c, d \nosymbol}) (V + \Gamma_{\nosymbol c, d}) + |
  \Gamma_{c, d \nosymbol} |^2 (2 \partial_{x_1} V_{- 1} V_1 + \partial_{z_1}
  \Gamma_{\nosymbol c, d}) .
\end{eqnarray*}
We write
\[ \int_{B (d \overrightarrow{e_1}, d^{\varepsilon'})}
   \mathfrak{R}\mathfrak{e} (\overline{\partial_d V} \partial_{z_1}
   (\tmop{NL}_V (\Gamma_{c, d}))) = I_1 + I_2 + I_3 + I_4 + I_5, \]
with
\[ I_1 = \int_{B (d \overrightarrow{e_1}, d^{\varepsilon'})}
   4\mathfrak{R}\mathfrak{e} (\overline{\partial_d V} \Gamma_{c, d})
   \mathfrak{R}\mathfrak{e} \left( \overline{\partial_{x_1} V_{- 1} V_1}
   \Gamma_{c, d \nosymbol} \right), \]
\[ I_2 = \int_{B (d \overrightarrow{e_1}, d^{\varepsilon'})}
   2\mathfrak{R}\mathfrak{e} (\overline{\partial_d V} \Gamma_{c, d})
   \mathfrak{R}\mathfrak{e} (\bar{V} \partial_{z_1} \Gamma_{c, d \nosymbol}),
\]
\[ I_3 = \int_{B (d \overrightarrow{e_1}, d^{\varepsilon'})}
   2\mathfrak{R}\mathfrak{e} (\bar{V} \Gamma_{\nosymbol c, d})
   \mathfrak{R}\mathfrak{e} (\overline{\partial_d V} \partial_{z_1} \Gamma_{c,
   d}), \]
\[ I_4 = \int_{B (d \overrightarrow{e_1}, d^{\varepsilon'})}
   2\mathfrak{R}\mathfrak{e} (\overline{\Gamma_{c, d \nosymbol}}
   \partial_{z_1} \Gamma_{c, d \nosymbol}) \mathfrak{R}\mathfrak{e}
   (\overline{\partial_d V} V) + 2\mathfrak{R}\mathfrak{e}
   (\overline{\Gamma_{c, d \nosymbol}} \partial_{z_1} \Gamma_{c, d \nosymbol})
   \mathfrak{R}\mathfrak{e} (\overline{\partial_d V} \Gamma_{c, d \nosymbol}),
\]
\[ I_5 = \int_{B (d \overrightarrow{e_1}, d^{\varepsilon'})} | \Gamma_{c, d}
   |^2 \mathfrak{R}\mathfrak{e} (\overline{\partial_d V} (2 \partial_{x_1}
   V_{- 1} V_1 + \partial_{z_1} \Gamma_{c, d \nosymbol})) . \]

{\tmem{\quad Estimate for $I_1$.}}

\

We estimate, by using $| \Gamma_{c, d} | \leqslant \frac{K}{(1 + r_1)^{9 /
16} d^{7 / 16}}$ that
\[ \left| \mathfrak{R}\mathfrak{e} (\overline{\partial_d V} \Gamma_{c, d})
   \mathfrak{R}\mathfrak{e} \left( \overline{\partial_{x_1} V_{- 1} V_1}
   \Gamma_{c, d \nosymbol} \right) \right| \leqslant | \Gamma_{c, d} |^2
   \frac{K}{(1 + r_1) d^{5 / 4}} \leqslant \frac{K}{(1 + r_1)^{2 + 1 / 8}
   d^{17 / 8}} \]
Then, since $17 / 8 > 2$,
\[ \int_{B (d \overrightarrow{e_1}, d^{\varepsilon'})}
   4\mathfrak{R}\mathfrak{e} (\overline{\partial_d V} \Gamma_{c, d})
   \mathfrak{R}\mathfrak{e} \left( \overline{\partial_{x_1} V_{- 1} V_1}
   \Gamma_{\nosymbol c, d} \right) = o_{d \rightarrow \infty} \left(
   \frac{1}{d^2} \right) . \]

{\tmem{\quad Estimate for $I_2$.}}

\

From (\ref{99}), we have
\[ \partial_{z_1} \Gamma_{c, d} = 2 V_1 \partial_{x_1} V_{- 1} (e^{\Psi_{c,
   d}} - 1) + V \partial_{z_1} \Psi_{c, d} e^{\Psi_{c, d}}, \]
therefore, on $B (d \overrightarrow{e_1}, d^{\varepsilon'})$, by Lemma
\ref{V-1est}, Proposition \ref{derg} and (\ref{1605235}), for any $1 > \sigma
> 0$,
\[ | \mathfrak{R}\mathfrak{e} (\bar{V} \partial_{z_1} \Gamma_{\nosymbol c, d})
   | \leqslant \frac{K (\sigma)}{d^{3 - \varepsilon' - \sigma}} + \frac{K}{(1
   + r_1)^{1 - \alpha} d^{1 + \lambda}} + \frac{K (\sigma)}{d^{2 + \lambda -
   \sigma} (1 + r_1)^{- \alpha}} . \]
Combining this with
\[ | \mathfrak{R}\mathfrak{e} (\overline{\partial_d V} \Gamma_{c, d
   \nosymbol}) | \leqslant \frac{K (\sigma)}{(1 + r_1)^{} d^{1 - \sigma}} \]
since $| \Gamma_{c, d} | \leqslant \frac{K (\sigma)}{d^{1 - \sigma}}$, we
infer
\begin{eqnarray*}
  \left| \int_{B (d \overrightarrow{e_1}, d^{\varepsilon'})}
  2\mathfrak{R}\mathfrak{e} (\overline{\partial_d V} \Gamma_{c, d \nosymbol})
  \mathfrak{R}\mathfrak{e} (\bar{V} \partial_{z_1} \Gamma_{c, d \nosymbol})
  \right| & \leqslant & \left| \int_{B (d \overrightarrow{e_1},
  d^{\varepsilon'})} \frac{K (\sigma)}{(1 + r_1)^{} d^{4 - \varepsilon' - 2
  \sigma}} \right|\\
  & + & \left| \int_{B (d \overrightarrow{e_1}, d^{\varepsilon'})} \frac{K
  (\sigma)}{(1 + r_1)^{2 - \alpha} d^{2 + \lambda - \sigma}} \right|\\
  & + & \left| \int_{B (d \overrightarrow{e_1}, d^{\varepsilon'})} \frac{K
  (\sigma)}{(1 + r_1)^{1 - \alpha} d^{3 + \lambda - 2 \sigma} } \right|,
\end{eqnarray*}
and since $\lambda + (1 - \alpha) \varepsilon' > 1$, we conclude, taking
$\sigma > 0$ small enough,
\[ \int_{B (d \overrightarrow{e_1}, d^{\varepsilon'})}
   2\mathfrak{R}\mathfrak{e} (\overline{\partial_d V} \Gamma_{c, d \nosymbol})
   \mathfrak{R}\mathfrak{e} (\bar{V} \partial_{z_1} \Gamma_{c, d \nosymbol}) =
   o_{d \rightarrow \infty} \left( \frac{1}{d^2} \right) . \]

{\tmem{\quad Estimate for $I_3$.}}

\

We have from (\ref{99}) that
\[ \partial_{z_1} \Gamma_{c, d} = 2 V_1 \partial_{x_1} V_{- 1} (e^{\Psi_{c,
   d}} - 1) + V \partial_{z_1} \Psi_{c, d} e^{\Psi_{c, d}}, \]
therefore
\[ | \mathfrak{R}\mathfrak{e} (\overline{\partial_d V} \partial_{z_1}
   \Gamma_{c, d}) | \leqslant K \left( \frac{1}{(1 + r_1) d^{3 - 2
   \varepsilon'}} + \frac{1}{(1 + r_1)^{1 - \alpha} d^{1 + \lambda}} \right),
\]
and $| \mathfrak{R}\mathfrak{e} (\bar{V} \Gamma_{\nosymbol c, d}) | \leqslant
K | \mathfrak{R}\mathfrak{e} (\Psi_{c, d}) |$, hence
\[ | \mathfrak{R}\mathfrak{e} (\bar{V} \Gamma_{c, d \nosymbol}) | \leqslant
   \frac{K (\sigma)}{(1 + r_1)^{1 + \sigma} d^{1 - \sigma}}, \]
then
\begin{eqnarray*}
  \left| \int_{B (d \overrightarrow{e_1}, d^{\varepsilon'})}
  2\mathfrak{R}\mathfrak{e} (\overline{\partial_d V} \partial_{z_1} \Gamma_{c,
  d}) \mathfrak{R}\mathfrak{e} (\bar{V} \Gamma_{c, d \nosymbol}) \right| &
  \leqslant & \int_{B (d \overrightarrow{e_1}, d^{\varepsilon'})} \frac{K
  (\sigma)}{(1 + r_1)^{2 + \sigma} d^{4 - 2 \varepsilon' - \sigma}}\\
  & + & \int_{B (d \overrightarrow{e_1}, d^{\varepsilon'})} \frac{K
  (\sigma)}{(1 + r_1)^{2 + \sigma - \alpha} d^{2 + \lambda - \sigma}}\\
  & = & o_{d \rightarrow \infty} \left( \frac{1}{d^2} \right)
\end{eqnarray*}
by taking $\sigma > 0$ small enough and using $\lambda + (1 - \alpha)
\varepsilon' > 1$.

\

{\tmem{\quad Estimate for $I_4$.}}

\

Recall that
\[ | \mathfrak{R}\mathfrak{e} (\overline{\partial_d V} V) | \leqslant
   \frac{K}{(1 + r_1)^3}, \]
and we have
\[ | \mathfrak{R}\mathfrak{e} (\overline{\partial_d V} \nobracket \Gamma_{c,
   d}) \nobracket | \leqslant \frac{K}{(1 + r_1)^{1 + 6 / 8} d^{2 / 8}} \]
since $| \Gamma_{c, d} | \leqslant \frac{K}{(1 + r_1)^{1 + 6 / 8} d^{2 / 8}}$.
Therefore, with $\frac{1}{d} \leqslant \frac{K}{(1 + r_1)}$,
\[ | \mathfrak{R}\mathfrak{e} (\overline{\partial_d V} \nobracket V)
   +\mathfrak{R}\mathfrak{e} (\overline{\partial_d V} \nobracket \Gamma_{c,
   d}) \nobracket \nobracket | \leqslant \frac{K}{(1 + r_1)^2} \]
Now, we use $| \Gamma_{c, d} | \leqslant \frac{K (\sigma)}{(1 + r_1)^{\sigma}
d^{1 - \sigma}}$ and Proposition \ref{derg} to get
\[ | \mathfrak{R}\mathfrak{e} (\overline{\Gamma_{c, d \nosymbol}}
   \partial_{z_1} \Gamma_{\nosymbol c, d}) | \leqslant \frac{K}{(1 +
   r_1)^{\sigma - \alpha} d^{2 + \lambda - \sigma'}} . \]
We conclude as for the previous estimates,
\[ \int_{B (d \overrightarrow{e_1}, d^{\varepsilon'})} 2
   (\mathfrak{R}\mathfrak{e} (\overline{\partial_d V} V)
   +\mathfrak{R}\mathfrak{e} (\overline{\partial_d V} \Gamma_{c, d}))
   \mathfrak{R}\mathfrak{e} (\overline{\Gamma_{c, d \nosymbol}} \partial_{z_1}
   \Gamma_{c, d \nosymbol}) = o_{d \rightarrow \infty} \left( \frac{1}{d^2}
   \right) . \]

{\tmem{\quad Estimate for $I_5$.}}

\

We have, by Proposition \ref{derg},
\[ | \mathfrak{R}\mathfrak{e} (\overline{\partial_d V} (\partial_{x_1} V_{- 1}
   V_1 + \partial_{z_1} \Gamma_{\nosymbol c, d})) | \leqslant \frac{K}{(1 +
   r_1)} \left( \frac{1}{d^{2 - \varepsilon'}} + \frac{1}{(1 + r_1)^{1 -
   \alpha} d^{2 + \lambda - \sigma}} \right) \]
and using $| \Gamma_{c, d} | \leqslant \frac{K (\sigma)}{d^{1 - \sigma}}$, we
have
\[ | \Gamma_{c, d} |^2 \leqslant \frac{K (\sigma)}{d^{2 - 2 \sigma}} . \]
Therefore, for $\sigma > 0$ small enough, since $\lambda + (1 - \alpha)
\varepsilon' > 1$,
\[ \int_{B (d \overrightarrow{e_1}, d^{\varepsilon'})} | \Gamma_{c, d} |^2
   \mathfrak{R}\mathfrak{e} (\overline{\partial_d V} (2 \partial_{x_1} V_{- 1}
   V_1 + \partial_{z_1} \Gamma_{\nosymbol c, d})) = o_{d \rightarrow \infty}
   \left( \frac{1}{d^2} \right) \]
which concludes the estimates.

\subsection{Proof of $\partial_c d_c = - \frac{1}{c^2} + o_{c \rightarrow 0}
\left( \frac{1}{c^2} \right)$}\label{37}

Recall that $d_c$ is defined by the implicit equation
\[ \int_{B (d \overrightarrow{e_1}, d^{\varepsilon'}) \cup B (- d
   \overrightarrow{e_1}, d^{\varepsilon'})} \mathfrak{R}\mathfrak{e}
   (\overline{\partial_d V} \tmop{TW}_c (Q_{c, d})) = 0. \]
We showed in subsection \ref{36} that
\[ \partial_d \int_{B (d \overrightarrow{e_1}, d^{\varepsilon'}) \cup B (- d
   \overrightarrow{e_1}, d^{\varepsilon'})} \mathfrak{R}\mathfrak{e}
   (\overline{\partial_d V} \tmop{TW}_c (Q_{c, d}))_{| d = d_c \nobracket} =
   \frac{- 2 \pi}{d_c^2} + o_{d_c \rightarrow \infty} \left( \frac{1}{d_c^2}
   \right) . \]
Therefore, by the implicit function theorem,
\[ \partial_c d_c = \frac{\partial_c \int_{B (d \overrightarrow{e_1},
   d^{\varepsilon'}) \cup B (- d \overrightarrow{e_1}, d^{\varepsilon'})}
   \mathfrak{R}\mathfrak{e} (\overline{\partial_d V} \tmop{TW}_c (Q_{c,
   d}))_{| d = d_c \nobracket}}{\frac{- 2 \pi}{d_c^2} + o_{d_c \rightarrow
   \infty} \left( \frac{1}{d_c^2} \right)} . \]
We compute for
\[ \tmop{TW}_c (Q_{c, d}) = - i c \partial_{x_2} Q_{c, d} - \Delta Q_{c, d} -
   (1 - | Q_{c, d} |^2) Q_{c, d} \]
that, with $\partial_c Q_{c, d} = \partial_c (V + \Gamma_{c, d}) = \partial_c
\Gamma_{c, d}$ at fixed $d$, we have (still at fixed $d$)
\[ \partial_c (\tmop{TW}_c (Q_{c, d})) = - i \partial_{x_2} Q_{c, d} -
   L_{Q_{c, d}} (\partial_c \Gamma_{c, d}), \]
where
\[ L_{Q_{c, d}} (h) \assign - \Delta h - i c \partial_{x_2} h - (1 - | Q_{c,
   d} |^2) h + 2\mathfrak{R}\mathfrak{e} (\overline{Q_{c, d}} h) Q_{c, d} . \]
We are left with the computation of

\[ \partial_c \int_{B (d \overrightarrow{e_1}, d^{\varepsilon'}) \cup B (- d
   \overrightarrow{e_1}, d^{\varepsilon'})} \mathfrak{R}\mathfrak{e}
   (\overline{\partial_d V} \tmop{TW}_c (Q_{c, d}))_{| d = d_c \nobracket} =
\]
\[ - \int_{B (d \overrightarrow{e_1}, d^{\varepsilon'}) \cup B (- d
   \overrightarrow{e_1}, d^{\varepsilon'})} \mathfrak{R}\mathfrak{e}
   (\overline{\partial_d V} (i \partial_{x_2} Q_{c, d}))_{| d = d_c
   \nobracket} \]
\[ - \int_{B (d \overrightarrow{e_1}, d^{\varepsilon'}) \cup B (- d
   \overrightarrow{e_1}, d^{\varepsilon'})} \mathfrak{R}\mathfrak{e}
   (\overline{\partial_d V} L_{Q_c} (\partial_c \Gamma_{c, d}))_{| d = d_c
   \nobracket} . \]
As above, we omit the subscript in $d_c$ for the computations.

\

\begin{tmindent}
  Step 1.  Proof of $\int_{B (d \overrightarrow{e_1}, d^{\varepsilon'}) \cup B
  (- d \overrightarrow{e_1}, d^{\varepsilon'})} \mathfrak{R}\mathfrak{e}
  (\overline{\partial_d V} (- i \partial_{x_2} Q_c))_{| d = d_c \nobracket} =
  2 \pi + o_{c \rightarrow 0} (1)$.
\end{tmindent}

\

We have $\partial_{x_2} Q_c = \partial_{x_2} V + \partial_{x_2} \Gamma_{c,
d}$, hence
\[ - \int_{B (d \overrightarrow{e_1}, d^{\varepsilon'}) \cup B (- d
   \overrightarrow{e_1}, d^{\varepsilon'})} \mathfrak{R}\mathfrak{e}
   (\overline{\partial_d V} (i \partial_{x_2} Q_c)) = \]
\[ - \int_{B (d \overrightarrow{e_1}, d^{\varepsilon'}) \cup B (- d
   \overrightarrow{e_1}, d^{\varepsilon'})} \mathfrak{R}\mathfrak{e} (i
   \overline{\partial_d V} \nobracket \partial_{x_2} V)) - \int_{B (d
   \overrightarrow{e_1}, d^{\varepsilon'}) \cup B (- d \overrightarrow{e_1},
   d^{\varepsilon'})} \mathfrak{R}\mathfrak{e} (i \overline{\partial_d V}
   \partial_{x_2} \Gamma_{c, d}) . \]
Since
\[ | \partial_d V | \leqslant \frac{K}{(1 + r_1)} \]
and
\[ | \partial_{x_2} \Gamma_{c, d} | \leqslant \frac{K}{(1 + r_1)^{1 + 1 / 2}
   d^{1 / 2}}, \]
we have
\[ \int_{B (d \overrightarrow{e_1}, d^{\varepsilon'})}
   \mathfrak{R}\mathfrak{e} (i \overline{\partial_d V} \partial_{x_2}
   \Gamma_{c, d}) = o_{c \rightarrow 0} (1) . \]
Furthermore,
\[ - \int_{B (d \overrightarrow{e_1}, d^{\varepsilon'})}
   \mathfrak{R}\mathfrak{e} (i \overline{\partial_d V} \partial_{x_2} V) =
   \int_{B (d \overrightarrow{e_1}, d^{\varepsilon'})}
   \mathfrak{R}\mathfrak{e} \left( i \overline{\partial_{x_{1 \nosymbol}} V_1}
   \partial_{x_2} V_1 \right) + o_{c \rightarrow 0} (1), \]
and we already computed in (\ref{bleu}) that
\[ \int_{\mathbbm{R}^2} \mathfrak{R}\mathfrak{e} \left( i \partial_{x_2} V_1
   \overline{\partial_{x_1} V_1} \right) = - \pi + o_{c \rightarrow 0} (c^{1 /
   4}) \]
hence
\[ \int_{B (d \overrightarrow{e_1}, d^{\varepsilon'}) \cup B (- d
   \overrightarrow{e_1}, d^{\varepsilon'})} \mathfrak{R}\mathfrak{e}
   (\overline{\partial_d V} (- i \partial_{x_2} Q_c))_{| d = d_c \nobracket} =
   2 \pi + o_{c \rightarrow 0} (1) . \]

\

\begin{tmindent}
  Step 2.  Proof of $\int_{B (d \overrightarrow{e_1}, d^{\varepsilon'}) \cup B
  (- d \overrightarrow{e_1}, d^{\varepsilon'})} \mathfrak{R}\mathfrak{e}
  (\overline{\partial_d V} L_{Q_c} (\partial_c \Gamma_{c, d}))_{| d = d_c
  \nobracket} = o_{c \rightarrow 0} (1)$.
\end{tmindent}

\

From the definition of $\Gamma_{c, d}$, at fixed $d$, we have
\begin{equation}
  \partial_c \Gamma_{c, d} = \eta V \partial_c \Psi_{c, d} + (1 - \eta) V
  \partial_c \Psi_{c, d} e^{\Psi_{c, d}} . \label{3292305}
\end{equation}
We have, by definition,
\[ L_{Q_c} (\partial_c \Gamma_{c, d}) = - i c \partial_{x_2} \partial_c
   \Gamma_{c, d} - \Delta \partial_c \Gamma_{c, d} - (1 - | Q_c |^2)
   \partial_c \Gamma_{c, d} + 2\mathfrak{R}\mathfrak{e} (\overline{Q_c}
   \partial_c \Gamma_{c, d}) Q_c, \]
and using $| \partial_d V | \leqslant \frac{K}{(1 + r_1)}$ with $|
\partial_{x_2} \partial_c \Gamma_{c, d} | \leqslant \frac{K c^{- 1 / 2}}{(1 +
r_1)^{1 + 1 / 2}}$ since $\left\| \frac{\partial_c \Gamma_{c, d}}{V}
\right\|_{\ast, 1 / 2, d} \leqslant K c^{- 3 / 4}$ from Lemma \ref{nl310} and
(\ref{3292305}), we have
\[ \left| \int_{B (d \overrightarrow{e_1}, d^{\varepsilon'})}
   \mathfrak{R}\mathfrak{e} (\overline{\partial_d V} (- i c \partial_{x_2}
   \partial_c \Gamma_{c, d})) \right| \leqslant K \int_{B (d
   \overrightarrow{e_1}, d^{\varepsilon'})} \frac{c^{1 / 4}}{(1 + r_1)^{2 + 1
   / 2}} = o_{c \rightarrow 0} (1) . \]
The estimate on $B (- d \overrightarrow{e_1}, d^{\varepsilon'})$ is similar.

\

We define
\[ \tilde{L}_{Q_c} (h) \assign - \Delta h - (1 - | Q_c |^2) h +
   2\mathfrak{R}\mathfrak{e} (\overline{Q_c} h) Q_c \]
and we are then left with the computation of
\[ \int_{B (d \overrightarrow{e_1}, d^{\varepsilon'})}
   \mathfrak{R}\mathfrak{e} (\overline{\partial_d V}  \tilde{L}_{Q_c}
   (\partial_c \Gamma_{c, d})), \]
the part on $B (- d \overrightarrow{e_1}, d^{\varepsilon'})$ being
symmetrical. We want to put the linear operator onto $\partial_d V$ since
$\tilde{L}_{Q_c} (\partial_d V)$ is close to $L_V (\partial_d V)$ which is
itself small. We then integrate by parts:
\[ \left| \int_{B (d \overrightarrow{e_1}, d^{\varepsilon'})}
   \mathfrak{R}\mathfrak{e} (\overline{\partial_d V}  \tilde{L}_{Q_c}
   (\partial_c \Gamma_{c, d})) \right| \leqslant \left| \int_{B (d
   \overrightarrow{e_1}, d^{\varepsilon'})} \mathfrak{R}\mathfrak{e}
   (\tilde{L}_{Q_c} (\partial_d V)  \overline{\partial_c \Gamma_{c, d}})
   \right| \]
\[ + \left| \int_{\partial B (d \overrightarrow{e_1}, d^{\varepsilon'})}
   \mathfrak{R}\mathfrak{e} (\overline{\partial_d V} \nabla \partial_c
   \Gamma_{c, d}) \right| + \left| \int_{\partial B (d \overrightarrow{e_1},
   d^{\varepsilon'})} \mathfrak{R}\mathfrak{e} (\nabla \overline{\partial_d V}
   \partial_c \Gamma_{c, d}) \right| . \]
We have on $\partial B (d \overrightarrow{e_1}, d^{\varepsilon'})$, that $|
\partial_d V | \leqslant \frac{K}{d^{3 / 4}}, | \nabla \partial_d V |
\leqslant \frac{K}{d^{3 / 2}}$ from Lemma \ref{ddVest}. Moreover, by $\left\|
\frac{\partial_c \Gamma_{c, d}}{V} \right\|_{\ast, 1 / 2, d} \leqslant K
(\sigma) c^{- 1 / 2 - \sigma}$ from Lemma \ref{nl310} and (\ref{3292305}), we
deduce $| \nabla \partial_c \Gamma_{c, d} | \leqslant \frac{K (\sigma) d^{1 /
2 + \sigma}}{d^{(3 / 4) (3 / 2)}} \leqslant \frac{K (\sigma)}{d^{5 / 8 -
\sigma}}$ and $| \partial_c \Gamma_{c, d} | \leqslant \frac{K (\sigma) d^{1 /
2 - \sigma}}{d^{(3 / 4) (1 / 2)}} \leqslant K (\sigma) d^{1 / 8 - \sigma}$. We
then obtain, for $\sigma > 0$ small enough,
\[ \left| \int_{\partial B (d \overrightarrow{e_1}, d^{\varepsilon'})}
   \mathfrak{R}\mathfrak{e} (\overline{\partial_d V} \nabla \partial_c
   \Gamma_{c, d}) \right| \leqslant \int_{\partial B (d \overrightarrow{e_1},
   d^{\varepsilon'})} | \partial_d V | | \nabla \partial_c \Gamma_{c, d} |
   \leqslant d^{3 / 4} \frac{K (\sigma) d^{2 \sigma}}{d^{3 / 4} d^{5 / 8}} =
   o_{c \rightarrow 0} (1), \]
\[ \left| \int_{\partial B (d \overrightarrow{e_1}, d^{\varepsilon'})}
   \mathfrak{R}\mathfrak{e} (\nabla \overline{\partial_d V} \partial_c
   \Gamma_{c, d}) \right| \leqslant \int_{\partial B (d \overrightarrow{e_1},
   d^{\varepsilon'})} | \nabla \partial_d V | | \partial_c \Gamma_{c, d} |
   \leqslant d^{3 / 4} \frac{K (\sigma) d^{1 / 8 + \sigma}}{d^{3 / 2}} = o_{c
   \rightarrow 0} (1) . \]
Therefore,
\[ \int_{B (d \overrightarrow{e_1}, d^{\varepsilon'})}
   \mathfrak{R}\mathfrak{e} (\overline{\partial_d V}  \tilde{L}_{Q_c}
   (\partial_c \Gamma_{c, d})) = \int_{B (d \overrightarrow{e_1},
   d^{\varepsilon'})} \mathfrak{R}\mathfrak{e} (\tilde{L}_{Q_c} (\partial_d V)
   \overline{\partial_c \Gamma_{c, d}}) + o_{c \rightarrow 0} (1) . \]
Now, from (\ref{bleu2}), we have that that
\[ | L_V (\partial_d V) | \leqslant \frac{K}{(1 + \tilde{r}^2) d} \]
and by Lemma \ref{nl310} and (\ref{3292305}), we have $| \partial_c \Gamma_{c,
d} | \leqslant \frac{K d^{1 / 4}}{(1 + r_1)^{1 / 2}}$, hence
\[ \left| \int_{B (d \overrightarrow{e_1}, d^{\varepsilon'})}
   \mathfrak{R}\mathfrak{e} (L_V (\partial_d V)  \overline{\partial_c
   \Gamma_{c, d}}) \right| \leqslant K \int_{B (d \overrightarrow{e_1},
   d^{\varepsilon'})} \frac{1}{(1 + r_1)^{2 + 1 / 2} d^{1 / 4}} = o_{c
   \rightarrow 0} (1) . \]
We deduce from this that
\[ \int_{B (d \overrightarrow{e_1}, d^{\varepsilon'})}
   \mathfrak{R}\mathfrak{e} (\overline{\partial_d V}  \tilde{L}_{Q_c}
   (\partial_c \Gamma_{c, d})) = \int_{B (d \overrightarrow{e_1},
   d^{\varepsilon'})} \mathfrak{R}\mathfrak{e} ((\tilde{L}_{Q_c} - L_V)
   (\partial_d V)  \overline{\partial_c \Gamma_{c, d}}) + o_{c \rightarrow 0}
   (1) . \]
We have $\tilde{L}_{Q_c} (h) = - \Delta h - (1 - | Q_c |^2) h +
2\mathfrak{R}\mathfrak{e} (\overline{Q_c} h) Q_c$ and $L_V (h) = - \Delta h -
(1 - | V |^2) h + 2\mathfrak{R}\mathfrak{e} (\bar{V} h) V$, therefore
\[ (\tilde{L}_{Q_c} - L_V) (\partial_d V) = (| Q_c |^2 - | V |^2) \partial_d V
   + 2\mathfrak{R}\mathfrak{e} (\bar{V} \partial_d V) (Q_c - V) +
   2\mathfrak{R}\mathfrak{e} (\overline{Q_c - V} \partial_d V) Q_c . \]
We have by (\ref{219}) that $| | Q_c |^2 - | V |^2 | \leqslant \frac{K c^{3 /
4}}{(1 + \tilde{r})^{1 + 1 / 4}}$, hence
\[ \left| \int_{B (d \overrightarrow{e_1}, d^{\varepsilon'})}
   \mathfrak{R}\mathfrak{e} ( (| Q_c |^2 - | V |^2) \partial_d V
   \overline{\partial_c \Gamma_{c, d}}) \right| \leqslant K \int_{B (d
   \overrightarrow{e_1}, d^{\varepsilon'})} \frac{c^{1 / 4}}{(1 + r_1)^{2 + 3
   / 4}} = o_{c \rightarrow 0} (1) . \]
We have from (\ref{218}) that $| Q_c - V | \leqslant \frac{c^{3 / 4}}{(1 +
\tilde{r})^{1 / 4}}$, and, in $B (d \overrightarrow{e_1}, d^{\varepsilon'})$,
we have (by Lemmas \ref{lemme3} and \ref{dervor}) that $|
\mathfrak{R}\mathfrak{e} (\bar{V} \partial_d V) | \leqslant \frac{K}{(1 +
r_1)^3}$, therefore
\[ \left| \int_{B (d \overrightarrow{e_1}, d^{\varepsilon'})}
   \mathfrak{R}\mathfrak{e} (2\mathfrak{R}\mathfrak{e} (\bar{V} \partial_d V)
   (Q_c - V) \overline{\partial_c \Gamma_{c, d}}) \right| \leqslant K \int_{B
   (d \overrightarrow{e_1}, d^{\varepsilon'})} \frac{c^{1 / 4}}{(1 + r_1)^{3 +
   3 / 4}} = o_{c \rightarrow 0} (1) . \]
Finally, by using the same estimates, we have
\[ \left| \int_{B (d \overrightarrow{e_1}, d^{\varepsilon'})}
   \mathfrak{R}\mathfrak{e} (2\mathfrak{R}\mathfrak{e} (\overline{Q_c - V}
   \partial_d V) Q_c \overline{\partial_c \Gamma_{c, d}}) \right| \leqslant K
   \int_{B (d \overrightarrow{e_1}, d^{\varepsilon'})} \frac{c^{3 / 4}}{(1 +
   r_1)^{1 + 1 / 4}} | \mathfrak{R}\mathfrak{e} (Q_c \overline{\partial_c
   \Gamma_{c, d}}) | . \]
We compute
\[ \mathfrak{R}\mathfrak{e} (Q_c \overline{\partial_c \Gamma_{c, d}})
   =\mathfrak{R}\mathfrak{e} (V \overline{\partial_c \Gamma_{c, d}})
   +\mathfrak{R}\mathfrak{e} (\Gamma_{c, d} \overline{\partial_c \Gamma_{c,
   d}}) . \]
By using $\left\| \frac{\partial_c \Gamma_{c, d}}{V} \right\|_{\ast, 1 / 2, d}
\leqslant K (\sigma) c^{- 1 / 2 - \sigma}$ from Lemma \ref{rest31} and
(\ref{3292305}), we have $\left| \mathfrak{R}\mathfrak{e} \left( V
\overline{\partial_c \Gamma_{c, d_c}} \right) \right| \leqslant \frac{K
(\sigma) c^{- 1 / 2 - \sigma}}{(1 + r_1)^{3 / 2}}$. Furthermore, with $|
\Gamma_{c, d} | \leqslant \frac{K c^{1 / 2}}{(1 + r_1)^{1 / 2}}$, we have $|
\mathfrak{R}\mathfrak{e} (\Gamma_{c, d} \overline{\partial_c \Gamma_{c, d}}) |
\leqslant \frac{K (\sigma) c^{- \sigma}}{(1 + r_1)} .$ With these estimates,
we infer, taking $\sigma > 0$ small enough,
\[ \left| \int_{B (d \overrightarrow{e_1}, d^{\varepsilon'})}
   \mathfrak{R}\mathfrak{e} (2\mathfrak{R}\mathfrak{e} (\overline{Q_c - V}
   \partial_d V) Q_c \overline{\partial_c \Gamma_{c, d}}) \right| = o_{c
   \rightarrow 0} (1) \]
which ends the proof of
\[ \int_{B (d \overrightarrow{e_1}, d^{\varepsilon'})}
   \mathfrak{R}\mathfrak{e} ((\tilde{L}_{Q_c} - L_V) (\partial_d V) 
   \overline{\partial_c \Gamma_{c, d}}) = o_{c \rightarrow 0} (1) . \]

\

\begin{tmindent}
  Step 3.  Conclusion.
\end{tmindent}

\

We showed that
\[ \partial_c d_c = \frac{2 \pi + o_{c \rightarrow 0} (1)}{\frac{- 2
   \pi}{d_c^2} + o_{d_c \rightarrow \infty} \left( \frac{1}{d_c^2} \right)},
\]
therefore, with $d_c = \frac{1 + o_{c \rightarrow 0} (1)}{c}$ from Proposition
\ref{dinc} we have
\[ \partial_c d_c = - \frac{1 + o_{c \rightarrow 0} (1)}{c^2} . \]
As a result of subsection \ref{36}, at fixed $c$,
\[ \partial_d \int_{B (d \overrightarrow{e_1}, d^{\varepsilon'}) \cup B (- d
   \overrightarrow{e_1}, d^{\varepsilon'})} \mathfrak{R}\mathfrak{e}
   (\overline{\partial_d V} \tmop{TW}_c (Q_{c, d}))_{| d = d_c \nobracket}
   \neq 0 \]
for $c$ small enough. By the implicit function theorem, taking some $0 <
c_{\ast} < c_0 (\sigma)$, we can construct a $C^1$ branch $c \mapsto d_c$ in a
vicinity of $c_{\ast}$. We define $\mathbf{C}$ as the set of $c_{\ast} >
c_{\circledast} \geqslant 0$ such that there exists a $C^1$ branch $c \mapsto
d_c$ on $] c_{\circledast}, c_{\ast} [$. We have just shown that $\mathbf{C}$
is not empty. Let us suppose that $c_{\circledast} \assign \inf \mathbf{C}
\neq 0$. Then, $c \mapsto d_c$ is uniformly bounded on $] c_{\circledast},
c_{\ast} [$ in $C^1$ by subsection \ref{37}, and can therefore be extended by
continuity to $c_{\circledast}$, and we denote $d_{\circledast}$ its value
there. We can construct the perturbation $\Phi_{c_{\circledast},
d_{\circledast}}$ by continuity since $c, d \mapsto \Phi_{c, d}$ are $C^1$
functions in the Banach space $\{ \Phi \in C^1 (\mathbbm{R}^2, \mathbbm{C}),
\| \Phi \|_{\ast, \sigma, d_{\circledast}} < + \infty \}$ for its canonical
norm (which is equivalent to $\| . \|_{\ast, \sigma, d}$ for any $d \in
[d_{\circledast}, d_{c_{\ast}}]$). By passing to the limit, we have $\|
\Phi_{c_{\circledast}, d_{\circledast}} \|_{\ast, \sigma, d_{\circledast}}
\leqslant K_0 (\sigma, \sigma') c_{\circledast}^{1 - \sigma'}$ for $K_0
(\sigma, \sigma')$ defined in Proposition \ref{contractionest}. By continuity
of $\lambda$, we check that we have $\lambda (c_{\circledast},
d_{\circledast}) = 0$ (for the perturbation $\Phi_{c_{\circledast},
d_{\circledast}}$). Therefore, by the implicit function theorem, there exists
a unique branch $c \mapsto d_c$ in a vicinity of $(c_{\circledast},
d_{\circledast})$ such that $\lambda (c, d_c) = 0$. This branch, by
uniqueness, corresponds to the branch we had on $] c_{\circledast}, c_{\ast}
[$, and is also $C^1$ by the implicit function theorem. Therefore $\inf
\mathbf{C}< c_{\circledast}$, which is in contradiction with $c_{\circledast}
= \inf \mathbf{C}$, and thus $\inf \mathbf{C}= 0$.

In particular, the travelling wave $Q_c$ on this branch is uniquely defined by
this construction and is a $C^1$ function of $c$. Indeed, we shall now show
that there is only one choice of $d_c$ such that $\lambda (c, d_c) = 0$ in
$\left] \frac{1}{2 c}, \frac{2}{c} \right[$. If there exist $d_1 \neq d_2$ in
$\left] \frac{1}{2 c}, \frac{2}{c} \right[$ such that $\lambda (c, d_1) =
\lambda (c, d_2) = 0$, by Subsection \ref{36}, we have
\[ \partial_d (\lambda (c, d))_{| d = d_1 \nobracket} < 0 \quad \tmop{and}
   \quad \partial_d (\lambda (c, d))_{| d = d_2 \nobracket} < 0, \]
therefore, there exists $d'$ such that $\lambda (c, d') = 0$ and $\partial_d
(\lambda (c, d'))_{| d = d' \nobracket} \geqslant 0$, but then, since $\lambda
(c, d') = 0$, we have $\partial_d (\lambda (c, d))_{| d = d' \nobracket} < 0$,
which is in contradiction with $\partial_d (\lambda (c, d'))_{| d = d'
\nobracket} \geqslant 0$. Now that we have uniqueness in the choice of $d_c$
(in $\left] \frac{1}{2 c}, \frac{2}{c} \right[$), we have uniqueness of
$\Phi_{c, d}$ in the set
\[ \{ \Phi \in C^1 (\mathbbm{R}^2, \mathbbm{C}), \| \Phi \|_{\ast, \sigma,
   d_c} \leqslant K_0 (\sigma, \sigma') c^{1 - \sigma'} \} \]
for $K_0 (\sigma, \sigma') > 0$ defined in Proposition \ref{contractionest}.

\subsection{Proof of the estimate on $\partial_c Q_c$\label{38}}

We conclude the proof of Theorem \ref{th1} with the following lemma.

\begin{lemma}
  \label{finallemma}For any $0 < \sigma < 1$, there exist $c_0 (\sigma) > 0$
  such that for any $c < c_0 (\sigma)$,
  \[ \left\| \frac{\partial_c Q_c}{V} + \left( \frac{1 + o_{c \rightarrow 0}
     (1)}{c^2} \right) \frac{\partial_d V_{| d = d_c \nobracket}}{V}
     \right\|_{\ast, \sigma, d_c} = o_{c \rightarrow 0} \left( \frac{1}{c^2}
     \right) . \]
\end{lemma}

With this estimate and by using the same computations as in the proof of Lemma
\ref{lema217}, we show that
\[ \left\| \partial_c Q_c + \left( \frac{1 + o_{c \rightarrow 0} (1)}{c^2}
   \right) \partial_d (V_1 (. - d \overrightarrow{e_1}) V_{- 1} (. + d
   \overrightarrow{e_1}))_{| d = d_c \nobracket} \right\|_p = o_{c \rightarrow
   0} \left( \frac{1}{c^2} \right) . \]
for all $+ \infty \geqslant p > 2$ if $c$ is small enough, which ends the
proof of Theorem \ref{th1}.

\

\begin{proof}
  From subsection $\ref{36}$, we know that $Q_c$ is a $C^1$ function of $c$.
  We have $Q_c = V + \Gamma_{c, d_c}$, hence
  \[ \partial_c Q_c = \partial_c V + \partial_c (\Gamma_{c, d_c}) = \frac{- 1
     + o_{c \rightarrow 0} (1)}{c^2} \partial_d V + \partial_c (\Gamma_{c,
     d_c}), \]
  where we used $\partial_c V = \left( - \frac{1}{c^2} + o_{c \rightarrow 0}
  \left( \frac{1}{c^2} \right) \right) \partial_d V$ thanks to subsection
  \ref{37}. $\Gamma_{c, d_c}$ depends on $c$ directly and through $d_c$. We
  will write $\partial_c \Gamma_{c, d_c}$ for the derivatives with respect to
  $c$ but at a fixed $d_c$, and $\partial_d \Gamma_{c, d_c}$ for the derivate
  with respect to $d_c$ but at fixed $c$. In particular,
  \[ \partial_c (\Gamma_{c, d_c}) = \partial_c \Gamma_{c, d_c} + \partial_c
     d_c \partial_d \Gamma_{c, d_c} . \]
  From Lemma \ref{nl310} and (\ref{3292305}), we showed that
  \[ \left\| \frac{\partial_c \Gamma_{c, d_c}}{V} \right\|_{\ast, \sigma,
     d_c} \leqslant K (\sigma, \sigma') c^{- \sigma'}, \]
  and from Lemma \ref{dderpsy} with the definition of $\Gamma_{c, d}$, we show
  easily that
  \[ \left\| \frac{\partial_d \Gamma_{c, d_c}}{V} \right\|_{\ast, \sigma, d_c}
     \leqslant K (\sigma, \sigma') c^{1 - \sigma'} . \]
  Finally, from subsection \ref{37}, we have $\partial_c d_c = \frac{1 + o_{c
  \rightarrow 0} (1)}{c^2}$, therefore
  \[ \left\| \frac{\partial_c (\Gamma_{c, d_c})}{V} \right\|_{\ast, \sigma,
     d_c} \leqslant K (\sigma, \sigma') (c^{- \sigma'} + c^{- 2} (1 + o_{c
     \rightarrow 0} (1)) c^{1 - \sigma'}) = o_{c \rightarrow 0} \left(
     \frac{1}{c^2} \right) \]
  since $0 < \sigma < \sigma' < 1$.
\end{proof}

This concludes the proof of Lemma \ref{finallemma}, which itself concludes the
proof of Theorem \ref{th1}.

\appendix\section{Proof of Lemma \ref{lemma7}}\label{AA}

\begin{proof}
  First we show that $L (\Phi) = (E - i c \partial_{x_2} V) \Psi + L' (\Psi)
  V$. We use $\Phi = V \Psi$ in $L (\Phi)$ to compute
  \[ L (\Phi) = - \Delta V \Psi - \Delta \Psi V - 2 \nabla \Psi . \nabla V -
     (1 - | V |^2) V \Psi + 2 | V |^2 V\mathfrak{R}\mathfrak{e} (\Psi) - i c V
     \partial_{x_2} \Psi - i c \partial_{x_2} V \Psi . \]
  We have that $E = - \Delta V - (1 - | V |^2) V$ hence $(E - i c
  \partial_{x_2} V) \Psi = - \Delta V \Psi - (1 - | V |^2) V \Psi - i c
  \partial_{x_2} V \Psi$ and the remaining terms are exactly equal to $V L'
  (\Psi)$.
  
  We denote $\zeta \assign 1 + \Psi - e^{\Psi}$. Remark that $\zeta$ is at
  least quadratic in $\Psi$. We compute the different terms in
  $(\tmop{TW}_c)$:
  \[ - i c \partial_{x_2} v - \Delta v - (1 - | v |^2) v = 0 \]
  with
  \[ v = \eta V (1 + \Psi) + (1 - \eta) V e^{\Psi} . \]
  We have $v = V + \Phi - (1 - \eta) \zeta$. In general, our goal in this
  computation is to factorize any term when possible by $V (\eta + (1 - \eta)
  e^{\Psi})$ and compute the other terms, which will be supported in the area
  $\eta (1 - \eta) \neq 0$. First compute
  \[ \partial_{x_2} v = \]
  \[ \eta (\partial_{x_2} V (1 + \Psi) + \partial_{x_2} \Psi V) +
     \partial_{x_2} \eta V (1 + \Psi) + (1 - \eta) e^{\Psi} (\partial_{x_2} V
     + \partial_{x_2} \Psi V) - \partial_{x_2} \eta V e^{\Psi}, \]
  therefore
  \begin{equation}
    - i c \partial_{x_2} v = V (\eta + (1 - \eta) e^{\Psi}) \left( - i c
    \frac{\partial_{x_2} V}{V} - i c \partial_{x_2} \Psi \right) - i c \eta
    \partial_{x_2} V \Psi - i c \partial_{x_2} \eta V \zeta . \label{com1}
  \end{equation}
  For the second term, we compute
  \begin{eqnarray*}
    \Delta v & = & \Delta \eta V (1 + \Psi - e^{\Psi}) + 2 \nabla \eta .
    \nabla (V (1 + \Psi - e^{\Psi}))\\
    & + & \eta (\Delta V (1 + \Psi) + 2 \nabla V. \nabla \Psi + V \Delta
    \Psi)\\
    & + & (1 - \eta) (\Delta V e^{\Psi} + 2 \nabla V. \nabla \Psi e^{\Psi} +
    V (\Delta \Psi + \nabla \Psi . \nabla \Psi) e^{\Psi}),
  \end{eqnarray*}
  hence
  \begin{eqnarray}
    - \Delta v & = & V (\eta + (1 - \eta) e^{\Psi}) \left( - \frac{\Delta
    V}{V} - 2 \frac{\nabla V}{V} . \nabla \Psi - \Delta \Psi \right)
    \nonumber\\
    & - & \eta \Delta V \Psi - (1 - \eta) V \nabla \Psi . \nabla \Psi
    e^{\Psi} - V \Delta \eta \zeta - 2 \nabla \eta . \nabla (V \zeta) . 
    \label{com2}
  \end{eqnarray}
  Finally, let us write $A \assign V (1 + \Psi)$ and $B \assign V e^{\Psi}$,
  so that $v = \eta A + (1 - \eta) B$, and remark that $V \zeta = A - B$. We
  then have
  \[ (1 - | v |^2) v = (1 - \eta^2 | A |^2 - (1 - \eta)^2 | B |^2 - 2 \eta (1
     - \eta) \mathfrak{R}\mathfrak{e} (A \bar{B})) (\eta A + (1 - \eta) B) .
  \]
  We want to bring out the terms not related to the interaction between $A$
  and $B$, namely $\eta (1 - | A |^2) A + (1 - \eta) (1 - | B |^2) B$. We have
  \begin{eqnarray*}
    (1 - | v |^2) v & = & \eta (1 - | A |^2) A + \eta A [(1 - \eta^2) | A |^2
    - (1 - \eta)^2 | B |^2 - 2 \eta (1 - \eta) \mathfrak{R}\mathfrak{e} (A
    \bar{B})]\\
    & + & (1 - \eta) (1 - | B |^2) B + (1 - \eta) B [(1 - (1 - \eta)^2) | B
    |^2 - \eta^2 | A |^2 - 2 \eta (1 - \eta) \mathfrak{R}\mathfrak{e} (A
    \bar{B})] .
  \end{eqnarray*}
  Now, factorizing $\eta (1 - \eta)$ we get
  \begin{eqnarray*}
    (1 - | v |^2) v & = & \eta (1 - | A |^2) A + (1 - \eta) (1 - | B |^2) B\\
    & + & \eta (1 - \eta) [(1 + \eta) A | A |^2 - (1 - \eta) A | B |^2 - 2
    \eta A\mathfrak{R}\mathfrak{e} (A \bar{B})]\\
    & + & \eta (1 - \eta) [(2 - \eta) B | B |^2 - \eta B | A |^2 - 2 (1 -
    \eta) B\mathfrak{R}\mathfrak{e} (A \bar{B})] .
  \end{eqnarray*}
  Remark that the last two lines yield $0$ if we take $A = B$, since $V \zeta
  = A - B$, we can write
  \[ (1 - | v |^2) v = \eta (1 - | A |^2) A + (1 - \eta) (1 - | B |^2) B +
     \eta (1 - \eta) (V \zeta G (\Psi) + \overline{V \zeta} H (\Psi)) \]
  where $G, H$ are functions satisfying $| H (\Psi) |, | G (\Psi) |, | \nabla
  H (\Psi) |, | \nabla G (\Psi) | \leqslant C (1 + | \Psi | + | \nabla \Psi |
  + | e^{\Psi} | + | \nabla \Psi e^{\Psi} |)$ for some universal constant $C >
  0$. We recall that $A = V (1 + \Psi)$ hence
  \[ (1 - | A |^2) A = (1 - | V |^2 | 1 + \Psi |^2) V (1 + \Psi), \]
  therefore we get a constant (in $\Phi$), a linear and a nonlinear part in
  $\Psi$:
  \begin{eqnarray*}
    (1 - | A |^2) A & = & (1 - | V |^2) V + (1 - | V |^2) V \Psi - 2 | V |^2
    V\mathfrak{R}\mathfrak{e} (\Psi)\\
    &  & - 2 | V |^2 V\mathfrak{R}\mathfrak{e} (\Psi) \Psi - | V \Psi |^2 V
    (1 + \Psi) .
  \end{eqnarray*}
  We have $B = V e^{\Psi}$, hence
  \[ (1 - | B |^2) B = e^{\Psi} ((1 - | V |^2) V - 2\mathfrak{R}\mathfrak{e}
     (\Psi) | V |^2 V - | V |^2 V S (\Psi)), \]
  where $S (\Psi) = e^{2\mathfrak{R}\mathfrak{e} (\Psi)} - 1 -
  2\mathfrak{R}\mathfrak{e} (\Psi)$ is nonlinear in $\Psi$. We add these
  relations and obtain
  \begin{eqnarray}
    \eta (1 - | A |^2) A + (1 - \eta) (1 - | B |^2) B & = & V (\eta + (1 -
    \eta) e^{\Psi}) ((1 - | V |^2) - 2\mathfrak{R}\mathfrak{e} (\Psi) | V |^2)
    \nonumber\\
    & + & \eta (1 - \eta) (V \zeta G (\Psi) + \overline{V \zeta} H (\Psi))
    \nonumber\\
    & + & \eta ((1 - | V |^2) V \Psi - 2 | V |^2 V\mathfrak{R}\mathfrak{e}
    (\Psi) \Psi - | V \Psi |^2 V (1 + \Psi)) \nonumber\\
    & - & (1 - \eta) e^{\Psi} | V^{\nosymbol} |^2 V S (\Psi) .  \label{com3}
  \end{eqnarray}
  Now adding the computations (\ref{com1}), (\ref{com2}) and (\ref{com3}) in
  $- i c \partial_{x_2} v - \Delta v - (1 - | v |^2) v = 0$ yields
  \begin{eqnarray}
    V (\eta + (1 - \eta) e^{\Psi}) \left( \frac{E - i c \partial_{x_2} V}{V} +
    L' (\Psi) \right) &  &  \nonumber\\
    + \eta ((E - i c \partial_{x_2} V) \Psi + 2 | V |^2
    V\mathfrak{R}\mathfrak{e} (\Psi) \Psi + | V \Psi |^2 V (1 + \Psi))    &  &
    \nonumber\\
    + V (1 - \eta) e^{\Psi} (| V |^2 S (\Psi) - \nabla \Psi . \nabla \Psi) & 
    &  \nonumber\\
    - i c \partial_{x_2} \eta V \zeta - V \Delta \eta \zeta - 2 \nabla \eta .
    \nabla (V \zeta) - \eta (1 - \eta) (V \zeta G (Z) + \overline{V \zeta} H
    (\Psi)) & = & 0.  \label{rouge}
  \end{eqnarray}
  We divide by $\eta + (1 - \eta) e^{\Psi}$, which is allowed since $\eta + (1
  - \eta) e^{\Psi} = 1 + (1 - \eta) (e^{\Psi} - 1)$ and in $\{ \eta \neq 1
  \}$, $| \Psi | \leqslant \frac{| \Phi |}{| V |} \leqslant K \| \Phi
  \|_{L^{\infty} (\mathbbm{R}^2)} \leqslant K C_0$ by our assumption $\| \Phi
  \|_{L^{\infty} (\mathbbm{R}^2)} \leqslant C_0$, therefore, choosing $C_0$
  small enough, in $\{ \eta \neq 1 \}$, we have $| e^{\Psi} - 1 | \leqslant 1
  / 2$. We also remark that
  \[ \frac{(1 - \eta) e^{\Psi}}{(\eta + (1 - \eta) e^{\Psi})} = (1 - \eta) +
     \eta (1 - \eta) \left( \frac{e^{\Psi} - 1}{\eta + (1 - \eta) e^{\Psi}}
     \right), \]
  therefore (\ref{rouge}) become
  \begin{eqnarray*}
    E - i c \partial_{x_2} V + V L' (\Psi) &  & \\
    + V (1 - \eta) (- \nabla \Psi . \nabla \Psi + | V |^2 S (\Psi)) &  & \\
    + \frac{\eta}{(\eta + (1 - \eta) e^{\Psi})} ((E - i c \partial_{x_2} V)
    \Psi + 2 | V |^2 V\mathfrak{R}\mathfrak{e} (\Psi) \Psi + | V \Psi |^2 V (1
    + \Psi)) &  & \\
    + R_1 (\Psi) & = & 0,
  \end{eqnarray*}
  where
  \begin{eqnarray*}
    R_1 (\Psi) & \assign & \frac{1}{(\eta + (1 - \eta) e^{\Psi})} (- i c
    \partial_{x_2} \eta V \zeta - V \Delta \eta \zeta - 2 \nabla \eta . \nabla
    (V \zeta) - \eta (1 - \eta) (V \zeta G (\Psi) + \overline{V \zeta} H
    (\Psi)))\\
    & + & V \eta (1 - \eta) \left( \frac{e^{\Psi} - 1}{\eta + (1 - \eta)
    e^{\Psi}} \right) (- \nabla \Psi . \nabla \Psi + | V |^2 S (\Psi)) .
  \end{eqnarray*}
  Remark that $R_1 (\Psi)$ is nonzero only in the rings where $\eta (1 - \eta)
  \neq 0$, i.e. $1 \leqslant \tilde{r} \leqslant 2$, since every term has
  either $\partial_{x_2} \eta, \Delta \eta$ or $\eta (1 - \eta)$ as a factor.
  Furthermore they all have as an additional factor $\zeta, \nabla \zeta, S$
  or $\nabla \Psi . \nabla \Psi$. Hence, if we suppose that $| \Psi |, |
  \nabla \Psi |, | \nabla^2 \Psi | \leqslant K C_0$ in the rings (which is a
  consequence of $\Phi = V \Psi$ and $\| \Phi \|_{C^2 (\mathbbm{R}^2)}
  \leqslant C_0$), then those terms can be bounded by $C \| \Psi \|^2_{C^1 (\{
  1 \leqslant \tilde{r} \leqslant 2 \})}$. Therefore if $| \Psi |, | \nabla
  \Psi |, | \nabla^2 \Psi | \leqslant K C_0$ in the rings, then
  \[ | R_1 (\Psi) | + | \nabla R_1 (\Psi) | \leqslant K \| \Psi \|^2_{C^2 (\{
     1 \leqslant \tilde{r} \leqslant 2 \})} \leqslant K \| \Phi \|^2_{C^2 (\{
     \tilde{r} \leqslant 2 \})} \]
  for some universal constant $K > 0$, since in the rings, $V$ is bounded from
  below by a nonzero constant. Now, we use
  \[ \frac{\eta}{(\eta + (1 - \eta) e^{\Psi})} = \eta + \eta (1 - \eta) 
     \frac{1 - e^{\Psi}}{\eta + (1 - \eta) e^{\Psi}} \]
  to compute
  \[ \frac{\eta}{(\eta + (1 - \eta) e^{\Psi})} (E - i c \partial_{x_2} V) \Psi
     = \eta (E - i c \partial_{x_2} V) \Psi + R_2 (\Psi), \]
  where
  \[ R_2 (\Psi) \assign \eta (1 - \eta)  \frac{(1 - e^{\Psi}) (E - i c
     \partial_{x_2} V)}{\eta + (1 - \eta) e^{\Psi}} \Psi . \]
  We show easily that $R_2 (\Psi)$ satisfies the same estimates as $R_1
  (\Psi)$. Remark that, using $\Phi = V \Psi$,
  \begin{eqnarray*}
    \left| \frac{\eta}{(\eta + (1 - \eta) e^{\Psi})} (2 | V |^2
    V\mathfrak{R}\mathfrak{e} (\Psi) \Psi + | V \Psi |^2 V (1 + \Psi)) \right|
    & = & \\
    \left| \frac{\eta}{(\eta + (1 - \eta) e^{\Psi})}
    (2\mathfrak{R}\mathfrak{e} (\Phi \bar{V}) \Phi + | \Phi |^2 (V + \Phi))
    \right| & \leqslant & K \| \Phi \|^2_{C^1 (\{ \tilde{r} \leqslant 2 \})}
  \end{eqnarray*}
  and
  \[ \left| \nabla \left( \frac{\eta}{(\eta + (1 - \eta) e^{\Psi})}
     (2\mathfrak{R}\mathfrak{e} (\Phi \bar{V}) \Phi + | \Phi |^2 (V + \Phi))
     \right) \right| \leqslant K \| \Phi \|^2_{C^1 (\{ \tilde{r} \leqslant 2
     \})} \]
  if $\| \Phi \|_{L^{\infty} (\mathbbm{R}^2)} \leqslant C_0$ (so that the term
  in $e^{\Psi}$ is bounded) since $\eta \neq 0$ only if $\tilde{r} \leqslant
  2$. We define
  \[ R (\Psi) \assign R_1 (\Psi) + R_2 (\Psi) + \frac{\eta}{(\eta + (1 - \eta)
     e^{\Psi})} (2 | V |^2 V\mathfrak{R}\mathfrak{e} (\Psi) \Psi + | V \Psi
     |^2 V (1 + \Psi)), \]
  which satisfies
  \[ | R (\Psi) |, | \nabla (R (\Psi)) | \leqslant K \| \Phi \|^2_{C^2 (\{
     \tilde{r} \leqslant 2 \})} \]
  for some universal constant $K > 0$, provided that $\| \Phi \|_{C^2
  (\mathbbm{R}^2)} \leqslant C_0$. The equation (\ref{rouge}) then becomes
  \begin{eqnarray*}
    E - i c \partial_{x_2} V + V L' (\Psi) + V (1 - \eta) (- \nabla \Psi .
    \nabla \Psi + | V |^2 S (\Psi)) &  & \\
    + \eta (E - i c \partial_{x_2} V) \Psi + R (\Psi) & = & 0.
  \end{eqnarray*}
  Now we finish by using $- i c V \partial_{x_2} \Psi = - \eta i c V
  \partial_{x_2} \Psi - (1 - \eta) i c V \partial_{x_2} \Psi$ and
  \[ \partial_{x_2} V \Psi + \partial_{x_2} \Psi V = \partial_{x_2} \Phi \]
  to obtain
  \[ V L' (\Psi) + \eta (E - i c \partial_{x_2} V) \Psi - i c \eta
     \partial_{x_2} \Phi + V (1 - \eta) (- \nabla \Psi . \nabla \Psi + | V |^2
     S (\Psi)) + R (\Psi) = 0. \]
  Finally, since we have shown that $L (\Phi) = (E - i c \partial_{x_2} V)
  \Psi + L' (\Psi) V$, we infer
  \[ V L' (\Psi) + \eta (E - i c \partial_{x_2} V) \Psi = \eta L (\Phi) + (1 -
     \eta) V L' (\Psi) . \]
  The proof is complete.
\end{proof}

\section{Elliptic computations}\label{AB}

\subsection{Proof of Lemma \ref{lapzeta}\label{AB1}}

\begin{proof}
  The uniqueness of such a function $\zeta$ is a consequence of the fact that
  $\zeta$ is bounded (by $\forall x \in \mathbbm{R}^2, | \zeta (x) | \leqslant
  \frac{K \varepsilon_{f, \alpha}}{(1 + \tilde{r})^{\alpha}}$), the linearity
  of the Laplacian, and that the only weak solution to $\Delta \zeta = 0$ that
  tends to $0$ at infinity is $0$. We define
  \[ \zeta \assign G \ast f, \]
  where $G$ is the fundamental solution of the Laplacian in dimension 2,
  namely $G (x) \assign \frac{1}{2 \pi} \ln (| x |)$. Since $\| f (x) (1 +
  \tilde{r})^{2 + \alpha} \|_{L^{\infty} (\mathbbm{R}^2)} < + \infty$, we
  check that $\zeta$ is well defined. Let us show that $\zeta \in C^1
  (\mathbbm{R}^2, \mathbbm{C})$. If $f \in C^{\infty}_c (\mathbbm{R}^2)$,
  then, for $j \in \{ 1, 2 \}$,
  \begin{eqnarray*}
    \frac{\zeta (x + h \vec{e}_j) - \zeta (x)}{| h |} & = & \frac{1}{2 \pi}
    \int_{\mathbbm{R}^2} \nobracket \ln (| x - Y |)) \frac{f (Y + h \vec{e}_j)
    - f (Y)}{| h |} d Y\\
    & \rightarrow & \frac{1}{2 \pi} \int_{\mathbbm{R}^2} \nobracket \ln (| x
    - Y |)) \partial_{Y_j} f (Y) d Y
  \end{eqnarray*}
  when $| h | \rightarrow 0$. Then, for $\varepsilon > 0$,
  \[ \left| \frac{1}{2 \pi} \int_{B (x, \varepsilon)} \nobracket \ln (| x - Y
     |)) \partial_{Y_j} f (Y) d Y \right| \leqslant K \varepsilon^2 | \ln
     (\varepsilon) | \| \nabla f \|_{L^{\infty} (\mathbbm{R}^2)} \]
  and by integration by parts,
  \begin{eqnarray*}
    \frac{1}{2 \pi} \int_{\mathbbm{R}^2 \backslash B (x, \varepsilon)}
    \nobracket \ln (| x - Y |)) \partial_{Y_j} f (Y) d Y & = & \frac{1}{2 \pi}
    \int_{\mathbbm{R}^2 \backslash B (x, \varepsilon)} \frac{x_j - Y_j}{| x -
    Y |^2} f (Y) d Y\\
    & - & \frac{1}{2 \pi} \int_{\partial B (x, \varepsilon)} \nobracket \ln
    (| x - Y |)) f (Y) \vec{e}_j . \vec{\nu} d \sigma
  \end{eqnarray*}
  and since $\left| \frac{1}{2 \pi} \int_{\partial B (x, \varepsilon)}
  \nobracket \ln (| x - Y |)) f (Y) \vec{e}_j . \vec{\nu} d \sigma \right|
  \leqslant K \| f \|_{L^{\infty} (\mathbbm{R}^2)} \varepsilon | \ln
  (\varepsilon) |$, taking $\varepsilon \rightarrow 0$ we deduce that
  \[ \frac{\zeta (x + h \vec{e}_j) - \zeta (x)}{| h |} \rightarrow \frac{1}{2
     \pi} \int_{\mathbbm{R}^2} \nobracket \ln (| x - Y |)) \partial_{Y_j} f
     (Y) d Y = \frac{1}{2 \pi} \int_{\mathbbm{R}^2} \frac{x_j - Y_j}{| x - Y
     |^2} f (Y) d Y \]
  when $| h | \rightarrow 0$. This implies that, for $f \in C^{\infty}_c
  (\mathbbm{R}^2)$,
  \[ \nabla \zeta (x) = \frac{1}{2 \pi} \int_{\mathbbm{R}^2} \frac{x - Y}{| x
     - Y |^2} f (Y) d Y. \]
  Now, for $f \in C^0 (\mathbbm{R}^2, \mathbbm{C})$ such that $\| f (x) (1 +
  \tilde{r})^{2 + \alpha} \|_{L^{\infty} (\mathbbm{R}^2)} < + \infty$, we take
  $f_n \in C^{\infty}_c (\mathbbm{R}^2, \mathbbm{C})$ such that $f_n
  \rightarrow f$ in $L^3 (\mathbbm{R}^2)$ and $(1 + \tilde{r})^{\alpha / 2}
  f_n \rightarrow (1 + \tilde{r})^{\alpha / 2} f$ in $L^1 (\mathbbm{R}^2)$ (we
  check easily that $f \in L^3 (\mathbbm{R}^2)$ and $(1 + \tilde{r})^{\alpha /
  2} f \in L^1 (\mathbbm{R}^2)$). In particular, $f_n \rightarrow f$ in $L^1
  (\mathbbm{R}^2)$. Then, for $\zeta_n$ such that $\Delta \zeta_n = f_n$, we
  check that, by H{\"o}lder inequality,
  \[ \left| \nabla \zeta_n (x) - \frac{1}{2 \pi} \int_{\mathbbm{R}^2} \frac{x
     - Y}{| x - Y |^2} f (Y) d Y \right| \leqslant \frac{1}{2 \pi}
     \int_{\mathbbm{R}^2} \frac{| \nobracket f_n (Y) - f (Y \nobracket) |}{| x
     - Y |} d Y, \]
  \[ \int_{\{ | x - Y | \leqslant 1 \}} \frac{| \nobracket f_n (Y) - f (Y
     \nobracket) |}{| x - Y |} d Y \leqslant \| f_n - f \|_{L^3
     (\mathbbm{R}^2)} \left( \int_{\{ | x - Y | \leqslant 1 \}} \frac{d Y}{| x
     - Y |^{3 / 2}} \right)^{2 / 3} \leqslant K \| f_n - f \|_{L^3
     (\mathbbm{R}^2)} \]
  and
  \[ \int_{\{ | x - Y | \geqslant 1 \}} \frac{| \nobracket f_n (Y) - f (Y
     \nobracket) |}{| x - Y |} d Y \leqslant \| f_n - f \|_{L^1
     (\mathbbm{R}^2)}, \]
  therefore $\nabla \zeta_n \rightarrow \frac{1}{2 \pi} \int_{\mathbbm{R}^2}
  \frac{x - Y}{| x - Y |^2} f (Y) d Y$ uniformly in $\mathbbm{R}^2$.
  
  Similarly, we estimate
  \[ \left| \zeta_n (x) - \frac{1}{2 \pi} \int_{\mathbbm{R}^2} \ln (| x - Y |)
     f (Y) d Y \right| \leqslant \frac{1}{2 \pi} \int_{\mathbbm{R}^2} | f_n
     (Y) - f (Y) |  | \ln (| x - Y |) | d Y, \]
  \begin{eqnarray*}
    \int_{\{ | x - Y | \leqslant 1 \}} | f_n (Y) - f (Y) |  | \ln (| x - Y |)
    | d Y & \leqslant & \| f_n - f \|_{L^3 (\mathbbm{R}^2)} \left( \int_{\{ |
    x - Y | \leqslant 1 \}}  | \ln (| x - Y |) |^{3 / 2} d Y \right)^{2 / 3}\\
    & \leqslant & K \| f_n - f \|_{L^3 (\mathbbm{R}^2)}
  \end{eqnarray*}
  and
  \[ \int_{\{ | x - Y | \geqslant 1 \}} | f_n (Y) - f (Y) | | \ln (| x - Y |)
     | d Y \leqslant K \| (1 + \tilde{r})^{\alpha / 2} f_n - (1 +
     \tilde{r})^{\alpha / 2} f \|_{L^1 (\mathbbm{R}^2)}, \]
  thus $\zeta_n \rightarrow G \ast f = \zeta$ uniformly in $\mathbbm{R}^2$,
  which implies by differentiation of a sequence of functions, that $\zeta \in
  C^1 (\mathbbm{R}^2, \mathbbm{C})$ and
  \[ \nabla \zeta (x) = \frac{1}{2 \pi} \int_{\mathbbm{R}^2} \frac{x - Y}{| x
     - Y |^2} f (Y) d Y. \]

  We check that $\zeta$ satisfies
  \[ \Delta \zeta = f \]
  in the distribution sense. Indeed, for $\varphi \in C^{\infty}_c
  (\mathbbm{R}^2)$, (see {\cite{MR2597943}}, chapter 2, Theorem 1)
  \[ \int_{\mathbbm{R}^2} (G \ast f) \Delta \varphi = \int_{\mathbbm{R}^2} f
     (G \ast \Delta \varphi) = \int_{\mathbbm{R}^2} f \varphi . \]
  It is also easy to check that
  \[ \forall (x_1, x_2) \in \mathbbm{R}^2, \quad \zeta (x_1, x_2) = - \zeta
     (x_1, - x_2) . \]
  Now, if $| x - d \vec{e}_1 | \leqslant 1$, we check that
  \[ | \nabla \zeta (x) | \leqslant K \int_{\mathbbm{R}^2} \frac{1}{| Y |} | f
     (x - Y) | d Y \leqslant K \varepsilon_{f, \alpha} \int_{\mathbbm{R}^2}
     \frac{d Y}{| Y | (1 + \tilde{r} (Y - x))^{2 + \alpha}} \leqslant K
     \varepsilon_{f, \alpha}, \]
  and, similarly,
  \[ | \zeta (x) | \leqslant K \varepsilon_{f, \alpha}, \]
  which is enough to show the required estimate of this lemma for these values
  of $x$. We can make the same estimate if $| x + d \overrightarrow{e_1} |
  \leqslant 1$, we therefore suppose from now on that $| x - d
  \overrightarrow{e_1} |, | x + d \overrightarrow{e_1} | \geqslant 1$.
  
  First, let us show that
  \begin{equation}
    \int_{\{ Y_1 \geqslant 0 \}} f (Y) d Y = \int_{\{ Y_1 \leqslant 0 \}} f
    (Y) d Y = 0. \label{vert}
  \end{equation}
  The integrals are well defined because $| f (x) | \leqslant
  \frac{\varepsilon_{f, \alpha}}{(1 + \tilde{r})^{2 + \alpha}}$ and therefore
  $f$ is integrable. Since $f$ is odd with respect to $x_2$, (\ref{vert})
  holds. We deduce that
  \begin{eqnarray*}
    | \nabla \zeta (x) | & \leqslant & \frac{1}{2 \pi} \left| \int_{\{ Y_1
    \geqslant 0 \}} \left( \frac{x - Y}{| x - Y |^2} - \frac{x - d
    \overrightarrow{e_1}}{| x - d \overrightarrow{e_1} |^2} \right) f (Y) d Y
    \right|\\
    & + & \frac{1}{2 \pi} \left| \int_{\{ Y_1 \leqslant 0 \}} \left( \frac{x
    - Y}{| x - Y |^2} - \frac{x + d \overrightarrow{e_1}}{| x + d
    \overrightarrow{e_1} |^2} \right) f (Y) d Y \right| .
  \end{eqnarray*}
  Now, using $| f (x) | \leqslant \frac{\varepsilon_{f, \alpha}}{(1 +
  \tilde{r})^{2 + \alpha}}$, we estimate
  \[ \begin{array}{lll}
       2 \pi | \nabla \zeta (x) | & \leqslant & \varepsilon_{f, \alpha}
       \int_{\{ Y_1 \geqslant 0 \}} \left| \frac{x - Y}{| x - Y |^2} - \frac{x
       - d \overrightarrow{e_1}}{| x - d \overrightarrow{e_1} |^2} \right|
       \frac{d Y}{(1 + r_1 (Y))^{2 + \alpha}}\\
       & + & \varepsilon_{f, \alpha} \int_{\{ Y_1 \leqslant 0 \}} \left|
       \frac{x - Y}{| x - Y |^2} - \frac{x + d \overrightarrow{e_1}}{| x + d
       \overrightarrow{e_1} |^2} \right| \frac{d Y}{(1 + r_{- 1} (Y))^{2 +
       \alpha}} .
     \end{array} \]
  By the change of variable $Y = Z + d \overrightarrow{e_1}$, we have
  \begin{eqnarray*}
    &  & \int_{\{ Y_1 \geqslant 0 \}} \left| \frac{x - Y}{| x - Y |^2} -
    \frac{x - d \overrightarrow{e_1}}{| x - d \overrightarrow{e_1} |^2}
    \right| \frac{d Y}{(1 + r_1 (Y))^{2 + \alpha}}\\
    & = & \int_{\{ Z_1 \geqslant - d \}} \left| \frac{(x - d
    \overrightarrow{e_1}) - Z}{| (x - d \overrightarrow{e_1}) - Z |^2} -
    \frac{x - d \overrightarrow{e_1}}{| x - d \overrightarrow{e_1} |^2}
    \right| \frac{d Z}{(1 + | Z |)^{2 + \alpha}},\\
    & \leqslant & \int_{\mathbbm{R}^2} \left| \frac{(x - d
    \overrightarrow{e_1}) - Z}{| (x - d \overrightarrow{e_1}) - Z |^2} -
    \frac{x - d \overrightarrow{e_1}}{| x - d \overrightarrow{e_1} |^2}
    \right| \frac{d Z}{(1 + | Z |)^{2 + \alpha}} .
  \end{eqnarray*}
  Now, if $| Z | \geqslant 2 | x - d \overrightarrow{e_1} |$, by triangular
  inequality, we check that
  \[ \left| \frac{(x - d \overrightarrow{e_1}) - Z}{| (x - d
     \overrightarrow{e_1}) - Z |^2} - \frac{x - d \overrightarrow{e_1}}{| x -
     d \overrightarrow{e_1} |^2} \right| \leqslant \frac{K}{| x - d \vec{e}_1
     |}, \]
  hence
  \begin{eqnarray}
    &  & \int_{\{ | Z | \geqslant 2 | x - d \overrightarrow{e_1} | \}} \left|
    \frac{(x - d \overrightarrow{e_1}) - Z}{| (x - d \overrightarrow{e_1}) - Z
    |^2} - \frac{x - d \overrightarrow{e_1}}{| x - d \overrightarrow{e_1} |^2}
    \right| \frac{d Z}{(1 + | Z |)^{2 + \alpha}} \nonumber\\
    & \leqslant & \frac{K}{| x - d \overrightarrow{e_1} |} \int_{\{ | Z |
    \geqslant 2 | x - d \overrightarrow{e_1} | \}} \frac{d Z}{(1 + | Z |)^{2 +
    \alpha}} \leqslant \frac{K (\alpha)}{| x - d \overrightarrow{e_1} |^{1 +
    \alpha}} . \label{r1} 
  \end{eqnarray}
  We now work for $| Z | \leqslant 2 | x - d \overrightarrow{e_1} |$. We
  remark that
  \begin{eqnarray*}
    &  & \left| \frac{(x - d \overrightarrow{e_1}) - Z}{| (x - d
    \overrightarrow{e_1}) - Z |^2} - \frac{x - d \overrightarrow{e_1}}{| x - d
    \overrightarrow{e_1} |^2} \right| | x - d \overrightarrow{e_1} |^2 | (x -
    d \overrightarrow{e_1}) - Z |^2\\
    & = & | (x - d \overrightarrow{e_1}) (| x - d \overrightarrow{e_1} |^2 -
    | (x - d \overrightarrow{e_1}) - Z |^2) - Z | x - d \overrightarrow{e_1}
    |^2 |\\
    & = & | (x - d \overrightarrow{e_1}) (2 (x - d \overrightarrow{e_1}) .
    \bar{Z} - | Z |^2) - Z | x - d \overrightarrow{e_1} |^2 |\\
    & = & | ((x - d \overrightarrow{e_1}) - Z) (2 (x - d
    \overrightarrow{e_1}) . \bar{Z} - | Z |^2) - Z | (x - d
    \overrightarrow{e_1}) - Z |^2 |\\
    & = & | (x - d \overrightarrow{e_1}) - Z | | Z | \left| \frac{(x - d
    \overrightarrow{e_1}) - Z}{| (x - d \overrightarrow{e_1}) - Z |} \left( 2
    (x - d \overrightarrow{e_1}) . \frac{\bar{Z}}{| Z |} - | Z | \right) -
    \frac{Z}{| Z |} | (x - d \overrightarrow{e_1}) - Z | \right|,
  \end{eqnarray*}
  and we estimate
  \begin{eqnarray*}
    &  & \left| \frac{(x - d \overrightarrow{e_1}) - Z}{| (x - d
    \overrightarrow{e_1}) - Z |} \left( 2 (x - d \overrightarrow{e_1}) .
    \frac{\bar{Z}}{| Z |} - | Z | \right) - \frac{Z}{| Z |} | (x - d
    \overrightarrow{e_1}) - Z | \right|\\
    & \leqslant & 2 | x - d \overrightarrow{e_1} | + \left| \frac{(x - d
    \overrightarrow{e_1}) - Z}{| (x - d \overrightarrow{e_1}) - Z |} (- | Z |)
    - \frac{Z}{| Z |} | (x - d \overrightarrow{e_1}) - Z | \right| .
  \end{eqnarray*}
  Furthermore,
  \begin{eqnarray*}
    &  & \left| \frac{(x - d \overrightarrow{e_1}) - Z}{| (x - d
    \overrightarrow{e_1}) - Z |} (- | Z |) - \frac{Z}{| Z |} | (x - d
    \overrightarrow{e_1}) - Z | \right|^2 | Z |^2 | (x - d
    \overrightarrow{e_1}) - Z |^2\\
    & = & | ((x - d \overrightarrow{e_1}) - Z) | Z |^2 + Z | (x - d
    \overrightarrow{e_1}) - Z |^2 |^2\\
    & = & | (x - d \overrightarrow{e_1}) - Z |^2 | Z |^4 + | Z |^2 | (x - d
    \overrightarrow{e_1}) - Z |^4 + 2 (\nobracket x - d \overrightarrow{e_1} -
    Z) . \bar{Z} \nobracket | Z |^2 | (x - d \overrightarrow{e_1}) - Z |^2\\
    & = & | (x - d \overrightarrow{e_1}) - Z |^2 | Z |^2 (- | Z |^2 + | (x -
    d \overrightarrow{e_1}) - Z |^2 + 2 \nobracket (x - d
    \overrightarrow{e_1}) . \bar{Z}) \nobracket\\
    & = & | (x - d \overrightarrow{e_1}) - Z |^2 | Z |^2 | x - d
    \overrightarrow{e_1} |^2,
  \end{eqnarray*}
  therefore
  \[ \left| \frac{(x - d \overrightarrow{e_1}) - Z}{| (x - d
     \overrightarrow{e_1}) - Z |^2} - \frac{x - d \overrightarrow{e_1}}{| x -
     d \overrightarrow{e_1} |^2} \right| \leqslant \frac{3 | Z |}{| x - d
     \overrightarrow{e_1} | \times | (x - d \overrightarrow{e_1}) - Z |} . \]
  We deduce that
  \begin{eqnarray*}
    &  & \int_{\{ | Z | \leqslant 2 | x - d \overrightarrow{e_1} | \}} \left|
    \frac{(x - d \overrightarrow{e_1}) - Z}{| (x - d \overrightarrow{e_1}) - Z
    |^2} - \frac{x - d \overrightarrow{e_1}}{| x - d \overrightarrow{e_1} |^2}
    \right| \frac{d Z}{(1 + | Z |)^{2 + \alpha}}\\
    & \leqslant & \frac{3}{| x - d \overrightarrow{e_1} |} \int_{\{ | Z |
    \leqslant 2 | x - d \overrightarrow{e_1} | \}} \frac{| Z | d Z}{| (x - d
    \overrightarrow{e_1}) - Z | (1 + | Z |)^{2 + \alpha}} .
  \end{eqnarray*}
  We remark that, either $| (x - d \overrightarrow{e_1}) - Z | \geqslant
  \frac{| x - d \overrightarrow{e_1} |}{2}$, and then
  \begin{eqnarray*}
    &  & \int_{\{ | Z | \leqslant 2 | x - d \overrightarrow{e_1} | \} \cap
    \left\{ | (x - d \overrightarrow{e_1}) - Z | \geqslant \frac{| x - d
    \overrightarrow{e_1} |}{2} \right\}} \frac{| Z | d Z}{| (x - d
    \overrightarrow{e_1}) - Z | (1 + | Z |)^{2 + \alpha}}\\
    & \leqslant & \frac{2}{| x - d \overrightarrow{e_1} |} \int_{\{ | Z |
    \leqslant 2 | x - d \overrightarrow{e_1} | \} \cap \left\{ | (x - d
    \overrightarrow{e_1}) - Z | \geqslant \frac{| x - d \overrightarrow{e_1}
    |}{2} \right\}} \frac{| Z | d Z}{(1 + | Z |)^{2 + \alpha}}\\
    & \leqslant & \frac{K (\alpha)}{| x - d \overrightarrow{e_1} |^{\alpha}}
  \end{eqnarray*}
  since $\alpha < 1$, or $| (x - d \overrightarrow{e_1}) - Z | \leqslant
  \frac{| x - d \overrightarrow{e_1} |}{2}$, and then $| Z | \geqslant \frac{|
  x - d \overrightarrow{e_1} |}{2}$, therefore
  \begin{eqnarray*}
    &  & \int_{\{ | Z | \leqslant 2 | x - d \overrightarrow{e_1} | \} \cap
    \left\{ | (x - d \overrightarrow{e_1}) - Z | \leqslant \frac{| x - d
    \overrightarrow{e_1} |}{2} \right\}} \frac{| Z | d Z}{| (x - d
    \overrightarrow{e_1}) - Z | (1 + | Z |)^{2 + \alpha}}\\
    & \leqslant & \int_{\left\{ \frac{| x - d_{} \overrightarrow{e_1} |}{2}
    \leqslant | Z | \leqslant 2 | x - d \overrightarrow{e_1} | \right\}}
    \frac{| Z | d Z}{| (x - d \overrightarrow{e_1}) - Z | (1 + | Z |)^{2 +
    \alpha}}\\
    & \leqslant & \frac{K}{| x - d \overrightarrow{e_1} |^{2 + \alpha}}
    \int_{\{ | Z - (x - d \overrightarrow{e_1}) | \leqslant 3 | x - d
    \overrightarrow{e_1} | \}} \frac{| Z | d Z}{| (x - d \overrightarrow{e_1})
    - Z |}\\
    & \leqslant & \frac{K}{| x - d \overrightarrow{e_1} |^{\alpha}} .
  \end{eqnarray*}
  We conclude that
  \begin{equation}
    \int_{\{ | Z | \leqslant 2 | x - d \overrightarrow{e_1} | \}} \left|
    \frac{(x - d \overrightarrow{e_1}) - Z}{| (x - d \overrightarrow{e_1}) - Z
    |^2} - \frac{x - d \overrightarrow{e_1}}{| x - d \overrightarrow{e_1} |^2}
    \right| \frac{d Z}{(1 + | Z |)^{2 + \alpha}} \leqslant \frac{K (\alpha)}{|
    x - d \overrightarrow{e_1} |^{1 + \alpha}} . \label{r2}
  \end{equation}
  Combining (\ref{r1}) and (\ref{r2}), and by symmetry, we deduce that
  \begin{eqnarray*}
    &  & \int_{\{ Y_1 \geqslant 0 \}} \left| \frac{x - Y}{| x - Y |^2} -
    \frac{x - d \overrightarrow{e_1}}{| x - d \overrightarrow{e_1} |^2}
    \right| \frac{d Y}{(1 + r_1 (Y))^{2 + \alpha}}\\
    & + & \int_{\{ Y_1 \leqslant 0 \}} \left| \frac{x - Y}{| x - Y |^2} -
    \frac{x + d \overrightarrow{e_1}}{| x + d \overrightarrow{e_1} |^2}
    \right| \frac{d Y}{(1 + r_{- 1} (Y))^{2 + \alpha}}\\
    & \leqslant & \frac{K (\alpha)}{| x - d \overrightarrow{e_1} |^{1 +
    \alpha}} + \frac{K (\alpha)}{| x + d \overrightarrow{e_1} |^{1 +
    \alpha}}\\
    & \leqslant & \frac{K (\alpha)}{\tilde{r} (x)^{1 + \alpha}},
  \end{eqnarray*}
  and therefore (recall that $| x - d \overrightarrow{e_1} |, | x + d
  \overrightarrow{e_1} | \geqslant 1$),
  \[ | \nabla \zeta (x) | \leqslant \frac{K \varepsilon_{f, \alpha}}{(1 +
     \tilde{r} (x))^{1 + \alpha}} . \]

  Now, let us show that $\zeta (x) \rightarrow 0$ when $| x | \rightarrow
  \infty$. We recall that
  \[ \zeta (x) = \frac{1}{2 \pi} \int_{\mathbbm{R}^2} \ln (| x - Y |) f (Y) d
     Y, \]
  and since $\int_{\mathbbm{R}^2} f (Y) d Y = 0$, for large values of $x$ (in
  particular $| x | \gg d$),
  \[ \zeta (x) = \frac{1}{2 \pi} \int_{\mathbbm{R}^2} \ln \left( \frac{| x - Y
     |}{| x |} \right) f (Y) d Y. \]
  If $| x - Y | \leqslant 1$, then $| f (Y) | \leqslant \frac{K
  \varepsilon_{f, \alpha}}{(1 + | x |)^{2 + \sigma}}$, hence
  \begin{eqnarray*}
    \int_{\{ | x - Y | \leqslant 1 \}} \left| \ln \left( \frac{| x - Y |}{| x
    |} \right) f (Y) \right| & \leqslant & \frac{K \varepsilon_{f, \alpha}}{(1
    + | x |)^{2 + \sigma}} \int_{\{ | x - Y | \leqslant 1 \}} | \ln (| x - Y
    |) - \ln (| x |) |\\
    & \leqslant & \frac{K \varepsilon_{f, \alpha} (1 + \ln (| x |))}{(1 + | x
    |)^{2 + \sigma}} \rightarrow 0
  \end{eqnarray*}
  when $x \rightarrow \infty$. If $| x - Y | \geqslant 1$, then $\ln \left(
  \frac{| x - Y |}{| x |} \right) \rightarrow 0$ when $| x | \rightarrow
  \infty$ and we recall that $f$ is bounded in $L^{\infty}$. We have, for $| x
  | \geqslant 2$ that $| x - Y | \leqslant | x | (| y | + 2)$ and therefore,
  for $| x - Y | \geqslant 1, | x | \geqslant 2$, $\left| \ln \left( \frac{| x
  - Y |}{| x |} \right) \right| \leqslant K \ln (| y | + 2)$, hence
  \[ \left| 1_{\{ | x - Y | \geqslant 1 \}} \ln \left( \frac{| x - Y |}{| x |}
     \right) f (Y) \right| \leqslant K \ln (| Y | + 2) f (Y) \in L^1
     (\mathbbm{R}^2, \mathbbm{C}) . \]
  By dominated convergence theorem, we deduce that $\zeta (x) \rightarrow 0$
  when $| x | \rightarrow \infty$. Now, to estimate $\zeta$, we integrate from
  infinity. For instance, in the case $x_1 \geqslant 0, x_2 \geqslant 0$, we
  estimate
  \[ | \zeta (x) | \leqslant \left| \int_{x_2}^{+ \infty} \partial_{x_2} \zeta
     (x_1, t) d t \right| \leqslant K \varepsilon_{f, \alpha} \int_{x_2}^{+
     \infty} \frac{d t}{(1 + | x_1 - d \overrightarrow{e_1} | + t)^{1 +
     \alpha}} \leqslant \frac{K \varepsilon_{f, \alpha}}{\alpha (1 + \tilde{r}
     (x))^{\alpha}} . \]
\end{proof}

\subsection{Proof of Lemma \ref{see}\label{AB2}}

\begin{proof}
  The fundamental solution of $- \Delta + 2$ in $\mathbbm{R}^2$ is $\frac{1}{2
  \pi} K_0 \left( \sqrt{2} | . | \right)$ where $K_0$ is the modified Bessel
  function of the second kind with the properties described in Lemma
  \ref{Kzero}. Since $\Psi \in H^1 (\mathbbm{R}^2)$ and the equation $- \Delta
  + 2$ is strictly elliptic, we have
  \[ \Psi = \frac{1}{2 \pi} K_0 \left( \sqrt{2} | . | \right) \ast h, \]
  therefore (using $K_0 \geqslant 0$), for $x \in \mathbbm{R}^2$,
  \[ | \Psi (x) | \leqslant K \| (1 + \tilde{r})^{\alpha} h \|_{L^{\infty}
     (\mathbbm{R}^2)} \int_{\mathbbm{R}^2} K_0 \left( \sqrt{2} | x - Y |
     \right) \frac{d Y}{(1 + \tilde{r} (Y))^{\alpha}} . \]
  If $| x - d \overrightarrow{e_1} | \leqslant 1$ or $| x + d
  \overrightarrow{e_1} | \leqslant 1$, we have
  \[ \int_{\mathbbm{R}^2} K_0 \left( \sqrt{2} | x - Y | \right) \frac{1}{(1 +
     \tilde{r} (Y))^{\alpha}} d Y \leqslant \tmmathbf{} \int_{\mathbbm{R}^2}
     K_0 \left( \sqrt{2} | x - Y | \right) d Y \leqslant \int_{\mathbbm{R}^2}
     K_0 \left( \sqrt{2} | Y | \right) d Y \leqslant K, \]
  therefore the estimate holds. We now suppose that $| x - d
  \overrightarrow{e_1} |, | x + d \overrightarrow{e_1} | \geqslant 1$. We
  decompose
  \begin{eqnarray*}
    \int_{\mathbbm{R}^2} K_0 \left( \sqrt{2} | x - Y | \right) \frac{1}{(1 +
    \tilde{r} (Y))^{\alpha}} d Y & = & \int_{\{ Y_1 \geqslant 0 \}} K_0 \left(
    \sqrt{2} | x - Y | \right) \frac{d Y}{(1 + | Y - d \overrightarrow{e_1}
    |)^{\alpha}}\\
    & + & \int_{\{ Y_1 \leqslant 0 \}} K_0 \left( \sqrt{2} | x - Y | \right)
    \frac{d Y}{(1 + | Y + d_n \overrightarrow{e_1} |)^{\alpha}},
  \end{eqnarray*}
  and we estimate, by a change of variable,
  \[ \int_{\{ Y_1 \geqslant 0 \}} K_0 \left( \sqrt{2} | x - Y | \right)
     \frac{d Y}{(1 + | Y - d \overrightarrow{e_1} |)^{\alpha}} \leqslant
     \int_{\mathbbm{R}^2} K_0 \left( \sqrt{2} | Y | \right) \frac{d Y}{(1 + |
     x - d \overrightarrow{e_1} - Y |)^{\alpha}} . \]
  Now, if $| Y | \leqslant \frac{| x - d \overrightarrow{e_1} |}{2}$, by Lemma
  \ref{Kzero} we have
  \begin{eqnarray*}
    &  & \int_{\left\{ | Y | \leqslant \frac{| x - d \overrightarrow{e_1}
    |}{2} \right\}} K_0 \left( \sqrt{2} | Y | \right) \frac{d Y}{(1 + | x - d
    \overrightarrow{e_1} - Y |)^{\alpha}}\\
    & \leqslant & \frac{K}{(1 + | x - d \overrightarrow{e_1} |)^{\alpha}}
    \int_{\left\{ | Y | \leqslant \frac{| x - d \overrightarrow{e_1} |}{2}
    \right\}} K_0 \left( \sqrt{2} | Y | \right) d Y\\
    & \leqslant & \frac{K}{(1 + | x - d \overrightarrow{e_1} |)^{\alpha}} .
  \end{eqnarray*}
  If $| Y | \geqslant \frac{| x - d_n \overrightarrow{e_1} |}{2}$, by Lemma
  \ref{Kzero} we have
  \begin{eqnarray*}
    &  & \int_{\left\{ | Y | \geqslant \frac{| x - d \overrightarrow{e_1}
    |}{2} \right\}} K_0 \left( \sqrt{2} | Y | \right) \frac{d Y}{(1 + | x - d
    \overrightarrow{e_1} - Y |)^{\alpha}}\\
    & \leqslant & K e^{- | x - d \overrightarrow{e_1} | / 4} \int_{\left\{ |
    Y | \geqslant \frac{| x - d \overrightarrow{e_1} |}{2} \right\}} e^{- | Y
    | / 4} d Y\\
    & \leqslant & \frac{K (\alpha)}{(1 + | x - d \overrightarrow{e_1}
    |)^{\alpha}} .
  \end{eqnarray*}
  By symmetry, we have
  \[ \int_{\{ Y_1 \leqslant 0 \}} K_0 \left( \sqrt{2} | x - Y | \right)
     \frac{d Y}{(1 + | Y + d \overrightarrow{e_1} |)^{\alpha}} \leqslant
     \frac{K}{(1 + | x + d \overrightarrow{e_1} |)^{\alpha}}, \]
  and this shows that
  \begin{equation}
    | \Psi (x) | \leqslant \frac{K (\alpha) \| (1 + \tilde{r})^{\alpha} h
    \|_{L^{\infty} (\mathbbm{R}^2)}}{(1 + \tilde{r} (x))^{\alpha}} .
    \label{ett3}
  \end{equation}
  For $\nabla \Psi$, we have the similar integral form
  \[ \nabla \Psi = \frac{1}{2 \pi} \nabla \left( K_0 \left( \sqrt{2} | . |
     \right) \right) \ast h. \]
  Once again, we can show the estimate if $| x - d \overrightarrow{e_1} |
  \leqslant 1$ or $| x + d \overrightarrow{e_1} | \leqslant 1$, and otherwise,
  we estimate as previously
  \begin{eqnarray*}
    | \nabla \Psi (x) | & \leqslant & K \| (1 + \tilde{r})^{\alpha} h
    \|_{L^{\infty} (\mathbbm{R}^2)} \int_{\mathbbm{R}^2} \left| \nabla K_0
    \left( \sqrt{2} | x - Y | \right) \right| \frac{1}{(1 + \tilde{r}
    (Y))^{\alpha}} d Y\\
    & \leqslant & K \| (1 + \tilde{r})^{\alpha} h \|_{L^{\infty}
    (\mathbbm{R}^2)} \int_{\mathbbm{R}^2} \left. - K'_0 \left( \sqrt{2} | x -
    Y | \right) \right| \frac{1}{(1 + \tilde{r} (Y))^{\alpha}} d Y
  \end{eqnarray*}
  since $K_0' < 0$ (from Lemma \ref{Kzero}). Now, we can do the same
  computation as for the estimation of $| \Psi |$, using the properties of
  $K_0'$ instead of $K_0$ in Lemma \ref{Kzero}. The same proof works, since
  the two main ingredients were the integrability near $0$ and an exponential
  decay at infinity of $K_0$, and $- K_0'$ verifies this too. We deduce
  \begin{equation}
    | \nabla \Psi (x) | \leqslant \frac{C (\alpha) \| (1 + \tilde{r})^{\alpha}
    h \|_{L^{\infty} (\mathbbm{R}^2)}}{(1 + \tilde{r} (x))^{\alpha}} .
    \label{ett4}
  \end{equation}
\end{proof}

\subsection{Proof of Lemma \ref{P3kerest1}\label{AB3}}

\begin{proof}
  First, since $\alpha > 0$, $h \in L^p (\mathbbm{R}^2, \mathbbm{C})$ for some
  large $p > 1$ (depdending on $\alpha$), and $\nabla K, K \in L^q
  (\mathbbm{R}^2, \mathbbm{C})$ for any $\frac{4}{3} > q > 1$ by Theorem
  \ref{P3gravejatftw}, thus $K \ast h$ and $\nabla K \ast h$ are well defined.
  We only look at the estimates for $x \in \mathbbm{R}^2$ with $x_1 \geqslant
  0$. The case $x_1 \leqslant 0$ can be done similarly. In this case, we have
  $\tilde{r} (x) = | x - d_c \overrightarrow{e_1} |$.
  
  \
  
  We first look at the case $0 < \alpha < 2$. By Theorem \ref{P3gravejatftw}
  and the change of variables $z = x - y$, we have
  \begin{eqnarray}
    &  & | K \ast h | (x) \nonumber\\
    & \leqslant & C \| h (1 + \tilde{r})^{\alpha} \|_{L^{\infty}
    (\mathbbm{R}^2)} \int_{\mathbbm{R}^2} \frac{d y}{| x - y |^{1 / 2} (1 + |
    x - y |)^{3 / 2} (1 + \tilde{r} (y))^{\alpha}} \nonumber\\
    & \leqslant & C (\alpha) \| h (1 + \tilde{r})^{\alpha} \|_{L^{\infty}
    (\mathbbm{R}^2)} \int_{\{ y_1 \geqslant 0 \}} \frac{d y}{| x - y |^{1 / 2}
    (1 + | x - y |)^{3 / 2} (1 + | y - d \overrightarrow{e_1} |)^{\alpha}}
    \nonumber\\
    & + & C (\alpha) \| h (1 + \tilde{r})^{\alpha} \|_{L^{\infty}
    (\mathbbm{R}^2)} \int_{\{ y_1 \leqslant 0 \}} \frac{d y}{| x - y |^{1 / 2}
    (1 + | x - y |)^{3 / 2} (1 + | y + d \overrightarrow{e_1} |)^{\alpha}}
    \nonumber\\
    & \leqslant & C (\alpha) \| h (1 + \tilde{r})^{\alpha} \|_{L^{\infty}
    (\mathbbm{R}^2)} \int_{\mathbbm{R}^2} \frac{d z}{| z |^{1 / 2} (1 + | z
    |)^{3 / 2} (1 + | z - (x - d \overrightarrow{e_1}) |)^{\alpha}}
    \nonumber\\
    & + & C (\alpha) \| h (1 + \tilde{r})^{\alpha} \|_{L^{\infty}
    (\mathbbm{R}^2)} \int_{\mathbbm{R}^2} \frac{d z}{| z |^{1 / 2} (1 + | z
    |)^{3 / 2} (1 + | z - (x + d \overrightarrow{e_1}) |)^{\alpha}} . 
    \label{P3070133}
  \end{eqnarray}
  We focus on the estimation of $\int_{\mathbbm{R}^2} \frac{d z}{| z |^{1 / 2}
  (1 + | z |)^{3 / 2} (1 + | z - (x - d \overrightarrow{e_1}) |)^{\alpha}}$.
  If $| x - d \overrightarrow{e_1} | \leqslant 1$, since $\alpha > 0$,
  \[ \int_{\mathbbm{R}^2} \frac{d z}{| z |^{1 / 2} (1 + | z |)^{3 / 2} (1 + |
     z - (x - d \overrightarrow{e_1}) |)^{\alpha}} \leqslant C (\alpha)
     \int_{\mathbbm{R}^2} \frac{d z}{| z |^{1 / 2} (1 + | z |)^{3 / 2} (1 + |
     z |)^{\alpha}} \leqslant C (\alpha) . \]
  Now, for $| x - d \overrightarrow{e_1} | \geqslant 1$, we decompose
  \begin{eqnarray*}
    &  & \int_{\mathbbm{R}^2} \frac{d z}{| z |^{1 / 2} (1 + | z |)^{3 / 2} (1
    + | z - (x - d \overrightarrow{e_1}) |)^{\alpha}}\\
    & = & \int_{\left\{ | z | \leqslant \frac{| x - d \overrightarrow{e_1}
    |}{2} \right\}} \frac{d z}{| z |^{1 / 2} (1 + | z |)^{3 / 2} (1 + | z - (x
    - d \overrightarrow{e_1}) |)^{\alpha}}\\
    & + & \int_{\left\{ | z | \geqslant \frac{| x - d \overrightarrow{e_1}
    |}{2} \right\}} \frac{d z}{| z |^{1 / 2} (1 + | z |)^{3 / 2} (1 + | z - (x
    - d \overrightarrow{e_1}) |)^{\alpha}} .
  \end{eqnarray*}
  In $\left\{ | z | \leqslant \frac{| x - d \overrightarrow{e_1} |}{2}
  \right\}$, we have $| z - (x - d \overrightarrow{e_1}) | \geqslant \frac{| x
  - d \overrightarrow{e_1} |}{2}$ and $| z - (x - d \overrightarrow{e_1}) |
  \geqslant | z |$, thus, since $\alpha - \alpha' > 0$ and $| x - d
  \overrightarrow{e_1} | \geqslant 1$,
  \begin{eqnarray*}
    &  & \int_{\left\{ | z | \leqslant \frac{| x - d \overrightarrow{e_1}
    |}{2} \right\}} \frac{d z}{| z |^{1 / 2} (1 + | z |)^{3 / 2} (1 + | z - (x
    - d \overrightarrow{e_1}) |)^{\alpha}}\\
    & \leqslant & \frac{C}{| x - d \overrightarrow{e_1} |^{\alpha'}}
    \int_{\mathbbm{R}^2} \frac{d z}{| z |^{1 / 2} (1 + | z |)^{3 / 2} (1 + | z
    |)^{\alpha - \alpha'}}\\
    & \leqslant & \frac{C (\alpha - \alpha')}{| x - d \overrightarrow{e_1}
    |^{\alpha'}}\\
    & \leqslant & \frac{C (\alpha - \alpha', \alpha')}{(1 + | x - d
    \overrightarrow{e_1} |)^{\alpha'}} .
  \end{eqnarray*}
  In $\left\{ | z | \geqslant \frac{| x - d \overrightarrow{e_1} |}{2}
  \right\}$, we have $| z | \geqslant \frac{| z - (x - d \overrightarrow{e_1})
  |}{3}$ since
  \[ | z - (x - d \overrightarrow{e_1}) | \leqslant | z | + | x - d
     \overrightarrow{e_1} | \leqslant | z | + 2 | z | \leqslant 3 | z |, \]
  and $| z | \geqslant K (1 + | z |)$ since $| z | \geqslant \frac{| x - d
  \overrightarrow{e_1} |}{2} \geqslant \frac{1}{2}$. We then estimate, with $0
  < \alpha' < \alpha < 2$,
  \begin{eqnarray*}
    &  & \int_{\left\{ | z | \geqslant \frac{| x - d \overrightarrow{e_1}
    |}{2} \right\}} \frac{d z}{| z |^{1 / 2} (1 + | z |)^{3 / 2} (1 + | z - (x
    - d \overrightarrow{e_1}) |)^{\alpha}}\\
    & \leqslant & \frac{C}{(1 + | x - d \overrightarrow{e_1} |)^{\alpha'}}
    \int_{\left\{ | z | \geqslant \frac{| x - d \overrightarrow{e_1} |}{2}
    \right\}} \frac{d z}{(1 + | z |)^{2 - \alpha'} (1 + | z - (x - d
    \overrightarrow{e_1}) |)^{\alpha}}\\
    & \leqslant & \frac{C (\alpha, \alpha')}{(1 + | x - d
    \overrightarrow{e_1} |)^{\alpha'}} \int_{\mathbbm{R}^2} \frac{d z}{(1 + |
    z - (x - d \overrightarrow{e_1}) |)^{2 + \alpha - \alpha'}}\\
    & \leqslant & \frac{C (\alpha, \alpha')}{(1 + | x - d
    \overrightarrow{e_1} |)^{\alpha'}} .
  \end{eqnarray*}
  With similar computations, we check that, since $x_1 \geqslant 0$,
  \[ \int_{\mathbbm{R}^2} \frac{d z}{| z |^{1 / 2} (1 + | z |)^{3 / 2} (1 + |
     z - (x + d \overrightarrow{e_1}) |)^{\alpha}} \leqslant \frac{C (\alpha -
     \alpha', \alpha')}{(1 + | x + d \overrightarrow{e_1} |)^{\alpha'}}
     \leqslant \frac{C (\alpha - \alpha', \alpha')}{(1 + | x - d
     \overrightarrow{e_1} |)^{\alpha'}} . \]
  Therefore, for $0 < \alpha < 2$, we have
  \[ | K \ast h | \leqslant \frac{C (\alpha - \alpha', \alpha') \| h (1 +
     \tilde{r})^{\alpha} \|_{L^{\infty} (\mathbbm{R}^2)}}{(1 +
     \tilde{r})^{\alpha'}} . \]
  Now, if we consider $\nabla K$ instead of $K$ and $\alpha < 3$, a similar
  proof gives the result. The only change is that we now use $3 - \alpha' > 0$
  since $\alpha' < \alpha < 3$ in the estimate of the integral in $\left\{ | z
  | \geqslant \frac{| x - d \overrightarrow{e_1} |}{2} \right\}$, with the
  extra decay coming from $\nabla K$ instead of $K$.
  
  \
  
  We now look at the case $2 < \alpha < 3$ and $\int_{\mathbbm{R}^2} h = 0$.
  In particular, since $\alpha > 2$, we indeed have $h \in L^1
  (\mathbbm{R}^2)$. For $\tilde{r} (x) = | x - d \overrightarrow{e_1} |
  \leqslant 1$, the proof is the same as in the case $\alpha < 2$.
  
  We now suppose that $\tilde{r} (x) = | x - d \overrightarrow{e_1} |
  \geqslant 1$. Since $\int_{\mathbbm{R}^2} h = 0$ and $\forall x \in
  \mathbbm{R}^2, h (- x_1, x_2) = h (x_1, x_2)$, we have
  \[ \int_{\{ y_1 \leqslant 0 \}} h (y) d y = \int_{\{ y_1 \geqslant 0 \}} h
     (y) d y = 0, \]
  hence
  \[ \int_{\{ y_1 \leqslant 0 \}} K (x + d \overrightarrow{e_1}) h (y) d y =
     \int_{\{ y_1 \geqslant 0 \}} K (x - d \overrightarrow{e_1}) h (y) d y =
     0. \]
  Therefore, we decompose
  \begin{eqnarray*}
    &  & | (K \ast h) (x) |\\
    & = & \left| \int_{\mathbbm{R}^2} K (x - y) h (y) d y \right|\\
    & = & \left| \int_{\{ y_1 \geqslant 0 \}} (K (x - y) - K (x - d
    \overrightarrow{e_1})) h (y) d y \right| + \left| \int_{\{ y_1 \leqslant 0
    \}} (K (x - y) - K (x + d \overrightarrow{e_1})) h (y) d y \right|\\
    & \leqslant & \int_{\{ y_1 \geqslant 0 \} \cap \{ | y - d
    \overrightarrow{e_1} | \leqslant | x - d \overrightarrow{e_1} | / 2 \}} |
    K (x - y) - K (x - d \overrightarrow{e_1}) | | h (y) | d y\\
    & + & \int_{\{ y_1 \geqslant 0 \} \cap \{ | x - y | \leqslant | x - d
    \overrightarrow{e_1} | / 2 \}} | K (x - y) - K (x - d
    \overrightarrow{e_1}) | | h (y) | d y\\
    & + & \int_{\{ y_1 \geqslant 0 \} \cap \{ | x - y | \geqslant | x - d
    \overrightarrow{e_1} | / 2 \} \cap \{ | y - d \overrightarrow{e_1} |
    \geqslant | x - d \overrightarrow{e_1} | / 2 \}} | K (x - y) - K (x - d
    \overrightarrow{e_1}) | | h (y) | d y.\\
    & + & \int_{\{ y_1 \leqslant 0 \}} | \nobracket K (x - y) - K (x + d
    \overrightarrow{e_1} \nobracket) | | h (y) | d y.
  \end{eqnarray*}
  In $\{ y_1 \geqslant 0 \} \cap \{ | y - d \overrightarrow{e_1} | \leqslant |
  x - d \overrightarrow{e_1} | / 2 \}$, by Theorem \ref{P3gravejatftw},
  \begin{eqnarray*}
    &  & | K (x - y) - K (x - d \overrightarrow{e_1}) |\\
    & \leqslant & | K ((x - d \overrightarrow{e_1}) - (y - d
    \overrightarrow{e_1})) - K (x - d \overrightarrow{e_1}) |\\
    & \leqslant & | y - d \overrightarrow{e_1} | \left( \sup_{B (x - d
    \overrightarrow{e_1}, | x - d \overrightarrow{e_1} | / 2)} | \nabla K |
    \right)\\
    & \leqslant & \frac{C | y - d \overrightarrow{e_1} |}{(1 + | x - d
    \overrightarrow{e_1} |)^3} .
  \end{eqnarray*}
  With $| x - d \overrightarrow{e_1} | \geqslant 1$, $\alpha < 3$ and the fact
  that in $\{ y_1 \geqslant 0 \} \cap \{ | y - d \overrightarrow{e_1} |
  \leqslant | x - d \overrightarrow{e_1} | / 2 \}$, $\tilde{r} (y) = | y - d
  \overrightarrow{e_1} |$, we estimate
  \begin{eqnarray*}
    &  & \int_{\{ y_1 \geqslant 0 \} \cap \{ | y - d \overrightarrow{e_1} |
    \leqslant | x - d \overrightarrow{e_1} | / 2 \}} | K (x - y) - K (x - d
    \overrightarrow{e_1}) | | h (y) | d y\\
    & \leqslant & \int_{\{ | y - d \overrightarrow{e_1} | \leqslant | x - d
    \overrightarrow{e_1} | / 2 \}} \frac{C \| h (1 + \tilde{r})^{\alpha}
    \|_{L^{\infty} (\mathbbm{R}^2)} | y - d \overrightarrow{e_1} |}{(1 + | x -
    d \overrightarrow{e_1} |)^3 (1 + | y - d \overrightarrow{e_1} |)^{\alpha}}
    d y\\
    & \leqslant & \frac{C \| h (1 + \tilde{r})^{\alpha} \|_{L^{\infty}
    (\mathbbm{R}^2)}}{(1 + | x - d \overrightarrow{e_1} |)^3} \int_{\{ | y - d
    \overrightarrow{e_1} | \leqslant | x - d \overrightarrow{e_1} | / 2 \}}
    \frac{| y - d \overrightarrow{e_1} |}{(1 + | y - d \overrightarrow{e_1}
    |)^{\alpha}} d y\\
    & \leqslant & \frac{C \| h (1 + \tilde{r})^{\alpha} \|_{L^{\infty}
    (\mathbbm{R}^2)}}{(1 + | x - d \overrightarrow{e_1} |)^3} \int_{\{ | z |
    \leqslant | x - d \overrightarrow{e_1} | / 2 \}} \frac{| z |}{(1 + | z
    |)^{\alpha}} d z\\
    & \leqslant & \frac{C (\alpha) \| h (1 + \tilde{r})^{\alpha}
    \|_{L^{\infty} (\mathbbm{R}^2)}}{(1 + | x - d \overrightarrow{e_1} |)^3}
    \times \frac{1}{(1 + | x - d \overrightarrow{e_1} |)^{\alpha - 3}}\\
    & \leqslant & \frac{C (\alpha) \| h (1 + \tilde{r})^{\alpha}
    \|_{L^{\infty} (\mathbbm{R}^2)}}{(1 + | x - d \overrightarrow{e_1}
    |)^{\alpha}} .
  \end{eqnarray*}
  Now, in $\{ y_1 \geqslant 0 \} \cap \{ | x - y | \leqslant | x - d
  \overrightarrow{e_1} | / 2 \}$, we have $| y - d \overrightarrow{e_1} |
  \geqslant | x - d \overrightarrow{e_1} | / 2$, and thus
  \[ | h (y) | \leqslant \frac{C (\alpha) \| h (1 + \tilde{r})^{\alpha}
     \|_{L^{\infty} (\mathbbm{R}^2)}}{(1 + | x - d \overrightarrow{e_1}
     |)^{\alpha}} . \]
  We deduce that
  \begin{eqnarray*}
    &  & \int_{\{ y_1 \geqslant 0 \} \cap \{ | x - y | \leqslant | x - d
    \overrightarrow{e_1} | / 2 \}} | K (x - y) - K (x - d
    \overrightarrow{e_1}) | | h (y) | d y\\
    & \leqslant & \frac{C \| h (1 + \tilde{r})^{\alpha} \|_{L^{\infty}
    (\mathbbm{R}^2)}}{(1 + | x - d \overrightarrow{e_1} |)^{\alpha}} \int_{\{
    y_1 \geqslant 0 \} \cap \{ | x - y | \leqslant | x - d
    \overrightarrow{e_1} | / 2 \}} | K (x - y) - K (x - d
    \overrightarrow{e_1}) | d y\\
    & \leqslant & \frac{C \| h (1 + \tilde{r})^{\alpha} \|_{L^{\infty}
    (\mathbbm{R}^2)}}{(1 + | x - d \overrightarrow{e_1} |)^{\alpha}} \left(
    \int_{\{ | x - y | \leqslant | x - d \overrightarrow{e_1} | / 2 \}} | K (x
    - y) | d y + | K (x - d \overrightarrow{e_1}) | \int_{\{ | x - y |
    \leqslant | x - d \overrightarrow{e_1} | / 2 \}} d y \right)\\
    & \leqslant & \frac{C \| h (1 + \tilde{r})^{\alpha} \|_{L^{\infty}
    (\mathbbm{R}^2)}}{(1 + | x - d \overrightarrow{e_1} |)^{\alpha}} \left(
    \int_{\{ | z | \leqslant | x - d \overrightarrow{e_1} | / 2 \}} | K (z) |
    d z + | K (x - d \overrightarrow{e_1}) | | x - d \overrightarrow{e_1} |^2
    \right)\\
    & \leqslant & \frac{C \| h (1 + \tilde{r})^{\alpha} \|_{L^{\infty}
    (\mathbbm{R}^2)}}{(1 + | x - d \overrightarrow{e_1} |)^{\alpha}} (\ln (1 +
    | x - d \overrightarrow{e_1} |) + 1)\\
    & \leqslant & \frac{C (\alpha - \alpha') \| h (1 + \tilde{r})^{\alpha}
    \|_{L^{\infty} (\mathbbm{R}^2)}}{(1 + | x - d \overrightarrow{e_1}
    |)^{\alpha'}}
  \end{eqnarray*}
  since $| x - d \overrightarrow{e_1} | \geqslant 1$.
  
  Now, in $\{ y_1 \geqslant 0 \} \cap \{ | x - y | \geqslant | x - d
  \overrightarrow{e_1} | / 2 \} \cap \{ | y - d \overrightarrow{e_1} |
  \geqslant | x - d \overrightarrow{e_1} | / 2 \}$, we have
  \[ | K (x - y) - K (x - d \overrightarrow{e_1}) | \leqslant | K (x - y) | +
     | K (x - d \overrightarrow{e_1}) | \leqslant \frac{C}{(1 + | x - d
     \overrightarrow{e_1} |)^2} \]
  and
  \[ | h (y) | \leqslant \frac{\| h (1 + \tilde{r})^{\alpha} \|_{L^{\infty}
     (\mathbbm{R}^2)}}{(1 + | x - d \overrightarrow{e_1} |)^{\alpha}}, \]
  as well as
  \[ | h (y) | \leqslant \frac{\| h (1 + \tilde{r})^{\alpha} \|_{L^{\infty}
     (\mathbbm{R}^2)}}{(1 + | y - d \overrightarrow{e_1} |)^{\alpha}} . \]
  We deduce, since $\alpha - \alpha' > 0$, that
  \begin{eqnarray*}
    &  & \int_{\{ y_1 \geqslant 0 \} \cap \{ | x - y | \geqslant | x - d
    \overrightarrow{e_1} | / 2 \} \cap \{ | y - d \overrightarrow{e_1} |
    \geqslant | x - d \overrightarrow{e_1} | / 2 \}} | K (x - y) - K (x - d
    \overrightarrow{e_1}) | | h (y) | d y.\\
    & \leqslant & \frac{C \| h (1 + \tilde{r})^{\alpha} \|_{L^{\infty}
    (\mathbbm{R}^2)}}{(1 + | x - d \overrightarrow{e_1} |)^{2 + (\alpha' -
    2)}} \int_{\mathbbm{R}^2} \frac{d y}{(1 + | y - d \overrightarrow{e_1}
    |)^{\alpha - \alpha' + 2}}\\
    & \leqslant & \frac{C (\alpha - \alpha') \| h (1 + \tilde{r})^{\alpha}
    \|_{L^{\infty} (\mathbbm{R}^2)}}{(1 + | x - d \overrightarrow{e_1}
    |)^{\alpha'}} .
  \end{eqnarray*}
  We are left with the estimation of $\int_{\{ y_1 \leqslant 0 \}} |
  \nobracket K (x - y) - K (x + d \overrightarrow{e_1} \nobracket) | | h (y) |
  d y$. We decompose it,
  \begin{eqnarray*}
    &  & \int_{\{ y_1 \leqslant 0 \}} | \nobracket K (x - y) - K (x + d
    \overrightarrow{e_1} \nobracket) | | h (y) | d y\\
    & = & \int_{\{ y_1 \leqslant 0 \} \cap \left\{ | y + d
    \overrightarrow{e_1} | \leqslant \frac{| x + d \overrightarrow{e_1} |}{2}
    \right\}} | \nobracket K (x - y) - K (x + d \overrightarrow{e_1}
    \nobracket) | | h (y) | d y\\
    & + & \int_{\{ y_1 \leqslant 0 \} \cap \left\{ | y + d
    \overrightarrow{e_1} | \geqslant \frac{| x + d \overrightarrow{e_1} |}{2}
    \right\}} | \nobracket K (x - y) - K (x + d \overrightarrow{e_1}
    \nobracket) | | h (y) | d y.
  \end{eqnarray*}
  In $\{ y_1 \leqslant 0 \} \cap \left\{ | y + d \overrightarrow{e_1} |
  \leqslant \frac{| x + d \overrightarrow{e_1} |}{2} \right\}$, we have
  \[ | h (y) | \leqslant \frac{\| h (1 + \tilde{r})^{\alpha} \|_{L^{\infty}
     (\mathbbm{R}^2)}}{(1 + | y + d \overrightarrow{e_1} |)^{\alpha}}, \]
  and
  \begin{eqnarray*}
    &  & | \nobracket K (x - y) - K (x + d \overrightarrow{e_1} \nobracket)
    |\\
    & = & | \nobracket K ((x + d \overrightarrow{e_1}) - (y + d
    \overrightarrow{e_1})) - K (x + d \overrightarrow{e_1} \nobracket) |\\
    & \leqslant & | y + d \overrightarrow{e_1} | \sup_{B (x + d
    \overrightarrow{e_1}, | x + d \overrightarrow{e_1} | / 2)} | \nabla K |\\
    & \leqslant & \frac{C | y + d \overrightarrow{e_1} |}{(1 + | x + d
    \overrightarrow{e_1} |)^3},
  \end{eqnarray*}
  thus
  \begin{eqnarray*}
    &  & \int_{\{ y_1 \leqslant 0 \} \cap \left\{ | y + d
    \overrightarrow{e_1} | \leqslant \frac{| x + d \overrightarrow{e_1} |}{2}
    \right\}} | \nobracket K (x - y) - K (x + d \overrightarrow{e_1}
    \nobracket) | | h (y) | d y\\
    & \leqslant & \frac{C \| h (1 + \tilde{r})^{\alpha} \|_{L^{\infty}
    (\mathbbm{R}^2)}}{(1 + | x + d \overrightarrow{e_1} |)^3} \int_{\left\{ |
    y + d \overrightarrow{e_1} | \leqslant \frac{| x + d \overrightarrow{e_1}
    |}{2} \right\}} \frac{| y + d \overrightarrow{e_1} |}{(1 + | y + d
    \overrightarrow{e_1} |)^{\alpha}} d y\\
    & \leqslant & \frac{C \| h (1 + \tilde{r})^{\alpha} \|_{L^{\infty}
    (\mathbbm{R}^2)}}{(1 + | x + d \overrightarrow{e_1} |)^3} \times \frac{C
    (\alpha)}{(1 + | x + d \overrightarrow{e_1} |)^{\alpha - 3}}\\
    & \leqslant & \frac{C (\alpha) \| h (1 + \tilde{r})^{\alpha}
    \|_{L^{\infty} (\mathbbm{R}^2)}}{(1 + | x + d \overrightarrow{e_1}
    |)^{\alpha}}\\
    & \leqslant & \frac{C (\alpha) \| h (1 + \tilde{r})^{\alpha}
    \|_{L^{\infty} (\mathbbm{R}^2)}}{(1 + | x - d \overrightarrow{e_1}
    |)^{\alpha}}
  \end{eqnarray*}
  since $x_1 \geqslant 0$ (which implies that $| x + d \overrightarrow{e_1} |
  \geqslant | x - d \overrightarrow{e_1} |$).
  
  Finally, in $\{ y_1 \leqslant 0 \} \cap \left\{ | y + d \overrightarrow{e_1}
  | \geqslant \frac{| x + d \overrightarrow{e_1} |}{2} \right\}$, we first
  suppose that $| x - y | \geqslant \frac{| x + d \overrightarrow{e_1} |}{2}$,
  thus
  \[ | K (x - y) - K (x + d \overrightarrow{e_1}) | \leqslant | K (x - y) | +
     | K (x + d \overrightarrow{e_1}) | \leqslant \frac{C}{(1 + | x + d
     \overrightarrow{e_1} |)^2}, \]
  and we have
  \[ | h (y) | \leqslant \frac{K (\alpha) \| h (1 + \tilde{r})^{\alpha}
     \|_{L^{\infty} (\mathbbm{R}^2)}}{(1 + | x + d \overrightarrow{e_1}
     |)^{\alpha}}, \]
  as well as
  \[ | h (y) | \leqslant \frac{K (\alpha) \| h (1 + \tilde{r})^{\alpha}
     \|_{L^{\infty} (\mathbbm{R}^2)}}{(1 + | y + d \overrightarrow{e_1}
     |)^{\alpha}} . \]
  We therefore estimate, since $\alpha - \alpha' > 0$, $| x + d
  \overrightarrow{e_1} | \geqslant | x - d \overrightarrow{e_1} |$,
  \begin{eqnarray*}
    &  & \int_{\{ y_1 \leqslant 0 \} \cap \left\{ | y + d
    \overrightarrow{e_1} | \geqslant \frac{| x + d \overrightarrow{e_1} |}{2}
    \right\} \cap \left\{ | x - y | \geqslant \frac{| x + d
    \overrightarrow{e_1} |}{2} \right\}} | \nobracket K (x - y) - K (x + d
    \overrightarrow{e_1} \nobracket) | | h (y) | d y\\
    & \leqslant & \frac{C \| h (1 + \tilde{r})^{\alpha} \|_{L^{\infty}
    (\mathbbm{R}^2)}}{(1 + | x + d \overrightarrow{e_1} |)^{2 + (\alpha' -
    2)}} \int_{\mathbbm{R}^2} \frac{1}{(1 + | y + d \overrightarrow{e_1}
    |)^{\alpha - \alpha' + 2}}\\
    & \leqslant & \frac{C (\alpha - \alpha') \| h (1 + \tilde{r})^{\alpha}
    \|_{L^{\infty} (\mathbbm{R}^2)}}{(1 + | x + d \overrightarrow{e_1}
    |)^{\alpha'}}\\
    & \leqslant & \frac{C (\alpha - \alpha') \| h (1 + \tilde{r})^{\alpha}
    \|_{L^{\infty} (\mathbbm{R}^2)}}{(1 + | x - d \overrightarrow{e_1}
    |)^{\alpha'}} .
  \end{eqnarray*}
  The other case is when $| x - y | \leqslant \frac{| x + d
  \overrightarrow{e_1} |}{2}$, where we still have $| h (y) | \leqslant
  \frac{\| h (1 + \tilde{r})^{\alpha} \|_{L^{\infty} (\mathbbm{R}^2)}}{(1 + |
  x + d \overrightarrow{e_1} |)^{\alpha}}$ and we estimate
  \begin{eqnarray*}
    &  & \int_{\{ y_1 \leqslant 0 \} \cap \{ | x - y | \leqslant | x + d
    \overrightarrow{e_1} | / 2 \}} | K (x - y) - K (x + d
    \overrightarrow{e_1}) | | h (y) | d y\\
    & \leqslant & \frac{C \| h (1 + \tilde{r})^{\alpha} \|_{L^{\infty}
    (\mathbbm{R}^2)}}{(1 + | x + d \overrightarrow{e_1} |)^{\alpha}} \int_{\{
    y_1 \leqslant 0 \} \cap \{ | x - y | \leqslant | x + d
    \overrightarrow{e_1} | / 2 \}} | K (x - y) - K (x + d
    \overrightarrow{e_1}) | d y\\
    & \leqslant & \frac{C \| h (1 + \tilde{r})^{\alpha} \|_{L^{\infty}
    (\mathbbm{R}^2)}}{(1 + | x + d \overrightarrow{e_1} |)^{\alpha}} \left(
    \int_{\{ | x - y | \leqslant | x + d \overrightarrow{e_1} | / 2 \}} | K (x
    - y) | d y + | K (x + d \overrightarrow{e_1}) | \int_{\{ | x - y |
    \leqslant | x + d \overrightarrow{e_1} | / 2 \}} d y \right)\\
    & \leqslant & \frac{C \| h (1 + \tilde{r})^{\alpha} \|_{L^{\infty}
    (\mathbbm{R}^2)}}{(1 + | x + d \overrightarrow{e_1} |)^{\alpha}} \left(
    \int_{\{ | z | \leqslant | x + d \overrightarrow{e_1} | / 2 \}} | K (z) |
    d z + | K (x + d \overrightarrow{e_1}) | | x + d \overrightarrow{e_1} |^2
    \right)\\
    & \leqslant & \frac{C \| h (1 + \tilde{r})^{\alpha} \|_{L^{\infty}
    (\mathbbm{R}^2)}}{(1 + | x + d \overrightarrow{e_1} |)^{\alpha}} (\ln (1 +
    | x + d \overrightarrow{e_1} |) + 1)\\
    & \leqslant & \frac{C (\alpha - \alpha') \| h (1 + \tilde{r})^{\alpha}
    \|_{L^{\infty} (\mathbbm{R}^2)}}{(1 + | x + d \overrightarrow{e_1}
    |)^{\alpha'}}\\
    & \leqslant & \frac{C (\alpha - \alpha') \| h (1 + \tilde{r})^{\alpha}
    \|_{L^{\infty} (\mathbbm{R}^2)}}{(1 + | x - d \overrightarrow{e_1}
    |)^{\alpha'}},
  \end{eqnarray*}
  which concludes the estimates of this lemma.
\end{proof}

\subsection{Proof of Lemma \ref{P3kerest2}\label{AB4}}

\begin{proof}
  We recall from {\cite{MR2086751}}  that
  \begin{eqnarray}
    (R_{j, k} \ast h) (x) & = & \frac{1}{2 \pi} \int_{| x - y | \geqslant 1}
    \frac{\delta_{j, k} | x - y |^2 - 2 (x - y)_j (x - y)_k}{| x - y |^4} h
    (y) d y \nonumber\\
    & + & \frac{1}{2 \pi} \int_{| x - y | \leqslant 1} \frac{\delta_{j, k} |
    x - y |^2 - 2 (x - y)_j (x - y)_k}{| x - y |^4} (h (y) - h (x)) d y. 
    \label{P3Rjk}
  \end{eqnarray}
  As in the proof of Lemma \ref{P3kerest1}, we suppose $x_1 \geqslant 0$. It
  implies that $\tilde{r} (x) = | x - d \overrightarrow{e_1} |$. The proof can
  be done similarly if $x_1 \leqslant 0$.
  
  \
  
  First, we look at the case $0 < \alpha < 2$. We check that
  \begin{eqnarray*}
    &  & \left| \int_{| x - y | \geqslant 1} \frac{\delta_{j, k} | x - y |^2
    - 2 (x - y)_j (x - y)_k}{| x - y |^4} h (y) d y \right|\\
    & \leqslant & K \int_{| x - y | \geqslant 1} \frac{| h (y) | d y}{(1 + |
    x - y |)^2}\\
    & \leqslant & K \| h (1 + \tilde{r})^{\alpha} \|_{L^{\infty}
    (\mathbbm{R}^2)} \int_{\mathbbm{R}^2} \frac{d y}{(1 + | x - y |)^2 (1 +
    \tilde{r} (y))^{\alpha}} .
  \end{eqnarray*}
  The estimate of $\int_{\mathbbm{R}^2} \frac{d y}{(1 + | x - y |)^2 (1 +
  \tilde{r} (y))^{\alpha}}$ can be done exactly as the estimate of
  \[ \int_{\mathbbm{R}^2} \frac{d y}{| x - y |^{1 / 2} (1 + | x - y |)^{3 / 2}
     (1 + \tilde{r} (y))^{\alpha}} \]
  in the proof of Lemma \ref{P3kerest1} (see equation (\ref{P3070133}) and the
  proof below). We deduce that
  \[ \left| \int_{| x - y | \geqslant 1} \frac{\delta_{j, k} | x - y |^2 - 2
     (x - y)_j (x - y)_k}{| x - y |^4} h (y) d y \right| \leqslant \frac{K
     (\alpha, \alpha') \| h (1 + \tilde{r})^{\alpha} \|_{L^{\infty}
     (\mathbbm{R}^2)}}{(1 + | x - d \overrightarrow{e_1} |)^{\alpha'}} . \]
  Now, if $| x - y | \leqslant 1$, for $0 < \alpha < 3$, we have
  \[ | h (y) - h (x) | \leqslant | y - x | \sup_{B (x, 1)} | \nabla h |
     \leqslant | y - x | \frac{\| \nabla h (1 + \tilde{r})^{\alpha}
     \|_{L^{\infty} (\mathbbm{R}^2)}}{(1 + \tilde{r} (x))^{\alpha}}, \]
  thus
  \begin{eqnarray*}
    &  & \left| \int_{| x - y | \leqslant 1} \frac{\delta_{j, k} | x - y |^2
    - 2 (x - y)_j (x - y)_k}{| x - y |^4} (h (y) - h (x)) d y \right|\\
    & \leqslant & \frac{K \| \nabla h (1 + \tilde{r})^{\alpha} \|_{L^{\infty}
    (\mathbbm{R}^2)}}{(1 + \tilde{r} (x))^{\alpha}} \int_{| x - y | \leqslant
    1} \frac{1}{| x - y |^2} | y - x | d y\\
    & \leqslant & \frac{K \| \nabla h (1 + \tilde{r})^{\alpha} \|_{L^{\infty}
    (\mathbbm{R}^2)}}{(1 + \tilde{r} (x))^{\alpha}} .
  \end{eqnarray*}
  This concludes the proof of the estimate in the case $\alpha < 2$. We now
  suppose that $2 < \alpha < 3$ and $\int_{\mathbbm{R}^2} h = 0$. We already
  have estimate the second integral in (\ref{P3Rjk}) (since the computations
  were done for $0 < \alpha < 3$), and for the first integral, the case $| x -
  d \overrightarrow{e_1} | \leqslant 1$ is done as previously.
  
  We now suppose that $| x - d \overrightarrow{e_1} | \geqslant 1$. We are
  left with the estimation of
  \[ \int_{| x - y | \geqslant 1} \frac{\delta_{j, k} | x - y |^2 - 2 (x -
     y)_j (x - y)_k}{| x - y |^4} h (y) d y. \]
  We define $F_{j, k} (z) \assign \frac{\delta_{j, k} | z |^2 - 2 z_j z_k}{| z
  |^4}$ and we check easily that, for $| z | \geqslant 1$,
  \[ | F_{j, k} (z) | \leqslant \frac{K}{| z |^2} . \]
  Since $\forall x \in \mathbbm{R}^2, h (- x_1, x_2) = h (x_1, x_2)$ and
  $\int_{\mathbbm{R}^2} h = 0$, we have
  \[ \int_{\{ y_1 \geqslant 0 \}} F_{j, k} (x - d \overrightarrow{e_1}) h (y)
     d y + \int_{\{ y_1 \leqslant 0 \}} F_{j, k} (x + d \overrightarrow{e_1})
     h (y) d y = 0. \]
  Furthermore, we estimate (since $| x - d \overrightarrow{e_1} | \geqslant
  1$)
  \begin{eqnarray*}
    &  & \int_{\{ y_1 \geqslant 0 \} \cap \{ | x - y | \leqslant 1 \}} |
    F_{j, k} (x - d \overrightarrow{e_1}) h (y) | d y\\
    & \leqslant & | F_{j, k} (x - d \overrightarrow{e_1}) | \int_{\{ y_1
    \geqslant 0 \} \cap \{ | x - y | \leqslant 1 \}} | h (y) | d y\\
    & \leqslant & \frac{K}{(1 + | x - d \overrightarrow{e_1} |)^2} \int_{\{
    y_1 \geqslant 0 \} \cap \{ | x - y | \leqslant 1 \}} | h (y) | d y.
  \end{eqnarray*}
  Now, in $\{ y_1 \geqslant 0 \} \cap \{ | x - y | \leqslant 1 \}$, we check
  that $| h (y) | \leqslant \frac{K (\alpha) \| h (1 + \tilde{r})^{\alpha}
  \|_{L^{\infty} (\mathbbm{R}^2)}}{(1 + | x - d \overrightarrow{e_1}
  |)^{\alpha}}$ and thus
  \[ \int_{\{ y_1 \geqslant 0 \} \cap \{ | x - y | \leqslant 1 \}} | F_{j, k}
     (x - d \overrightarrow{e_1}) h (y) | d y \leqslant \frac{K (\alpha) \| h
     (1 + \tilde{r})^{\alpha} \|_{L^{\infty} (\mathbbm{R}^2)}}{(1 + | x - d
     \overrightarrow{e_1} |)^{2 + \alpha}} . \]
  Similarly, since $| x + d \overrightarrow{e_1} | \leqslant | x - d
  \overrightarrow{e_1} |$ since $x_1 \geqslant 0$,
  \[ \int_{\{ y_1 \leqslant 0 \} \cap \{ | x - y | \leqslant 1 \}} | F_{j, k}
     (x + d \overrightarrow{e_1}) h (y) | d y \leqslant \frac{K (\alpha) \| h
     (1 + \tilde{r})^{\alpha} \|_{C^0 (\mathbbm{R}^2)}}{(1 + | x - d
     \overrightarrow{e_1} |)^{2 + \alpha}} . \]
  Therefore, we estimate
  \begin{eqnarray*}
    &  & \left| \int_{| x - y | \geqslant 1} \frac{\delta_{j, k} | x - y |^2
    - 2 (x - y)_j (x - y)_k}{| x - y |^4} h (y) d y \right|\\
    & \leqslant & \int_{\{ y_1 \geqslant 0 \} \cap \{ | x - y | \geqslant 1
    \}} | F_{j, k} (x - y) - F_{j, k} (x - d \overrightarrow{e_1}) | | h (y) |
    d y\\
    & + & \int_{\{ y_1 \leqslant 0 \} \cap \{ | x - y | \geqslant 1 \}} |
    F_{j, k} (x - y) - F_{j, k} (x + d \overrightarrow{e_1}) | | h (y) | d y\\
    & + & \frac{K (\alpha) \| h (1 + \tilde{r})^{\alpha} \|_{C^0
    (\mathbbm{R}^2)}}{(1 + | x - d \overrightarrow{e_1} |)^{2 + \alpha}} .
  \end{eqnarray*}
  Now, we conclude as in the proof of Lemma \ref{P3kerest1} for the estimation
  of the two remaining integrals, replacing the function $K$ by $F_{j, k}$,
  and having the domain of all integrals restricted to $\{ | x - y | \geqslant
  1 \}$. We check that, in $\{ | z | \geqslant 1 \}$,
  \[ | F_{j, k} (z) | \leqslant \frac{K}{| z |^2} \leqslant \frac{K}{(1 + | z
     |)^2}, \]
  and, in $\{ | x - y | \geqslant 1 \}$,
  \[ | F_{j, k} (x - y) - F_{j, k} (x) | \leqslant \frac{K | y |}{(1 + | x
     |)^3} . \]
  With these estimates replacing Theorem \ref{P3gravejatftw}, we can do the
  proof of the estimates as in Lemma \ref{P3kerest1}, in the case $2 < \alpha
  < 3$ and $\int_{\mathbbm{R}^2} h = 0$.
\end{proof}

\subsection{Proof of Lemma \ref{funky}}\label{AB5}

\begin{proof}
  First, we check that, as a solution of $\eta L (\Phi) + (1 - \eta) V L'
  (\Psi) = V h$, $\Phi \in C^2 (\mathbbm{R}^2, \mathbbm{C})$ and
  \[ \| \Phi \|_{L^{\infty} (\{ r < 10 / c^2 \})} + \| \nabla \Phi
     \|_{L^{\infty} (\{ r < 10 / c^2 \})} + \| \nabla^2 \Phi \|_{L^{\infty}
     (\{ r < 10 / c^2 \})} \leqslant K (c, \| \Phi \|_{H_{\infty}}, \| h
     \|_{\ast \ast, \sigma'}) < + \infty . \]
  Since $\Phi \in C^2 (\mathbbm{R}^2, \mathbbm{C})$ and it satisfies the
  symmetries and the orthogonality condition, to show that $\Phi = V \Psi \in
  \mathcal{E}_{\ast, \sigma}$, we only have to show that $\| \Psi \|_{\ast,
  \sigma, d} < + \infty$. Now, similarly as in the proof of Proposition
  \ref{invertop}, we add a cutoff function $\chi_R$, writing $\tilde{\Psi} =
  \tilde{\Psi}_1 + i \tilde{\Psi}_2 = \chi_R \Psi, \tilde{h} = \tilde{h}_1 + i
  \tilde{h}_2 = \chi_R h$ but this time its value is $1$ if $r \geqslant 10 /
  c^2$ and $0$ if $r \leqslant 5 / c^2$. In particular, its support is far
  from both vortices. We check similarly that, with the same notations, we
  obtain the equation (\ref{systss}) that we write in real and imaginary
  parts:
  \begin{equation}
    \left\{ \begin{array}{l}
      \Delta \tilde{\Psi}_1 - 2 | V |^2 \tilde{\Psi}_1 = - \tilde{h}_1 -
      2\mathfrak{R}\mathfrak{e} \left( \frac{\nabla V}{V} . \nabla
      \tilde{\Psi} \right) + c \partial_{x_2} \tilde{\Psi}_2 + \tmop{Loc}_1
      (\Psi)\\
      \Delta \tilde{\Psi}_2 + c \partial_{x_2} \tilde{\Psi}_1 = - \tilde{h}_2
      - 2\mathfrak{I}\mathfrak{m} \left( \frac{\nabla V}{V} . \nabla
      \tilde{\Psi} \right) + \tmop{Loc}_2 (\Psi),
    \end{array} \right. \label{systnew}
  \end{equation}
  where $\tmop{Loc} (\Psi) = \tmop{Loc}_1 (\Psi) + i \tmop{Loc}_2 (\Psi)$, and
  this time the local terms is in $\{ 5 / c^2 \leqslant r \leqslant 10 / c^2
  \}$. Recall that $\tilde{\Psi} = 0$ on $\{ r \leqslant 5 / c^2 \}$. In
  particular, we look only at values of $x$ such that $| x | \geqslant 5 /
  c^2$. Now, we define a function $\zeta$, solution of $\Delta \zeta = -
  \tilde{h}_2 - 2\mathfrak{I}\mathfrak{m} \left( \frac{\nabla V}{V} . \nabla
  \tilde{\Psi} \right) + \tmop{Loc}_2 (\Psi)$ as in Lemma \ref{lapzeta}. With
  Lemma \ref{nonmodV} and $\nabla \tilde{\Psi} \in L^2 (\mathbbm{R}^2)$ (since
  $\Phi \in H_{\infty}$), we have $Y \mapsto (1 + \tilde{r})^{1 / 10} (\ln | x
  - Y |) \mathfrak{I}\mathfrak{m} \left( \frac{\nabla V}{V} . \nabla
  \tilde{\Psi} \right) (Y) \in L^1 (\mathbbm{R}^2)$ (hence $Y \mapsto (\ln | x
  - Y |) \mathfrak{I}\mathfrak{m} \left( \frac{\nabla V}{V} . \nabla
  \tilde{\Psi} \right) (Y) \in L^1 (\mathbbm{R}^2)$) and thus $\zeta$ is well
  defined. By H{\"o}lder inequality, we can check that
  $\mathfrak{I}\mathfrak{m} \left( \frac{\nabla V}{V} . \nabla \tilde{\Psi}
  \right) \in L^3 (\mathbbm{R}^2)$ . We check, with the same computations as
  in the proof of Lemma \ref{lapzeta} (with $\alpha = 1 / 10$ in the
  computations), that $\zeta \in C^1 (\mathbbm{R}^2)$ and that we have
  \[ | \nabla \zeta (x) | \leqslant \frac{1}{2 \pi} \int_{\mathbbm{R}^2}
     \frac{1}{| x - Y |} \left| - \tilde{h}_2 - 2\mathfrak{I}\mathfrak{m}
     \left( \frac{\nabla V}{V} . \nabla \tilde{\Psi} \right) + \tmop{Loc}_2
     (\Psi) \right| (Y) d Y, \]
  under the condition that $\nabla \tilde{\Psi} \in L^2 (\mathbbm{R}^2) \cap
  L^3 (\mathbbm{R}^2)$. With the upcoming estimates, we will check in
  particular that this condition is satisfied (by Sobolev embedding). From the
  proof of Lemma \ref{lapzeta}, we check that, since $V h \in
  \mathcal{E}_{\ast \ast, \sigma'}$ and $\frac{1 + \sigma}{2} < 1$,
  \[ \sup_{x \in \mathbbm{R}^2} (1 + | x |)^{\frac{1 + \sigma}{2}} \left.
     \int_{\mathbbm{R}^2} \frac{1}{| x - Y |} | - \tilde{h}_2 + \tmop{Loc}_2
     (\Psi) | (Y) d Y \right| < + \infty \]
  (here, its size may depend on \ $\sigma, \sigma', c, R, \| \Phi
  \|_{H_{\infty}}$ and $\| h \|_{\ast \ast, \sigma'}$). Now, from Lemma
  \ref{nonmodV}, we have, outside of $\{ \chi_R = 0 \}$ that $| \nabla V |
  \leqslant \frac{K (c)}{(1 + r)^2}$. We deduce
  \[ \int_{\mathbbm{R}^2} \frac{1}{| x - Y |} \left| \mathfrak{I}\mathfrak{m}
     \left( \frac{\nabla V}{V} . \nabla \tilde{\Psi} \right) \right| (Y) d Y
     \leqslant K (c, R) \int_{\mathbbm{R}^2} \frac{| \nabla \tilde{\Psi} |
     (Y)}{| x - Y | (1 + | Y |)^2} d Y. \]
  We focus now on the estimation of $\int_{\mathbbm{R}^2} \frac{| \nabla
  \tilde{\Psi} | (Y)}{| x - Y | (1 + | Y |)^2} d Y$. From {\cite{MR1814364}},
  Theorem 8.8, we check that $\| \nabla \tilde{\Psi} \|_{H^1 (\mathbbm{R}^2)}
  \leqslant K (c, R, \| \Phi \|_{H_{\infty}}, \| h \|_{\ast \ast, \sigma'})$.
  In particular, by Sobolev embedding, $\| \nabla \tilde{\Psi} \|_{L^3
  (\mathbbm{R}^2)} \leqslant K (c, R, \| \Phi \|_{H_{\infty}}, \| h \|_{\ast
  \ast, \sigma'})$. In the area $\{ | x - Y | \leqslant 1 \}$, we have $(1 + |
  Y |)^2 \geqslant K (1 + | x |)^2$ and therefore, by H{\"o}lder inequality,
  \begin{eqnarray*}
    \int_{\{ | x - Y | \leqslant 1 \}} \frac{| \nabla \tilde{\Psi} | (Y)}{| x
    - Y | (1 + | Y |)^2} d Y & \leqslant & \frac{K}{(1 + | x |)^2} \int_{\{ |
    x - y | \leqslant 1 \}} \frac{| \nabla \tilde{\Psi} | (Y)}{| x - Y |} d
    Y\\
    & \leqslant & \frac{K \| \nabla \tilde{\Psi} \|_{L^3 (\mathbbm{R}^2)}}{(1
    + | x |)^2} \left( \int_{\{ | x - Y | \leqslant 1 \}} \frac{d Y}{| x - Y
    |^{3 / 2}} \right)^{2 / 3}\\
    & \leqslant & \frac{K (c, R, \| \Phi \|_{H_{\infty}}, \| h \|_{\ast \ast,
    \sigma})}{(1 + | x |)^2} .
  \end{eqnarray*}
  In the area $\{ 1 \leqslant | x - Y | \leqslant | x | / 2 \}$, we have $| Y
  | \geqslant \frac{| x - Y |}{2}$ and $| Y | \geqslant \frac{| x |}{2}$,
  therefore, by Cauchy-Schwarz (since $\frac{1 + \sigma}{2} < 1$),
  \begin{eqnarray*}
    &  & \int_{\{ 1 \leqslant | x - Y | \leqslant | x | / 2 \}} \frac{|
    \nabla \tilde{\Psi} | (Y) d Y}{| x - Y | (1 + | Y |)^2}\\
    & \leqslant & \frac{K (\sigma, c, R)}{(1 + | x |)^{\frac{1 + \sigma}{2}}}
    \int_{\{ 1 \leqslant | x - Y | \leqslant | x | / 2 \}} \frac{| \nabla
    \tilde{\Psi} | (Y) d Y}{| x - Y | (1 + | x - Y |)^{2 - \left( \frac{1 +
    \sigma}{2} \right)}}\\
    & \leqslant & \frac{K (\sigma, c, R)}{(1 + | x |)^{\frac{1 + \sigma}{2}}}
    \sqrt{\int_{\{ 1 \leqslant | x - Y | \leqslant | x | / 2 \}} | \nabla
    \tilde{\Psi} |^2 (Y) d Y \int_{\{ 1 \leqslant | x - Y | \leqslant | x | /
    2 \}} \frac{d Y}{| x - Y |^{3 - \left( \frac{1 + \sigma}{2} \right)}}}\\
    & \leqslant & \frac{K (c, R, \sigma, \| \Phi \|_{H_{\infty}})}{(1 + | x
    |)^{\frac{1 + \sigma}{2}}} .
  \end{eqnarray*}
  Finally, in the area $\{ | x - Y | \geqslant | x | / 2 \}$, we estimate by
  Cauchy-Schwarz that
  \begin{eqnarray*}
    &  & \int_{\{ | x - Y | \geqslant | x | / 2 \}} \frac{| \nabla
    \tilde{\Psi} | (Y)}{| x - Y | (1 + | Y |)^2} d Y\\
    & \leqslant & \frac{K}{1 + | x |} \sqrt{\int_{\{ | x - Y | \geqslant | x
    | / 2 \}} | \nabla \tilde{\Psi} |^2 \int_{\{ | x - Y | \geqslant | x | / 2
    \}} \frac{d Y}{(1 + | Y |)^4}}\\
    & \leqslant & \frac{K (\| \Phi \|_{H_{\infty}})}{1 + | x |} .
  \end{eqnarray*}
  Combining these estimates, we conclude that
  \[ | \nabla \zeta | (x) \leqslant \frac{K (c, R, \sigma, \sigma', \| \Phi
     \|_{H_{\infty}}, \| h \|_{\ast \ast, \sigma})}{(1 + | x |)^{\frac{1 +
     \sigma}{2}}} . \]
  Now, we write $\tilde{\Psi}'_2 = \tilde{\Psi}_2 - \zeta$, and the system
  becomes
  \[ \left\{ \begin{array}{l}
       \Delta \tilde{\Psi}_1 - 2 \tilde{\Psi}_1 - c \partial_{x_2}
       \tilde{\Psi}_2' = - \tilde{h}_1 - 2\mathfrak{R}\mathfrak{e} \left(
       \frac{\nabla V}{V} . \nabla \tilde{\Psi} \right) + \tmop{Loc}_1 (\Psi)
       - c \partial_{x_2} \zeta - 2 (1 - | V |^2) \tilde{\Psi}_1\\
       \Delta \tilde{\Psi}_2' + c \partial_{x_2} \tilde{\Psi}_1 = 0.
     \end{array} \right. \]
  We deduce, as for equation (\ref{P32932}), that for $j \in \{ 1, 2 \}$,
  \[ \partial_{x_j} \tilde{\Psi}_2' = c K_j \ast \left( - \tilde{h}_1 -
     2\mathfrak{R}\mathfrak{e} \left( \frac{\nabla V}{V} . \nabla \tilde{\Psi}
     \right) + \tmop{Loc}_1 (\Psi) - c \partial_{x_2} \zeta - 2 (1 - | V |^2)
     \tilde{\Psi}_1 \right) . \]
  We check that, with Lemma \ref{P3kerest1} (for $1 > \alpha = \frac{1 +
  \sigma}{2} > 0$, $\alpha' = \sigma < \alpha$),
  \[ | K_j \ast (- \tilde{h}_1 + \tmop{Loc}_1 (\Psi) - c \partial_{x_2} \zeta)
     | \leqslant \frac{K (c, R, \sigma, \| \Phi \|_{H_{\infty}}, \| h \|_{\ast
     \ast, \sigma'})}{(1 + | x |)^{\sigma}}, \]
  since
  \[ | - \tilde{h}_1 + \tmop{Loc}_1 (\Psi) - c \partial_{x_2} \zeta |
     \leqslant \frac{K (c, R, \sigma, \| \Phi \|_{H_{\infty}}, \| h \|_{\ast
     \ast, \sigma'})}{(1 + | x |)^{\frac{1 + \sigma}{2}}} . \]
  Furthermore, from Lemma \ref{nonmodV}, outside of $\{ \chi_R = 0 \}$, $|
  \nabla V | \leqslant \frac{K (c)}{(1 + r)^2}$. We check, with Theorem
  \ref{P3gravejatftw}, that on $\left\{ | x - Y | \leqslant \frac{| x |}{2}
  \right\}$, we have $| Y | \geqslant \frac{| x |}{2}$ and
  \begin{eqnarray*}
    &  & \int_{\{ | x - Y | \leqslant | x | / 2 \}} \left| K_j (x - Y)
    \mathfrak{R}\mathfrak{e} \left( \frac{\nabla V}{V} . \nabla \tilde{\Psi}
    \right) (Y) \right| d Y\\
    & \leqslant & \frac{K (c, R)}{(1 + | x |)^2} \int_{\{ | x - Y | \leqslant
    | x | / 2 \}} \frac{| \nabla \tilde{\Psi} | (Y) d Y}{| x - Y |^{1 / 2} (1
    + | x - Y |)^{3 / 2}} .
  \end{eqnarray*}
  By Cauchy-Schwarz, we estimate
  \begin{eqnarray*}
    &  & \int_{\{ | x - Y | \leqslant | x | / 2 \}} \frac{| \nabla
    \tilde{\Psi} | (Y) d Y}{| x - Y |^{1 / 2} (1 + | x - Y |)^{3 / 2}}\\
    & \leqslant & \| \nabla \tilde{\Psi} \|_{L^2 (\mathbbm{R}^2)}
    \sqrt{\int_{\{ | x - Y | \leqslant | x | / 2 \}} \frac{d Y}{| x - Y | (1 +
    | x - Y |)^3}}\\
    & < & + \infty,
  \end{eqnarray*}
  and in $\left\{ | x - Y | \geqslant \frac{| x |}{2} \right\}$, we estimate
  \[ \int_{\{ | x - Y | \geqslant | x | / 2 \}} \left| K_j (x - Y)
     \mathfrak{R}\mathfrak{e} \left( \frac{\nabla V}{V} . \nabla \tilde{\Psi}
     \right) (Y) \right| \leqslant \frac{K (c, R)}{(1 + | x |)^2} \int_{\{ | x
     - Y | \leqslant | x | / 2 \}} \frac{| \nabla \tilde{\Psi} | (Y) d Y}{(1 +
     | Y |)^2}, \]
  and we conclude by Cauchy-Schwarz that
  \[ \int_{\{ | x - Y | \geqslant | x | / 2 \}} \left| K_j (x - Y)
     \mathfrak{R}\mathfrak{e} \left( \frac{\nabla V}{V} . \nabla \tilde{\Psi}
     \right) (Y) \right| d Y \leqslant \frac{K (c, R, \| \Phi
     \|_{H_{\infty}})}{(1 + | x |)^2} . \]
  Since $\| \tilde{\Psi}_1 \|_{L^2 (\mathbbm{R}^2)} \leqslant K (c, R, \| \Phi
  \|_{H_{\infty}})$, we estimate similarly
  \[ \int_{\mathbbm{R}^2} | K_j (x - Y) (1 - | V |^2) \tilde{\Psi}_1 (Y) | d Y
     \leqslant \frac{K (c, R, \| \Phi \|_{H_{\infty}})}{(1 + | x |)^2}, \]
  and we conclude that $| \partial_{x_j} \tilde{\Psi}_2' | \leqslant \frac{K
  (c, R, \| \Phi \|_{H_{\infty}})}{(1 + | x |)^2}$. Therefore, since
  $\tilde{\Psi}_2 = \zeta + \tilde{\Psi}'_2$,
  \[ | \nabla \tilde{\Psi}_2 | \leqslant \frac{K (c, R, \sigma, \sigma', \|
     \Phi \|_{H_{\infty}}, \| h \|_{\ast \ast, \sigma'})}{(1 + | x
     |)^{\sigma}} . \]
  By integration from the origin (using $\| \tilde{\Psi}_2 \|_{L^{\infty} (\{
  r < 10 / c^2 \})} \leqslant K (c, \| \Phi \|_{H_{\infty}}, \| h \|_{\ast
  \ast})$), we deduce also that
  \begin{equation}
    | \tilde{\Psi}_2 | \leqslant \frac{K (c, R, \sigma, \sigma', \| \Phi
    \|_{H_{\infty}}, \| h \|_{\ast \ast, \sigma'})}{(1 + | x |)^{- 1 +
    \sigma}} . \label{limitin}
  \end{equation}
  With these estimates and the equation
  \[ \Delta \tilde{\Psi}_1 - 2 \tilde{\Psi}_1 = - \tilde{h}_1 + c
     \partial_{x_2} \tilde{\Psi}_2 + \tmop{Loc}_1 (\Psi) -
     2\mathfrak{R}\mathfrak{e} \left( \frac{\nabla V}{V} . \nabla \tilde{\Psi}
     \right) - 2 (1 - | V |^2) \tilde{\Psi}_1, \]
  we check that $| - \tilde{h}_1 + c \partial_{x_2} \tilde{\Psi}_2 +
  \tmop{Loc}_1 (\Psi) | \leqslant \frac{K (c, R, \sigma, \sigma', \| \Phi
  \|_{H_{\infty}}, \| h \|_{\ast \ast, \sigma'})}{(1 + | x |)^{\sigma}}$, and
  by Lemma \ref{see} (for $\alpha = \sigma > 0$),
  \[ | \tilde{\Psi}_1 | + | \nabla \tilde{\Psi}_1 | \leqslant \frac{K (c, R,
     \sigma, \sigma', \| \Phi \|_{H_{\infty}}, \| h \|_{\ast \ast,
     \sigma'})}{(1 + | x |)^{\sigma}} \]
  (where the estimation for the terms $\mathfrak{R}\mathfrak{e} \left(
  \frac{\nabla V}{V} . \nabla \tilde{\Psi} \right)$ and $2 (1 - | V |^2)
  \Psi_1$ are similar to what has already been done since we only have $\nabla
  \tilde{\Psi}, \Psi_1 \in L^2 (\mathbbm{R}^2)$ at this point).
  
  \
  
  With this first set of estimates, looking at equation (\ref{systnew}), we
  have enough to show that
  \[ | \Delta \tilde{\Psi}_1 - 2 \tilde{\Psi}_1 - c \partial_{x_2}
     \tilde{\Psi}_2 | \leqslant \frac{K (c, R, \sigma, \| \Phi
     \|_{H_{\infty}}, \| h \|_{\ast \ast, \sigma'})}{(1 + | x |)^{1 + \sigma}}
  \]
  and
  \[ | \Delta \tilde{\Psi}_2 + c \partial_{x_2} \tilde{\Psi}_1 | \leqslant
     \frac{K (c, R, \sigma, \| \Phi \|_{H_{\infty}}, \| h \|_{\ast \ast,
     \sigma'})}{(1 + | x |)^{2 + \sigma}} . \]
  From the computations at the beginning of subsection \ref{gpkernel}, we have
  that, for $j \in \{ 1, 2 \}$,
  \[ \partial_{x_j} \tilde{\Psi}_1 = \partial_{x_j} K_0 \ast (\Delta
     \tilde{\Psi}_1 - 2 \tilde{\Psi}_1 - c \partial_{x_2} \tilde{\Psi}_2) + c
     K_j \ast (\Delta \tilde{\Psi}_2 + c \partial_{x_2} \tilde{\Psi}_1), \]
  therefore, by Lemma \ref{P3kerest1}, taking $\alpha = 1 + \sigma < 2$ and
  $\alpha' = 1 + \sigma' < \alpha$, we have
  \[ | \nabla \tilde{\Psi}_1 | \leqslant \frac{K (c, R, \sigma, \sigma', \|
     \Phi \|_{H_{\infty}}, \| h \|_{\ast \ast, \sigma'})}{(1 + | x |)^{1 +
     \sigma'}} . \]
  Furthermore, by Lemma \ref{P3kerest1}, $| K_j \ast (\Delta \tilde{\Psi}_2 +
  c \partial_{x_2} \tilde{\Psi}_1) | \leqslant \frac{K (c, R, \sigma, \sigma',
  \| \Phi \|_{H_{\infty}})}{(1 + | x |)^{2 + \sigma / 2}}$, hence, since for
  $x_j > 0$,
  \[ \tilde{\Psi}_1 = K_0 \ast (\Delta \tilde{\Psi}_1 - 2 \tilde{\Psi}_1 - c
     \partial_{x_2} \tilde{\Psi}_2) + c \int_{x_j}^{+ \infty} K_j \ast (\Delta
     \tilde{\Psi}_2 + c \partial_{x_2} \tilde{\Psi}_1) d y_j \]
  by integration from infinity, we also have (with a similar computation if
  $x_j < 0$)
  \[ | \tilde{\Psi}_1 | \leqslant \frac{K (c, R, \sigma, \sigma', \| \Phi
     \|_{H_{\infty}}, \| h \|_{\ast \ast, \sigma'})}{(1 + | x |)^{1 + \sigma /
     2}} . \]
  Now, using Theorem 8.10 from {\cite{MR1814364}}, we have for any $x \in
  \mathbbm{R}^2$ that
  \[ | \nabla^2 \tilde{\Psi} | (x) \leqslant K (\| \Delta \tilde{\Psi}
     \|_{L^{\infty} (B (x, 1))} + \| \tilde{\Psi} \|_{L^{\infty} (B (x, 1))} +
     \| \nabla \tilde{\Psi} \|_{L^{\infty} (B (x, 1))}), \]
  therefore (the limiting decay coming from (\ref{limitin}))
  \[ | \nabla^2 \tilde{\Psi} | \leqslant \frac{K (c, R, \sigma, \sigma', \|
     \Phi \|_{H_{\infty}}, \| h \|_{\ast \ast, \sigma'})}{(1 + | x |)^{- 1 +
     \sigma}} . \]
  With these estimates, we have that $\tilde{\Psi} \in \mathcal{E}_{\otimes, -
  3 + \sigma, \infty}$. Now, we define
  \[ \check{h} \assign \tilde{h} + 2 \frac{\nabla V}{V} . \nabla \tilde{\Psi}
     + 2 (1 - | V |^2) \mathfrak{R}\mathfrak{e} (\tilde{\Psi}) + \tmop{Loc}
     (\Psi), \]
  and we infer that, for any $\alpha \leqslant \sigma'$
  \begin{equation}
    \| \check{h} \|_{\otimes \otimes, \alpha, \infty} \leqslant K (\alpha, c,
    R, \sigma, \sigma', \delta, \| \Phi \|_{H_{\infty}}, \| h \|_{\ast \ast,
    \sigma'}) (1 + \| \tilde{\Psi} \|_{\otimes, \delta, \infty})
    \label{jenaimar}
  \end{equation}
  given that $\delta \geqslant - 2 + \alpha$. Indeed, we have that, for
  $\alpha \leqslant \sigma'$, $\| \tilde{h} \|_{\otimes \otimes, \alpha,
  \infty} \leqslant K (\alpha, \sigma') \| h \|_{\ast \ast, \sigma'}$, and
  \[ \| \tmop{Loc} (\Psi) \|_{\otimes \otimes, \alpha, \infty} \leqslant K (c,
     \alpha) \| \Phi \|_{C^2 (\{ r \leqslant 10 / c^2 \})} \leqslant K (c,
     \alpha, \| \Phi \|_{H_{\infty}}, \| h \|_{\ast \ast, \sigma'}) . \]
  We recall that $(1 - | V |^2) \mathfrak{R}\mathfrak{e} (\tilde{\Psi})$ is a
  real-valued term, and with Lemma \ref{nonmodV}, $0 < \sigma < \sigma' < 1$,
  we estimate
  \[ \| (1 + \tilde{r})^{1 + \alpha} (1 - | V |^2) \mathfrak{R}\mathfrak{e}
     (\tilde{\Psi}) \|_{L^{\infty} (\mathbbm{R}^2)} \leqslant K \left\|
     \frac{(1 + \tilde{r})^{1 + \alpha}}{(1 + \tilde{r})^{3 + \delta}}
     \right\|_{L^{\infty} (\mathbbm{R}^2)} \| \tilde{\Psi} \|_{\otimes,
     \delta, \infty} \leqslant K (\alpha, \delta) \| \tilde{\Psi} \|_{\otimes,
     \delta, \infty} \]
  if $1 + \alpha \geqslant 3 + \delta$ (which is a consequence of $\delta
  \geqslant - 2 + \alpha$), and
  \[ \| (1 + \tilde{r})^{2 + \alpha} \nabla ((1 - | V |^2)
     \mathfrak{R}\mathfrak{e} (\tilde{\Psi})) \|_{L^{\infty} (\mathbbm{R}^2)}
     \leqslant K \left\| \frac{(1 + \tilde{r})^{2 + \alpha}}{(1 +
     \tilde{r})^{4 + \delta}} \right\|_{L^{\infty} (\mathbbm{R}^2)} \|
     \tilde{\Psi} \|_{\otimes, \delta, \infty} \leqslant K (\alpha, \delta) \|
     \tilde{\Psi} \|_{\otimes, \delta, \infty} . \]
  Now, we estimate similarly (still using Lemma \ref{nonmodV})
  \[ \left\|  (1 + \tilde{r})^{1 + \alpha} \mathfrak{R}\mathfrak{e} \left(
     \frac{\nabla V}{V} . \nabla \tilde{\Psi} \right) \right\|_{L^{\infty}
     (\mathbbm{R}^2)} \leqslant K (c) \left\|  \frac{(1 + \tilde{r})^{1 +
     \alpha}}{(1 + \tilde{r})^{3 + \delta}} \right\|_{L^{\infty}
     (\mathbbm{R}^2)} \| \tilde{\Psi} \|_{\otimes, \delta, \infty} \leqslant K
     (c, \alpha, \delta) \| \tilde{\Psi} \|_{\otimes, \delta, \infty}, \]
  
  \[ \left\|  (1 + \tilde{r})^{2 + \alpha} \nabla \mathfrak{R}\mathfrak{e}
     \left( \frac{\nabla V}{V} . \nabla \tilde{\Psi} \right)
     \right\|_{L^{\infty} (\mathbbm{R}^2)} \leqslant K (c) \left\|  \frac{(1 +
     \tilde{r})^{2 + \alpha}}{(1 + \tilde{r})^{4 + \delta}}
     \right\|_{L^{\infty} (\mathbbm{R}^2)} \| \tilde{\Psi} \|_{\otimes,
     \delta, \infty} \leqslant K (c, \alpha, \delta) \| \tilde{\Psi}
     \|_{\otimes, \delta, \infty}, \]
  and since
  \[ \mathfrak{I}\mathfrak{m} \left( \frac{\nabla V}{V} . \nabla \tilde{\Psi}
     \right) =\mathfrak{I}\mathfrak{m} \left( \frac{\nabla V}{V} \right)
     .\mathfrak{R}\mathfrak{e} (\nabla \tilde{\Psi}) +\mathfrak{R}\mathfrak{e}
     \left( \frac{\nabla V}{V} \right) .\mathfrak{I}\mathfrak{m} (\nabla
     \tilde{\Psi}), \]
  with Lemma \ref{nonmodV} and estimate at the end of the proof of Proposition
  \ref{invertop}, we infer that
  \begin{eqnarray*}
    &  & \left\|  (1 + \tilde{r})^{2 + \alpha} \mathfrak{I}\mathfrak{m}
    \left( \frac{\nabla V}{V} . \nabla \tilde{\Psi} \right)
    \right\|_{L^{\infty} (\mathbbm{R}^2)}\\
    & \leqslant & \left\|  (1 + \tilde{r})^{2 + \alpha}
    \mathfrak{I}\mathfrak{m} \left( \frac{\nabla V}{V} \right)
    .\mathfrak{R}\mathfrak{e} (\nabla \tilde{\Psi}) \right\|_{L^{\infty}
    (\mathbbm{R}^2)} + \left\|  (1 + \tilde{r})^{2 + \alpha}
    \mathfrak{R}\mathfrak{e} \left( \frac{\nabla V}{V} \right)
    .\mathfrak{I}\mathfrak{m} (\nabla \tilde{\Psi}) \right\|_{L^{\infty}
    (\mathbbm{R}^2)}\\
    & \leqslant & K (c) \left\|  \frac{(1 + \tilde{r})^{2 + \alpha}}{(1 +
    \tilde{r})^{4 + \delta}} \right\|_{L^{\infty} (\mathbbm{R}^2)} \|
    \tilde{\Psi} \|_{\otimes, \delta, \infty} + K \left\| \frac{(1 +
    \tilde{r})^{2 + \alpha}}{(1 + \tilde{r})^{4 + \delta}}
    \right\|_{L^{\infty} (\mathbbm{R}^2)} \| \tilde{\Psi} \|_{\otimes, \delta,
    \infty}\\
    & \leqslant & K (c, \alpha, \delta) \| \tilde{\Psi} \|_{\otimes, \delta,
    \infty},
  \end{eqnarray*}
  and with similar estimates,
  \[ \left\|  (1 + \tilde{r})^{2 + \alpha} \nabla \mathfrak{I}\mathfrak{m}
     \left( \frac{\nabla V}{V} . \nabla \tilde{\Psi} \right)
     \right\|_{L^{\infty} (\mathbbm{R}^2)} \leqslant K (c, \alpha, \delta) \|
     \tilde{\Psi} \|_{\otimes, \delta, \infty} . \]
  This concludes the proof of (\ref{jenaimar}). With $\tilde{\Psi} \in
  \mathcal{E}_{\otimes, - 3 + \sigma, \infty}$, we therefore deduce that for
  $\varepsilon > 0$ a small constant, $\| \check{h} \|_{\otimes \otimes, - 1 +
  \sigma - \varepsilon, \infty} < + \infty$, hence $\check{h} \in
  \mathcal{E}_{\otimes \otimes, - 1 + \sigma - \varepsilon}$. With estimate
  (\ref{jenaimar}), Lemma \ref{P3starest} and
  \[ - \Delta \tilde{\Psi} - i c \partial_{x_2} \tilde{\Psi} +
     2\mathfrak{R}\mathfrak{e} (\tilde{\Psi}) = \check{h}, \]
  and with the symmetries on $\tilde{\Psi}$ and $\check{h}$, we can bootstrap
  our estimates on $\tilde{\Psi}$ and then on $\check{h}$, and we conclude
  that $\tilde{\Psi} \in \mathcal{E}_{\otimes, \sigma}$ (since $\sigma <
  \sigma'$).
\end{proof}

\section{Estimations for the differentiability}\label{AC}

\subsection{Proof of Lemma \ref{dderpsy}}

\begin{proof}
  We fix $0 < c < c_0 (\sigma)$. We define, for $d \in \left] \frac{1}{2 c},
  \frac{2}{c} \right[ \cap \left] d_{\circledast} - \frac{\delta}{2},
  d_{\circledast} + \frac{\delta}{2} \right[$, the function
  \[ \mathbbm{H}_d : \Phi \mapsto (\eta L (.) + (1 - \eta) V L' (. / V))_d^{-
     1} (\Pi_d^{\bot} (F_d (\Phi / V))) \]
  from $\mathcal{E}_{\circledast, \sigma, d_{\circledast}}$ to
  $\mathcal{E}_{\circledast, \sigma, d_{\circledast}}$, so that
  \[ H (\Phi, c, d) =\mathbbm{H}_d (\Phi) + \Phi . \]
  We took the same convention as in the proof of Lemma \ref{rest31}: we added
  a subscript in $d$ in the operators to describe at which values of $d$ this
  operator is taken.
  
  \
  
  \begin{tmindent}
    Step 1.  Differentiability of $\mathbbm{H}_d$ with respect to $d$.
  \end{tmindent}
  
\  
  
  To apply the implicit function theorem, we have to check that $H (\Phi, c,
  d)$ (or, equivalently $\mathbbm{H}_d (\Phi)$) is differentiable with respect
  to $d$, and that $\partial_d H (\Phi, c, d) \in \mathcal{E}_{\circledast,
  \sigma, d_{\circledast}}$. By definition of the operator $(\eta L (.) + (1 -
  \eta) V L' (. / V))^{- 1}$, we have, in the distribution sense,
  \[ \left( \eta L (\mathbbm{H}_{d + \varepsilon} (\Phi)) + (1 - \eta) V L'
     \left( \frac{\mathbbm{H}_{d + \varepsilon} (\Phi)}{V} \right) \right)_{d
     + \varepsilon} + \Pi_{d + \varepsilon}^{\bot} (F_{d + \varepsilon} (\Phi
     / V_{d + \varepsilon})) = 0 \]
  and
  \[ \left( \eta L (\mathbbm{H}_d (\Phi)) + (1 - \eta) V L' \left(
     \frac{\mathbbm{H}_d (\Phi)}{V} \right) \right)_d + \Pi_d^{\bot} (F_d
     (\Phi / V_d)) = 0. \]
  From Lemma \ref{lemma7}, we have, for any $\Phi = V_d \Psi \in
  \mathcal{E}_{\circledast, \sigma, d_{\circledast}}$ that
  \[ \left( \eta L (.) + (1 - \eta) V L' \left( \frac{.}{V} \right) \right)_d
     (\Phi) = L_d (\Phi) - (1 - \eta_d) (E - i c \partial_{x_2} V)_d \Psi, \]
  and with the definition of $L_d$ (in Lemma \ref{lemma7}), we check that, for
  any $\Phi \in \mathcal{E}_{\circledast, \sigma, d_{\circledast}}$, in the
  distribution sense,
  \begin{eqnarray*}
    &  & \left( \left( \eta L (.) + (1 - \eta) V L' \left( \frac{.}{V}
    \right) \right)_{d + \varepsilon} - \left( \eta L (.) + (1 - \eta) V L'
    \left( \frac{.}{V} \right) \right)_d \right) (\Phi)\\
    & = & (| V_{d + \varepsilon} |^2 - | V_d |^2) \Phi +
    2\mathfrak{R}\mathfrak{e} (\overline{V_{d + \varepsilon}} \Phi) V_{d +
    \varepsilon} - 2\mathfrak{R}\mathfrak{e} (\overline{V_d} \Phi) V_d\\
    & - & (1 - \eta_{d + \varepsilon}) (E - i c \partial_{x_2} V)_{d +
    \varepsilon} + (1 - \eta_d) (E - i c \partial_{x_2} V)_d .
  \end{eqnarray*}
  We therefore compute that, in the distribution sense,
  \begin{eqnarray*}
    &  & \left( \eta L (.) + (1 - \eta) V L' \left( \frac{.}{V} \right)
    \right)_d (\mathbbm{H}_{d + \varepsilon} (\Phi) -\mathbbm{H}_d (\Phi))\\
    & = & - ((| V_{d + \varepsilon} |^2 - | V_d |^2) \mathbbm{H}_{d +
    \varepsilon} (\Phi) + 2\mathfrak{R}\mathfrak{e} (\overline{V_{d +
    \varepsilon}} \mathbbm{H}_{d + \varepsilon} (\Phi)) V_{d + \varepsilon} -
    2\mathfrak{R}\mathfrak{e} (\overline{V_d} \mathbbm{H}_{d + \varepsilon}
    (\Phi)) V_d)\\
    & + & ((1 - \eta_{d + \varepsilon}) (E - i c \partial_{x_2} V)_{d +
    \varepsilon} - (1 - \eta_d) (E - i c \partial_{x_2} V)_d) \mathbbm{H}_{d +
    \varepsilon} (\Phi)\\
    & - & (\Pi_{d + \varepsilon}^{\bot} (F_{d + \varepsilon} (\Phi / V_{d +
    \varepsilon})) - \Pi_d^{\bot} (F_d (\Phi / V_d))) .
  \end{eqnarray*}
  Since
  \[ \partial_d^2 V = \partial_{x_1}^2 V_1 V_{- 1} + \partial_{x_1}^2 V_{- 1}
     V_1 - 2 \partial_{x_1} V_1 \partial_{x_1} V_{- 1}, \]
  with Lemmas \ref{lemme3}, \ref{ddVest} and equation (\ref{230210}), we check
  easily that
  \[ | V_{d + \varepsilon} |^2 - | V_d |^2 = \varepsilon \partial_d (| V |^2)
     + \frac{O_{\varepsilon \rightarrow 0}^{c, d} (\varepsilon^2)}{(1 +
     \tilde{r})^3} \]
  and
  \[ \nabla (| V_{d + \varepsilon} |^2) - \nabla (| V_d |^2) = \varepsilon
     \partial_d (\nabla | V |^2) + \frac{O_{\varepsilon \rightarrow 0}^{c, d}
     (\varepsilon^2)}{(1 + \tilde{r})^3} . \]
  It implies in particular that $(| V_{d + \varepsilon} |^2 - | V_d |^2)
  \mathbbm{H}_{d + \varepsilon} (\Phi) \in \mathcal{E}_{\circledast
  \circledast, \gamma (\sigma), d_{\circledast}}$, with
  \[ \| (| V_{d + \varepsilon} |^2 - | V_d |^2) \mathbbm{H}_{d + \varepsilon}
     (\Phi) \|_{\circledast \circledast, \gamma (\sigma), d_{\circledast}}
     \rightarrow 0 \]
  when $\varepsilon \rightarrow 0$. We check similarly
  \begin{eqnarray*}
    &  & 2\mathfrak{R}\mathfrak{e} (\overline{V_{d + \varepsilon}}
    \mathbbm{H}_{d + \varepsilon} (\Phi)) V_{d + \varepsilon} -
    2\mathfrak{R}\mathfrak{e} (\overline{V_d} \mathbbm{H}_{d + \varepsilon}
    (\Phi)) V_d\\
    & = & \varepsilon (2\mathfrak{R}\mathfrak{e} (\overline{\partial_d V}
    \mathbbm{H}_{d + \varepsilon} (\Phi)) V_d + 2\mathfrak{R}\mathfrak{e}
    (\overline{V_d} \mathbbm{H}_{d + \varepsilon} (\Phi)) \partial_d V_d) +
    O^{c, d}_{\| . \|_{\circledast \circledast, \gamma (\sigma),
    d_{\circledast}}} (\varepsilon^2),
  \end{eqnarray*}
  and that $2\mathfrak{R}\mathfrak{e} (\overline{\partial_d V} \mathbbm{H}_{d
  + \varepsilon} (\Phi)) V_d + 2\mathfrak{R}\mathfrak{e} (\overline{V_d}
  \mathbbm{H}_{d + \varepsilon} (\Phi)) \partial V_d \in
  \mathcal{E}_{\circledast \circledast, \gamma (\sigma), d_{\circledast}}$. We
  continue, still with Lemmas \ref{lemme3}, \ref{ddVest} and equation
  (\ref{230210}), we infer
  \begin{eqnarray*}
    &  & ((1 - \eta_{d + \varepsilon}) (E - i c \partial_{x_2} V)_{d +
    \varepsilon} - (1 - \eta_d) (E - i c \partial_{x_2} V)_d) \mathbbm{H}_{d +
    \varepsilon} (\Phi)\\
    & = & \varepsilon \partial_d ((1 - \eta_d) (E - i c \partial_{x_2} V)_d)
    \mathbbm{H}_{d + \varepsilon} (\Phi) + O^{c, d}_{\| . \|_{\circledast
    \circledast, \gamma (\sigma), d_{\circledast}}} (\varepsilon^2)
  \end{eqnarray*}
  and $\partial_d ((1 - \eta_d) (E - i c \partial_{x_2} V)_d) \mathbbm{H}_{d +
  \varepsilon} (\Phi) \in \mathcal{E}_{\circledast \circledast, \gamma
  (\sigma), d_{\circledast}}$. Finally, we recall that
  \[ F_d (\Psi) = (E - i c \partial_{x_2} V)_d + V_d (1 - \eta) (- \nabla
     \Psi . \nabla \Psi + | V |^2 S (\Psi)) + R_d (\Psi), \]
  and we check similarly that
  \[ \Pi_{d + \varepsilon}^{\bot} (F_{d + \varepsilon} (\Phi / V_{d +
     \varepsilon})) - \Pi_d^{\bot} (F_d (\Phi / V_d)) = \varepsilon \partial_d
     (\Pi_d^{\bot} (F_d (\Phi / V_d))) + O^{c, d}_{\| . \|_{\circledast
     \circledast, \gamma (\sigma), d_{\circledast}}} (\varepsilon^2) . \]
  We have
  \[ \partial_d (\Pi_d^{\bot} (F_d (\Phi / V_d))) = (\partial_d \Pi_d^{\bot})
     (F_d (\Phi / V_d)) + \Pi_d^{\bot} (\partial_d (F_d (\Phi / V_d))), \]
  and since $(\partial_d \Pi_d^{\bot}) (F_d (\Phi / V))$ is compactly
  supported, $(\partial_d \Pi_d^{\bot}) (F_d (\Phi / V)) \in
  \mathcal{E}_{\circledast \circledast, \gamma (\sigma), d_{\circledast}}$. We
  will check in the next step that $\partial_d (F_d (\Phi / V_d)) \in
  \mathcal{E}_{\circledast \circledast, \gamma (\sigma), d_{\circledast}}$.
  Let us suppose this result for now and finish the proof of the
  differentiability.
  
  Combining the different estimates, we have in particular that
  \[ \left( \eta L (.) + (1 - \eta) V L' \left( \frac{.}{V} \right) \right)_d
     (\mathbbm{H}_{d + \varepsilon} (\Phi) -\mathbbm{H}_d (\Phi)) \rightarrow
     0 \]
  in $\mathcal{E}_{\circledast \circledast, \gamma (\sigma), d_{\circledast}}$
  when $\varepsilon \rightarrow 0$. By Proposition \ref{invertop} (from
  $\mathcal{E}_{\circledast \circledast, \gamma (\sigma), d_{\circledast}}$ to
  $\mathcal{E}_{\circledast, \sigma, d_{\circledast}}$), this implies that
  \[ \mathbbm{H}_{d + \varepsilon} (\Phi) \rightarrow \mathbbm{H}_d (\Phi) \]
  in $\mathcal{E}_{\circledast, \sigma, d_{\circledast}}$ when $\varepsilon
  \rightarrow 0$. Now, taking the equation
  \begin{eqnarray*}
    &  & \left( \eta L (.) + (1 - \eta) V L' \left( \frac{.}{V} \right)
    \right)_d (\mathbbm{H}_{d + \varepsilon} (\Phi) -\mathbbm{H}_d (\Phi))\\
    & = & - ((| V_{d + \varepsilon} |^2 - | V_d |^2) \mathbbm{H}_{d +
    \varepsilon} (\Phi) + 2\mathfrak{R}\mathfrak{e} (\overline{V_{d +
    \varepsilon}} \mathbbm{H}_{d + \varepsilon} (\Phi)) V_{d + \varepsilon} -
    2\mathfrak{R}\mathfrak{e} (\overline{V_d} \mathbbm{H}_{d + \varepsilon}
    (\Phi)) V_d)\\
    & + & ((1 - \eta_{d + \varepsilon}) (E - i c \partial_{x_2} V)_{d +
    \varepsilon} - (1 - \eta_d) (E - i c \partial_{x_2} V)_d) \mathbbm{H}_{d +
    \varepsilon} (\Phi)\\
    & - & (\Pi_{d + \varepsilon}^{\bot} (F_{d + \varepsilon} (\Phi / V_{d +
    \varepsilon})) - \Pi_d^{\bot} (F_d (\Phi / V_d)))
  \end{eqnarray*}
  and dividing it by $\varepsilon$, and then taking $\varepsilon \rightarrow
  0$, we check that $d \mapsto \mathbbm{H}_d (\Phi)$ is a $C^1$ function in
  $\mathcal{E}_{\circledast, \sigma, d_{\circledast}}$, with
  \[ \partial_d H (\Phi, c, d) = \partial_d \mathbbm{H}_d (\Phi) = \left(
     \eta L (.) + (1 - \eta) V L' \left( \frac{.}{V} \right) \right)^{- 1} (G
     (d, \Phi)), \]
  with
  \begin{eqnarray*}
    G (d, \Phi) & \assign & \partial_d (| V |^2) \mathbbm{H}_d (\Phi) +
    2\mathfrak{R}\mathfrak{e} (\overline{\partial_d V} \mathbbm{H}_d (\Phi))
    V_d + 2\mathfrak{R}\mathfrak{e} (\overline{V_d} \mathbbm{H}_d (\Phi))
    \partial_d V_d\\
    & + & \partial_d ((1 - \eta_d) (E - i c \partial_{x_2} V)_d)
    \mathbbm{H}_d (\Phi) - \partial_d (\Pi_d^{\bot} (F_d (\Phi / V_d))) .
  \end{eqnarray*}
  By the implicit function theorem, with Lemma \ref{lemme20}, since $\|
  \Psi_{c, d_{}} \|_{\ast, \sigma, d} \leqslant K (\sigma, \sigma') c^{1 -
  \sigma'}$ this implies that, for $c$ small enough, $d \mapsto \Phi_{c, d}$
  is a $C^1$ function, and
  \[ \partial_d \Phi_{c, d} = - d_{\Phi} H^{- 1} (\partial_d H (\Phi_{c,
     d_{}}, d, c)) . \]
  Now, let us check that indeed $\partial_d (F_d (\Phi / V_d)) \in
  \mathcal{E}_{\circledast \circledast, \gamma (\sigma), d_{\circledast}}$ for
  $\Phi \in \mathcal{E}_{\circledast, \sigma, d_{\circledast}}$.
  
\
  
  \begin{tmindent}
    Step 2.  Proof of $\left\| \frac{\partial_d (F_d (\Phi / V_d))}{V}
    \right\|_{\ast \ast, \gamma (\sigma), d} \leqslant K (\sigma) c^{1 -
    \gamma (\sigma)} + K \| \Psi \|_{\ast, \sigma, d}$.
  \end{tmindent}
  
\
  
  By the equivalence of the $\ast$ and $\circledast$ norms, these estimates
  imply that $\partial_d (F_d (\Phi / V_d)) \in \mathcal{E}_{\circledast
  \circledast, \gamma (\sigma), d_{\circledast}}$. We suppose from now on that
  $\| \Psi \|_{\ast, \sigma, d} \leqslant 1$. From Lemma \ref{lemma7}, we have
  \[ F_d \left( \frac{\Phi}{V_d} \right) = (E - i c \partial_{x_2} V)_d + R_d
     \left( \frac{\Phi}{V_d} \right) + V_d (1 - \eta_d) \left( - \nabla \left(
     \frac{\Phi}{V_d} \right) . \nabla \left( \frac{\Phi}{V_d} \right) + | V_d
     |^2 S \left( \frac{\Phi}{V_d} \right) \right) . \]
  It is easy to check that at fixed $\Phi, c$,
  \[ \left\| \frac{\partial_d \left( R_d \left( \frac{\Phi}{V_d} \right)
     \right)}{V} \right\|_{\ast \ast, \gamma (\sigma), d} \leqslant K (\sigma)
     c^{1 - \gamma (\sigma)} + K \| \Psi \|_{\ast, \sigma, d}, \]
  since it is localized near the vortices. For the nonlinear part, we have
  \begin{eqnarray*}
    \frac{\partial_d \left( V (1 - \eta) \left( - \nabla \left( \frac{\Phi}{V}
    \right) . \nabla \left( \frac{\Phi}{V} \right) + | V |^2 S \left(
    \frac{\Phi}{V} \right) \right) \right)}{V} & = & \frac{\partial_d V}{V} (1
    - \eta) (- \nabla \Psi . \nabla \Psi + | V |^2 S (\Psi))\\
    & - & \partial_d \eta (- \nabla \Psi . \nabla \Psi + | V |^2 S (\Psi))\\
    & + & (1 - \eta)  \left( - 2 \nabla \Psi . \partial_d \left( \nabla
    \left( \frac{\Phi}{V_d} \right) \right) \right)\\
    & + & (1 - \eta) 2\mathfrak{R}\mathfrak{e} (\bar{V} \partial_d V) S
    (\Psi)\\
    & + & (1 - \eta) | V |^2 \partial_d \left( S \left( \frac{\Phi}{V_d}
    \right) \right) .
  \end{eqnarray*}
  For the first line, from Lemma \ref{ddVest}, $\| \Psi \|_{\ast, \sigma, d}
  \leqslant 1$ and the definition of $\| . \|_{\ast, \sigma, d}$, we have
  \begin{eqnarray*}
    \left| \frac{\partial_d V}{V} (1 - \eta) (- \nabla \Psi . \nabla \Psi + |
    V |^2 S (\Psi)) \right| & \leqslant & \frac{K \| \Psi \|_{\ast, \sigma,
    d}^2}{(1 + \tilde{r})^3} \leqslant \frac{K \| \Psi \|_{\ast, \sigma,
    d}}{(1 + \tilde{r})^3}
  \end{eqnarray*}
  and
  \[ \begin{array}{lll}
       \left| \nabla \left( \frac{\partial_d V}{V} (1 - \eta) (- \nabla \Psi .
       \nabla \Psi + | V |^2 S (\Psi)) \right) \right| & \leqslant & \frac{K
       \| \Psi \|_{\ast, \sigma, d}^2}{(1 + \tilde{r})^3} \leqslant \frac{K \|
       \Psi \|_{\ast, \sigma, d}}{(1 + \tilde{r})^3},
     \end{array} \]
  which is enough the estimate. Similarly, since $\partial_d \eta$ is
  compactly supported, we have
  \[ | \partial_d \eta (- \nabla \Psi . \nabla \Psi + | V |^2 S (\Psi)) | + |
     \nabla (\partial_d \eta (- \nabla \Psi . \nabla \Psi + | V |^2 S (\Psi)))
     | \leqslant \frac{K \| \Psi \|_{\ast, \sigma, d}^2}{(1 + \tilde{r})^3}
     \leqslant \frac{K \| \Psi \|_{\ast, \sigma, d}}{(1 + \tilde{r})^3} . \]
  Now, we develop
  \[ \partial_d \left( \nabla \left( \frac{\Phi}{V} \right) \right) = -
     \frac{\partial_d V \nabla \Phi}{V^2} - \frac{\nabla \partial_d V
     \Phi}{V^2} + \frac{\partial_d V \Phi \nabla V}{V^3}, \]
  and we check, with Lemma \ref{ddVest}, that
  \[ \left| (1 - \eta)  \left( - 2 \nabla \Psi . \partial_d \left( \nabla
     \left( \frac{\Phi}{V_d} \right) \right) \right) \right| \leqslant \frac{K
     \| \Psi \|_{\ast, \sigma, d}^2}{(1 + \tilde{r})^3} \leqslant \frac{K \|
     \Psi \|_{\ast, \sigma, d}}{(1 + \tilde{r})^3}, \]
  as well as
  \[ \left| \nabla \left( (1 - \eta)  \left( - 2 \nabla \Psi . \partial_d
     \left( \nabla \left( \frac{\Phi}{V_d} \right) \right) \right) \right)
     \right| \leqslant \frac{K \| \Psi \|_{\ast, \sigma, d}^2}{(1 +
     \tilde{r})^3} \leqslant \frac{K \| \Psi \|_{\ast, \sigma, d}}{(1 +
     \tilde{r})^3} . \]
  Since $| \mathfrak{R}\mathfrak{e} (\bar{V} \partial_d V) | \leqslant
  \frac{K}{(1 + \tilde{r})^3}$ from Lemma \ref{ddVest} and $| S (\Psi) |
  \leqslant K | \mathfrak{R}\mathfrak{e} (\Psi) |$ (since $\| \Psi \|_{\ast,
  \sigma, d} \leqslant 1$), we have similarly
  \[ | (1 - \eta) 2\mathfrak{R}\mathfrak{e} (\bar{V} \partial_d V) S (\Psi) |
     \leqslant \frac{K \| \Psi \|_{\ast, \sigma, d}}{(1 + \tilde{r})^3}, \]
  and finally, since
  \[ \partial_d \left( S \left( \frac{\Phi}{V_d} \right) \right) = -
     2\mathfrak{R}\mathfrak{e} \left( \frac{\Phi \partial_d V}{V^2} \right)
     (e^{2\mathfrak{R}\mathfrak{e} (\Psi)} - 1^{}) \]
  is real-valued, we check that
  \[ \left| \partial_d \left( S \left( \frac{\Phi}{V_d} \right) \right)
     \right| \leqslant \frac{K \| \Psi \|_{\ast, \sigma, d}^2}{(1 +
     \tilde{r})^{2 + 2 \sigma}} \leqslant \frac{K \| \Psi \|_{\ast, \sigma,
     d}}{(1 + \tilde{r})^{1 + \gamma (\sigma)}} \]
  and
  \[ \left| \nabla \partial_d \left( S \left( \frac{\Phi}{V_d} \right) \right)
     \right| \leqslant \frac{K \| \Psi \|_{\ast, \sigma, d}^2}{(1 +
     \tilde{r})^{3 + 2 \sigma}} \leqslant \frac{K \| \Psi \|_{\ast, \sigma,
     d}}{(1 + \tilde{r})^{2 + \gamma (\sigma)}} . \]
  and this is enough for the estimate. Finally, we will show that for any $0 <
  \sigma < 1$,
  \[ \left\| \frac{\partial_d (E - i c \partial_{x_2} V)}{V} \right\|_{\ast
     \ast, \sigma, d} \leqslant K (\sigma) c^{1 - \sigma}, \]
  which would conclude the proof of this step (taking $\gamma (\sigma)$
  instead of $\sigma$).
  
  Let us show first that
  \begin{equation}
    | \partial_d E | \leqslant \frac{K c^{1 - \sigma}}{(1 + \tilde{r})^{2 +
    \sigma}} . \label{ddEes}
  \end{equation}
  We have from (\ref{E2}) that
  \[ E = - 2 \nabla V_1 . \nabla V_{- 1} + (1 - | V_1 |^2) (1 - | V_{- 1} |^2)
     V_1 V_{- 1}, \]
  hence
  \[ \partial_d E = 2 \nabla \partial_{x_1} V_1 . \nabla V_{- 1} - 2 \nabla
     V_1 . \nabla \partial_{x_1} V_{- 1} + \partial_d ((1 - | V_1 |^2) (1 - |
     V_{- 1} |^2) V_1 V_{- 1}) . \]
  With Lemmas \ref{lemme3} and \ref{dervor}, we easily check that
  \[ | \nabla \partial_{x_1} V_1 . \nabla V_{- 1} | \leqslant \frac{K}{(1 +
     r_1)^2 (1 + r_{- 1})}, \]
  \[ | \nabla V_1 . \nabla \partial_{x_1} V_{- 1} | \leqslant \frac{K}{(1 +
     r_1) (1 + r_{- 1})^2} \]
  and
  \[ | \partial_d ((1 - | V_1 |^2) (1 - | V_{- 1} |^2) V_1 V_{- 1}) |
     \leqslant \frac{K}{(1 + r_1)^3 (1 + r_{- 1})^2} + \frac{K}{(1 + r_1)^2 (1
     + r_{- 1})^3} . \]
  In the right half-plane, where $r_1 \leqslant r_{- 1}$ and $r_{- 1}
  \geqslant d$, we use
  \[ \frac{1}{(1 + r_{- 1})^{1 - \sigma}} \leqslant K c^{1 - \sigma} \]
  and
  \[ \frac{1}{(1 + r_1)^{\alpha}} + \frac{1}{(1 + r_{- 1})^{\alpha}} \leqslant
     \frac{2}{(1 + \tilde{r})^{\alpha}} \]
  for $\alpha > 0$ on the three previous estimates to show that
  \[ | \partial_d E | \leqslant \frac{K c^{1 - \sigma}}{(1 + \tilde{r})^{2 +
     \sigma}} \]
  in the right half-plane. Similarly, the result holds in the left half-plane,
  and this proves (\ref{ddEes}). With similar computations, we can estimate
  $\nabla \left( \frac{\partial_d E}{V} \right)$ and show that
  \[ \left\| \frac{\partial_d E}{V} \right\|_{\ast \ast, \sigma, d} \leqslant
     K (\sigma) c^{1 - \sigma} . \]
  Let us now prove that
  \begin{equation}
    \left\| \frac{\partial_d (i c \partial_{x_2} V)}{V} \right\|_{\ast \ast,
    \sigma, d} \leqslant K (\sigma) c^{1 - \sigma} . \label{lrouge}
  \end{equation}
  We show easily that
  \[ \| i c \partial_{x_2} \partial_d V \|_{C^1 (\{ \tilde{r} \leqslant 3 \})}
     \leqslant K c \leqslant K c^{1 - \sigma}, \]
  and since $\partial_{x_2} \partial_d V = - \partial_{x_1 x_2} V_1 V_{- 1} +
  \partial_{x_1 x_2} V_{- 1} V_1 - \partial_{x_1} V_1 \partial_{x_2} V_{- 1} +
  \partial_{x_1} V_{- 1} \partial_{x_2} V_1$, by Lemma \ref{dervor} we have
  \[ | \partial_{x_2} \partial_d V | \leqslant \frac{K}{(1 + \tilde{r})^2}, |
     \nabla \partial_{x_2} \partial_d V | \leqslant \frac{K}{(1 +
     \tilde{r})^3} \]
  therefore
  \[ \left\| \tilde{r}^{1 + \sigma} \mathfrak{R}\mathfrak{e} \left( \frac{i c
     \partial_{x_2} \partial_d V}{V} \right) \right\|_{L^{\infty} (\{
     \tilde{r} \geqslant 2 \})} + \left\| \tilde{r}^{2 + \sigma} \nabla \left(
     \frac{i c \partial_{x_2} \partial_d V}{V} \right) \right\|_{L^{\infty}
     (\{ \tilde{r} \geqslant 2 \})} \leqslant K c \leqslant K c^{1 - \sigma} .
  \]
  This proves that (\ref{lrouge}) is true for the real part contribution. We
  are left with the proof of
  \[ \left\| c \tilde{r}^{2 + \sigma} \mathfrak{I}\mathfrak{m} \left( \frac{i
     \partial_{x_2} \partial_d V}{V} \right) \right\|_{L^{\infty} (\{
     \tilde{r} \geqslant 2 \})} \leqslant K (\sigma) c^{1 - \sigma}, \]
  which is more delicate and relies on some cancelations. We compute
  \[ \mathfrak{I}\mathfrak{m} \left( \frac{i \partial_{x_2} \partial_d V}{V}
     \right) = -\mathfrak{R}\mathfrak{e} \left( - \frac{\partial_{x_1 x_2}
     V_1}{V_1} + \frac{\partial_{x_1 x_2} V_{- 1}}{V_{- 1}} \right)
     -\mathfrak{R}\mathfrak{e} \left( - \frac{\partial_{x_1} V_1}{V_1}
     \frac{\partial_{x_2} V_{- 1}}{V_{- 1}} + \frac{\partial_{x_1} V_{-
     1}}{V_{- 1}} \frac{\partial_{x_2} V_1}{V_{- 1}} \right) . \]
  From Lemma \ref{dervor}, we have
  \[ \frac{\partial_{x_1} V_1}{V_1} = - \frac{i}{r_1} \sin (\theta_1) + O_{r_1
     \rightarrow \infty} \left( \frac{1}{r_1^3} \right) \]
  and the part in $O_{r_1 \rightarrow \infty} \left( \frac{1}{r_1^3} \right)$
  can be estimated as in the proof of Lemma \ref{fineestimate} for $\left\|
  \frac{i c \partial_{x_{2 \nosymbol}} V}{V} \right\|_{\ast \ast, \sigma, d}$.
  In particular, we will just compute the terms of order less than
  $\frac{1}{r_1^3}$ or $\frac{1}{r_{- 1}^3}$. From Lemma \ref{dervor}, we have
  also
  \[ \frac{\partial_{x_2} V_1}{V_1} = - \frac{i}{r_1} \cos (\theta_1) + O_{r_1
     \rightarrow \infty} \left( \frac{1}{r_1^3} \right) \]
  and
  \[ \mathfrak{R}\mathfrak{e} \left( \frac{\partial_{x_1 x_2} V_1}{V_1}
     \right) = \frac{\cos (\theta_1) \sin (\theta_1)}{r_1^2} + O_{r_1
     \rightarrow \infty} \left( \frac{1}{r_1^3} \right) . \]
  These two estimates hold by changing $i \rightarrow - i$, $\theta_1
  \rightarrow \theta_{- 1}$, $r_1 \rightarrow r_{- 1}$ and $V_1 \rightarrow
  V_{- 1}$. We then deduce that
  \begin{eqnarray}
    \mathfrak{I}\mathfrak{m} \left( \frac{i \partial_{x_2} \partial_d V}{V}
    \right) & = & - \left( - \frac{\cos (\theta_1) \sin (\theta_1)}{r_1^2} +
    \frac{\cos (\theta_{- 1}) \sin (\theta_{- 1})}{r_{- 1}^2} \right)
    \nonumber\\
    & - & \left( - \frac{\sin (\theta_1)}{r_1} \frac{\cos (\theta_{-
    1})}{r_{- 1}} + \frac{\sin (\theta_{- 1})}{r_{- 1}} \frac{\cos
    (\theta_1)}{r_1} \right) \nonumber\\
    & + & O_{r_1 \rightarrow \infty} \left( \frac{1}{r_1^3} \right) + O_{r_{-
    1} \rightarrow \infty} \left( \frac{1}{r_{- 1}^3} \right) . 
    \label{horrible}
  \end{eqnarray}
  We start with the second term of (\ref{horrible}) which is the easiest one.
  We use for $\epsilon = \pm 1$ that
  \[ \cos (\theta_{\epsilon}) = \frac{x_1 - d \epsilon}{r_{\epsilon}} \quad
     \tmop{and} \quad \sin (\theta_{\epsilon}) = \frac{x_2}{r_{\epsilon}} \]
  to compute
  \[ \sin (\theta_1) \cos (\theta_{- 1}) = \frac{(x_1 + d) x_2}{r_1 r_{- 1}}
  \]
  and
  \[ \sin (\theta_{- 1}) \cos (\theta_1) = \frac{(x_1 - d) x_2}{r_1 r_{- 1}},
  \]
  therefore
  \[ - \frac{\sin (\theta_1)}{r_1} \frac{\cos (\theta_{- 1})}{r_{- 1}} +
     \frac{\sin (\theta_{- 1})}{r_{- 1}} \frac{\cos (\theta_1)}{r_1} = \frac{2
     d x_2}{(r_1 r_{- 1})^2} . \]
  We have, in the right half-plane, where $r_1 \leqslant r_{- 1}$ and $r_{- 1}
  \geqslant d \geqslant \frac{K}{c}$,
  \[ \left| c \tilde{r}^{2 + \sigma} \frac{2 d x_2}{(r_1 r_{- 1})^2} \right| =
     2 \left| c d \frac{\tilde{r}^{2 + \sigma}}{r_1^2 r_{- 1}^{\sigma}}
     \frac{x_2}{r_{- 1}} \frac{1}{r_{- 1}^{1 - \sigma}} \right| \leqslant K
     c^{1 - \sigma} \]
  since $\frac{\tilde{r}^{2 + \sigma}}{r_1^2 r_{- 1}^{\sigma}} \leqslant 1$,
  $\frac{| x_2 |}{r_{- 1}} \leqslant 1 \tmop{and} c d \leqslant K$. Similarly,
  we have the same estimate in the left half-plane.
  
  Now for the first term of (\ref{horrible}), we have, for $\epsilon = \pm 1$,
  \[ \sin (\theta_{\epsilon}) \cos (\theta_{\epsilon}) = \frac{(x_1 - \epsilon
     d) x_2}{r_{\epsilon}^2} . \]
  Therefore,
  \[ - \frac{\cos (\theta_1) \sin (\theta_1)}{r_1^2} + \frac{\cos (\theta_{-
     1}) \sin (\theta_{- 1})}{r_{- 1}^2} = \frac{x_2}{(r_1 r_{- 1})^4} (r_1^4
     (x_1 + d) - r_{- 1}^4 (x_1 - d)) . \]
  We compute, for $\epsilon = \pm 1$,
  \[ r_{\epsilon}^4 = ((x_1 - \epsilon d)^2 + x_2^2)^2 = (x_1 - \epsilon d)^4
     + 2 (x_1 - \epsilon d)^2 x_2^2 + x_2^4, \]
  hence
  \begin{eqnarray*}
    &  & - \frac{\cos (\theta_1) \sin (\theta_1)}{r_1^2} + \frac{\cos
    (\theta_{- 1}) \sin (\theta_{- 1})}{r_{- 1}^2}\\
    & = & \frac{x_2}{(r_1 r_{- 1})^4} (x_1 - d) (x_1 + d) ((x_1 - d)^3 - (x_1
    + d)^3 + 2 x_2^2 ((x_1 - d) - (x_1 + d)))\\
    & + & \frac{x_2}{(r_1 r_{- 1})^4} x_2^4 (x_1 + d - (x_1 - d)) .
  \end{eqnarray*}
  We simplify this equation to
  \begin{equation}
    - \frac{\cos (\theta_1) \sin (\theta_1)}{r_1^2} + \frac{\cos (\theta_{-
    1}) \sin (\theta_{- 1})}{r_{- 1}^2} = \frac{- x_2 (x_1 - d) (x_1 +
    d)}{(r_1 r_{- 1})^4} (2 d^3 + 6 x_1^2 d - 4 x_2^2 d) + \frac{2 x_2^5
    d}{(r_1 r_{- 1})^4} . \label{382901}
  \end{equation}
  We now estimate separately each contribution of (\ref{382901}). We have, in
  the right half-plane, where $r_1 \leqslant r_{- 1}$ and $r_{- 1} \geqslant d
  \geqslant \frac{K}{c}$,
  \[ \left| c \tilde{r}^{2 + \sigma} \frac{2 x_2^5 d}{(r_1 r_{- 1})^4} \right|
     = 2 \left| c d \frac{x_2^5}{r_1^2 r_{- 1}^3} \frac{\tilde{r}^{2 +
     \sigma}}{r_1^2 r_{- 1}^{\sigma}} \frac{1}{r_{- 1}^{1 - \sigma}} \right|
     \leqslant K c^{1 - \sigma} \]
  since $| x_2 | \leqslant r_1, | x_2 | \leqslant r_{- 1}$ and
  $\frac{\tilde{r}^{2 + \sigma}}{r_1^2 r_{- 1}^{\sigma}} \leqslant 1$. Still
  in the right half-plane,
  \[ \left| c \tilde{r}^{2 + \sigma} \frac{x_2 (x_1 - d) (x_1 + d)}{(r_1 r_{-
     1})^4} 2 d^3 \right| = 2 \left| c d \frac{d^2}{r_{- 1}^2} \frac{(x_1 -
     d)}{r_1} \frac{(x_1 + d)}{r_{- 1}} \frac{x_2}{r_1} \frac{\tilde{r}^{2 +
     \sigma}}{r_1^2 r_{- 1}^{\sigma}} \frac{1}{r_{- 1}^{1 - \sigma}} \right|
     \leqslant K c^{1 - \sigma} \]
  since $d \leqslant K r_{- 1}$, $| x_1 - d | \leqslant r_1$ and $| x_1 + d |
  \leqslant r_{- 1}$. For the next term, we write $x_1^2 = x_1^2 - d^2 + d^2$
  in
  \[ \frac{x_2 (x_1 - d) (x_1 + d)}{(r_1 r_{- 1})^4} 6 x_1^2 d = \frac{x_2
     (x_1 - d) (x_1 + d)}{(r_1 r_{- 1})^4} 6 (x_1^2 - d^2) d + \frac{x_2 (x_1
     - d) (x_1 + d)}{(r_1 r_{- 1})^4} 6 d^3 . \]
  In the right half-plane, using $x_1^2 - d^2 = (x_1 - d) (x_1 + d)$,
  \[ \left| c \tilde{r}^{2 + \sigma} \frac{x_2 (x_1 - d) (x_1 + d)}{(r_1 r_{-
     1})^4} 6 (x_1^2 - d^2) d \right| = 6 \left| c d \frac{(x_1 - d)^2}{r_1^2}
     \frac{(x_1 + d)^2}{r_{- 1}^2} \frac{x_2}{r_{- 1}} \frac{\tilde{r}^{2 +
     \sigma}}{r_1^2 r_{- 1}^{\sigma}} \frac{1}{r_{- 1}^{1 - \sigma}} \right|
     \leqslant K c^{1 - \sigma} \]
  using previous estimates. We continue in the right half-plane with
  \[ \left| c \tilde{r}^{2 + \sigma} \frac{x_2 (x_1 - d) (x_1 + d)}{(r_1 r_{-
     1})^4} 6 d^3 \right| = 6 \left| c d \frac{(x_1 - d)^{}}{r_1} \frac{(x_1 +
     d)^{}}{r_{- 1}} \frac{d^2}{r_{- 1}^2} \frac{x_2}{r_1} \frac{\tilde{r}^{2
     + \sigma}}{r_1^2 r_{- 1}^{\sigma}} \frac{1}{r_{- 1}^{1 - \sigma}} \right|
     \leqslant K c^{1 - \sigma} \]
  and
  \[ \left| c \tilde{r}^{2 + \sigma} \frac{x_2 (x_1 - d) (x_1 + d)}{(r_1 r_{-
     1})^4} 4 x_2^2 d \right| = 4 \left| c d \frac{(x_1 - d)^{}}{r_1}
     \frac{(x_1 + d)^{}}{r_{- 1}} \frac{x_2^3}{r_1 r_{- 1}^2}
     \frac{\tilde{r}^{2 + \sigma}}{r_1^2 r_{- 1}^{\sigma}} \frac{1}{r_{- 1}^{1
     - \sigma}} \right| \leqslant K c^{1 - \sigma} \]
  using previous estimates. Similarly, all these estimates hold in the left
  half-plane, and for $\nabla \left( \frac{\partial_d (i c \partial_{x_2}
  V)}{V} \right)$, which ends the proof of
  \[ \left\| \frac{\partial_d (i c \partial_{x_2} V)}{V} \right\|_{\ast \ast,
     \sigma, d_c} \leqslant K c^{1 - \sigma} . \]
\end{proof}

\subsection{Proof of Lemma \ref{nl310}}\label{AC2}

\begin{proof}
  From the proof (and with the notations) of Lemma \ref{rest31},
  \[ \begin{array}{lll}
       &  & \left( \tmop{Id} + \left( \eta L (.) + (1 - \eta) V L' \left(
       \frac{.}{V} \right) \right)_c^{- 1} (\Pi_d^{\bot} (d_{\Psi} F_c (. /
       V))) \right) ((\Phi_{c + \varepsilon, d} \nobracket - \nobracket
       \Phi_{c, d}))\\
       & = & \left( \eta L (.) + (1 - \eta) V L' \left( \frac{.}{V} \right)
       \right)_c^{- 1} \left( - \varepsilon \Pi_d^{\bot} (- i \partial_{x_2}
       V) - i \varepsilon \left( \eta \partial_{x_2} \Phi_{c + \varepsilon, d}
       + (1 - \eta) V \partial_{x_2} \left( \frac{\Phi_{c + \varepsilon,
       d}}{V} \right) \right) \right)\\
       & + & \left( \eta L (.) + (1 - \eta) V L' \left( \frac{.}{V} \right)
       \right)_c^{- 1} (O^{\sigma, c}_{\| . \|_{\ast \ast, \sigma, d}}
       (\varepsilon^2)),
     \end{array} \]
  thus, taking $\varepsilon \rightarrow 0$, we deduce that (with Lemma
  \ref{rest31})
  \begin{eqnarray*}
    &  & \left( \tmop{Id} + \left( \eta L (.) + (1 - \eta) V L' \left(
    \frac{.}{V} \right) \right)^{- 1} (\Pi_d^{\bot} (d_{\Psi} F_c (. / V)))
    \right) (\partial_c \Phi_{c, d})\\
    & = & \left( \eta L (.) + (1 - \eta) V L' \left( \frac{.}{V} \right)
    \right)^{- 1} \left( \Pi_d^{\bot} (\partial_c F (\Phi_{c, d} / V)) - i
    \eta \partial_{x_2} \Phi_{c, d} + (1 - \eta) V \partial_{x_2} \left(
    \frac{\Phi_{c, d}}{V} \right) \right) .
  \end{eqnarray*}
  Since at $d = d_c$, $\lambda (c, d_c) = 0$, we have
  \[ \Pi_d^{\bot} (\partial_c F (\Phi_{c, d} / V)) - i \eta \partial_{x_2}
     \Phi_{c, d} + (1 - \eta) V \partial_{x_2} \left( \frac{\Phi_{c, d}}{V}
     \right)_{| d = d_c \nobracket} = - i \partial_{x_2} Q_c, \]
  hence, with Proposition \ref{invertop},
  \begin{eqnarray*}
    \| \partial_c \Psi_{c, d | d = d_c \nobracket} \|_{\ast, \sigma, d} &
    \leqslant & K \| \partial_c \Psi_{c, d | d = d_c \nobracket}
    \|_{\circledast, \sigma, d_{\circledast}}\\
    & \leqslant & K (\sigma, \sigma') \left\| \frac{i \partial_{x_2} Q_c}{V}
    \right\|_{\circledast \circledast, \frac{\sigma + \sigma'}{2},
    d_{\circledast}} .
  \end{eqnarray*}
  We will conclude by showing that for any $0 < \sigma < \sigma' < 1$,
  \[ \left\| \frac{i \partial_{x_2} Q_c}{V} \right\|_{\circledast \circledast,
     \sigma, d_{\circledast}} \leqslant K (\sigma, \sigma') c^{- \sigma'} ; \]
  which, applied to $0 < \frac{\sigma + \sigma'}{2} < \sigma' < 1$, concludes
  the proof.
  
  \
  
  By Lemma \ref{fineestimate}, we have
  \[ \left\| \frac{i \partial_{x_2} V}{V} \right\|_{\circledast \circledast,
     \sigma, d_{\circledast}} \leqslant K (\sigma) c^{- \sigma}, \]
  and using $\| \Psi_{c, d_c} \|_{\ast, \sigma, d_c} \leqslant K (\sigma,
  \sigma') c^{1 - \sigma'}$ with Lemma \ref{nonmodV}, we check easily that,
  for $c$ small enough,
  \[ \left\| \frac{i \partial_{x_2} Q_c}{V} \right\|_{\circledast \circledast,
     \sigma, d_{\circledast}} \leqslant K (\sigma, \sigma') c^{- \sigma'} . \]

  We now focus on the estimation of $\partial_d \Phi_{c, d | d = d_c
  \nobracket}$. At the end of step 1 of the proof of Lemma \ref{dderpsy}, we
  have shown that
  \[ \partial_d \Phi_{c, d | d = d_c \nobracket} = - d_{\Phi} H^{- 1}
     (\partial_d H (\Phi_{c, d_c}, c, d_c)) . \]
  From Lemma \ref{lemme20}, we have that, at $d = d_c, \Phi = \Phi_{c, d_c}$,
  the operator $d_{\Phi} H^{- 1}$ is invertible from
  $\mathcal{E}_{\circledast, \sigma, d_{\circledast}}$ to
  $\mathcal{E}_{\circledast, \sigma, d_{\circledast}}$, with an operator norm
  with size $1 + o_{c \rightarrow 0}^{\sigma} (1)$. We therefore only have to
  check that
  \[ \| \partial_d H (\Phi_{c, d_c}, c, d_c) \|_{\ast, \sigma, d_c} \leqslant
     K (\sigma, \sigma') c^{1 - \sigma'} . \]
  Since $\partial_d H (\Phi_{c, d_c}, c, d_c) = \left( \eta L (.) + (1 - \eta)
  V L' \left( \frac{.}{V} \right) \right)^{- 1} (G (d_c, \Phi_{c, d_c}))$, By
  Proposition \ref{invertop} (from $\mathcal{E}_{\circledast \circledast,
  \sigma', d_{\circledast}}$ to $\mathcal{E}_{\circledast, \sigma,
  d_{\circledast}}$), it will be a consequence of
  \[ \left\| \frac{G (d_c, \Phi_{c, d_c})}{V} \right\|_{\ast \ast, \sigma, d}
     \leqslant K (\sigma, \sigma') c^{1 - \sigma'} \]
  for any $0 < \sigma < \sigma' < 1$.
  
  We have, since $\mathbbm{H}_{d_c} (\Phi_{c, d_c}) = \Phi_{c, d_c}$, that
  \[ \begin{array}{lll}
       \frac{G (d_c, \Phi_{c, d_c})}{V} & = & \partial_d (| V |^2)
       \frac{\Phi_{c, d_c}}{V} + 2\mathfrak{R}\mathfrak{e}
       (\overline{\partial_d V} \Phi_{c, d_c}) + 2\mathfrak{R}\mathfrak{e}
       (\bar{V} \Phi_{c, d_c}) \frac{\partial_d V}{V}\\
       & + & \partial_d ((1 - \eta) (E - i c \partial_{x_2} V))_{| d = d_c
       \nobracket} \frac{\Phi_{c, d_c}}{V} - \frac{1}{V} \partial_d
       (\Pi_d^{\bot} (F_d (\Phi / V)))_{| d = d_c \nobracket} .
     \end{array} \]
  Since $\partial_d (| V |^2) = 2\mathfrak{R}\mathfrak{e} (\partial_d V
  \bar{V})$, we check, with Lemma \ref{ddVest} that
  \[ \left| \partial_d (| V |^2) \frac{\Phi_{c, d_c}}{V} \right| + \left|
     2\mathfrak{R}\mathfrak{e} (\bar{V} \Phi_{c, d_c}) \frac{\partial_d V}{V}
     \right| \leqslant \frac{K (\sigma, \sigma') c^{1 - \sigma'}}{(1 +
     \tilde{r})^{2 + \sigma}}, \]
  and
  \[ | \mathfrak{R}\mathfrak{e} (\overline{\partial_d V} \Phi_{c, d_c}) |
     \leqslant \frac{K (\sigma, \sigma') c^{1 - \sigma'}}{(1 + \tilde{r})^{1 +
     \sigma}}, \]
  as well as
  \[ \left| \nabla \left( \partial_d (| V |^2) \frac{\Phi_{c, d_c}}{V} +
     2\mathfrak{R}\mathfrak{e} (\bar{V} \Phi_{c, d_c}) \frac{\partial_d V}{V}
     +\mathfrak{R}\mathfrak{e} (\overline{\partial_d V} \Phi_{c, d_c}) \right)
     \right| \leqslant \frac{K (\sigma, \sigma') c^{1 - \sigma'}}{(1 +
     \tilde{r})^{2 + \sigma}}, \]
  and this estimate a real valued quantity. From step 2 of the proof of Lemma
  \ref{dderpsy}, we have
  \[ \left\| \frac{1}{V} \partial_d ((1 - \eta) (E - i c \partial_{x_2} V))
     \right\|_{\ast \ast, \sigma, d} \leqslant K (\sigma) c^{1 - \sigma}, \]
  which is enough to show that
  \[ \left\| \partial_d ((1 - \eta) (E - i c \partial_{x_2} V))_{| d = d_c
     \nobracket} \frac{\Phi_{c, d_c}}{V} \right\|_{\ast \ast, \sigma, d}
     \leqslant K (\sigma, \sigma') c^{1 - \sigma'} . \]
  Finally, in step 2 of the proof of Lemma \ref{dderpsy}, we have shown that
  (taking the estimate for $\Phi = \Phi_{c, d_c}$)
  \[ \left\| \frac{1}{V} \partial_d (\Pi_d^{\bot} (F_d (\Phi / V)))_{| d = d_c
     \nobracket} \right\|_{\ast \ast, \sigma, d} \leqslant K (\sigma, \sigma')
     c^{1 - \sigma'}, \]
  which conclude the proof of this lemma.
\end{proof}

\end{document}